\documentclass[reqno, 11pt]{article}

\usepackage[utf8]{inputenc}	
\usepackage[T1]{fontenc}
\usepackage[english]{babel} 
\usepackage{amsmath}
\usepackage{comment,ulem}
\usepackage{mathtools}
\usepackage{mathrsfs}
\usepackage{amssymb,amsfonts}
\usepackage{amsthm}
\usepackage{array}
\usepackage{xcolor}
\usepackage{multirow}
\usepackage{todonotes}
\usepackage{lineno}
\usepackage[pdftex,
colorlinks=true, 
linkcolor=blue, 
citecolor=blue, 
pagebackref=false, 
]{hyperref}

\usepackage{dsfont}
\usepackage{transparent}

\usepackage{enumitem}

\usepackage{graphicx}
\usepackage{subcaption}

\usepackage{dsfont}

\usepackage{hyperref}
\usepackage{vmargin}
\setmarginsrb{3.5cm}{2cm}{3.5cm}{2cm}{0cm}{1cm}{0cm}{1.5cm}

\allowdisplaybreaks[3]

 {\everymath{\displaystyle\everymath{}}\array}%
 {\endarray}

\newcommand\addtxt[1]{\textcolor{brown}{#1}}

\theoremstyle{plain}
\newtheorem{theorem}{Theorem}
\newtheorem{proposition}[theorem]{Proposition}
\newtheorem{lemma}[theorem]{Lemma}
\newtheorem{corollary}[theorem]{Corollary}
\newtheorem{definition}[theorem]{Definition}

\theoremstyle{remark}
\newtheorem{example}[theorem]{Example}
\newtheorem{remark}[theorem]{Remark}

\newenvironment{frcseries}{\fontfamily{frc}\selectfont}{}

\def\<{{\langle}}
\def\>{{\rangle}}

\def\1Lip{1\text{-Lip}}

\def\dd{\mathrm{d}}

\def\Tr{\mathrm{Tr}}

\DeclareMathOperator*{\argmax}{arg\,max}
\DeclareMathOperator*{\argmin}{arg\,min}

\DeclareMathOperator*{\esssup}{ess\,sup}

\title{Robust mean field control: stochastic maximum principle and variational mean field games}

\author{
François Delarue 
\thanks{
Université Côte d'Azur, CNRS, Laboratoire J.A. Dieudonné,   
06108 Nice, France;
Emails: \texttt{francois.delarue@univ-cotedazur.fr}, 
\texttt{pierre.lavigne@univ-cotedazur.fr};
F. Delarue and P. Lavigne  acknowledge the financial support of the European Research Council (ERC) under the European Union's Horizon Europe research and innovation program (ELISA project, Grant agreement No. 101054746). Views and opinions expressed are however
those of the author(s) only and do not necessarily reflect those of the European Union or the
European Research Council Executive Agency. Neither the European Union nor the granting
authority can be held responsible for them.
}
\and
Pierre Lavigne \footnotemark[1]
}

\date{\vspace{-2em}}

\begin{document}
\maketitle

\begin{abstract}
We introduce a class of robust control problems formulated in min--max form, in which the principal agent is viewed as a central planner facing Nature. The agent's cost is a nonlinear function of all its possible realizations, encompassing in particular the mean field regime where the cost depends on the distribution of the states.
 In parallel, Nature favors the occurrence of outcomes that are least favorable to the agent, at an entropic cost. We establish existence and uniqueness of solutions under appropriate assumptions, including suitable convexity--concavity conditions, and derive a related stochastic maximum principle. 
We further address a corresponding class of robust  variational mean field games in which the interaction term is subject to ambiguity, and prove existence and uniqueness of solutions.
\end{abstract}

{\small Keywords: Robust mean field control, Stochastic maximum principle, Risk-averse control, Quadratic backward stochastic differential equation, Entropic penalties.
\vskip 4pt

}
\tableofcontents

\section{Introduction}

In this work we introduce a zero-sum non-local stochastic game in finite horizon between two players. 
Throughout, the first player is referred to as `Nature' and the second one to 
as `the central planner'.

\paragraph{Formulation of the problem.}
The problem is defined on a finite interval $[0,T]$ 
and a probability space $(\Omega,{\mathcal F},{\mathbb P})$
equipped with a $d$-dimensional Brownian motion 
$W=(W_t)_{t \in [0,T]}$ and an independent $n$-dimensional random variable 
$\eta$ representing the initial condition of the central planner. 
Here, $n\in\mathbb{N}^\star$ is the state dimension of the central planner and $d\in\mathbb{N}^\star$ the noise dimension
to which the central planner is subjected. 
The ${\mathbb P}$-complete filtration generated by 
$(\eta,W)$ is denoted by
$\mathbb F = (\mathcal F_t)_{0 \le t \leq T}$. On this probabilistic set-up, we consider the
following inf-sup non-local (in the sense that the cost ${\mathcal G}$ below takes the entire random variables $q_T$ and $X_T^\psi$, and not only their realizations, as inputs) stochastic control problem:
\begin{equation} \label{pb:min-max-G} \tag{P}
    \sup_{q \in \mathcal{Q}}  \inf_{\psi \in \mathcal{A}}  \mathcal{J}(q,\psi), \quad \mathcal{J}(q,\psi) \coloneqq \mathcal{R}(q,\psi) - \mathcal{S}(q),
\end{equation}
where 
\begin{align}
    \mathcal{R}(q,\psi)  & \coloneqq  \mathcal{G}\bigl(q_T,X^{\psi}_T\bigr) +  
    \mathbb{E}\left[ \int_0^T
    q_s \ell(s,\psi_s)  \dd s  \right], \label{def:F}\\
    \mathcal{S}(q) & \coloneqq \mathbb{E}\left[ \int_0^T q_s  f^{\star}(s,Y^\star_s,Z^\star_s) \dd s  \right]. \label{def:alpha}
\end{align}  
In this formulation, equilibria are sought over open loop controls.
Nature optimizes with respect to $q \in \mathcal{Q}$ and the central planner with respect to $\psi \in \mathcal{A}$, where the admissible sets ${\mathcal Q}$ and ${\mathcal A}$ can be roughly described as follows:
\begin{itemize}
    \item The set $\mathcal{Q}$ is a class of ${\mathbb F}$-progressively measurable, positive-valued processes with finite entropy ${\mathcal S}$, accounting for changes in the historical measure ${\mathbb P}$ under uncertainty from Nature (here and throughout, `positive' is understood in the sense of strictly positive). Precisely, a process $q$ belongs to ${\mathcal Q}$ if
     \begin{align}
        &\mathcal{S}(q) < + \infty,
     \label{eq:h:entropy}
\\
  {\rm and} \quad &q_t = 1 + \int_0^t q_s Y_s^\star \dd s + \int_0^t q_s Z_s^\star \cdot \dd W_s,
\quad t \in [0,T],
    \label{eq:q} 
  \end{align}
    where $Y^\star=(Y_t^\star)_{t \in [0,T]}$
    and $Z^\star=(Z_t^\star)_{t \in [0,T]}$
    are two ${\mathbb F}$-progressively measurable processes
    with 
    values in ${\mathbb R}$ and ${\mathbb R}^d$, 
    respectively. The process $q$ admits an explicit expression in terms of $Y^\star$ and $Z^\star$:
\begin{equation}
\label{eq:q:explicit:factorization}
    q_t = e^{\int_0^t Y^\star_s\dd s } \mathcal{E}_t\biggl(\int_0^\cdot Z^\star_s \cdot \dd W_s\biggr),
\end{equation}
where $(\mathcal{E}_t(\int_0^\cdot Z^\star_s \cdot \dd W_s))_{t \in [0,T]}$ is the stochastic exponential associated to $Z^\star$. From now on, we denote $q_T  \mathbb{P}$ the equivalent (non-normalized) measure defined as $\int_A q_T \dd \mathbb{P}$ for all $A \in \mathcal{F}$.
When $|Y^\star|=0$, $q = (q_t)_{t \in [0,T]}$ is a Doléans-Dade exponential and defines 
an equivalent probability measure $q_T  {\mathbb P}$. When 
$|Y^\star| >0$, $q$ defines a collection of equivalent
non-normalized measures $( q_t   {\mathbb P})_{t \in [0,T]}$, which we refer to as ‘discounted measures'. We refer the reader to Appendix \ref{appendix:representation} for more details about the representation of $q$.

    \item The set $\mathcal{A}$ consists in a class of  
    ${\mathbb F}$-progressively-measurable,
    ${\mathbb R}^n$-valued processes 
    $ \psi=(\psi_t)_{t \in [0,T]}$ such that
    \begin{equation}
    \label{eq:expo:bound:psi}
    \mathcal{S}^\star(\psi) < +\infty, \quad 
       \mathcal{S}^\star(\psi) \coloneq \sup_{q \in \mathcal{Q}} \left\{\mathbb{E}\left[ \int_0^T q_s |\psi_s|^2 \dd s \right] - \gamma \mathcal{S}(q)\right\}.
    \end{equation}
    The coefficient $\gamma$ has to be fixed carefully and will be clearly defined in the Assumption \ref{hyp:gamma} below, but we already mention that it should depend on the other data of the problem.
For a given control $\psi \in \mathcal{A}$, the state $X^\psi=(X_t^\psi)_{t \in [0,T]}$ of the central planner is the solution to
\begin{equation}
\label{eq:intro:X}
    \dd X_t = b(t,X_t,\psi_t) \dd t + \sigma(t,\psi_t) \dd W_t, \quad X_0 = \eta,
\end{equation}
where the drift $b \colon \Omega \times [0,T] \times \mathbb{R}^n \times \mathbb{R}^n \to \mathbb{R}^n$ and the volatility $\sigma \colon \Omega \times [0,T] \times \mathbb{R}^n \to \mathbb{R}^{n \times d}$ are possibly random.
Implicitly, 
$b$ and $\sigma$ are required to be ${\mathbb F}$-progressively measurable. The precise assumptions on the two of them will be clarified later in the article; see Subsection \ref{subsec:main-result}. In particular the state equation 
\eqref{eq:intro:X}
will be assumed to be linear, but we keep it under general form for the exposition.

\end{itemize}

Returning to \eqref{def:F} and \eqref{def:alpha}, $\ell \colon \Omega \times [0,T] \times \mathbb{R}^n \to \mathbb{R}$ is referred to as the running cost. The coefficient $f^\star \colon \Omega \times [0,T] \times \mathbb{R} \times \mathbb{R}^d \to \mathbb{R}$ is called the ‘convex dual' driver (for reasons explained below). This function $f^\star$ is typically viewed as (a perturbation of) the square of its last argument.
The function ${\mathcal G}$ represents the terminal cost. In its most general form, it is defined as a (measurable) real-valued mapping on $\Omega \times L^1(\Omega,{\mathcal F}_T,{\mathbb P};{\mathbb R}_+) \times L^2(\Omega, {\mathcal F}_T, {\mathbb P}; {\mathbb{R}}^d)$. This formulation encompasses mean field functions with arguments such as $ \mathbb{Q} \circ (X_T)^{-1}$ with $\mathbb{Q} = \mathcal{E}_T(\int_0^{\cdot}Z^\star_s \cdot \dd W_s)\mathbb{P}$; this example motivates the term  central planner for the player optimizing over $\psi$.
The problem is thus called non-local, since the functional $\mathcal{G}$ requires the full information on the terminal random variables $(q_T,X^\psi_T)$ to be evaluated.
In principle, we could incorporate a running cost of a similar structure in \eqref{def:F}, but for the sake of simplicity and clarity, we will omit this term from the remainder of the article.

The cost functions can be interpreted as follows: when the central planner chooses a strategy $\psi$, Nature tries to adjust the historical probability ${\mathbb P}$ by weighting it with $q$ in the worst possible way for the planner, thus maximizing the cost ${\mathcal R}(q,\psi)$.
Conversely, once the weighting $q$ is chosen, the planner aims to select the best strategy $\psi$ to minimize ${\mathcal R}(q,\psi)$. This is an ‘almost classic' stochastic control problem, depending on the form of the terminal cost ${\mathcal G}$. When ${\mathcal G}(q_T, X_T^\psi)$ is written as an expectation ${\mathbb E}[q_T g(X_T)]$, the planner solves a standard problem under the discounted measure $q  {\mathbb P}$. When ${\mathcal G}(q_T, X_T^\psi)$ takes the form $G((q_T  \mathbb{P}) \circ (X_T^\psi)^{-1})$, with $G$ being a cost function defined on the space ${\mathcal M}_+({\mathbb R}^n)$ of positive measures on ${\mathbb R}^n$, the planner solves a mean field control problem under the measures $q  {\mathbb P}$. In both cases, the running cost $\ell$ can be chosen to be quadratic or to grow quadratically in $\psi$.

\paragraph{A preview: risk averse control problem and BSDEs.} 
To better understand the problem \eqref{pb:min-max-G}, we focus in this paragraph on the first of the two cases above, namely, we assume that there exists a function $g \colon \mathbb{R}^n \to \mathbb{R}$ such that
\begin{equation*}
    \mathcal{G}\bigl(q_T,X_T^\psi\bigr) \coloneqq \mathbb{E}\left[q_T g(X_T^\psi)\right]= \mathbb{E}^{\mathbb{Q}}\left[ e^{\int_0^T Y^\star_s\dd s }  g(X_T^\psi) \right],
\end{equation*}
where $\mathbb{Q}$ is the equivalent probability measure defined by $ \mathbb{Q} = \mathcal{E}_T(\int_0^\cdot Z_s^\star
\cdot \dd W_s)  \mathbb{P}$. Here the term $Y^\star$ can be understood as an actualization rate, which might be negative.
In this framework, the problem
\eqref{pb:min-max-G} becomes
\begin{equation} \label{pb:min-max-J} \tag{P\textsubscript{L}}
    \sup_{q \in \mathcal{Q}}  \inf_{\psi \in \mathcal{A}}  J(q,\psi), \quad J(q,\psi) \coloneqq R(q,\psi) - \mathcal{S}(q),
\end{equation}
where 
\begin{align*}
    R(q,\psi)  & \coloneqq \mathbb{E}^{\mathbb{Q}}\left[ e^{\int_0^T Y^\star_s\dd s }  g(X_T^\psi) + \int_0^T  e^{\int_0^s Y^\star_u\dd u } \ell(s,\psi_s) \dd s \right].
\end{align*}
When $\psi \in {\mathcal A}$ is fixed, the penalty ${\mathcal S}(q)$ prevents Nature from choosing a singular measure (relative to the historical probability ${\mathbb P}$) that would only assign weight to the worst outcome for the central planner. In fact, the problem solved by Nature coincides with the risk-aversion problem presented in \cite[Chapter 6.4]{pham2009continuous}, with the key difference being that the variable $Z^\star$ is bounded in \cite{pham2009continuous}, which greatly simplifies the analysis. In particular, \cite{pham2009continuous} provides a representation of the value of the problem (corresponding here to the problem solved by Nature) in the form of a Backward Stochastic Differential Equation (BSDE) driven by coefficients with at most linear growth. In our framework, this BSDE may become quadratic, as explained in the next paragraph.

To further fix the ideas about the ‘linear' problem \eqref{pb:min-max-J}, assume
that $f^\star(t,y^\star,z^\star) = \frac{1}{2}|z^\star|^2$, for all $(t,y^\star,z^\star) \in  [0,T] \times \mathbb{R} \times \mathbb{R}^d$, and $Y^\star \equiv 0$. Because the actualization rate $Y^\star $ is null, $q_T$ is the Radon-Nikodym derivative of $\mathbb{Q}$ with respect to $\mathbb{P}$, that is $\dd {\mathbb Q} = q_T \dd {\mathbb P}$. Due to the specific form of $f^\star$, the penalty ${\mathcal S}(q)$ is equal to 
$\mathrm{H}({\mathbb Q} \vert {\mathbb P})$ where 
$\mathrm{H}(\mathbb{Q}\vert \mathbb{P}) \coloneqq{\mathbb E}^{\mathbb Q}[\ln(\dd {\mathbb Q}/\dd {\mathbb P})]$ denotes the relative entropy of $\mathbb{Q}$ with respect to $\mathbb{P}$. Then, the cost simplifies to 
\begin{equation} \label{def:J-psi-q}
    J(q,\psi) =  \mathbb{E}^{\mathbb{Q}}\left[ g(X_T^\psi) + \int_0^T \ell(s,\psi_s) \dd s \right] -  \mathrm{H}(\mathbb{Q}\vert \mathbb{P}).
\end{equation}
Nature's problem then coincides with a maximization problem that frequently appears in large deviation theory. Indeed, the Donsker–Varadhan variational formula provides an interpretation of Nature's optimal value as the log-Laplace transform of a cost function defined on the Wiener space, as seen in works like \cite{budhiraja2019analysis, dupuis2011weak}. This problem has a long history in economic and finance literature \cite{chen2002ambiguity,hansen2001robust}, and can be found under different names (ambiguity, robust or risk sensitive control problem, depending on the interpretation) 
We also refer to the recent contribution \cite{bourdais2023entropy} for a systematic analysis of entropy-penalized stochastic optimal control problems.

When considering an optimizer $q \in \mathcal{Q}$ for Nature's problem in \eqref{pb:min-max-J}, 
the remaining central planner minimization problem over $\psi$ can be reformulated as a control problem
over BSDEs:
\begin{equation*}
    \inf_{\psi \in \mathcal{A}} \mathbb{E}\left[Y_0^\psi\right],
\end{equation*}
where $(Y,Z)$ is the solution to,
\begin{equation} \label{eq:quad-bsde}
    - \dd Y_t  =  \left(f(t,Y_t,Z_t) + \ell(t,\psi_t) \right) \dd t - Z_t \cdot \dd W_t, \quad Y_T = g(X_T^\psi),
\end{equation}
and $f$ is the Fenchel transform of $f^\star$ (see \eqref{eq:f:star}). 
This connection is presented in \cite[Chapter 6.4]{pham2009continuous} 
in the particular case of linear growth drivers $f$.
When $f$ is quadratic in the variable $z$,  as considered throughout the remainder of the article, solving the BSDE in equation \eqref{eq:quad-bsde} becomes more challenging.
The study of quadratic BSDEs began with the seminal work of \cite{kobylanski2000backward} on equations driven by bounded terminal conditions. For a comprehensive presentation of the standard theory, see \cite{zhang2017backward}, which includes additional references. Subsequent research has extended the results on existence and uniqueness to unbounded terminal conditions, under the assumption of finite exponential order moments \cite{briand2006bsde, briand2008quadratic, delbaen2011uniqueness}. We will return to these references in the core of the article, as our analysis is typically conducted in the context where the terminal value 
$g$ is unbounded.

The BSDE in equation \eqref{eq:quad-bsde} can be interpreted as a nonlinear conditional expectation, specifically a $g$-expectation \cite{peng1997backward}. When the criterion $J$ is given by \eqref{def:J-psi-q}, that is, when $f(s,y,z) = \frac{1}{2}|z|^2$, the first component $Y$ of the BSDE is known in the literature as the entropic risk measure of the cost $g(X_T^\psi) + \int_0^T \ell(\psi_t)  \dd t$. Entropic risk measures have been extensively studied in the $L^\infty$ case, i.e., for bounded costs, see \cite{barrieu2009pricing}.

\paragraph{First contribution: From risk neutral to robust  mean field control.}
The main objective of our paper is twofold: first, from a technical perspective, to relax the growth conditions of the various cost functionals in the problem \eqref{pb:min-max-J}; and second, from a modeling perspective, to consider a mean field version, whose general form is given in \eqref{pb:min-max-G}. In this regard, the problem \eqref{pb:min-max-G} encompasses not only mean field control problems with risk aversion but, more generally, problems in which the central planner is subject to uncertainty, here perceived as an adverse action of Nature.
A series of examples are provided in Subsections \ref{sec:example-application} and 
\ref{subsec:examples:mean field} to illustrate these concepts.

In the risk-neutral case, stochastic mean field control problems are typically introduced as the limiting behavior of optimal control problems defined over large interacting particle systems. In these settings, a central planner seeks to optimize an objective function that depends on the collective dynamics of the particles. This class of problems has attracted significant attention in recent years \cite{andersson2011maximum, bonnet2019pontryagin2, bonnet2019pontryagin1, buckdahn2011general, daudin2023optimal, daudin2024optimal, djete2022mckean, lacker2017limit, lauriere2014dynamic}. For a comprehensive introduction to the subject, we refer to \cite{bensoussan2013mean, carmona2018probabilistic-v1,carmonadelaruev2}.

In this article, we establish the stochastic maximum principle for the problem \eqref{pb:min-max-G}. The stochastic maximum principle is a powerful tool for solving stochastic control problems, first introduced by \cite{kushner1965stochastic} and further developed by \cite{bismut1976linear}, \cite{hu2017stochastic}, \cite{peng1990general}, and \cite{wu2013general}. It plays a central role in the stochastic mean field control and mean fied game literature \cite{buckdahn2016stochastic,carmona2018probabilistic-v1,carmonadelaruev2}.
The standard theory of the stochastic maximum principle applies to risk-neutral
control problems and is typically formulated in an $L^2$ framework, where both
the state variables and the adjoint processes are assumed to belong to $L^2$.
 To establish the stochastic maximum principle, three key steps are typically followed: first, proving the existence of a solution to the control problem \cite{haussmann1990existence}; second, deriving the necessary conditions \cite{peng1993backward}; and third, demonstrating the sufficient conditions, which can be shown using a simple verification argument.

Here, we move beyond the scope of the standard theory for two main reasons, which align with the two primary objectives of our work. The first is to address a mean field problem with a risk-averse min-max structure. Extensions of the stochastic maximum principle to risk-averse problems have been studied in the context of optimal control of Forward-Backward Stochastic Differential Equations (FBSDEs). For example, see \cite{oksendal2010maximum} for cases with linear growth drivers and jumps. The second objective is to allow the terminal condition $g$ to be unbounded. While bounded terminal conditions enable the use of the $\mathrm{BMO}$ theory for quadratic BSDEs \cite{hu2022global}, such assumptions are too restrictive for some applications. Moreover, they are rather incompatible with the convexity constraints typically required in the sufficient condition of the maximum principle.  One natural approach to obtain stronger exponential integrability properties on
$\psi$, compatible with those required in the theory of quadratic BSDEs, would be
to follow the methodology of \cite{cheridito2008dual,cheridito2009risk} and work
within an Orlicz space framework. Indeed, Orlicz spaces generalize $L^p$ spaces and,
in particular, include random variables with finite exponential moments of
arbitrary order, together with their dual space, which consists of random
variables with finite entropy $\mathrm H$. Such a dual space would be a natural
candidate for carrying the variable $q$.
That said, adopting this approach in our setting would require working with a
quadratic driver of the form $f(t,y,z)=\frac{1}{\gamma}|z|^2$. For $\gamma$ large
enough, this would provide the level of exponential integrability needed to apply
the theory of quadratic BSDEs. However, for small values of $\gamma$, to the best
of our knowledge, the stochastic maximum principle is not available even in this
simpler setting.
In contrast, our analysis goes one step further: the driver $f$ is only assumed to
have at most quadratic growth, and may in fact exhibit subquadratic growth. Our
strategy is to extend the duality inherent to Orlicz spaces of random variables to
a setting involving dual spaces of stochastic processes. This perspective
motivates the introduction of the mappings $\mathcal S$ and $\mathcal S^\star$,
which define the admissible sets $\mathcal Q$ and $\mathcal A$.

The first major contribution of this article is the proof of the stochastic maximum principle for the problem \eqref{pb:min-max-G}. Under appropriate concavity-convexity conditions, we show that this problem has a unique solution, where the minimizer is fully characterized by the solution of a FBSDE.  
To establish this result, we begin by considering a constrained version of
\eqref{pb:min-max-G}, for which we identify a topological structure ensuring
semi-continuity, convexity/concavity, and compactness of the criterion
$\mathcal J$ in each variable.
 This preliminary analysis allows us to apply Sion’s min--max
theorem and to deduce the existence of a saddle point for the constrained problem.
We then relax the constraints by showing that there exists a level at which they
are in fact nonbinding, which in turn yields the existence of a saddle point for
the original problem \eqref{pb:min-max-G}. The necessary and sufficient
optimality conditions are obtained by coupling the first-order conditions
associated with Nature’s problem and the central planner’s problem.
It is worth emphasizing that the necessary conditions provided by the stochastic
maximum principle require solving FBSDEs that go beyond the scope of the standard
theory. The sufficient
conditions ensure the uniqueness of these solutions. In the course of the
analysis, we revisit the connection between entropy-type optimization problems
and quadratic BSDEs with unbounded terminal conditions, a connection previously
established for linear functionals $\mathcal G$ in \cite{delbaen2011uniqueness}.

\paragraph{Second contribution: Robust mean field control and variational mean field games.}
Mean field control (MFC) problems constitute a class of stochastic optimal control
problems in which both the system dynamics and the associated cost functional.may 
depend on the distribution of the controlled state process. Such problems
naturally arise in the modeling of large populations of weakly interacting
particles, where the influence of each individual is mediated through the
empirical distribution of the population. Typical applications can be found in
economics, statistical physics, and mathematical finance.
In recent years, these problems have attracted significant attention; see, for
instance,
\cite{buckdahn2016stochastic,cardaliaguet2023algebraic,carmona2015forward,
djete2022mckean,lacker2017limit,pham2017dynamic}, among many others. When treated
from a probabilistic perspective, they are often addressed via the
stochastic maximum principle.

In this article, we introduce a \textit{robust} version of this problem, where the measure encoding the mean field interaction is biased by Nature. For a real-valued function $G$, defined on the space of non-negative measures on ${\mathbb R}^n$, we thus consider the min-max problem
\begin{equation} \tag{MFC} \label{pb:mckean-vlasov}
     \inf_{\psi \in \mathcal{A}} \sup_{q \in \mathcal{Q}} \left\{  G(\mathbb{Q}_{X}^q)  + \mathbb{E}\left[\int_0^T q_s \ell(s,\psi_s) \dd s \right] - \mathcal{S}(q)\right\},
\end{equation}
where $(X^\psi_t)_{t \in [0,T]}$ is the solution to the controlled stochastic differential equation \eqref{eq:intro:X} and 
\begin{equation*}
    \mathbb{Q}^q_X \coloneqq \mathbb{Q} \circ X^{-1}, \quad \mathbb{Q} =
    \exp\left(\int_0^T Y^\star_s   \dd  s \right) 
    \mathcal{E}_T\left(\int_0^\cdot Z^\star_s \cdot \dd W_s \right)  \mathbb{P}.
\end{equation*}
This problem is a specification of the problem \eqref{pb:min-max-G} when $\mathcal{G}(q,X) = G(\mathbb{Q}^q_X)$.
Building upon the stochastic maximum principle  established for \eqref{pb:min-max-G}, we derive the stochastic maximum principle for the problem \eqref{pb:mckean-vlasov} under the assumption that the mapping $G$ is Lions differentiable, Lions convex and flat concave. 

In addition, we also study a variational mean field game (MFG) problem. In contrast
to MFC problems, which are cooperative in essence, MFGs are
competitive problems. They are defined over a continuum of players whose
interactions arise through a mean field functional. The theory of MFGs was
introduced independently in \cite{huang2007large} and
\cite{LL06cr1,LL06cr2}, and has since been extensively developed; see, for
instance,
\cite{bertucci2019some,BHP-schauder,cardaliaguet2019master,
cardaliaguet2015second,carmona2018probabilistic-v1,carmonadelaruev2,carmona-delarue-lacker}.
MFGs have found numerous applications in economics and finance
\cite{achdou2023simple,cardaliaguet2018mean,feron2021price,mouzouni2019topic},
environmental studies \cite{kobeissi2024tragedy,lavigne2023decarbonization}, and
electricity markets \cite{alasseur2020extended}, to name just a few. We also refer
to \cite{carmona2018probabilistic-v1,carmonadelaruev2} for a comprehensive monograph.
The classical theory of MFGs typically considers risk-neutral agents. A natural
extension is therefore to investigate models with risk-averse agents. Several
approaches have been proposed in this direction, each relying on different ways
of incorporating risk aversion into the representative agent’s cost functional.
Risk-sensitive MFGs
\cite{moon2016linear,tembine2013risk} introduce criteria depending on the
variance of the state, while risk-averse MFGs
\cite{cheng2023risk,escribe2024mean,lavigne2020} incorporate risk measures directly
into the cost functional. MFGs in which agents optimize a worst-case
criterion, using $H^\infty$ control techniques, were introduced in
\cite{bauso2016robust}.
The theory of MFGs is closely related to that of MFC, particularly
through variational (or potential) MFGs, which form a special class of MFGs. In
brief, a variational MFG can be formulated as the first-order optimality
conditions of a stochastic MFC problem
\cite{benamou2019entropy,benamou2017variational,bonnans2021discrete,
briani2018stable,graber2018variational,graber2020weak}. In particular, any solution
to the MFC problem yields an equilibrium of the associated
variational game. Moreover, when the MFC problem is strictly convex
and coercive, the corresponding variational MFG admits a unique solution.

In this article, we study the following  MFG problem.
Given a non-negative measure $\mu$ on ${\mathbb R}^n$, representing the
mean field coupling, a representative agent (in the 
continuum) minimizes a risk-averse objective functional
\begin{equation*}
    \inf_{\psi \in \mathcal{A}} \sup_{q \in \mathcal{Q}} \mathcal{J}[\mu](q,\psi) \coloneq \mathbb{E}\left[q_T \frac{\delta G}{\delta \mu}(X^\psi_T, \mu_T) + \int_0^T q_s \ell(s,\psi_s)\dd s \right] - \mathcal{S}(q),
\end{equation*}
where the controlled state process  $(X_t^\psi)_{t \in [0,T]}$ satisfies the dynamics given in \eqref{eq:intro:X}, 
and
$\delta G / \delta \mu$ is 
the so-called flat derivative of $G$, see Section \ref{sec:mfg} for a reminder. 
For a saddle point $(\psi,q) \in \mathcal{A} \times \mathcal{Q}$, 
the mean field equilibrium condition is defined as follows: the measure $\mu$ is required to coincide with the law of the terminal state $X_T^\psi$ under the probability measure $\mathbb{Q}^q$ induced by Nature, that is,
\begin{equation} \tag{MFG-eq}\label{eq:equilibrium-condition}
    \mu = \mathbb{Q}^q_{X^\psi} \coloneqq \mathbb{Q}^q \circ (X^\psi_T)^{-1}, \quad \mathbb{Q}^q = \mathcal{E}\left(\int_0^\cdot Z^\star_s \cdot \dd W_s \right).
\end{equation}
In other words, the MFG problem consists in finding a triple $(q,\psi,\mu)$, with $ (q,\psi) \in \mathcal{Q} \times \mathcal{A}$ and $\mu$ being a non-negative measure, 
  such that
 \begin{equation} \tag{MFG} \label{pb:mfg}
     \mathcal{J}[\mu](q,\psi) = \inf_{\psi' \in \mathcal{A}} \sup_{q' \in \mathcal{Q}} \mathcal{J}[\mu](q',\psi'), \quad \mu = \mathbb{Q}^q \circ (X^\psi_T)^{-1}.
 \end{equation}
The optimization problem faced by the representative player can be interpreted as a risk-averse (non mean field) control problem. To find the optimal strategy, the representative agent solves a risk-averse stochastic control problem that falls within the scope of (non-mean field) control problems addressed in this work. 

Under the same regularity and concavity--convexity assumptions on $G$ as those used
in the analysis of the robust MFC \eqref{pb:mckean-vlasov}, we establish the
existence and uniqueness of an equilibrium, which ultimately coincides with the
solution of \eqref{pb:mckean-vlasov}. As such, this article is the first to
identify a variational structure for risk-averse MFGs. The analysis of such robust
MFGs is pursued further in our companion work \cite{DelarueLavigne2}, where we go
beyond the variational setting.

\paragraph{Organization of the article.}

The article is organized as follows. In Section~\ref{sec:notations}, we introduce
the main notations and definitions used throughout the paper. Section
\ref{sec:SMP-pb-P} contains our main result, Theorem~\ref{theorem:SMP}, which
establishes a stochastic maximum principle for the problem
\eqref{pb:min-max-G}, together with first examples of applications. Section
\ref{sec:mfg} is devoted to the mean field setting. There, we establish the
existence and uniqueness of solutions to a class of robust mean field control
problems in Corollary~\ref{corollary:mckean-vlasov}, and we then consider a related
class of robust variational mean field games, proving existence and uniqueness of equilibria in
Corollary~\ref{corollary:mfg}. Additional examples are provided in
Subsection~\ref{subsec:examples:mean field}. Finally, Section~\ref{sec:proof} is
dedicated to the proof of Theorem~\ref{theorem:SMP}.

\section{Notations} \label{sec:notations}

In this section, we introduce the main notations used in the article. 
Throughout, we work on the same filtered
complete
probability space $(\Omega,{\mathcal F},{\mathbb F},{\mathbb P})$
as in the definition of the problem 
\eqref{pb:min-max-G}.

\paragraph{Spaces of random variables and random processes.} 

We begin by introducing the spaces of variables and stochastic processes on which
our analysis relies. Unless otherwise stated, all notations are understood to be
with respect to the probability measure $\mathbb P$. When a different measure,
say $\mathbb Q$, is used, this will be made explicit. For example, in the context
of the first example below, we will write $L^p(\ldots, \mathbb Q)$ to indicate the
underlying measure.
Moreover, for each of the spaces defined below, we
will often omit the notation $\mathbb{R}^k$ when $k = 1$. 

\vskip 4pt

\noindent \textit{Usual random variable spaces.} 
For a given $k \in {\mathbb N}^*$ and for each $t \in [0,T]$,
we denote by $L^0(\mathcal{F}_t,\mathbb{R}^k)$ the set of $\mathbb{R}^k$ valued and $\mathcal{F}_t$-measurable random variables 
(r.v.'s in short). And then,  we define the sets

\begin{itemize}
\item $L^p(\mathcal{F}_t,\mathbb{R}^k)$  of r.v.'s
$X \in L^0(\mathcal{F}_t,\mathbb{R}^d)$ s.t. $ \| X \|_{L^p(\mathcal{F}_t,\mathbb{R}^k)} \coloneqq \mathbb E[|X|^p]<+\infty$, for $p< + \infty$,
\item $L^\infty(\mathcal{F}_t,\mathbb{R}^k)$  of r.v.'s $X \in L^0(\mathcal{F}_t,\mathbb{R}^k)$
s.t. $\| X \|_{L^\infty(\mathcal{F}_t,\mathbb{R}^k)} \coloneqq 
\underset{\omega \in \Omega}{\esssup} \underset{i \in \{1,\ldots,d\}}{\sup} |X^i(\omega)| < + \infty$.
\end{itemize}

\noindent \textit{Usual random process spaces.} 
We denote by $L^0(\mathbb F,\mathbb{R}^k)$ the space of $\mathbb F$-progressively measurable random processes (r.p.'s in short) with values in $\mathbb{R}^k$, and by $S^0({\mathbb F},{\mathbb R}^k)$ the subset of $L^0({\mathbb F},{\mathbb R}^k)$ comprising processes with continuous trajectories. 
We define the sets
\begin{itemize}
    \item $L^{p}(\mathbb F,\mathbb{R}^k)$ of r.p.'s $X \in L^0(\mathbb F,\mathbb{R}^k)$
    s.t. $\| X \|_{L^{p}(\mathbb F,\mathbb{R}^k)} \coloneqq \mathbb E\left[ \left(\displaystyle \int_0^T |X_t|^p \dd t \right)^{1/p} \right]<+\infty,$ for 
    $p<+\infty$,
    \item $M^{p}(\mathbb F,\mathbb{R}^k)$ of r.p.'s $X \in L^0(\mathbb F,\mathbb{R}^k)$
    s.t. $\| X \|_{M^{p}(\mathbb F,\mathbb{R}^k)} \coloneqq \mathbb E\left[ \left(\displaystyle \int_0^T |X_t|^2 \dd t \right)^{p/2} \right]<+\infty,$
    \item $L^\infty(\mathbb F,\mathbb{R}^d)$  of 
    r.p.'s
    $X \in L^0(\mathbb F,\mathbb{R}^k)$ s.t.  $\| X \|_{L^\infty(\mathbb F,\mathbb{R}^k)} \coloneqq \underset{t \in [0,T]}{\sup} \| X_t\|_{L^\infty(\mathcal{F}_t,\mathbb{R}^k)} < + \infty$,
     \item $S^p(\mathbb F,\mathbb{R}^k)$ of r.p.'s $X \in S^0(\mathbb F,\mathbb{R}^k)$
    s.t. $\|X\|_{S^p(\mathbb F,\mathbb{R}^k)}  \coloneqq  \mathbb E\left[
    \underset{t \in [0,T]}{\sup} |X_t|^p \right]<+\infty.$
    \item $D({\mathbb F},{\mathbb R}^k)$ of r.p.'s $X \in S^0({\mathbb F},{\mathbb R}^k)$ such that the family 
    $(\vert X_{\tau}\vert)_{\tau}$, with 
    $\tau$ running over the set of $[0,T]$-valued ${\mathbb F}$-stopping times, is uniformly integrable.
\end{itemize}
The class 
$D({\mathbb F},{\mathbb R}^k)$, which is the least standard among the above classes, was introduced in \cite[Definition 20]{dellacherie:meyer:b}.

Moreover, 
for a process $X=(X_t)_{t \in [0,T]}$, with continuous trajectories and with values in ${\mathbb R}^k$,
we denote by 
$(X_t^* \coloneqq \sup_{s \in [0,t]} \vert X_s \vert)_{t \in [0,T]}$ the running maximum of the norm of $X$.
\vskip 5pt

\noindent \textit{Orlicz} spaces. Following \eqref{eq:h:entropy}, we define the entropy function $h \colon \mathbb{R}_+ \to \mathbb{R}$:
\begin{equation}
\label{eq:entropy:definition:h}
    h(x) \coloneqq x(\ln(x) - 1),
\end{equation}
together with the two sets
 \begin{itemize}
    \item  $L \log L(\mathcal{F}_T)$ of 
    r.v.'s
    $X \in L^0(\mathcal F_T)$ s.t. $ \mathbb{E} \left[ h(X_T) \right] < + \infty$,
    \item  $L \log L(\mathbb{F})$ of r.p.'s $X \in L^0(\mathcal F_t)$
    s.t.
    $\underset{t \in [0,T]}{\sup} \mathbb{E} \left[ h(X_t) \right] < + \infty$.
\end{itemize}
For any $X \in L^0(\mathcal{F}_T)$, with non-negative values, we call  entropic risk measure of level $\vartheta >0$ of $X$
the quantity
\begin{equation}
\label{eq:rho:vartheta}
     \rho_{\vartheta}[X] \coloneqq \frac{1}{\vartheta} \ln \mathbb{E}\left[\exp\left(\vartheta X \right) \right],
\end{equation}
which makes it possible to define the sets
\begin{itemize}
    \item $L^{p,\vartheta}_{\exp}(\mathcal{F}_t,\mathbb{R}^k)$ of r.v.'s $X \in L^0(\mathcal{F}_t,\mathbb{R}^d)$
    s.t. $\rho_\vartheta[|X|^p] < + \infty,$
    \item $L^{p,\vartheta}_{\exp}(\mathbb{F},\mathbb{R}^k)$ of r.p.'s  $X \in L^0(\mathbb{F},\mathbb{R}^k)$
    s.t. $\rho_\vartheta\biggl[\displaystyle \int_0^T|X_s|^p\dd s\biggr] < + \infty,$
    \item $S^{p,\vartheta}_{\exp}(\mathbb{F},\mathbb{R}^k)$ of r.p.'s 
    $X \in L^0(\mathbb{F},\mathbb{R}^k)$ s.t. $\rho_\vartheta\bigl[|X^*_T|^p\bigr] < + \infty,$
\end{itemize}
for $k \in {\mathbb N}^*$, $p >0$ and $t \in [0,T]$.
When $X \in L^0({\mathcal F}_T)$ and $\vert X \vert \in L^{1,\vartheta}_{\rm \exp}({\mathcal F}_T)$, the right-hand side \eqref{eq:rho:vartheta} still makes sense
and we can define $\rho_{\vartheta}[X]$ accordingly. Moreover, 
for $p>0$,
we denote $L^{p}_{\exp}(\mathcal{F}_t,\mathbb{R}^k)$ the set of random variables $X \in L^0(\mathcal{F}_t,\mathbb{R}^k)$ such that $X \in L^{p,\vartheta}_{\exp}(\mathcal{F}_t,\mathbb{R}^k)$ for some 
$\vartheta>0$. The sets $L^{p}_{\exp}(\mathbb{F},\mathbb{R}^d)$ and $S^{p}_{\exp}(\mathbb{F},\mathbb{R}^d)$ are defined in an analogous way.

\paragraph{Spaces of measures.} For a metric space $(\mathcal{X},d)$, we call
${\mathcal B}({\mathcal X})$ its Borel $\sigma$-field,  
$\mathcal{P}(\mathcal{X})$ the set of probability measures on ${\mathcal X}$, and $\mathcal{M}(\mathcal{X})$ the set of finite non-negative measures on $\mathcal{X}$. Let $p \geq 1$ we define the sets
\begin{itemize}
    \item $\mathcal{P}_p(\mathcal{X})$ of $\mu \in \mathcal{P}(\mathcal{X})$ s.t. $\int_{\mathcal{X}} d(x_0,x)^p \dd \mu(x) < +\infty$ for some $x_0 \in {\mathcal X}$, 
    \item $\mathcal{M}_p(\mathcal{X})$ of $\mu \in \mathcal{M}(\mathcal{X})$ s.t. $\int_{\mathcal{X}} d(x_0,x)^p \dd \mu(x) < +\infty$ for some $x_0 \in {\mathcal X}$. 
\end{itemize}
For any finite measure $\mathbb{Q}$ on $\Omega$ and measurable mapping $X \colon \Omega \to \mathcal{X}$, we denote   $\mathbb{Q} \circ X^{-1}$, or 
${\mathbb Q}_X$, the image measure of ${\mathbb Q}$ by $X$.
When $f$ is a non-normalized non-negative measurable
function on $\Omega$, we denote by $f  \mathbb{P}$ the equivalent non-normalized measure $\mathbb{Q} :  \mathcal{F}
\ni A \mapsto 
    \mathbb{Q}(A) \coloneqq \int_{A} f \dd \mathbb{P}$.
    In particular, for $X$ as before, $(f {\mathbb P})_X$ stands for the image of $f {\mathbb P}$ by $X$. 

Lastly, for $\mu^1,\mu^2 \in \mathcal P(\mathcal X)$, we define the relative entropy
by
\[
\mathrm H(\mu^1 \vert \mu^2) \coloneqq
\int_{\mathcal X} \ln\!\left(\frac{\dd \mu^1}{\dd \mu^2}\right)\, \dd \mu^1,
\]
if $\mu^1$ is absolutely continuous with respect to $\mu^2$, and we set
$\mathrm H(\mu^1 \vert \mu^2)=+\infty$ otherwise. Additional material on the metric
structures of $\mathcal P(\mathcal X)$ and $\mathcal M(\mathcal X)$ is introduced
in Subsection~\ref{sec:mean-field-control}.

\paragraph{Duality.}
We end this section with duality results.
\medskip  

\noindent \textit{Fenchel transform.} The following duality is used repeatedly all along the article. By Fenchel duality, we have, 
for any $x^\star \in \mathbb{R}$ and $x \in \mathbb{R}_+$ (recalling the definition of 
$h$ in 
\eqref{eq:entropy:definition:h}),
\begin{equation} \label{ineq:entropiqueduality}
\exp(x^\star) + h(x) \geq  x^\star x.
\end{equation}
We often make use of the duality inequality 
\eqref{ineq:entropiqueduality}, when
reformulated in the form
\begin{align}
\label{eq:ineq:duality:with:r}
x^\star x =  \left(\vartheta x^\star\right) \frac{x}{\vartheta}
&\leq h\left(\frac{x}\vartheta \right) + \exp (\vartheta x^\star)
\\ \nonumber
&= \frac1{\vartheta} h(x) - \ln(\vartheta) x + \exp(\vartheta x^\star),
\end{align}
for all $\vartheta >0$ and any $x,x^\star >0$.
\medskip

\noindent \textit{Duality between $\mathcal{S}$ and $\mathcal{S}^\star$.}
By definition of $\mathcal{S}$ and $\mathcal{S}^\star$ in \eqref{def:alpha} and \eqref{eq:expo:bound:psi}, we have for any $\mathbb{F}$-progressively measurable processes $q$ and $\zeta$, valued in $\mathbb{R}$, 
\begin{equation} \label{ineq:S-S-star}
    \mathcal{S}(q) + \mathcal{S}^\star(\zeta) \geq \frac{1}{\gamma} \mathbb{E}\left[\int_0^T q_s |\zeta_s|^2 \dd s \right],
\end{equation}
where $\mathcal{S}(q)$ and $\mathcal{S}^\star(\zeta)$ might take infinite values. This inequality is a direct consequence of the definition of $\mathcal{S}^\star$. Equality holds whenever 
\begin{equation*}
    q \in \argmax_{q' \in \mathcal{Q}} \left\{\mathbb{E}\left[ \int_0^T q_s |\zeta_s|^2 \dd s \right] - \gamma \mathcal{S}(q)\right\}.
\end{equation*}

\medskip 
\noindent \textit{Dual Donsker-Varadhan variational formula.} 
Let $\mu \in \mathcal{P}(\mathbb{R}^n)$ and $\vartheta >0$ be such that $\int_{\mathbb{R}^n} \exp(\alpha \cdot x) \dd \mu(x) < + \infty$ for all $|\alpha| \leq \vartheta$. If there is a finite constant $L>0$ such that $|k(x)| \leq L(1+|x|)$ for any $x \in \mathbb{R}^n$ then 
\begin{equation} \label{eq:donsker-varadhan}
    - \ln \int_{\mathbb{R}^n} \exp\left(-k(x) \right) \dd \mu (x) = \inf_{m \in \mathcal{P}(\mathbb{R}^n) : H(m\vert \mu)<\infty} \left\{\mathrm{H}(m \vert \mu) + \int_{\mathbb{R}^n} k(x) \dd m(x) \right\}.
\end{equation}
This formula can be found in \cite[Proposition~2.3]{budhiraja2019analysis}. It
remains valid even when $\mu$ does not satisfy exponential integrability,
provided that $k$ is bounded from above, with no assumption on its growth from
below.

\paragraph{Miscellaneous.} 
Throughout the article, we use a generic constant \( C > 0 \) that depends only on the data of the problem. The value of \( C \) may change from line to line. 
As for the data of the problem themselves,
they are introduced and specified in the assumptions section.

When $x$ and $y$ are vectors of finite dimension, $x \cdot y $ denotes the scalar product between $x$ and $y$.
\section{Stochastic maximum principle} 

\label{sec:SMP-pb-P}
 
In this section, we establish the stochastic maximum principle for the (non–mean field) problem \eqref{pb:min-max-G}. The section is organized into two  main subsections.
The main result, Theorem \ref{theorem:SMP}, which presents the stochastic maximum principle for the problem \eqref{pb:min-max-G}, is stated in Subsection \ref{subsec:main-result}. Its proof relies on an application of Sion’s min–max theorem (recalled below), together with the stochastic maximum principles for both Nature’s problem and the central planner’s problem. These two problems are treated independently in Section \ref{sec:proof}.
Two application examples are discussed in Subsection \ref{sec:example-application}.

\begin{theorem}[Sion \cite{sion1958}] \label{thm:Sion}
    Let $M$ be a compact convex subset of a linear topological space and $N$ a convex subset of a linear topological space. Let $v \colon M \times N \to \mathbb{R}$ be such that
\begin{enumerate}
    \item $v(\cdot,y)$ is lower semi-continuous and convex on $M$ for each $y \in N$,
    \item $v(x,\cdot)$ is upper semi-continuous and concave on $N$ for each $x \in M$.
\end{enumerate}
Then we have
\begin{equation*}
    \min_{x \in M} \sup_{y \in N} v(x,y) = \sup_{y \in N} \min_{x \in M}  v(x,y)
\end{equation*}
and supremum is attained whenever $N$ is compact.
\end{theorem}

\subsection{Main result} \label{subsec:main-result}

We first present the assumptions used throughout the paper, even though some intermediate results are stated under weaker conditions. Additional assumptions are introduced in Section \ref{sec:mfg} when discussing the mean field setting.

\paragraph*{Assumptions.}
The assumptions are stated in terms of two constants, $L>0$ and $r \in \{0,1\}$.
They also make use of the notion of  progressively-measurable field: for a metric space $({\mathcal X},d)$
and an integer $k \in {\mathbb N}^*$, a random field 
$\mathcal{G} : \Omega \times [0,T] \times {\mathcal X} \rightarrow {\mathbb R}^k$ is said to be progressively-measurable if, for any $t \in [0,T]$,
its restriction
to $\Omega \times [0,t] \times {\mathcal X}$
is ${\mathcal F}_t \otimes {\mathcal B}([0,t]) \otimes {\mathcal B}({\mathcal X})/{\mathcal B}({\mathbb R}^k)$
measurable.

\begin{enumerate}[label*=A\arabic*]
    \item \label{hyp:b} \textit{Initial condition and drift.}
The initial condition 
$\eta$ in \eqref{eq:intro:X} belongs to $L^\infty({\mathcal F}_0,{\mathbb R}^n)$, i.e.
\begin{equation*} 
    \|\eta \|_{L^\infty(\mathcal{F}_0,\mathbb{R}^{n})}  <+\infty.
\end{equation*}
The drift $b \colon \Omega \times [0,T] \times \mathbb{R}^n \times \mathbb{R}^n \to \mathbb{R}^n$ is linear and  of separated form
\begin{align*}
    b(t,x,\psi) = a_t + b_t x + c_t \psi,
\end{align*}
where $a$, $b$ and $c$ belong respectively to 
$L^\infty({\mathbb F},{\mathbb R}^n)$,
$L^\infty({\mathbb F},{\mathbb R}^{n \times n})$
and
$L^\infty({\mathbb F},{\mathbb R}^{n \times n})$, i.e.
 $\|a\|_{L^\infty(\mathbb{F},\mathbb{R}^{n})} + \|b\|_{L^\infty(\mathbb{F},\mathbb{R}^{n \times n})} +  \|c\|_{L^\infty(\mathbb{F},\mathbb{R}^{n \times n})} < +\infty$. In particular, $b$ is a progressively-measurable random field.

 \item \textit{Volatility.} \label{hyp:sigma} The volatility $\sigma \colon \Omega \times [0,T] \times \mathbb{R}^n \to \mathbb{R}^{n \times d}$ is linear in the control variable.
Precisely, the $n \times d$ entries of the matrix $\sigma$ are of the form
\begin{align*}
    (\sigma(t,\psi))_{i,j} = (\nu_t)_{i,j} + r (\sigma_t)_{i,j,k} \psi_k,
\end{align*}
$(i,j,k) \in \{1,\ldots,n\} \times \{1,\ldots,d\} \times  \{1,\ldots,n\}$, 
with $\nu$ and $\sigma$ belonging respectively to 
$L^\infty(\mathbb{F},\mathbb{R}^{n\times d})$
and
$L^\infty(\mathbb{F},\mathbb{R}^{n\times d})$
and satisfying
$\|\nu\|_{L^\infty(\mathbb{F},\mathbb{R}^{n\times d})} + \|\sigma\|_{L^\infty(\mathbb{F},\mathbb{R}^{n\times d \times n})} \leq L$.

\item \textit{Driver.} \label{eq:assumption1-f}  
The driver $f \colon \Omega \times [0,T] \times \mathbb{R} \times \mathbb{R}^d \to \mathbb{R}$ is  progressively-measurable, 
convex and twice differentiable with respect to its last two variables, the corresponding derivatives of order 2 are bounded by $L$. Moreover, 
there exist two constants $\alpha,\beta  \geq 0$ such that
$t$,
\begin{equation*}
        f(t,y,z) \leq |f_t^0| + \alpha |y| + \frac{\beta}{2} |z|^2,
    \quad (y,z) \in {\mathbb R} \times {\mathbb R}^d,
\end{equation*}
where  $f^0 \coloneqq f(0,0) \in L^\infty(\mathbb{F})$.

\item \textit{Running cost.}  \label{hyp:L}
The running cost $\ell \colon \Omega \times [0,T] \times \mathbb{R}^n \to \mathbb{R}$  is progressively-measurable and twice differentiable in the last variable. It satisfies
\begin{equation*} 
    \bigl(\nabla_\psi \ell(t,\psi) - \nabla_\psi \ell(t,\psi'\bigr))\cdot (\psi - \psi') \geq \frac{1}{L}|\psi - \psi'|^2, \quad |D^2_{\psi} \ell(t,\psi) | \leq L,
\end{equation*}
and $\vert \ell(t,0) \vert \leq L$ for any $t\in[0,T]$ and $\psi,\psi' \in \mathbb{R}^n$.
In particular, 
$\ell$ is strongly convex in the last variable, with a quadratic growth, uniformly in the other variables.

\item \textit{Coefficients.} \label{hyp:gamma} The coefficient $\gamma$ in 
\eqref{eq:expo:bound:psi}
is chosen as
\begin{align*}
    \gamma = & 8 \beta \max(1,L) e^{\alpha T} \|\Gamma\|_{L^\infty(\mathbb{F},\mathbb{R}^{n \times n})} \|\Gamma^{-1}\|_{L^\infty(\mathbb{F},\mathbb{R}^{n \times n})} \\& \times \left(\|\nu\|_{L^\infty(\mathbb{F},\mathbb{R}^{n \times d})} + 12 \max (1,L) e^{\alpha T} \|\sigma\|_{L^\infty(\mathbb{F},\mathbb{R}^{n \times d \times n})} \right),
\end{align*}
where $\Gamma$ is the resolvent of the linear ODE driven by $b$, i.e. the solution to 
\begin{equation*}
        \frac{\dd}{\dd t} \Gamma_t =  b_t \Gamma_t, \quad 
        t \in [0,T],
        \quad 
        \Gamma_0 = I_n,
    \end{equation*}
    with $I_n$ standing for the $n \times n$ identity matrix.

Moreover, when $r=0$, the following smallness condition is satisfied:
\begin{equation*}
    4 \beta e^{\alpha T} L \| \Gamma \|^2_{L^\infty({\mathbb F},{\mathbb R}^{n \times n})} \|\Gamma^{-1} \|^2_{L^\infty({\mathbb F},{\mathbb R}^{n \times n})} \| \nu \|^2_{L^\infty({\mathbb F},{\mathbb R}^{n \times d})} T < 1,
\end{equation*} 
The choice of $\gamma$ is discussed in Remark \ref{remark:gamma} below, and the smallness condition in Remark \ref{rem:p-L-beta-alpha}.

\item \textit{Growth of the mapping $\mathcal{G}$ and its derivatives.}  \label{eq:G-growth} 
Denoting by $\mathscr{G}$ the set of pairs $(q,X)$ of  $\mathcal{F}_T$-measurable random variables with values in $\mathbb{R}_+ \times \mathbb{R}^n$ such that ${\mathbb E}[q \vert X \vert^{2-r}] < + \infty$, 
the cost $\mathcal{G}$ 
is a real-valued function on $\mathscr{G}$. Together with some mappings 
$\delta_q  {\mathcal G} : {\mathscr G} \rightarrow L^0({\mathcal F},{\mathbb R})$ and 
$\delta_X \mathcal{G}
: {\mathscr G}
\rightarrow L^0({\mathcal F},{\mathbb R}^d)
$, which are interpreted below as derivatives of ${\mathcal G}$ in the directions $q$ and $X$ respectively, 
it satisfies the growth properties:
\begin{align*} 
    |\mathcal{G}(q,X)| & \leq L\left(1+ \mathbb{E}\left[ (1+q) |X|^{2-r}\right]\right),\\
    - L\left(1+ |X| + \mathbb{E}\left[ q|X|^{2-r}\right] \right)\leq  \delta_q \mathcal{G}(q,X) & \leq L\left(1+ |X|^{2-r} + \mathbb{E}\left[ q|X|^{2-r}\right] \right),  \\
    \left |\delta_X \mathcal{G}(q,X)\right | & \leq Lq \left(1+ |X|^{1-r} + \mathbb{E}\left[ q|X|^{2-r}\right]\right).
\end{align*}

\item \textit{First order Taylor expansion of the mapping $\mathcal{G}$.} 
\label{hyp:DgD_XG}
With $\mathscr{G}$, 
$\delta_q {\mathcal G}$
and 
$\delta_X {\mathcal G}$ as in the previous condition, the mapping $\mathcal{G}$ admits the following two first order expansions in $X$ and $q$ respectively: 
\begin{equation*}
\begin{split}
&{\mathcal G}(q,X') = {\mathcal G}(q,X)
+ {\mathbb E} \left[ \delta_X {\mathcal G}(q,X) \cdot (X'- X)\right]
+ O\left(
{\mathbb E} \left[ q\vert X'- X \vert^2 \right]
\right) 
\\
&{\mathcal G}(q',X) = {\mathcal G}(q,X) + {\mathbb E} \left[ \delta_q \mathcal{G}(q,X) (q'-q)\right]
+ o\left( 
{\mathbb E} \left[(1+ \vert X \vert^{2-r})\vert q'-q \vert\right]\right),
\end{split} 
\end{equation*}
where 
$\vert O(r) \vert \leq c   r $ for a constant $c$ that only depends on $(q,X',X)$
 via (any bound for) 
 ${\mathbb E}[q \vert X \vert^{2-r}]$ and 
${\mathbb E}[q \vert X' \vert^{2-r}]$,
and where $\vert o(r) \vert \leq \eta(r) r$ for a function 
$\eta$
 that tends $0$ with $r$ and that only depends on $(q,q',X)$
 via (any bound for) 
 ${\mathbb E}[q \vert X \vert^{2-r}]$ and 
${\mathbb E}[q \vert X' \vert^{2-r}]$.

\item \textit{Concavity-convexity of the mapping $\mathcal{G}$.} \label{eq:G-concave-convex} The mapping $\mathcal{G}$ is concave with respect to the variable $q$ and convex with respect to the variable $X$, i.e., 
for any
$(q,X)$, $(q^1,X^1)$
and $(q^2,X^2)$ in $\mathscr{G}$,
with
${\mathscr G}$ as in 
\eqref{eq:G-growth}, 
and for any $\theta \in [0,1]$.
\begin{equation*}
\begin{split}
    \mathcal{G}\left(\theta q^1 + (1-\theta) q^2 , X \right) & \geq \theta \mathcal{G}(q^1, X)  + (1-\theta) \mathcal{G}(q^2, X), \\
     \mathcal{G}\left(q,\theta X^1 + (1-\theta) X^2\right) & \leq \theta \mathcal{G}(q,X^{1}) + (1-\theta)\mathcal{G}(q,X^{2}).
\end{split}
\end{equation*}
\end{enumerate}

\paragraph{Comments and examples.}

We provide several comments and examples to clarify the assumptions. 
\begin{remark}
 \label{eq:assumption1-nabla-f-star} \textit{Lower bound on the dual driver.} 
In \eqref{def:alpha}, $f^\star \colon \Omega \times [0,T] \times \mathbb{R} \times \mathbb{R}^d  \to \mathbb{R}$ is the Fenchel transform of the driver $f$ with respect to its variables $(y,z)$, i.e.,
\begin{equation*}
\label{eq:f:star}
    f^\star (t,y^\star,z^\star) \coloneqq \sup_{(y,z) \in \mathbb{R}\times \mathbb{R}^d} \left \{ \langle (y^\star,z^\star),(y,z) \rangle - f(t,y,z) \right\}.
\end{equation*}
Since
$f$ is continuous 
in $(y,z)$, the supremum in the definition of $f^\star$ can be reduced to a supremum 
over a countable set. We easily deduce that $f^\star$
is progressively-measurable. 

Moreover, because $f$ is twice differentiable in 
$(y,z)$ with bounded second-order derivatives, see 
\ref{eq:assumption1-f},
$f^\star$ is $c$-strongly convex with respect to its last two variables, for a constant $c>0$, see for instance 
\cite{DelarueLavigne2} for an explicit proof.

Assumption \ref{eq:assumption1-f} also implies that, for any $(t,y^\star,z^\star) \in [0,T] \times \mathbb{R} \times \mathbb{R}^d$,
\begin{align} \label{ineq:duality-fstar}
    f^\star (t,y^\star,z^\star) \geq - |f^0_t| + \chi_{\mathcal{B}}(y^\star/\alpha) + \frac{1}{2\beta} |z^\star|^2,
\end{align}
 where $\chi_{\mathcal{B}}$ denotes the indicator function of the unit ball $\mathcal{B} \coloneqq \{x \in \mathbb{R}, \;|x|\leq 1 \}$, i.e., 
\begin{equation*}
    \chi_{\mathcal{B}}(y^\star) = \begin{cases}
    0, & \mathrm{if}\; y^\star \in \mathcal{B},\\
    +\infty, & \mathrm{otherwise}.
    \end{cases}
\end{equation*}
Indeed by the growth condition in \ref{eq:assumption1-f}, we have, for any
$t \in [0,T]$,
$(y^\star,z^\star) \in \mathbb{R} \times \mathbb{R}^d$ and $(y,z) \in \mathbb{R} \times \mathbb{R}^d$,
\begin{align*}
    f^\star (t,y^\star,z^\star) & \geq \langle (y^\star,z^\star),(y,z) \rangle - f(t,y,z) \\
    & \geq \langle (y^\star,z^\star),(y,z) \rangle - |f^0_t| - \alpha |y| - \frac{\beta}{2}|z|^2.
\end{align*}
Taking, on both sides, the supremum   with respect to $(y,z) \in \mathbb{R} \times \mathbb{R}^d$ and recalling that
the absolute value $|\cdot|$ and the indicator function $\chi_{\mathcal{B}}$ are in duality, 
we get \eqref{ineq:duality-fstar}.
\end{remark}

\begin{example}\textit{Mean field structure of $\mathcal{G}$.} As we already mentioned, the problem addressed in this section is not of mean field type. 
It is only in Section~\ref{sec:mfg} that we clarify our application to the mean field case, 
by considering cost functions $\mathcal G$ of the form
\begin{equation*}
    \mathcal G(q,X) = G\left( (q\mathbb P)_X \right),
\end{equation*}
where $(q\mathbb P)_X$ denotes the law of $X$ under $q\mathbb P$, assuming that $q$ is a non-negative random variable, 
and $G$ is a cost function defined on the space of non-negative measures. 
\end{example}

\begin{remark} 
    \textit{Linearity of the state equation.}  The linearity of the state equation, as guaranteed by Assumptions~\ref{hyp:b} and~\ref{hyp:sigma}, ensures 
the concavity of the mapping $\mathcal Q \ni q \mapsto \mathcal J(q,\psi)$
and the convexity of the mapping $\mathcal A \ni \psi \mapsto \mathcal J(q,\psi)$,
which are proved in Proposition~\ref{prop:concave-J} and Lemma~\ref{lemma:convexity-J}, respectively.
Additionally, the assumption that the volatility is independent of the state variable is crucial for guaranteeing the existence of a finite exponential moment of $L|X_T^{0,*}|$, where $X^{0}$ denotes the solution to the state equation when $\psi \equiv 0$. Such a property would generally fail if the volatility depended linearly on the state variable.
See also Remark~\ref{rem:p-L-beta-alpha} for further comments on the exponential integrability of $X_T^{0,*}$.
\end{remark}

\begin{remark}\textit{On the constant $\gamma$.} \label{remark:gamma} The choice of $\gamma$, as specified in~\ref{hyp:gamma}, stems from Lemma~\ref{lemma:psi-in-A}. 
Roughly speaking, the latter provides an  a priori bound on the component $\psi$ of any saddle point $(q,\psi)$ of~\eqref{pb:min-max-G}. 
This bound is formulated in terms of a bound on $\mathcal S^\star(\psi)$ and therefore requires an appropriate choice of $\gamma$.
\end{remark}

\begin{remark}\textit{On the constant $r$.} \label{rem:alternative-assumption} 
The volatility is controlled when \( r = 1 \) and uncontrolled when \( r = 0 \). As suggested in the previous remark, the state variable \( X \) has finite exponential moments of sufficiently small order when the volatility is controlled (\( r = 1 \)). When the volatility is uncontrolled (\( r = 0 \)), stronger results can be established, showing that \( X \) has finite quadratic exponential moments of small order, as detailed in Lemma \ref{lemma:reg-X}. The fact that the integrability properties are stronger when \( r = 0 \) explains why the growth assumption  \ref{eq:G-growth} is more general in this case.  
\end{remark}

\begin{remark}\label{rem:p-L-beta-alpha} \textit{On the smallness condition on $\| \nu \|^2_{L^\infty({\mathbb F},{\mathbb R}^{n \times d})}$.} 
Part of our analysis relies on an a priori bound for the component $q$ of an arbitrary saddle point $(q,\psi)$
of~\eqref{pb:min-max-G}.
This bound is established in Lemma~\ref{lemma:adjoint-existence-uniqueness}.
To make the proof work, we require the existence of some $\psi \in \mathcal A$
(in fact, for simplicity, we choose $\psi \equiv 0$)
such that $L|X_T^{\psi,*}|^{2-r}$ admits an exponential moment of sufficiently large order~$\upsilon$
(with $\upsilon$ depending explicitly on the other parameters in the assumptions).
When $r=1$, the random variable
$|X_T^{0,*}|^{2-r} = |X_T^{0,*}|$
admits exponential moments of all orders.
When $r=0$, the random variable
$L|X_T^{0,*}|^{2-r} = L|X_T^{0,*}|^{2}$
admits an exponential moment of order~$\upsilon$ provided the smallness condition stated in~\ref{hyp:gamma}
is in force
(see Lemma~\ref{lemma:reg-X}).

It must be stressed that we require stronger integrability properties of $|X_T^{0,*}|$ in the case $r=0$
than in the case $r=1$, due to the growth conditions imposed on $\mathcal G$.
If we were to work with the same growth conditions as in the case $r=1$,
the smallness condition would no longer be needed.

Finally, we note that the smallness condition imposed here is reminiscent of the integrability assumptions
appearing in the analysis of quadratic BSDEs with unbounded terminal data;
see, for instance, \cite{briand2008quadratic,delbaen2011uniqueness}.
This is not surprising, since the characterization of the saddle points of~\eqref{pb:min-max-G}
resulting from our analysis (see Theorem~\ref{theorem:SMP})
relies on a forward--backward SDE that may be quadratic (if $f$ is).
In this respect, it is worth emphasizing that our smallness condition is not imposed at the level of the saddle point itself,
but rather at the level of a single controlled trajectory.
As such, it is more explicit and easier to verify.
\end{remark}
 
\begin{remark} \textit{On the running cost $\ell$.}
  In our analysis, the running cost $\ell$ is assumed to be independent of the state variable. 
This assumption may be restrictive for certain applications. 
However, our approach also allows one to consider a running cost 
$\ell' \colon [0,T] \times \mathbb{R}^n \times \mathbb{R}^n \to \mathbb{R}$
that depends on the state variable $x$ and is of separated form. 
More precisely, there exists a function 
$c' \colon [0,T] \times \mathbb{R}^n \to \mathbb{R}$ such that
\begin{equation*}
    \ell'(t,x,\psi) = c'(t,x) + \ell(t,\psi).
\end{equation*}
    One has to assume that $c$ is convex with respect to its second variable, and satisfies the growth condition
    \begin{equation*}
        |c'(t,x)| \leq L(1 + |x|^{2-r}).
    \end{equation*}
The latter implies
that 
\begin{equation*}
        \mathbb{E}\left[\int_0^T q_s \ell'(s,X^0_s,0) \dd s\right] \leq L\left(1 + \mathbb{E}\left[\int_0^T (1+q_s) |X^0_s|^{2-r} \dd s\right] \right), 
    \end{equation*}
     which is, in particular, enough to reproduce the proofs of the two key Lemmas \ref{lemma:adjoint-existence-uniqueness} 
    and
    \ref{lemma:psi-in-A} (up to an adaptation of the
two    conditions in \ref{hyp:gamma}).
\end{remark}

Inequality \eqref{ineq:duality-fstar} has an important consequence, which we formalize
in the following statement: 

\begin{lemma}
\label{lem:about:q}
Let 
$q=(q_t)_{t \in [0,T]}$
be an ${\mathbb F}$-progressively measurable positive-valued continuous process 
such that ${\mathcal S}(q)< + \infty$. 
Then, 
\begin{equation}
\label{eq:about:q:1}
{\mathbb P} \otimes \textrm{\rm Leb}_{[0,T]}
\left( 
\left\{ (\omega,t) \in \Omega \times [0,T], \; 
 \vert   Y_t^\star \vert > \alpha   \right\}
\right) = 0.
\end{equation}
Moreover, 
\begin{equation}
\label{eq:about:q:2}
{\mathbb P} 
\left( 
\left\{ 
\int_0^T \vert Z_s^\star \vert^2 \dd s < + \infty
\right\}
\right) = 1,
\end{equation}
and $({\mathcal E}_t(\int_0^{\cdot} 
Z_s^\star \cdot \dd W_s))_{t \in [0,T]}$ in \eqref{eq:q:explicit:factorization} is a `true' martingale.
\end{lemma}

While the first claim, 
\eqref{eq:about:q:1}, is quite obvious, the second one, 
\eqref{eq:about:q:2}, is more subtle. Indeed, we 
deduce from
\eqref{ineq:duality-fstar} that 
\begin{equation*} 
{\mathbb P} 
\left( 
\left\{ 
\int_0^T q_s \vert Z_s^\star \vert^2 \dd s < + \infty
\right\}
\right) = 1.
\end{equation*}
And then, it is by continuity and strict positivity of $q$ that
\eqref{eq:about:q:2} follows. In particular, it must be observed that the
stochastic integral in \eqref{eq:q:explicit:factorization} is necessarily
well-defined. It is then clear that \eqref{eq:q} and
\eqref{eq:q:explicit:factorization} are equivalent: starting from
\eqref{eq:q}, one obtains \eqref{eq:q:explicit:factorization} by applying
It\^o’s formula, while the converse implication also follows from It\^o’s formula,
applied to the process $(\ln(q_t))_{t \in [0,T]}$, which is well defined since
$q$ takes strictly positive values.
The fact that 
$({\mathcal E}_t(\int_0^{\cdot} 
Z_s^\star \cdot \dd W_s))_{t \in [0,T]}$ is a true martingale is a follows from Lemma 
\ref{lemma:representation-q}, proved in  Appendix \ref{appendix:representation}. 

\paragraph{Main result.} 
The main result of this section is presented in Theorem \ref{theorem:SMP} below.
We introduce the pre-Hamiltonian of the system $\mathcal{H} \colon \Omega \times [0,T] \times \mathbb{R}_+ \times \mathbb{R} \times  \mathbb{R}^d \times \mathbb{R} \times  \mathbb{R}^d \times \mathbb{R}^{n} \times \mathbb{R}^{n} \times \mathbb{R}^{n} \times \mathbb{R}^{n\times d} \to \mathbb{R}$,
\begin{align}
\label{eq:def:Hamiltonian:mathcalH}
    \mathcal{H}(t,q,y^\star,z^\star,y,z,x,\psi,p,k) \coloneqq q(yy^\star + z\cdot z^\star - f^\star(t,y^\star,z^\star) + \ell(t,\psi))\\ + p \addtxt{\cdot} b(t,x,\psi)+ \Tr (k \sigma^{\top}(t,\psi) ). \nonumber
\end{align}
Although the pre-Hamiltonian ${\mathcal H}$ explicitly appears in Isaac's condition discussed in Remark \ref{rem:isaac} below, for the purpose of our analysis, it is more convenient to split it into two parts, each corresponding to the pre-Hamiltonian used by either the central planner or Nature:
\begin{equation}
    \label{eq:def:F:H}    
    \begin{split}
F(t,q,y^\star,z^\star,y,z,\psi) & \coloneqq q\left(yy^\star + z\cdot z^\star - f^\star(t,y^\star,z^\star) + \ell(t,\psi)\right) , \\
    H(t,x,\psi,p,k,q) & \coloneqq q \ell(t,\psi) + p \cdot b(t,x,\psi)+ \Tr (k \sigma^{\top}(t,\psi) ).
\end{split}
\end{equation}

Given $q \in \mathcal{Q}$, we say that a tuple $(\psi,p,k,X)$ satisfies the first order condition \eqref{optim:condition-primal} for the central planner problem if $(\psi,p,k,X)$ is a solution to
\begin{equation} \label{optim:condition-primal} \tag{Opt\textsubscript{C}}
    \left\{ \begin{array}{rll}
        - \dd p_t  & =   \nabla_x H(t,X_t,\psi_t,p_t,k_t,q_t) \dd t - k_t \dd W_t, & 
        p_T = \delta_X \mathcal{G}(q_T,X_T^\psi), \\[0.5em]
        \dd X_t & = b(t,X_t,\psi_t) \dd t + \sigma(t,\psi_t) \dd W_t,  & X_0 = \eta,\\[0.5em]
        \psi_t & \in \argmin_{\alpha} H(t,X_t,\alpha,p_t,k_t,q_t), & \dd \mathbb{P}  \otimes \dd t \text{-a.s.}
    \end{array} \right.
\end{equation}
The first equation is interpreted as the adjoint equation for the central planner, the second equation as the state equation, and the last equation as the optimality condition. Because the last equation couples the two preceding equations, the system above is an FBSDE.
The wordings `first order condition' and `optimality condition' are fully justified by 
the statement of Theorem 
\ref{theorem:SMP}
below.

Given $\psi \in \mathcal{A}$, we say that a tuple $(Y,Z,q)$ satisfies the first order condition \eqref{optim:condition-dual} for Nature problem if  $(Y,Z,q)$ is a solution to 
\begin{equation} \label{optim:condition-dual} \tag{Opt\textsubscript{N}}
    \left\{ \begin{array}{rll}
        - \dd Y_t & =  \partial_q F(t,q_t,Y^{\star}_t,Z^{\star}_t,Y_t,Z_t,\psi_t) \dd t - Z_t \cdot \dd W_t,  & Y_T =\delta_q \mathcal{G}(q_T, X_T^\psi), \\[0.5em]
         \dd q_t &  = q_t Y^\star_t  \dd t +  q_t Z^\star_t  \cdot \dd W_t, &
        q_0  = 1,\\[0.5em]
        (Y^\star_t,Z^\star_t) & \in \argmax_{(Y^{\star \prime},Z^{\star \prime})}F(t,q_t,Y^{\star \prime},Z^{\star \prime},Y_t,Z_t,\psi_t),  & \dd \mathbb{P}  \otimes \dd t\text{-a.s.}
    \end{array} \right.
\end{equation}
The first equation is interpreted as the adjoint equation for Nature, the second equation describes the dynamics of the control variable, and the last equation is the optimality condition. Similar to 
 the previous one, this system of equations is also an FBSDE.

The two systems \eqref{optim:condition-primal} and \eqref{optim:condition-dual} above are presented in an abstract form. To clarify the result, we now give an explicit formulation using the concrete expressions of the coefficients. We start with the system \eqref{optim:condition-primal}.
Computing the gradient of the Hamiltonian $\nabla_x H$ and the optimality condition, we have 
\begin{equation} \label{eq:BSDE:compact:p:k:X}
    \left\{ \begin{array}{rll}
        - \dd p_t  & =  b_t^\top p_t \dd t - k_t \dd W_t, & 
        p_T = \delta_X \mathcal{G}(q_T,X_T^\psi), \\[0.5em]
        \dd X_t & = b(t,X_t,\psi_t) \dd t + \sigma(t,\psi_t) \dd W_t,  & X_0 = \eta,\\[0.5em]
        0 & = q_t \nabla_{\psi} \ell(t,\psi_t) +    
        c_t^{\top}
        p_t + r \Tr (\sigma_t^\top k_t), & \dd \mathbb{P}  \otimes \dd t \text{-a.s,}
    \end{array} \right.
\end{equation}
where we denote, by convention,
\begin{equation} \label{def:trace-sigma-k}
    {\rm Tr}\left( \sigma_t^{\top} k_t \right) =
\left( 
 \sum_{i=1}^n
 \sum_{j=1}^d
( \sigma_t)_{i,j,\ell} (k_t)_{i,j} 
\right)_{\ell=1,\ldots,d}.
\end{equation}
We now turn to the system \eqref{optim:condition-dual}. The optimality condition is given by 
\begin{equation}
    \label{eq:Y:Z<->Ystar:Zstar}(Y^\star_t,Z^\star_t) = (\partial_y f(t,Y_t,Z_t),\partial_z f(t,Y_t,Z_t)).
\end{equation}
Computing the derivative of the Hamiltonian $\partial_q F$ and plugging the optimality condition into the backward equation, the latter equation becomes a (possibly quadratic) BSDE by Fenchel's duality 
\begin{equation}
\label{eq:BSDE:compact:Y:Z}
    \left\{ \begin{array}{rll}
        - \dd Y_t & = (f(t,Y_t,Z_t) + \ell(t,\psi_t))\dd t - Z_t \cdot \dd W_t,  & Y_T =\delta_q \mathcal{G}(q_T, X_T^\psi), \\[0.5em]
         \dd q_t &  = q_t Y^\star_t  \dd t +  q_t Z^\star_t  \cdot \dd W_t, &
        q_0  = 1.
    \end{array} \right.
\end{equation}

For the purpose of analyzing these two systems, 
we define the following two spaces. The first is the space of solutions to the system \eqref{optim:condition-primal} given $\psi$ within a certain sub-level set of ${\mathcal A}$ (which will be specified when necessary), and the second is the space of solutions to \eqref{optim:condition-dual} given \( q \in \mathcal{Q} \):
\begin{align} \label{def:mathscr-A}
     \mathscr{A} & \coloneqq \mathcal{A} \times D({\mathbb F})
    \times (
    \underset{\beta \in (0,1)}{\cap} M^{\beta}({\mathbb F},{\mathbb R}^d)) \times S^{2-r}(\mathbb F,\mathbb{R}^n,\mathbb{Q}),\\ \label{def:mathscr-Q}
    \mathscr{Q} & \coloneqq  \left\{ 
    (q,Y,Z) \in   {\mathcal Q}\times  D({\mathbb F},{\mathbb Q})
    \times (
    \underset{\beta \in (0,1)}{\cap} M^{\beta}({\mathbb F},{\mathbb R}^d,{\mathbb Q})    
    )  \right\},
\end{align}
where $\mathbb{Q}$ in the first line is the measure $q_T {\mathbb P}$.
Here is now our main statement regarding
the inf-sup mean field stochastic control problem
\eqref{pb:min-max-G}.


\begin{theorem} \label{theorem:SMP}
     There exists 
    a unique saddle point
    $(\bar{q},\bar{\psi}) \in  \mathcal{Q} \times \mathcal{A}$ to the problem \eqref{pb:min-max-G}, i.e. 
        \begin{equation*}
            \min_{\psi \in \mathcal{A}} \max_{q \in \mathcal{Q}} \mathcal{J}(q,\psi) = \max_{q \in \mathcal{Q}} \min_{\psi \in \mathcal{A}}  \mathcal{J}(q,\psi) = \mathcal{J}(\bar{q},\bar{\psi}).
        \end{equation*}
Moreover, if a pair  
$(\psi,q) \in \mathcal{A} \times \mathcal{Q}$ is a solution to the problem \eqref{pb:min-max-G}, then
the tuples $(\psi,p,k,X)$, obtained by solving in ${\mathscr A}$ the two decoupled equations in 
\eqref{optim:condition-dual},
and $(q,Y,Z)$, obtained by solving in ${\mathscr Q}$ the two decoupled equations in \eqref{optim:condition-primal},
satisfy the optimality conditions in 
\eqref{optim:condition-dual}
and 
\eqref{optim:condition-primal}
respectively. Conversely,
if
$(\psi,p,k,X,q,Y,Z) \in \mathscr{A} \times  \mathscr{Q} $ is a solution to \eqref{optim:condition-primal}-\eqref{optim:condition-dual}, then the pair 
$(\psi,q) \in \mathcal{A} \times \mathcal{Q}$ is the solution to the problem \eqref{pb:min-max-G}.
\end{theorem}

We provide a sketch of the proof based on the results established in the core of the article. We believe this presentation will help the reader gain a global overview of the structure of the arguments.

The strategy relies on introducing two truncation parameters. For $c_1,c_2 >0$,  we define the two following sets:
\begin{align}
    \mathcal{Q}_{c_1}  &\coloneqq \left\{q \in \mathcal{Q}, \; \mathcal{S}(q) \leq  c_1 \right\}, \label{eq:def:Qc1}
    \\
    \mathcal{A}_{c_2} & \coloneqq \left\{\psi \in L^2(\mathbb{F},\mathbb{R}^n), \;  \;\mathcal{S}^\star(\psi) \leq c_2 \right\}.
    \label{eq:def:Ac2}
\end{align}
    Accordingly, we define the following min-max problem, analogous to the problem \eqref{pb:min-max-G}, but with the above two sets as restricted admissible sets: 
    \begin{equation} \label{pb:min-max-G-c1-c2} \tag{P'}
    \sup_{q \in \mathcal{Q}_{c_1}}  \inf_{\psi \in \mathcal{A}_{c_2}}  \mathcal{J}(q,\psi).
\end{equation}

\begin{proof}

    \noindent \textit{Step 1: Existence of a saddle point to \eqref{pb:min-max-G-c1-c2}.} 
   The problem \eqref{pb:min-max-G-c1-c2}
    is studied in Subsection \ref{sec:properties-J}. Existence of a saddle point is established in Lemma \ref{lemma:existence-of-saddle-point-p-prime}.
    \vskip 4pt

    \noindent \textit{Step 2: Interior solutions.} By definition, any saddle point $(q,\psi)$ to \eqref{pb:min-max-G-c1-c2} satisfies
    \begin{equation} \label{ineq:0-psi-q-q0}
        \mathcal{J}(q,0) \geq \mathcal{J}(q,\psi) \geq  \mathcal{J}(q^0,\psi),
    \end{equation}
    where $q^0 =(q_t^0=1)_{t \in [0,T]} \in \mathcal{Q}_{c_1}$ denotes the solution to $q_t = 1 + \int_0^t q_s Y^\star_s \dd s + \int_0^t q_s Z^\star_s\cdot \dd W_s$ with $(Y^\star,Z^\star) \equiv (0,0)$.
    Then, by the two forthcoming Lemmas \ref{lemma:adjoint-existence-uniqueness} and \ref{lemma:psi-in-A}, there exist two constants $c_1',c_2'>0$ only depending on the data and independent of $c_1$ and $c_2$ such that (more precisely $c_2'$ depends on $c_1'$, which is only depending on the data) such that 
    \begin{equation*}
        (q,\psi) \in \left( \mathcal{Q}_{c_1} \cap \mathcal{Q}_{c_1'}\right) \times
        \left(\mathcal{A}_{c_2} \cap \mathcal{A}_{c_2'}\right).
    \end{equation*} 
    Now choosing $c_1$ and $c_2$ such that $c_1 > c_1'$ and $c_2 > c_2'$ yields that $(q,\psi) \in \mathcal{Q}_{c_1'} \times  \mathcal{A}_{c_2'}$ and thus $(q,\psi)$ is an interior solution to the problem \eqref{pb:min-max-G-c1-c2}, in the sense that
        $\mathcal{S}(q)$
    and  $\mathcal{S}^\star(\psi)$
    are respectively strictly less than $c_1$ and $c_2$.
\vskip 4pt

    \noindent \textit{Step 3: Nature's problem.}
    Let $(q,\psi) \in  \mathcal{Q}_{c_1} \times {\mathcal A}_{c_2}$ be a saddle point to \eqref{pb:min-max-G-c1-c2}, for $c_1>c_1'$. By the previous step, $q$ lies in the interior of $\mathcal{Q}_{c_1}$. Theorem \ref{thm:sto-max-princ-dual} (whose statement and proof are the main objectives of 
    Subsection \ref{sec:Q-BSDE} below) says that $q$ is a maximizer of the problem  
    \begin{equation} \label{pb:Nature-c1}
        \sup_{q' \in {\mathcal Q_{c_1}}} {\mathcal J}(q',\psi),
    \end{equation}
    if and only if the triple
    $(q,Y,Z) \in \mathscr{Q}$ obtained by solving the decoupled FBSDE in \eqref{optim:condition-dual}
    satisfies the optimality condition in \eqref{optim:condition-dual}. 
    Theorem \ref{thm:sto-max-princ-dual} also guarantees that
the maximizer of the problem \eqref{pb:Nature-c1} is unique.
   \vskip 4pt
   
    \noindent \textit{Step 4: Central planner's problem.}
    Let $(q,\psi) \in  \mathcal{Q}_{c_1} \times {\mathcal A}_{c_2}$ be a saddle point to \eqref{pb:min-max-G-c1-c2}, for $c_2>c_2'$. By Step 2, $\psi$ lies in the interior of ${\mathcal A}_{c_2}$. Then Theorem \ref{thm:sto-max-princ-central-planner} (which is the main result of Subsection \ref{sec:central-planner} below)
    establishes that $\psi \in {\mathcal A}_{c_2}$ is a minimizer of the problem
    \begin{equation} \label{pb:central-planner-c2}
        \inf_{\psi \in {\mathcal A}_{c_2}} {\mathcal J}(q,\psi),
    \end{equation}
    if and only if the tuple $(\psi,p,k,X) \in \mathscr{A}$ obtained by solving the decoupled FBSDE in \eqref{optim:condition-primal} satisfies the optimality condition in \eqref{optim:condition-primal}.
    Theorem \ref{thm:sto-max-princ-central-planner} also guarantees that the minimizer of the problem \eqref{pb:central-planner-c2} is unique.
    \vskip 4pt
    
    \noindent \textit{Step 5: Conclusion.} To conclude the proof, we show that 
    the problems \eqref{pb:min-max-G-c1-c2} and \eqref{pb:min-max-G} have the same set of solutions if $c_1 >c_1'$ and  $c_2>c_2'$. We first show that any solution to \eqref{pb:min-max-G-c1-c2} is solution to \eqref{pb:min-max-G}. 
    To do so, we consider
    the tuple $(q,Y,Z,\psi,p,k,X) \in \mathscr{Q} \times \mathscr{A}$, solution to the coupled system of FBSDEs \eqref{optim:condition-primal}-\eqref{optim:condition-dual} (which solution is given by the previous two steps).
    By the sufficiency 
property of the two first order conditions 
\eqref{optim:condition-primal} and \eqref{optim:condition-dual}
    (see again the 
    previous two steps), 
$(q,Y,Z,\psi,p,k,X)$
    satisfies the following two properties: 
    \begin{itemize}
    \item 
    for any $c_1'' > c_1'$, 
    $q$ is a maximizer of \eqref{pb:Nature-c1}, with $c_1$ being replaced by 
    $c_1''$ therein, and thus 
    ${\mathcal J}(q,\psi) \geq {\mathcal J}(q',\psi)$ for any 
    $q' \in {\mathcal Q}$;
    \item for any $c_2'' > c_2'$, 
    $\psi$ is a minimizer of 
\eqref{pb:central-planner-c2}, with $c_2$ being replaced 
by $c_2''$ therein, and thus 
    ${\mathcal J}(q,\psi) \leq {\mathcal J}(q,\psi')$ for any 
    $\psi' \in {\mathcal A}$.
    \end{itemize}
Therefore,
$(q,\psi)$ is also a solution to the problem \eqref{pb:min-max-G}, which proves in particular that the  problem \eqref{pb:min-max-G} admits at least a solution. 

We now show that any solution to \eqref{pb:min-max-G} is also a solution to \eqref{pb:min-max-G-c1-c2}, when $c_1>c_1'$ and $c_2>c_2'$. Any solution $(q,\psi) \in \mathcal{Q} \times \mathcal{A}$ to \eqref{pb:min-max-G} necessarily belongs to 
$\mathcal{Q}_{c_1''} \times \mathcal{A}_{c_2''}$
for some $c_1''>0$ and $c_2''>0$ (since 
${\mathcal Q}=\cup_{c_1''>0} {\mathcal Q}_{c_1''}$
and
${\mathcal A}^{\vartheta}=\cup_{c_2''>0} {\mathcal A}^{\vartheta}_{c_2''}$). This implies that $(q,\psi)$ also lies in $\mathcal{Q}_{c_1'} \times \mathcal{A}_{c_2'}$ by the same argument as in Step 2. Then, by repeating the arguments of Step 3 and 4, we deduce that $(q,\psi)$ is a solution to \eqref{pb:min-max-G-c1-c2}, concluding the proof.

Uniqueness follows readily. Suppose that there exist two distinct saddle points
$(q,\psi)$ and $(q',\psi')$ in $\mathcal Q \times \mathcal A$, and hence in
$\mathcal Q_{c_1'} \times \mathcal A_{c_2'}$ by the analysis above. Then at least
one of the following holds: $q' \neq q$ or $\psi' \neq \psi$.
If $q' \neq q$, we use the fact that the optimization problem
\eqref{pb:Nature-c1} admits a unique maximizer to deduce that
$\mathcal J(q,\psi) > \mathcal J(q',\psi)$. 
By the saddle-point property of $(q',\psi')$, this implies
\[
\mathcal J(q,\psi) > \mathcal J(q',\psi) \geq \mathcal J(q',\psi').
\]
This is a contradiction, since both extreme terms are equal to
$\min_{\tilde \psi \in \mathcal A} \max_{\tilde q \in \mathcal Q}
\mathcal J(\tilde q,\tilde \psi)$. 
Similarly, if $\psi \neq \psi'$, then
$\mathcal J(q,\psi) < \mathcal J(q,\psi') \leq \mathcal J(q',\psi')$, 
which again leads to a contradiction by the saddle-point property.
This concludes the proof of uniqueness.
\end{proof}

The system described by equations \eqref{optim:condition-primal} and \eqref{optim:condition-dual} is inherently coupled. Specifically, the control $q$ played by Nature appears both in the pre-Hamiltonian $H$ and in the terminal condition for the adjoint variables $(p,k)$ of the central planner, as shown in \eqref{optim:condition-primal}. Similarly, the control $\psi$ of the central planner is present in the driver and in the terminal condition for the Nature adjoint variables $(Y,Z)$.
As a consequence of Theorem \ref{theorem:SMP}, this coupled system has a unique solution, 
which characterizes the (unique) saddle point to 
the problem \eqref{pb:min-max-G}.

\begin{remark}  \label{rem:isaac}\textit{Isaac's condition.}
     At optimality, the following Isaac's condition holds at the optimum:
    \begin{equation}
\label{eq:isaac:min-max:max-min}
    \begin{split}
        &\min_{\psi}  \max_{(Y^\star,Z^\star)}  
        \mathcal{H}(t,q_t,Y^\star,Z^\star,Y_t,Z_t,X_t,\psi,p_t,k_t)
        \\
        &= \max_{(Y^\star,Z^\star)} \min_{\psi}  \mathcal{H}(t,q_t,Y^\star,Z^\star,Y_t,Z_t,X_t,\psi,p_t,k_t),
    \end{split}
    \end{equation}
    $\dd \mathbb{P} \otimes \dd t$-almost surely. This 
   follows from the combination of 
   \eqref{optim:condition-dual} and 
   \eqref{optim:condition-primal}, which say that, 
   $\dd \mathbb{P} \otimes \dd t$-almost surely, 
   \begin{equation*}
\begin{split}
   (Y^\star_t,Z^\star_t)  &\in \argmax_{(Y^{\star \prime},Z^{\star \prime})}F(t,q_t,Y^{\star \prime},Z^{\star \prime},Y_t,Z_t,\psi_t),  
   \\
    \psi_t  &\in \argmin_{\alpha} H(t,X_t,\alpha,p_t,k_t,q_t).
\end{split}
\end{equation*}
    Returning back to the definition 
    \eqref{eq:def:Hamiltonian:mathcalH}
of ${\mathcal H}$, these two lines can be rewritten as
\begin{equation*}
\begin{split}
   (Y^\star_t,Z^\star_t)  &\in \argmax_{(Y^{\star \prime},Z^{\star \prime})}
   \mathcal{H}(t,q_t,Y^{\star \prime},Z^{\star \prime},Y_t,Z_t,X_t,\psi_t,p_t,k_t),
   \\
    \psi_t  &\in \argmin_{\alpha} 
    \mathcal{H}(t,q_t,Y_t^\star,Z_t^\star,Y_t,Z_t,X_t,\alpha,p_t,k_t),
    \end{split}
    \end{equation*}
from 
which the bound $\min \max \leq \max \min$
 in 
\eqref{eq:isaac:min-max:max-min} indeed follows, the converse bound being always true.
\end{remark}

\subsection{Examples of applications} \label{sec:example-application}
We provide two examples of applications of 
Theorem 
\ref{theorem:SMP}. 
 On purpose, the presentation is informal and contains no mathematical statement. 
Further examples are given in Subsection \ref{subsec:examples:mean field}.

\paragraph{Risk averse portfolio management with trading costs.}
The first example is inspired by \cite{fleming2002risk} and considers the regime ``without investment control constraints,'' in the absence of a risk-free asset, and over a finite time horizon.

Consider a financial market consisting of $d \in \mathbb{N}^\star $ stocks, whose prices per share are encoded in the form of an $d$-dimensional process $S = (S^1, S^2, \dots, S^d)$, satisfying the following SDE:
\begin{equation*}
    \frac{\dd S_t^i}{S^i_t} = c^i_t \dd t + (\sigma_t \dd W_t)^i, \quad S^i_0 = \zeta^i, \quad \forall i \in \{1,\ldots,d\},
\end{equation*}
where 
$\zeta=(\zeta^1,\ldots,\zeta^d) \in L^\infty(\mathcal{F}_0,\mathbb{R}^d)$ are the initial prices, ${c}=(c^1,\ldots,c^d) \in L^\infty(\mathbb{F};\mathbb{R}^d)$ is the vector of  stock appreciation rates, $\sigma \in L^\infty(\mathbb{F},\mathbb{R}^{d\times d})$ is the volatility matrix, and $W=(W^1,\ldots,W^d)$ is a  $d$-dimensional Brownian motion. For simplicity we thus assume that the number of assets is equal to the number of noise sources. We further assume that there is no bond available on the market. For a given vector of amounts (or allocation strategies) $\psi \in \mathcal{A}$, the dynamics of the self-financing portfolio $X^\psi$ is given by
\begin{equation*}
    \dd X_t = \psi_t \cdot \frac{\dd S_t}{S_t}, \quad X_0 = 1,
\end{equation*}
where the dot appearing on the right-hand side stands for the inner product in 
${\mathbb R}^d$, 
the initial condition is arbitrarily chosen to be unitary, and 
(consistently with the fact there is no bond)
the interest rate of the market is assumed to be null for simplicity. 
The problem of the risk averse investor under a min-max form is given by
\begin{equation} \label{pb:inve-risk)-averse}
    \sup_{q \in \mathcal{Q}} \inf_{\psi\in \mathcal{A}} \mathcal{J}(q,\psi),
\end{equation}
where \begin{equation*}
    \mathcal{J}(q,\psi) = \mathbb{E}^{\mathbb{Q}}\left[X_T^\psi + \frac{1}{2} \int_0^T |\psi_s|^2 \dd s \right] - \lambda \mathrm{H}(\mathbb{Q} \vert \mathbb{P}),\quad \mathbb{Q} =  q_T  \mathbb{P}
\end{equation*}
and $q_T = \mathcal{E}_T(\int_0^\cdot Z_s^\star \cdot \dd W_s)$. The problem can be interpreted as follows. Given a probability measure $\mathbb{Q}$ equivalent to $\mathbb{P}$, the investor optimizes the average return of the portfolio while incurring a trading cost. Given an investment strategy chosen by the investor, Nature then selects the worst-case probability measure $\mathbb{Q}$, while being penalized by an entropic cost.
 The parameter $\lambda>0$ models the level of risk aversion of the investor. 

This is a sub-case of our setting. The random processes $a$ and $b$ are null, $c$ is valued in $\mathbb{R}^d$ instead of $\mathbb{R}^{1 \times d}$ and $\sigma$ in $\mathbb{R}^{d \times d}$ instead of $\mathbb{R}^{d \times 1 \times d}$. The terminal cost is linear in the measure $\mathbb{Q} \circ (X_T^\psi)^{-1}$ and the actualization rate $Y^\star$ is null.
By Theorem \ref{theorem:SMP}, the problem \eqref{pb:inve-risk)-averse}
admits a unique solution $(\psi,q)$, characterized by $\nabla_{\psi} H(t,X_t,\psi_t,p_t,k_t,q_t) = 0$ and $Z^\star_t = \partial_z f(Z_t)$, which, after computations (with 
$k_t$ being viewed as a vector of dimension $d$), 
gives 
\begin{equation*}
    \psi_t = - q_t^{-1}\left( p_t   c_t +  \sigma_t^{\top} k_t   \right), \quad Z_t = \frac{1}{\lambda} Z^\star_t,
\end{equation*}
where the tuple of state and adjoint processes $(Y,Z,X,p,k,q)$ is the solution to
\begin{equation*} 
    \left\{ \begin{array}{rll}
        - \dd Y_t  & =  (\frac{1}{2\lambda} |Z_t|^2 + \vert \psi_t \vert^2) \dd t - Z_t \cdot \dd W_t, & Y_T = X_T, \\[0.5em]
        \dd X_t & = \psi_t \cdot c_t \dd t + \psi_t \cdot ( \sigma_t \dd W_t), & X_0 = 1, \\[0.5em]
        - \dd p_t  & = - k_t \cdot \dd W_t, & 
        p_T = q_T, \\[0.5em]
         \dd q_t &  =   \lambda  q_t Z_t \cdot  \dd W_t, &
        q_0  = 1.
    \end{array} \right.
\end{equation*}
By the last two equations, we have that $p_t = q_t$ and $k_t =  \lambda q_t Z_t$. Then the solution simplifies to
\begin{equation*}
    \psi_t = -c_t - \frac{1}{\lambda}  \sigma_t^{\top} Z_t ,
\end{equation*}
and 
\begin{equation}  \label{FBSDE:opt-risk-averse-pb}
    \left\{ \begin{array}{rll}
        - \dd Y_t  & =  (\frac{1}{2\lambda} |Z_t|^2 + \vert \psi_t \vert^2) \dd t - Z_t \cdot \dd W_t, & Y_T = X_T, \\[0.5em]
        \dd X_t & = \psi_t \cdot c_t \dd t + \psi_t \cdot \left(\sigma_t \dd W_t\right), & X_0 = 1, 
    \end{array} \right.
\end{equation}
which reduces the problem to a quadratic FBSDE. It seems that,
due to the unboundedness of the terminal condition, 
the latter system is out of the scope of the theory of FBSDEs with a quadratic driver (in the backward equation) \cite{fromm2015theory,luo2017solvability,jackson2024quasilinear}.
Very briefly, existing results on the solvability of quadratic BSDEs require the terminal state variable $X_T$ 
and the cost 
$\int_0^T \vert \psi_t \vert^2 \dd t$ (with 
$\psi$ standing for the optimal control) to admit an exponential moment with a sufficiently large exponent, depending on the parameters of the problem. In the present setting, we are only able to establish exponential integrability for small exponents. This result is not stated explicitly in the article, as it holds only in the case where the function $f$ is genuinely quadratic. In that case, the functional ${\mathcal S}$ coincides with the standard entropy and, by a Donsker--Varadhan-type duality (see \eqref{eq:donsker-varadhan}), bounds on the conjugate functional ${\mathcal S}^\star$ yield bounds on certain exponential moments of $X_T$.

Here, the solution $(Y,Z)$ to the quadratic BSDE is obtained in the rather weak space
\[
D(\mathbb{F}, \mathbb{Q}) \times \bigcap_{\beta \in (0,1)} M^\beta(\mathbb{F}, \mathbb{R}^d,\mathbb{Q}).
\]
Because the process $Y$ does not enjoy strong integrability properties, we are not able to justify the following identity, which is frequently used to 
establish the
 connection between the min--max and the risk-averse formulations:
\begin{equation*}
    \frac{1}{\lambda} \ln \mathbb{E}\big[\exp(\lambda Y_0)\big]
    =
    \rho_{\lambda}\!\left[
    X_T + \frac{1}{2} \int_0^T |\psi_s|^2 \, \dd s
    \right],
\end{equation*}
where $\rho_{\lambda}$ is defined in \eqref{eq:rho:vartheta}.
The standard proof of this identity relies on the Hopf--Cole transform for quadratic BSDEs. However, it would require the random variable
\[
X_T + \tfrac12 \int_0^T |\psi_s|^2 \, \dd s
\]
to admit an exponential moment of exponent $\lambda$, a property which appears to be out of reach in our framework.

\paragraph{Control of systemic risk measure.}

Systemic risk measures are risk assessment tools that evaluate the macro-level risk of a system composed of multiple interacting agents. The concept was first introduced axiomatically in \cite{chen2013axiomatic} and has since been extensively explored in both management science \cite{armenti2018multivariate,biagini2019unified} and mathematical finance literature \cite{kromer2016systemic}.

Denoting by $N$ the number of agents in the system, we consider the  product space $(\Omega^{\times N},{\mathcal F}^{\times N},{\mathbb P}^N \coloneqq {\mathbb P}^{\times N})$, 
and 
we equip its $i$-th factor with an ${\mathbb R}^d$-valued Brownian motion 
$W^i=(W^i_t)_{t \in [0,T]}$. We denote by ${\mathbb F}^N=({\mathcal F}_t^N)_{t \in [0,T]}$ the completion of the filtration generated by 
$(W^1,\ldots,W^N)$, and by ${\mathcal Q}^{(N)}$ 
the analogue of 
${\mathcal Q}$ but on the product space, i.e. 
${\mathcal Q}^{(N)}$ is the set of $q^N \in L \log L(\mathbb{F}^N)$ such that 
\begin{equation}
\label{eq:example:2:def:qN}
    \dd q^N_t = q^N_t \left(\sum_{i = 1}^N Z^{\star,i}_t \cdot \dd W^i_t\right), \quad q_0 = 1, \quad   \mathrm{H}(q^N \mathbb{P}^N \vert \mathbb{P}^N) < +\infty.
\end{equation}
Below, we write $\mathbb{Q}^N \coloneqq q^N \mathbb{P}^{N} = q^N {\mathbb P}^{\times N}$.

The function $f^\star$ only depends on its last variable and is given, for a certain $\lambda >0$, by $f^\star(z^{(N)}) = \tfrac{\lambda}2 \sum_{i = 1}^N |z^i|^2$ for any $z^{(N)}=(z^1,\ldots,z^N) \in [\mathbb{R}^{d}]^N$.
We then denote by 
${\mathcal A}^{(N)}$ the analogue of 
${\mathcal A}$ but on the product space. 

To simplify the presentation, we assume that
the states of the agents follow dynamics similar to the one presented in the first example, but with each driven by its own noise 
$W^i$. We also assume that
the coefficients are deterministic, which avoids the need to track how each player’s coefficients depend on the various sources of noise.
For $i\in \{1,\ldots,N\}$, the state of the $i$-th agent is thus
given by the solution $X^{i,\psi^i}$ of the state equation:
\begin{equation}
\label{eq:particle:system}
\dd X^{i,\psi^i}_t  = \psi^i_t \cdot c_t \dd t + \psi^i_t \cdot \sigma_t \dd W_t^i, \quad  X_0 = 1,
\end{equation}
where
$\psi^i$ is the control to player $i$. 

We now address the construction 
of a risk measure for the system formed by the $N$ agents. 
An initial approach would consist in summing individual risk measures associated to each of the agents. However, as emphasized in \cite[Section~2]{biagini2019unified}, this approach may fail to capture systemic risk effects in financial systems.
Motivated by the latter article, we propose an alternative construction in which individual states are first aggregated through an increasing, convex, and nonlinear function
\( g \colon \mathbb{R} \to \mathbb{R} \), then summed, and finally evaluated via an individual risk measure, the nonlinearity of the aggregation function being essential for practical relevance.
A typical example for $g$ is the cost function \( g(x) = \max\{x - x_0, 0\} \), where \( x_0 \in \mathbb{R} \) is a finite threshold. That said, in order to fit within our framework, in which $g$ is typically required to be differentiable, we consider instead a smooth version of it, sill convex, obtained for instance via regularization.
We thus define the systemic risk measure $\mathfrak{p}_{\lambda}$, 
by letting
\begin{equation}
\label{eq:def:pN:lambda}
    \mathfrak{p}^N_{\lambda}(\psi^1,\ldots,\psi^N) = \rho_{\lambda}\left[\sum_{i = 1}^N \left(g(X_T^{i,\psi^i}) +  \frac{1}{2} \int_0^T |\psi^i_t|^2 \dd t \right) \right].
\end{equation}
\textit{Choice of the normalization.} We emphasize that the sum inside $\rho_\lambda$ in the definition of
$\mathfrak{p}^N_\lambda$ diverges as $N$ tends to $+\infty$. In contrast,
if we normalize this sum (inside the risk measure) by an additional factor
$1/N$, we obtain, in the limit $N \to +\infty$, a model in which risk
aversion disappears.
To see this, assume that the independence property of the noises
$W^1,\ldots,W^N$ is asymptotically transmitted to the optimal controls, as
in a standard MFC problem without risk aversion. Equivalently, restrict the
definition of $\mathfrak{p}^N_\lambda$ to controls
$\psi^1,\ldots,\psi^N$ that are each constructed as a common progressively
measurable function of $W^i$, for the corresponding index
$i \in \{1,\ldots,N\}$. Then, a purely formal application of the law of large
numbers (without further justification) allows one to pass to the limit
\textit{inside} $\rho_\lambda$, yielding
\begin{align}
\label{eq:example:2:lln:1}
     \lim_{N \to +\infty} \rho_{\lambda}\left[\frac{1}{N}\sum_{i = 1}^N \left(g(X_T^{i,\psi^i}) +  \int_0^T |\psi^i_t|^2 \dd t \right) \right] =  \int g \dd \mu_T^\psi + \rho_{\lambda}\left[  \int_0^T |\psi_t|^2 \dd t \right],
\end{align}
where $\mu_T^\psi = \mathcal{L}(X_T^\psi)$ and the dynamics of $X^\psi$ is given by
\begin{equation} \label{dyn:X-risk-measure}
\dd X_t  = \psi_t \cdot c_t \dd t + \psi_t \cdot (\sigma_t \dd W_t), \quad  X_0 = 1.
\end{equation}
As announced, the risk aversion has disappeared in the limit. Intuitively, the limiting problem (obtained by letting $N \rightarrow + \infty$) is a standard MFC problem. 
For this reason, we propose below an alternative construction of risk measures for $N$-particles system.
\vskip 4pt

\noindent \textit{Solving the $N$-fixed problem.} 
Instead, we want to keep
the sum over 
$g(X_T^{i,\psi})$
unnormalized in the definition 
\eqref{eq:def:pN:lambda}
of 
$p^N_\lambda$. Accordingly, 
our objective is to explain, at least informally, what is the behaviour 
of $\tfrac1N p^N_\lambda$
as $N$ tends to $+ \infty$. 

The first step is to observe from 
the Donsker-Varadhan formula 
\eqref{eq:donsker-varadhan} that 
the minimization of 
${\mathfrak p}^N_\lambda$ can be reformulated as 
as a min-max problem, i.e., 
\begin{align*}
     &\inf_{\psi^N \in \mathcal{A}^{(N)}} \tfrac1N  \mathfrak{p}_{\lambda}^{N}(\psi^1,\ldots,\psi^N)  
     \\
     &=  \inf_{\psi^{(N)} \in \mathcal{A}^{(N)}}  \sup_{q^N \in \mathcal{Q}^{(N)}}    \frac{1}{N}\left\{ 
    \mathbb{E}^{\mathbb{Q}^N} \left[\sum_{i = 1}^N \left( g\left(X_T^{i,\psi^{i}}\right) + \int_0^T |\psi^{i}_t|^2 \dd t \right)\right] - \lambda \mathrm{H}\left(\mathbb{Q}^N\vert \mathbb{P}^N \right) \right\}. 
\end{align*}
As 
$g$ is convex, we observe that the min-max problem
appearing in the right-hand side  enters the framework of Theorem \ref{theorem:SMP}, 
with 
$n=N$ and $d$ replaced by 
$d \times N$, 
with
$q^N$ and $Z^{\star,(N)} = (Z^{\star,i})_{i \in \{1,\ldots,N\}}$ satisfying 
\eqref{eq:example:2:def:qN} (implicitly, 
$\alpha =0$ and $Y^{\star,(N)} = (Y^{\star,i})_{i \in \{1,\ldots,N\}} \equiv 0$), 
with
$\psi^{(N)} = (\psi^{i})_{i \in \{1,\ldots,N\}}$ and $X^{(N)} = (X^{i,\psi^i})_{i \in \{1,\ldots,N\}}$
solving 
\eqref{eq:particle:system}, 
 and 
 with the terminal cost functions
 \begin{align*}
    &\mathcal{G}\left(q^N_T,X^{(N)}_T\right)  = \mathbb{E}\left[ q^N_T \sum_{i = 1}^N g\left(X_T^{i,\psi^i} \right) \right], 
    \\
    &\ell\left(t,\psi^{(N)}\right)  =  \frac{1}{2}\sum_{i = 1}^N |\psi^{i}_t|^2, 
    \ 
    f^\star\left(t,Z^{\star,(N)}\right)
     = \frac{\lambda}{2}\sum_{i = 1}^N |Z^{\star,i}_t|^2.
\end{align*}
 The saddle-point is characterized by
\begin{equation*}
    \psi^{i}_t = - (q_t^N)^{-1}\left(  p^i_t  c_t  +   \sigma_t^{\top}
    k^{i}_t \right), \quad Z^{i}_t = \frac1{\lambda} Z^{\star,i}_t,
\end{equation*}
for each $i \in \{1,\ldots,N\}$, 
where 
$p_t^i$ takes values in ${\mathbb R}$
and $k_t^i$ in ${\mathbb R}^d$, and
where the tuple of state and adjoint processes $(Y^{N},Z^{(N)}=(Z^1,\ldots,Z^N),X^{(N)},p^{(N)}=(p^1,\cdots,p^N),k^{(N)}=(k^{1},\ldots,k^N),q^N)$ is the solution to
\begin{equation*} 
    \left\{ \begin{array}{rll}
        - \dd Y^N_t  & =  \left(\frac{1}{2\lambda} \sum_{j = 1}^N |Z^{j}_t|^2 +
        \frac12 \sum_{j = 1}^N  |\psi^{j}_t|^2 \right) \dd t - \sum_{j = 1}^N Z^{j}_t \cdot \dd W^j_t,  \\[0.5em]
         \dd X^{i}_t  & =   \psi^{i}_t \cdot c_t \dd t + \psi^{i}_t \cdot (\sigma_t \dd W^i_t), \\[0.5em]
        - \dd p^{i}_t  & = - \sum_{j = 1}^N k^{i,j}_t \cdot \dd W^j_t, \\[0.5em]
         \dd q^N_t &  =   \lambda  q^N_t  \sum_{j = 1}^N Z^{j}_t \cdot \dd W^j_t,
    \end{array} \right.
\end{equation*}
with terminal conditions
\begin{equation*} 
Y^N_T =  \sum_{j = 1}^N g(X^{j,\psi^j}_T), \quad X^{N,i}_0 = 1, \quad   p^{i}_T = q^N_T  g'\left(X_T^{i,\psi^i}\right), \quad 
        q^N_0  = 1,
\end{equation*}
for any $i \in \{1,\ldots,N\}$.
\vskip 4pt

\noindent \textit{Towards a robust MFC problem.}
We now provide a heuristic derivation of the limiting problem as
$N \to +\infty$. The purpose of this discussion is solely to identify the
structure of the limiting model; no claim of rigor is made at this stage.
Proceeding as in~\eqref{eq:example:2:lln:1}, we assume that, at the saddle
point, the controls $\psi^1,\ldots,\psi^N$ are each constructed as a common
progressively measurable function of the individual Brownian motion $W^i$,
for the corresponding index $i \in \{1,\ldots,N\}$, and similarly for the
controls $Z^{\star,1},\ldots,Z^{\star,N}$. Strictly speaking, such an
independence structure does not hold at the finite-$N$ saddle point.
However, this assumption can be justified a posteriori by reverse
engineering: starting from a solution to the limiting problem, one may
construct an approximate optimizer for the finite-$N$ problem, which is a
standard approach in mean field control theory.
Under this assumption, and applying  Girsanov’s theorem (all similar changes of probability measures
will be justified in the core of the article, but we prefer not to address
such technical questions in this informal discussion), the vector
$X^{(N)} = (X^1,\ldots,X^N)$ satisfies the following dynamics under
$\mathbb{Q}^N$:
\begin{equation*}
    \dd X_t^{i}
    =
    \psi_t^{i} \cdot \bigl(c_t + \sigma_t Z_t^{\star,i}\bigr)\,\dd t
    +
    \psi_t^{i} \cdot \sigma_t\,\dd W_t^{i,\mathbb{Q}^N},
\end{equation*}
where
\[
\bigl(
W_t^{i,\mathbb{Q}^N}
=
W_t^{i}
-
\int_0^t Z_s^{\star,i}\,\dd s
\bigr)_{t \in [0,T],\, i=1,\ldots,N}
\]
is (expected to be) an $N$-dimensional Brownian motion under
$\mathbb{Q}^N$.

And then, 
under 
${\mathbb Q}^N$, the 
processes $(X^i,\psi^i)_{i=1,\cdots,N}$ are independent and identically distributed, which makes it possible to derive, by a new application of the law of large numbers, the following approximation for the cost underpinning the min-max problem:
\begin{align*}
    & \frac{1}{N} \left\{ \mathbb{E}^{\mathbb{Q}^N} \left[\sum_{i = 1}^N \left( g\left(X_T^{i,\psi^{N}}\right) + \int_0^T |\psi^{N,i}_t|^2 \dd t \right)\right] - \lambda \mathrm{H}\left(\mathbb{Q}^N\vert \mathbb{P}^N \right)\right\} 
    \\
    &\approx G\left( (q^1_T {\mathbb P})_{X_T^{1,\psi^1}}\right) + \mathbb{E} \left[q_T^1 \int_0^T |\psi_t^1|^2 \dd t \right] - \lambda \mathrm{H}\left( q^1_T {\mathbb P}\vert \mathbb{P}\right).
\end{align*}
as $N \to +\infty$, where 
$q^1_T {\mathbb P})_{X_T^{1,\psi^1}}$
denotes the law of 
$X_T^{1,\psi^1}$, regarded as a random variable on $(\Omega,{\mathcal F})$ equipped with the probability measure $q^1_T {\mathbb P}$. 

We thus conjecture that, in the limit, we are faced with a robust control
problem involving a single agent, but with a cost depending on the law of
the agent under the distribution resulting from Nature’s choice. In other
words, we expect that the asymptotic problem (obtained by letting
$N \to \infty$) consists of the following min–max problem:
\begin{align*}
\inf_{\psi \in \mathcal{A}} \sup_{q \in \mathcal{Q}}
\left\{
G\bigl(\mathcal{L}^{\mathbb{Q}}(X_T^\psi)\bigr)
+
\mathbb{E}^{\mathbb{Q}}\!\left[\int_0^T |\psi_t|^2 \,\dd t\right]
-
\lambda\, \mathrm{H}\bigl(\mathbb{Q}\vert \mathbb{P}\bigr)
\right\},
\end{align*}
formulated on $(\Omega,\mathcal{F},\mathbb{P})$, where
$\mathbb{Q} \coloneqq q_T\,\mathbb{P}$
denotes the probability measure induced by Nature’s strategy.
A more general treatment of this problem is provided in the next section, which is devoted to the mean field regime.

\section{Applications to mean field models} \label{sec:mfg}
In this section, we develop a robust formulation of classical mean field control problems and subsequently analyze an associated variational mean field game problem.  
Subsection~\ref{sec:mean-field-control} is devoted to the robust mean field control problem~\eqref{pb:mckean-vlasov}. Relying on Theorem~\ref{theorem:SMP} from the previous section, we establish in Corollary~\ref{corollary:mckean-vlasov} the existence and uniqueness of a saddle point, together with the corresponding stochastic maximum principle. Subsection~\ref{subsec:examples:mean field} provides two examples of applications in this context. 
Subsection~\ref{sec:variationnal-mean-field-games} then turns to a class of variational mean field games. Building on Corollary~\ref{corollary:mckean-vlasov}, we prove in Corollary~\ref{corollary:mfg} the existence and uniqueness of a Nash equilibrium, which is fully characterized by a McKean--Vlasov forward--backward stochastic differential equation.

\subsection{Robust mean field control} \label{sec:mean-field-control}

This subsection is dedicated to the study of the robust mean field control problem \eqref{pb:mckean-vlasov}, which we recall here
\begin{equation} \tag{MFC}
     \inf_{\psi \in \mathcal{A}} \sup_{q \in \mathcal{Q}} \left\{  G\left((q_T {\mathbb P})_{X_T^\psi}\right)  + \mathbb{E}\left[\int_0^T q_s \ell(s,\psi_s) \dd s \right] - \mathcal{S}(q)\right\},
\end{equation}
    where $G \colon \mathcal{M}_+(\mathbb{R}^n) \to \mathbb{R}$ is a mean field mapping of the positive measure $(q_T {\mathbb P})_{X_T^{\psi}}
= (q_T {\mathbb P}) \circ X^{\psi,-1}_T$ and  $(X^\psi_t)_{t \in [0,T]}$ denotes the solution to the controlled stochastic differential equation \eqref{eq:intro:X}. 

This problem can be recast as problem \eqref{pb:min-max-G} assuming that  the mapping $\mathcal{G}$ is specified as follows

\begin{equation}
\label{eq:G:subsection4.1}
{\mathcal G}(q,X) = G\left((q {\mathbb P})_{X}\right).
\end{equation}
When the actualization rate $Y^\star$ is null, the measure $q_T {\mathbb P}$
is a probability measure, equal to 
${\mathcal E}_T(\int_0^\cdot Z_s^\star \cdot \dd W_s)$, and the domain of definition of $G$
can be reduced to ${\mathcal P}({\mathbb R}^n)$. In the latter case, $G$ is a true mean field function. By extension, we still call the model `mean field' even if the mass of $q_T$ is unnormalized. 

As the domain of definition of $G$ is larger than ${\mathcal P}_1({\mathbb R}^n)$, we are led, under the assumptions below, to redefine implicitly the 
notion of flat and Lions derivatives. Since the objects thus redefined coincide, in the mean field case, with the true flat and Lions derivatives, we nevertheless use the same notations $\delta G/\delta \mu$ and $\partial_\mu G$ as in the introduction. This is the rational behind the introduction of the following distances.

\paragraph{Spaces of positive measures}

We introduce, for any $p \geq 1$, a variant of the total variation distance,
adapted to elements of 
${\mathcal M}_p({\mathbb R}^n) \coloneqq 
\{ \mu \in {\mathcal M}_+({\mathbb R}^n), \ M_p(\mu) < + \infty\}$,  where here and throughout  
$M_p(\mu) \coloneqq
\int_{{\mathbb R}^n}
\vert x \vert^p \dd \mu(x)$.  For such a  $p$, we let 
(the proof of the fact that the right-hand side below defines a distance is left to the reader):
\begin{equation} \label{def:distance-d_p}
d_p(\mu,\nu)
\coloneqq
\sup_{ \varphi}
\int_{{\mathbb R}^n} \varphi(x) \dd \left( \mu - \nu \right) (x), \quad \mu,\nu 
\in {\mathcal M}_p({\mathbb R}^n),
\end{equation}
where 
the supremum is taken over measurable functions 
$\varphi : {\mathbb R}^n \rightarrow {\mathbb R}$ such that $\left\vert \varphi(x) \right\vert \leq 1 + \vert x \vert^p$.

We also use the standard $p$-Wasserstein distance, when restricted 
to subsets of measures with equal mass. 
Below, we refer to these subsets as ``isomass subsets''. 
For the same $p$ as above, and for $\mu$ and $\nu$
in ${\mathcal M}_p({\mathbb R}^n)$ such that 
$\mu({\mathbb R}^n) = \nu({\mathbb R}^n)$, we let
\begin{equation*}
W_{p}(\mu,\nu) \coloneqq \inf_{\substack{ \pi \in 
{\mathcal M}_p({\mathbb R}^n \times {\mathbb R}^n), \\  \pi \circ e_1^{-1} = \mu, \;\pi \circ e_2^{-1}=\nu}} \;
    \int_{{\mathbb R}^n \times {\mathbb R}^n} \vert x-x' \vert^{p} \dd \pi(x,x')
    \leq \upsilon,
    \end{equation*}
where $e_i: {\mathbb R}^n \times {\mathbb R}^n \ni (x_1,x_2) \mapsto x_i$, 
for $i=1,2$.

We combine the two distances $d_p$ and $W_p$ by considering functions, defined on ${\mathcal M}_p({\mathbb R}^n)$, that are continuous with respect to $d_p$ on the entire 
${\mathcal M}_p({\mathbb R}^n)$, and that are continuous with respect to $W_p$ on any isomass subset of ${\mathcal M}_p({\mathbb R}^n)$. We prove in Subsection  \ref{se:app:gen:wasserstein} of the Appendix that those functions are continuous with respect to the so-called
generalized $p$-Wasserstein distance, and conversely. That said, we feel easier, in our specific framework, to use separately the two distances $d_p$ and $W_p$, instead of the 
single generalized Wasserstein distance.

\paragraph{Assumptions}
We now state the required assumptions on $G$. We still assume \ref{hyp:b}-\ref{hyp:gamma} to hold. We recall that, the parameter $r$ used throughout, is defined in \ref{hyp:sigma}.
\begin{enumerate}[label*=A\arabic*,resume]
    \item \label{assumption:g} 
    We assume that there exist two functions
\begin{equation*}
\begin{split}    
\frac{\delta G}{\delta \mu} : {\mathcal M}_{2-r}({\mathbb R}^n) \times {\mathbb R}^n \rightarrow {\mathbb R}, 
\quad
\partial_\mu G : {\mathcal M}_{2-r}({\mathbb R}^n) 
\times {\mathbb R}^n \rightarrow {\mathbb R}^n,
    \end{split}  \end{equation*}
with $\delta G/\delta \mu$
being 
differentiable in the second argument when the first one is fixed such that, 
for any $\mu \in {\mathcal M}_{2-r}({\mathbb R}^n)$ and any $x \in {\mathbb R}^n$,
\begin{equation}
\label{eq:flat:derivative:def}
\frac{\delta G}{\delta \mu}(\mu,x) = \frac{\dd}{\dd \varepsilon} \vert_{\varepsilon = 0+}
G\left( \mu +\varepsilon \delta_x \right), 
\end{equation}
where $\delta_x$ is the delta mass at point $x$,
and 
\begin{equation}
\label{eq:lions:derivative:def}
\partial_\mu G(\mu,x) = \nabla_x  \frac{\delta}{\delta \mu} G(\mu,x).
\end{equation}

a) We assume that 
    these three mappings 
    satisfy the following growth conditions:
    \begin{equation} \label{ass:growth-G-mean-field}
        \begin{split}
        |G(\mu)| &\leq L\left(1 +M_{2-r}(\mu)\right), \\
        -L 
        \left(1 +
        M_{2-r}(\mu) 
        +  |x|\right) \leq \frac{\delta G}{\delta \mu}(\mu)(x)  &\leq L \left(1 + M_{2-r}(\mu) 
        +  |x|^{2-r}\right), \\
        \left |\partial_\mu G(\mu,x) \right| & \leq L\left(1 +
        M_{2-r}(\mu)
        +  |x|^{1-r} \right).
        \end{split}
    \end{equation}

    b)
    We assume that
$G$ is continuous with respect to $d_{2-r}$ on the entire ${\mathcal M}_{2-r}({\mathbb R}^n)$, and continuous with respect to $W_{2-r}$ on isomass subsets. 

    As for the derivative 
    $\delta G / \delta \mu$ (which is already required to be locally Lipschitz, uniformly in $\mu$, thanks to \eqref{ass:growth-G-mean-field}), we assume it to be continuous with
    respect to 
     $d_{2-r}$, 
     with a modulus of continuity 
     that grows 
     at most 
     like $\vert x \vert^{2-r}$, 
and that is uniform in 
     $\mu \in {\mathcal M}_{2-r}({\mathbb R}^n)$ satisfying $M_{2-r}(\mu) \leq C$ for some $C >0$. More precisely, 
for any $\varepsilon >0$ and $C>0$, we assume that there exists $\upsilon>0$,
such that, for any  $\mu, \mu' \in \mathcal{M}_{2-r}(\mathbb{R}^n)$ satisfying  
$M_{2-r}(\mu),M_{2-r}(\mu') \leq C$
and 
$d_{2-r}(\mu,\mu') \leq \upsilon$, we have
    \begin{equation} \label{ass:unif-flat-dG-mean-field}
        \sup_{x \in \mathbb{R}^n}\left[ \frac1{1+\vert x \vert^{2-r}} \left|\frac{\delta G}{\delta \mu}(\mu)(x) -  \frac{\delta G}{\delta \mu}(\mu')(x)\right|\right] \leq \varepsilon.
    \end{equation}
We also assume that, for any $x \in {\mathbb R}^n$, 
$\mu \mapsto [\delta G / \delta \mu](\mu,x)$ is continuous with respect to $W_{2-r}$ on isomass subsets. 

At last, we require that, when the first argument is restricted to
the isomass subset
$\{ \mu \in {\mathcal M}_{2-r}({\mathbb R}^n), 
\mu({\mathbb R}^n) = c \}$, for some
$c \geq 0$, 
the function
    $\partial_\mu G$
is locally Lipschitz continuous in $(\mu,x)$ in the following sense: for any   $C>0$, there exists $L_C\geq 0$ such that, for any $\mu, \mu' \in \mathcal{M}_{2-r}(\mathbb{R}^n)$
    with 
    $\mu({\mathbb R}^n) = 
    \mu'({\mathbb R}^n)$
    and $M_{2-r}(\mu),M_{2-r}(\mu')\leq C$,
    and any $x,x' \in {\mathbb R}^n$,  
    \begin{equation} \label{ass:unif-Lions-dG-mean-field}
    \begin{split}
        \frac1{1+ \vert x \vert^{1-r}}
        \left|\partial_\mu G(\mu,x) -  \partial_\mu G(\mu',x)\right|
       & \leq L_C W_{2-r}(\mu,\mu'),
       \\
       \left|\partial_\mu G(\mu,x) -  \partial_\mu G(\mu,x')\right| &\leq L_C \vert x-x' \vert.
       \end{split}
    \end{equation}
c)     We finally assume that $G$ is flat concave, i.e.,
for $\mu,\mu' \in {\mathcal M}_{2-r}({\mathbb R}^n)$, 
    \begin{align} \label{ass:flat-monotnony}
        \int_{{\mathbb R}^n}
        \left( \frac{\delta G}{\delta \mu}(\mu,x) - 
        \frac{\delta G}{\delta \mu}(\mu',x)
        \right) \dd ( \mu - \mu')(x) &\leq 0,
        \end{align}
and $G$ is displacement convex on isomass subsets, i.e., for any $\mu,\mu' \in {\mathcal M}_{2-r}({\mathbb R}^n)$ with $\mu({\mathbb R}^n)=
\mu'({\mathbb R}^n)$, 
for 
any measure $\pi
\in {\mathcal M}_{2-r}({\mathbb R}^n \times {\mathbb R}^n)$ with 
$\mu$ and $\mu'$ as marginals, 
\begin{align}
        \int_{{\mathbb R}^n} \int_{{\mathbb R}^n} \left(\partial_{\mu} G(\mu, x) - \partial_{\mu} G(\mu',x') \right) \cdot (x-x') \dd \pi(x,x') &\geq 0. \label{ass:lions-monotnony}
    \end{align}
\end{enumerate}

\begin{remark} 
    For presentation purpose we only consider a mean field terminal cost. But one could also consider mean field running cost of the separated form
    \begin{equation*}
        \ell'(t,\psi_t,\mu_t) = \ell(t,\psi_t) + c(X_t,\mu_t),
    \end{equation*}
    where $\mu_t$ is the marginal law of $X_t$ under the probability measure $q_T \mathbb{P}$ induced by  Nature. 
    The assumptions on $c$ should be analogous to the assumptions required for $g$ above (growth, flat differentiable and Lions differentiable, with the appropriate regularity, flat concave and displacement convex). 
\end{remark}

\begin{remark}
\label{rem:G(c)}
The following comments are in order. 

The first remark is that it suffices, for our purpose, to have all the 
above conditions satisfied for 
$\mu$ and $\mu'$ of mass less than $\exp(\alpha T)$. This follows from the fact that, in our applications, ${\mathbb E}[q_T] \leq \exp(\alpha T)$. 

The second observation is that the notion of displacement convexity, as mentioned in \eqref{ass:lions-monotnony}, is usually reserved to functions defined on the space of probability measures. In 
\eqref{ass:lions-monotnony}, we can easily recover the case when 
$\mu$ and $\mu'$ are probability measures by normalizing them. 
Indeed, 
for a given $c>0$ representing the common mass of $\mu$ and $\mu'$, we can consider the function 
$G^{(c)} : {\mathcal P}_{2-r}({\mathbb R}^n) \ni \mu \mapsto G(c\mu)$. Obviously, the standard flat and Lions derivatives (according to their usual definitions for functionals defined on 
${\mathcal P}_{2-r}({\mathbb R}^n)$,  the common construction of the Lions derivative being restricted to the case $r=0$) are 
\begin{equation*}
\frac{\delta G^{(c)}}{\delta \mu}(\mu,x) = 
c \frac{\delta G}{\delta \mu}(\mu,x), \quad 
\partial_\mu G^{(c)}(\mu,x) = 
c \partial_\mu G(\mu,x). 
\end{equation*}
If $G^{(c)}$ has second-order derivatives in $\mu$ and $x$, then it satisfies \eqref{ass:lions-monotnony} for any two probability measures $\mu$
and $\mu'$ if 
\begin{equation}
\label{eq:hessian:example}
\begin{split}
&\int_{{\mathbb R}^n \times {\mathbb R}^n}
{\rm Tr} \left[ 
\partial^2_{\mu} G^{(c)}(\mu,x,x') \beta(x) \otimes \beta(x')  \right] \dd \mu(x) \dd \mu(x') 
\\
& + \int_{{\mathbb R}^n}
{\rm Tr} \left[ \partial_x 
\partial_{\mu} G^{(c)}(\mu,x) \beta(x) \otimes \beta(x) \right] \dd \mu(x) \geq 0,
\end{split}
\end{equation}
for any bounded measurable function $\beta$ from ${\mathbb R}^n$ to itself. The above can be found in 
\cite[Chapter 5]{carmona2018probabilistic-v1}, when $r=0$. 
Returning back to unnormalized measures (i.e., changing $\mu$ into $c \mu$), it easy to see that, when 
$r=0$, \eqref{ass:lions-monotnony} is true (whathever the mass of $\mu$ and $\mu'$) if 
\eqref{eq:hessian:example}
is true with $G$ being substituted for $G^{(c)}$. 
In fact, \eqref{eq:hessian:example} remains also a sufficient condition when $r=1$: It 
implies 
\eqref{ass:lions-monotnony}
when $\mu$ and $\mu'$ therein have finite second-order moments; by a standard 
approximation argument, the inequality remains true when $\mu$ and $\mu'$ are just in ${\mathcal M}_1({\mathbb R}^n)$.
Below, we thus call Hessian of $G$ in the direction $\beta$ the 
quantity
\begin{equation*}
\begin{split}
{\mathcal H}_G(\beta) \coloneqq & 
\int_{{\mathbb R}^n \times {\mathbb R}^n}
{\rm Tr}\left[ \partial^2_{\mu} G(\mu,x,x') \beta(x) \otimes \beta(x')\right] \dd \mu(x) \dd \mu(x') 
\\
& + \int_{{\mathbb R}^n}
{\rm Tr}\left[
\partial_x 
\partial_{\mu} G(\mu,x) \beta(x) \otimes \beta(x) 
\right]\dd \mu(x).
\end{split}
\end{equation*}

\end{remark}

\paragraph{Constructing flat concave and displacement convex functions}

We first note that any linear functional of the form
\begin{equation}
\label{eq:linear:example}
G_1(\mu)=\int_{\mathbb{R}^n} v_1(x)\dd \mu(x),
\end{equation}
where \(v_1:\mathbb{R}^n\to\mathbb{R}\) is smooth and convex, is an ideal candidate to satisfy Assumption~\ref{assumption:g}. Indeed it is flat concave as it is linear in $\mu$ and displacement convex by convexity of $v_1$. The main point is to check that $v_1$ satisfies the required integrability properties, depending on whether $r=0$ or $r=1$, which prompts us to distinguish between these two cases  below.

Regardless of the integrability properties, 
$G_1$ satisfies 
\begin{equation*}
\frac{\delta G_1}{\delta \mu}(\mu,x) = v_1(x), 
\end{equation*}
for any $x \in \mathbb{R}^n$, so that \eqref{ass:flat-monotnony} is trivially satisfied, and 
\begin{equation*}
\partial_\mu G_1(\mu,x) = \nabla_x v_1(x), 
\end{equation*}
so that \eqref{ass:lions-monotnony} is expected to be satisfied if 
$v_1$ is convex.
In particular, the Hessian 
${\mathcal H}_{G_1}(\beta)$ is equal to 
\begin{equation}
\label{eq:Hessian:G1}
{\mathcal H}_{G_1}(\beta)
= \int_{{\mathbb R^n}}
{\rm Tr}\left[ 
\nabla_x^2 v_1(x) \beta(x) \otimes \beta(x)\right] \dd \mu(x),
\end{equation}
which is obviously non-negative when $v_1$ is convex. 

Before we discuss more in depth the integrability properties, we notice, as $G_1$ is linear in $\mu$, that any composition of $G_1$
 by a (smooth) concave function 
 $\varphi : {\mathbb R}
\rightarrow {\mathbb R}$ 
is expected to be flat concave. Such an example can be written as 
\begin{equation}
\label{eq:concave:composition:linear:example}
G_2(\mu) = \varphi\left( \int_{{\mathbb R}^n} v_2(x) \dd \mu(x) 
\right),
\end{equation}
where $v_2$ satisfies the required integrability constraints (similar to $v_1$, as discussed below), 
and $\varphi$ is smooth and concave. 
In this situation,  
we have (at least formally), 
\begin{equation*}
\partial_\mu G_2(\mu,x) 
= 
\varphi'\left(\int_{{\mathbb R}^n} v_2(x) \dd \mu(x)  \right) 
\nabla_x v_2(x), 
\end{equation*}
and then,
\begin{equation*}
\begin{split}
&\partial_\mu^2 G_2(\mu,x,x') 
= 
\varphi''\left(\int_{{\mathbb R}^n} v_2(x) \dd \mu(x) \right) 
\nabla_x v_2(x) \otimes 
\nabla_x v_2(x'),
\\
&\nabla_x \partial_\mu G_2(\mu,x) 
=
\varphi'\left(\int_{{\mathbb R}^n} v_2(x) \dd \mu(x) \right) 
\nabla_x^2 v_2(x).\end{split}
\end{equation*}
In particular,
by Cauchy-Schwarz inequality, it is quite easy to see that, for 
$\mu({\mathbb R}^n) = \mu'({\mathbb R}^n) \leq \exp(\alpha T)$, 
\begin{equation}
\label{eq:Hessian:G2}
\left\vert 
{\mathcal H}_{G_2}(\beta) 
\right\vert \leq C(\varphi,v_2,\mu)
\int_{{\mathbb R}^n}
\left(\vert \nabla_x v_2(x) \vert^2
+ \vert \nabla_x^2 v_2(x) 
\vert 
\right) \vert \beta(x) \vert^2 \dd \mu(x),
\end{equation}
where $ C(\varphi,v_2,\mu) \coloneqq
\max(\vert \varphi' \vert (\int_{{\mathbb R}^n} v_2 \dd \mu),
\exp(\alpha T)
\vert \varphi'' \vert (\int_{{\mathbb R}^n} v_2 \dd \mu))$.

To produce a wider class class of functions that 
are flat concave and displacement convex, we can sum 
$G_1$ and $G_2$. Indeed, we observe that $G_1+G_2$ is always flat concave. To obtain that the sum is displacement convex, we only need to ensure that 
\begin{equation*}
{\mathcal H}_{G_1}
(\beta)
 +
 {\mathcal H}_{G_2}
(\beta) \geq 0.
\end{equation*}
Combining 
\eqref{eq:Hessian:G1}
and 
\eqref{eq:Hessian:G2}, the latter inequality holds true if 
\begin{equation}
\label{eq:Hessian:v1:v2}
\forall x,y \in {\mathbb R}^n, \quad {\rm Tr}\left[ \nabla^2_x v_1(x) y \otimes y \right] \geq 
C(\varphi,v_2,\mu)
\left(\vert \nabla_x v_2(x)\vert^2
+ \vert \nabla_x^2 v_2(x) 
\vert 
\right)
\vert y \vert^2. 
\end{equation}
We stress that the inequality must be true for any 
$\mu$ with a mass less than $\exp(\alpha T)$. This puts an additional constraint due to the dependence of the constant $C(\varphi,v_2,\mu)$ on $\mu$. That said, when 
$v_2$ is bounded, the constant $C(\varphi,v_2,\mu)$ can be bounded independently of $\mu$,
since $\mu({\mathbb R}^n) \leq \exp(\alpha T)$; in that case, we can substitute 
$C(\varphi,v_2)$ for $C(\varphi,v_2,\mu)$ and then get a condition that is independent of $\mu$.

In order to give more explicit examples, we need to take into account the integrability conditions of
$\mu$, as the latter dictate the growth properties of the derivatives of $v_1$ and $v_2$.
\vskip 4pt

\textit{Case $r=0$.}
When $r=0$, a prototypal example is $v_1(x) = \lambda \vert x \vert^2/2$, for $\lambda >0$.
Then, \eqref{eq:Hessian:v1:v2}
holds if 
\begin{equation*}
C(\varphi,v_2,\mu)
\left(\vert \nabla_x v_2\vert^2
+ \vert \nabla_x^2 v_2(x) 
\vert 
\right) \leq \lambda. 
\end{equation*} 
An interesting example is
$v_2 = x$ and $\varphi(u)=-u^2/2$, in which case $G=G_1+G_2$
writes
\begin{equation*}
G(\mu) =  \frac{\lambda}2\int_{{\mathbb R}^n  }
\vert x\vert^2 \dd \mu(x) -  \frac12 \left\vert  \int_{{\mathbb R}^n}
x  \dd \mu(x) \right\vert^2. 
\end{equation*}
Here, 
\begin{equation*}
\frac{\delta G}{\delta \mu}(\mu,x) = \frac{\lambda}2 \vert x \vert^2 - x \cdot \left( \int_{{\mathbb R}^n} x' \dd \mu(x') \right), 
\quad 
\partial_\mu G(\mu,x) = 
\lambda x - 
\int_{{\mathbb R}^n}
x' \dd \mu(x'),
\end{equation*}
and it is easy to check 
\eqref{ass:growth-G-mean-field}, 
\eqref{ass:unif-flat-dG-mean-field}
and \eqref{ass:unif-Lions-dG-mean-field}. 
Moreover,
for any measure $\pi \in {\mathcal M}_2({\mathbb R}^n \times {\mathbb R}^n)$ with $\mu$ and $\mu'$ as marginal measures, 
\begin{equation*}
\begin{split}
&\int_{{\mathbb R}^n \times {\mathbb R}^n}
\left( 
\partial_\mu G(\mu,x) - 
\partial_\mu G(\mu',x') 
\right) 
\cdot (x-x') \dd \pi(x,x') 
\\
&= 
\lambda 
\int_{{\mathbb R}^n \times {\mathbb R}^n}
\vert x-x' \vert^2 
\dd \pi(x,x') -
\left\vert 
\int_{{\mathbb R}^n \times {\mathbb R}^n}
(x-x') \dd \pi(x,x') 
\right\vert^2
\\
&\geq \left( \lambda
- \mu({\mathbb R}^n) 
\right)
\int_{{\mathbb R}^n \times {\mathbb R}^n}
\vert x-x' \vert^2 
\dd \pi(x,x'),
\end{split}
\end{equation*}
with the last line following from Cauchy-Schwarz inequality, and from the fact that 
$\pi({\mathbb R}^n \times {\mathbb R}^n) 
= \mu({\mathbb R}^n)$. 
This shows that, for 
$\lambda \geq \exp(\alpha T)$, 
\eqref{ass:lions-monotnony} is satisfied for any $\mu,\mu'$ such that $\mu({\mathbb R}^n) = 
\mu'({\mathbb R}^n) \leq \exp(\alpha T)$. 
At the threshold 
$\lambda = \mu({\mathbb R}^n) (= 
\mu'({\mathbb R}^n))$,
\begin{equation*}
G(\mu) = \frac14 
\int_{{\mathbb R}^n \times {\mathbb R}^n}
\vert x - x' \vert^2 
\dd \mu(x) \dd \mu(x').
\end{equation*}
Still for $v_1(x) = \lambda \vert x\vert^2/2$, we can choose  $v_2$ bounded, with bounded derivatives of order $1$ and $2$. In that case, it is easy to check \eqref{ass:growth-G-mean-field}, 
\eqref{ass:unif-flat-dG-mean-field} and 
\eqref{ass:unif-Lions-dG-mean-field}. Moreover, 
\eqref{eq:Hessian:v1:v2} holds true if 
\begin{equation*}
\lambda \geq \sup_{\vert c \vert \leq \exp(\alpha T) \| v_2\|_\infty}
\max\left(\vert \varphi'\vert (c),
\exp(\alpha T) 
\vert \varphi''\vert(c)
\right)\left( \| \nabla_x v_2 \|^2_\infty + 
\| \nabla^2_{x} v_2 \|_\infty \right).    
\end{equation*}
\vskip 4pt

\textit{Case $r=1$.}
When $r=1$, we can no longer choose $v_1$ of quadratic growth (since $v_1$ must have a finite integral with respect to elements of 
${\mathcal M}_1({\mathbb R}^n)$). 
Instead, we can work with 
\begin{equation*}
v_1(x) = \lambda 
v_1^0(x), 
\quad \textrm{\rm with} \quad v_1^0(x) \coloneqq\left( 1 + \vert x \vert^2 \right)^{1/2}, 
\end{equation*}
for some $\lambda >0$. 
Then, 
for any coordinates $i,j \in \{1,\ldots,n\}^2$,
\begin{equation*}
\partial_{x_i} v_1^0(x) =  \frac{x_i}{v_1^0(x)}, 
\quad 
\partial^2_{x_i x_j} v_1^0(x) 
=  \frac{\delta_{i,j}}{v_1^0(x)} - \frac{x_i x_j}{v_1^0(x)^3},
\end{equation*}
where $\delta$ is the Kronecker delta here, which gives for any 
$y \in {\mathbb R}^n$,
\begin{equation*}
{\rm Tr}\left[ \nabla^2_x v_1(x) y \otimes y\right]
= \lambda 
\left[\frac{\vert y\vert^2}{v_1^0(x)}
- \frac{(x \cdot y)^2}{v_1^0(x)^3}
\right]
\geq  \lambda \frac{\vert y\vert^2}{v_1^0(x)^3}.
\end{equation*}
If we assume that  $v_2$ is bounded, with bounded derivatives of order $1$ and $2$, it is easy to check \eqref{ass:growth-G-mean-field}, 
\eqref{ass:unif-flat-dG-mean-field} and 
\eqref{ass:unif-Lions-dG-mean-field}. Moreover, 
\eqref{eq:Hessian:v1:v2} holds true if 
\begin{equation*}
\frac{\lambda}{v_1^0(x)^3} \geq \sup_{\vert c \vert \leq \exp(\alpha T) \| v_2\|_\infty}
\max\left(\vert \varphi'\vert (c),
\exp(\alpha T) 
\vert \varphi''\vert(c)
\right)\left( | \nabla_x v_2(x)|^2  + 
| \nabla^2_{x} v_2(x) | \right).    
\end{equation*}
For instance, the above holds true if 
$v_2$ is compactly supported and $\lambda$
is large enough.

\paragraph{Main result}
The following result is standard in the literature (see for instance \cite{carmona2018probabilistic-v1,RenWang}). 
For completeness, the proof is given in the Appendix, see 
Subsection 
\ref{subse:differentiability:Mp}.

\begin{lemma}
\label{lem:differentiation:flat:lions}
Let $G$
satisfy \ref{assumption:g}. Then, on the same  probability space 
$(\Omega,{\mathcal F},{\mathbb P})$
as before, for any two random variables $q,q'$ with values in ${\mathbb R}_+$, such that 
${\mathbb E}[q],{\mathbb E}[q']< + \infty$, and any random variable $X$ with values in ${\mathbb R}^n$, such that ${\mathbb E}[q \vert X \vert^{2-r}]$, 
${\mathbb E}[q' \vert X \vert^{2-r}]< + \infty$,  
\begin{equation}
\label{eq:derivatives:flat:expansion}
\begin{split}
&G\left( (q' {\mathbb P})_X\right)
    - G\left( (q  {\mathbb P})_X\right)
    \\[0.5em]
    &
\hspace{15pt}    = \int_0^1 
    \left[ 
\int_{{\mathbb R}^n}
    \frac{\delta G}{\delta \mu}\left(\left( 
(\theta q' + (1-\theta) q) {\mathbb P}\right)_X,x\right) \dd \left[ (q' {\mathbb P})_X - (q {\mathbb P})_X\right](x) \right] \dd \theta,
\end{split}
\end{equation}
and, for any random variable $X'$
with values in ${\mathbb R}^n$, such that 
${\mathbb E}[q \vert X'\vert^{2-r}] < + \infty$, 
\begin{equation}
\label{eq:derivatives:lions:expansion}
\begin{split}
&G\left( (q {\mathbb P})_{X'}\right) 
- G\left( (q {\mathbb P})_X \right) 
\\[0.5em]
&\hspace{15pt} 
= \int_0^1
{\mathbb E}
\left[ 
\partial_\mu 
G\left( (q {\mathbb P})_{\theta X' + (1-\theta) X},
\theta X' + (1-\theta) X\right) \cdot \left( X' - 
X \right) \right]
\dd \theta.
\end{split}
\end{equation}
\end{lemma}

Here is now the main result of this section:

\begin{corollary} \label{corollary:mckean-vlasov}
Let Assumptions 
\ref{hyp:b}--\ref{hyp:gamma} and \ref{assumption:g} be satisfied. Then, there exists 
    a unique saddle point
    $(\bar{q},\bar{\psi}) \in  \mathcal{Q} \times \mathcal{A}$ to the problem \eqref{pb:mckean-vlasov}.
Moreover, if a pair  
$(q,\psi) \in  \mathcal{Q} \times \mathcal{A}$ is a solution to the problem \eqref{pb:min-max-G}, then
the tuples $(\psi,p,k,X)$, obtained by solving in ${\mathscr A}$ the two decoupled equations in 
\eqref{optim:condition-dual} with the terminal condition being specified by 
\begin{equation*}
    p_T = q_T \partial_{\mu} G\left((q_T {\mathbb P})_{X_T},X_T\right),
\end{equation*}
and $(q,Y,Z)$, obtained by solving in ${\mathscr Q}$ the two decoupled equations in \eqref{optim:condition-primal} with the terminal condition being specified by 
\begin{equation*}
    Y_T = \frac{\delta G}{\delta \mu}\left((q_T {\mathbb P})_{X_T},X_T\right),
\end{equation*}
satisfy the optimality conditions in 
\eqref{optim:condition-dual}
and 
\eqref{optim:condition-primal}
respectively. Conversely,
if
$(\psi,p,k,X,q,Y,Z) \in \mathscr{A} \times  \mathscr{Q} $ is a  
solution to \eqref{optim:condition-primal}-\eqref{optim:condition-dual} with the terminal condition specified above, then the pair 
$(\psi,q) \in \mathcal{A} \times \mathcal{Q}$ is the unique solution to the problem \eqref{pb:mckean-vlasov}.
\end{corollary}

\begin{proof}
As the result is a direct application of Theorem \ref{theorem:SMP},  we just need to check that the mapping $\mathcal{G}$
defined in 
\eqref{eq:G:subsection4.1}
satisfies the Assumptions  \ref{eq:G-growth}-\ref{eq:G-concave-convex} of the previous section. 
\vskip 4pt

\noindent \textit{Step 1: $\mathcal{G}$ verifies \ref{eq:G-growth}-\ref{hyp:DgD_XG}.} Let $(q,X), (q',X') \in \mathscr{G}$ (the definition of $\mathscr{G}$ can be found in Assumption \ref{eq:G-growth}), satisfying ${\mathbb E}[q \vert X' \vert^{2-r}]$ and ${\mathbb E}[q' \vert X \vert^{2-r}]< + \infty$.
By \eqref{ass:growth-G-mean-field}, we can easily check 
\ref{eq:G-growth} with
\begin{equation*}
{\mathcal G}(q,X) = 
G\left( (q {\mathbb P})_X \right), 
\  
\delta_\mu {\mathcal G}(q,X) 
= \frac{\delta G}{\delta \mu}\left( (q {\mathbb P})_X,X \right), 
\ 
\delta_X {\mathcal G}(q,X) 
= \partial_\mu 
G \left( (q {\mathbb P})_X, 
X \right). 
\end{equation*}
In fact, the main point is to check that 
$\delta_\mu {\mathcal G}$
and $\delta_X {\mathcal G}$ are the derivatives of 
${\mathcal G}$, in the directions $q$ and $X$ respectively, as required in 
\ref{hyp:DgD_XG}. 
By Lemma 
\ref{lem:differentiation:flat:lions}, we know that
\begin{equation}
\label{eq:RMFC:derivatives}
\begin{split}
     G\left( (q' \mathbb{P})_X \right)  = \; & G\left( 
     (q \mathbb{P})_X
     \right) 
     \\[0.5em]
     & + \int_0^1
     \left[\int_{\mathbb{R}^n} \frac{\delta G}{\delta \mu}\left( (q^\theta {\mathbb P})_{X},x \right) \dd \left[ (q ' \mathbb{P})_X - (q \mathbb{P})_X\right] (x) \right] \dd \theta,
       \\[0.5em]
    G\left( (q \mathbb{P})_{X'}\right)  = \; &  G\left( (q \mathbb{P})_{X}\right) + \int_0^1\mathbb{E}\left[q  \partial_{\mu} G \left( ( q \mathbb{P})_{X^\theta},X^\theta \right) \cdot (X'-X) \right] \dd \theta,
\end{split}
\end{equation}
with the convenient notation $X^\theta \coloneqq \theta X' + (1-\theta)X$ and $q^\theta \coloneqq \theta q' + (1-\theta) q$. 
Let us first prove the first line in \ref{hyp:DgD_XG}.
For a constant $C \geq 0$
and for $q$, $X$ and $X'$
satisfying 
${\mathbb E}[q \vert X \vert^{2-r}], {\mathbb E}[ q \vert X'\vert^{2-r}] \leq C$, 
we rewrite the second line in \eqref{eq:RMFC:derivatives}
as 
\begin{align}\label{eq:lions-fundamental-integration}
    G\left( ( q {\mathbb P})_{X'}\right)   =\; &  G\left( ( q {\mathbb P})_{X}\right) + \mathbb{E}\left[q  \partial_{\mu} G \left( (q \mathbb{P})_{X},X \right) \cdot (X'-X) \right] \\[0.5em]
    & + \int_0^1 \mathbb{E}\left[q  \left( \partial_{\mu} G \left(( q \mathbb{P})_{X^\theta},X^\theta \right) - \partial_{\mu} G \left( (q \mathbb{P})_{X},X \right) \right) \cdot (X'-X) \right] \dd \theta.  \nonumber
\end{align}
By  \eqref{ass:unif-Lions-dG-mean-field} we have 
\begin{equation*}
\begin{split}
    &\left| \partial_{\mu} G \left( (q \mathbb{P})_{X^\theta},X^\theta \right) - \partial_{\mu} G \left( (q \mathbb{P})_{X},X \right) \right| 
    \\[0.5em]
    &  \leq 
  L_C \left( 1+ \vert X \vert^{1-r}\right) {\mathbb E}\left[ q \vert X - X' \vert^{2-r} \right]^{1/(2-r)}
    + L_C 
\vert X' - X \vert,
\end{split}
    \end{equation*}
from which we deduce, by Cauchy-Schwarz inequality, that 
\begin{equation*}
\begin{split}
    &\left |\int_0^1 \mathbb{E}\left[q  \left( \partial_{\mu} G \left( ( q \mathbb{P})_{X^\theta},X^\theta\right) - \partial_{\mu} G \left( ( q \mathbb{P})_{X},X \right) \right) \cdot (X'-X) \right] \dd \theta \right|
    \\[0.5em]
    &\leq  
     L_C' {\mathbb E}\left[ q \vert X' - X \vert^{2}\right],
     \end{split}
\end{equation*}
for a constant $L_C'$ depending on $L_C$ and $C$. Inserting the above display in 
\eqref{eq:lions-fundamental-integration}, this proves 
the first line in \ref{hyp:DgD_XG}. 
We now establish the second line in \ref{hyp:DgD_XG}. 
where the modulus of continuity $\varpi_L$ might increase. 
We rewrite the first line in 
\eqref{eq:RMFC:derivatives}
as
\begin{equation}
\begin{split}
     &G\left( (q' {\mathbb P})_X\right)   =  G\left( (q {\mathbb P})_X\right) + \int_{\mathbb{R}^n} \frac{\delta G}{\delta \mu}
     \left( (q {\mathbb P})_X,x\right)
     \dd \left[ 
     (q' {\mathbb P})_X 
     - (q {\mathbb P})_X\right] (x)  
     \\[0.5em]
     &+ \int_0^1 \left[ \int_{\mathbb{R}^n} \left(\frac{\delta G}{\delta \mu}\left(
(q^\theta {\mathbb P})_X,x 
     \right) -  \frac{\delta G}{\delta \mu}\left(
     (q {\mathbb P})_X,x 
     \right) \right)\dd \left[ (q' {\mathbb P})_X - (q {\mathbb P})_X\right](x)
     \right] \dd \theta. 
     \end{split}
     \label{eq:flat-fundamental-integration}
\end{equation}
By assumption \eqref{ass:unif-flat-dG-mean-field}, 
\begin{align*}
     \left|\frac{\delta G}{\delta \mu}\left(
     (q^\theta {\mathbb P})_X,x\right)-  \frac{\delta G}{\delta \mu}\left((q {\mathbb P})_X,x\right)\right| & \leq
     \left( 1 +\vert x \vert^{2-r} \right) 
     \varpi \left( d_{2-r}\left( (q^\theta {\mathbb P})_X, (q {\mathbb P})_X\right)\right) 
     \\
     &\leq \left(1 + \vert x \vert^{2-r} \right) \varpi \left( \sup_{\varphi} 
     {\mathbb E}
\left[ \varphi(X) (q-q')\right]
\right)
     \\[0.5em]
     & \leq \left(1 + \vert x \vert^{2-r} \right)\varpi \left( 
     {\mathbb E}\left[ (1+\vert X \vert^{2-r}) \vert q-q'\vert\right]\right),
\end{align*}
where 
$\varpi$ in the first line is the modulus of continuity of $\delta G/\delta \mu$ in the first argument, 
and is (here) independent of $x$ but depends on $q$, $q'$ and $X$
via $C$. As for $\varphi$ on the second line, it satisfies 
$\vert \varphi(x) \vert \leq 1+ \vert x\vert^{2-r}$. 

Combining the last two displays, we obtain 
\begin{equation*}
\begin{split}
     G\left( (q'{\mathbb P})_X\right) = \; & G\left(
     (q {\mathbb P})_X
     \right) + \int_{\mathbb{R}^n} \frac{\delta G}{\delta \mu}\left( (q {\mathbb P})_X
     ,x
     \right)  \dd 
     \left[
     (q' {\mathbb P})_X
     -
     (q {\mathbb P})_X
     \right]     
     (x)
     \\[0.5em]
&     + 
     o
     \left( 
     {\mathbb E}
     \left[ ( 1+ \vert X\vert^{2-r}) \vert q-q' \vert \right]
     \right),
\end{split}
\end{equation*}
where $o(r)/\vert r \vert \rightarrow 0$ as $r$ tends to $0$, uniformly in $q$, $q'$ and $X$ satisfying the two bounds 
${\mathbb E}[q \vert X \vert^{2-r}],{\mathbb E}[q' \vert X \vert^{2-r}] \leq C$. 
This proves 
the second line in \ref{hyp:DgD_XG}.
\vskip 4pt

\noindent \textit{Step 2: $\mathcal{G}$ verifies \ref{eq:G-concave-convex}.} Consider again $(q,X), (q',X') \in \mathscr{G}$  such that  ${\mathbb E}[q \vert X' \vert^{2-r}]$ and ${\mathbb E}[q' \vert X \vert^{2-r}]< + \infty$.
By \eqref{eq:lions-fundamental-integration} and then
\eqref{ass:lions-monotnony}, we have 
\begin{align*}
    G\left( (q {\mathbb P})_{X'}  \right) 
    = \; &  
    G\left( (q {\mathbb P})_{X}  \right)  + \mathbb{E}\left[q  \partial_{\mu} G \left( (q {\mathbb P})_X,X \right) \cdot (X'-X) \right]  \\[0.5em]
    & +  \int_0^1 \frac{1}{\theta}  \mathbb{E}\left[ q \left( \partial_{\mu} G \left(
    (q^\theta {\mathbb P})_X, X^\theta \right) 
-    
 \partial_{\mu} G \left( (q {\mathbb P})_X,X \right)  \right) \cdot (X^\theta -X) \right] \dd \theta 
    \\[0.5em]
    \geq & \;  
    G( (q {\mathbb P})_X) + \mathbb{E}\left[q  \partial_{\mu} G \left(
    ( q {\mathbb P})_X
    ,X \right) \cdot (X'-X) \right],
\end{align*}
which proves the second condition in \ref{eq:G-concave-convex}. 
In order to establish the first condition in 
\ref{eq:G-concave-convex}, we notice that, in 
\eqref{eq:flat-fundamental-integration}, 
\begin{align*}
    ( {q^\theta} {\mathbb P})_{X} & = \theta (q' {\mathbb P})_X  + (1-\theta) (q {\mathbb P})_X
     = (q {\mathbb P})_X
    + \theta \left((q' {\mathbb P})_X-
    (q {\mathbb P})_X    \right),
\end{align*}
and then, 
\begin{align*}
    G\left((q' {\mathbb P})_X\right) 
    &=  G\left((q {\mathbb P})_X\right) + \int_{\mathbb{R}^n} \frac{\delta G}{\delta \mu}\left((q {\mathbb P})_X,x\right) \dd \left[ (q' {\mathbb P})_X -  (q {\mathbb P})_X \right](x) 
    \\[0.5em]
     + & \int_0^1 \int_{\mathbb{R}^n} \frac{1}{\theta} \left[ \frac{\delta G}{\delta \mu}\left((q^\theta {\mathbb P})_X,x\right) - \frac{\delta G}{\delta \mu}\left((q {\mathbb P})_X,x\right)\right]\dd \left[ (q^\theta\mathbb{P})_{X}  -
    (q \mathbb{P})_{X}
    \right]
    (x) \dd \theta
    \\[0.5em]
    &\leq  G((q {\mathbb P})_X) + \int_{\mathbb{R}^n} \frac{\delta G}{\delta \mu}\left((q {\mathbb P})_X,x\right) \dd \left[(q' {\mathbb P})_X - (q {\mathbb P})_X\right](x),
\end{align*}
where the last line follows by the monotonicity assumption \eqref{ass:flat-monotnony}.
\end{proof}

\paragraph{Perspectives}

\textit{Common noise.}
Our approach, based on the stochastic maximum principle, would allow us to introduce a common noise into the model in a direct manner. Similar to \cite{cardaliaguet2019master}, we can think of an additive white noise manifesting in the form of an extra term \( \sigma^0 \dd W^0 \) in the dynamics of $X$, where \( W^0 \) is a Brownian motion independent of \( (W, \eta) \). Alternatively, we could randomize the coefficients independently of \( (W, \eta) \).
In any case, this additional source of randomness could be represented by tensorizing the space \( \Omega \) (which carries the idiosyncratic noises) with a new space \( \Omega^0 \) (which carries the common noise). This approach is used in \cite{carmonadelaruev2}.

To incorporate this, the following changes would be necessary:

\begin{itemize}
\item The terminal cost would read 
\begin{equation*}
    {\mathbb E}^0[{\mathcal G}(q_T(\omega^0, \cdot), X_T(\omega^0, \cdot))]=
{\mathbb E}^0[
G((q_T(\omega^0,\cdot) 
{\mathbb P})_{X_T(\omega^0,\cdot)})],
\end{equation*}
for any element \( \omega^0 \in \Omega^0 \). This accounts for the fact that the common noise induces a conditioning.
\item Assuming, without significant loss of generality, that the filtration on \( \Omega^0 \) is generated by a Brownian motion (denoted \( W^0 \)), all the backward equations would include an additional penalization term in the form of a stochastic integral with respect to \( W^0 \), i.e., \( \int_0^\cdot Z^0_s \cdot \dd W^0_s \). If the filtration were not Brownian, the penalization could instead be written as a (possibly discontinuous) martingale, which would make the model more complex to study.
\item 
If the model were extended to incorporate risk aversion with respect to the common
noise, the dynamics of $q^0$ would include an additional term of the form
$q^0_s Z^{0,\star}_s \cdot \dd W^0_s$. As a consequence, the driver of the BSDE for
$Y$ would also depend on the additional variable $Z^0$, where $Z^0$ arises from the
martingale representation above. Accordingly, both the adjoint process and the
Hamiltonian would have to be modified to account for this additional dependence.
\end{itemize}
 \textit{$N$-particles system.}
 A natural question is how the robust mean field model arises as the limit of an
$N$-particle control problem. A thorough and rigorous analysis of this convergence
process is beyond the scope of the present article and is left for future work.
Nevertheless, we hope that the formal arguments provided in the second example of
Subsection~\ref{sec:example-application}, as well as in the forthcoming examples
presented in Subsection~\ref{subsec:examples:mean field}, will help the reader
to identify, at least at an intuitive level, the underlying mechanisms from which
the mean field model can be expected to emerge.

\subsection{Examples}
\label{subsec:examples:mean field}
In this paragraph, we provide two examples that lead to a robust mean field control problem.

\paragraph{Feynman-Kac path particle models.}

Inspired by the monograph 
\cite{DelMoral}, 
we consider a large system of $N$ weakly interacting 
$d$-dimensional particles,
with Gibbs distributions on the path space 
${\mathcal C}([0,T],\mathbb{R}^n)^N$:
\begin{equation*}
\begin{split}
\exp \biggl( & \beta \biggl[ 
\frac1{2N}\sum_{i,j=1}^N 
G \left(X_T^{i,\psi^i},X_T^{j,\psi^j}
\right) 
\\
& + \frac1{2N}\sum_{i,j=1}^N  \int_0^T 
F \left(X_t^{i,\psi^i},X_t^{j,\psi^j}
\right) \dd t 
+ \frac1{2}\sum_{i=1}^N  \int_0^T 
\vert \psi^i_t \vert^2 
\dd t \biggr] \biggr) 
\cdot {\mathbb P}^{\times N},
\end{split}
\end{equation*}
where $\beta >0$ and ${\mathbb P}$ is the Wiener measure on 
$\Omega\coloneqq {\mathcal C}([0,T],\mathbb{R}^n)$, 
$\psi^i : {\mathcal C}([0,T],\mathbb{R}^n)^N \rightarrow 
{\mathbb R}^d$
is a progressively-measurable control for each 
$i \in\{1,\ldots,N\}$ and
\begin{equation*}
X_t^{i,\psi^i}(\omega^1,\ldots,\omega^N) 
=
\omega_t^i + \int_0^t \psi^i_s
(\omega^1,\ldots,\omega^N)  \dd s, \quad t \in [0,T]. 
\end{equation*}
The goal is then to minimize, with respect to $(\psi^1,\ldots,\psi^N)$, the free energy
given (up to a logarithmic transformation) by
\begin{equation*}
\begin{split}
{\mathbb E}^{\times N}
\biggl[  \exp \biggl(&  \beta \biggl[ 
\frac1{2N}\sum_{i,j=1}^N 
G \left(X_T^{i,\psi^i},X_T^{j,\psi^j}
\right) 
\\
& + \frac1{2N}\sum_{i,j=1}^N  \int_0^T 
F \left(X_t^{i,\psi^i},X_t^{j,\psi^j}
\right) \dd t 
+ \frac1{2}\sum_{i=1}^N  \int_0^T 
\vert \psi^i_t \vert^2 
\dd t \biggr] \biggr) \biggr].
\end{split}
\end{equation*}
In order to simplify, we assume below that
the running cost $F$ is equal to $0$, but the analysis 
would be the same if $F$ were not trivial. 

Thanks to Donsker-Varadhan's formula (see \eqref{eq:donsker-varadhan}) for a remainder, the free energy can be rewritten in the form 
\begin{equation}
\label{eq:example3:def}
\begin{split}
\sup_{q^{(N)}}
\biggl\{ 
& \beta 
{\mathbb E}^{\times N}
\left[  q_T^{(N)}
\left( 
\frac1{2N} \sum_{i,j=1}^N G\left(X_T^{i,\psi^i},
X_T^{j,\psi^j}\right) 
+ \frac1{2}\sum_{i=1}^N  \int_0^T 
\vert \psi^i_t \vert^2 
\dd t 
\right) \right]
\\
& - \mathrm{H} \left( q^{(N)}  {\mathbb P}^{\times N} \vert {\mathbb P}^{\times N}
\right)
\biggr\},
\end{split}
\end{equation}
where $q^{(N)}$ is taken in the space of densities on 
$\Omega^{\times N}$ with a finite entropy. 
We observe that the normalization in the potential is consistent with that used in the paragraph on risk measures in Subsection 
\ref{sec:example-application}. This therefore constitutes a nonlinear version (in the sense that the potential now depends on the empirical measure through a second-order functional) of the previous example; for simplicity, this example is also presented in the case of uncontrolled volatility.
\vskip 5pt

\noindent \textit{Characterization of the saddle-point.} If $G$, viewed as a real-valued function on ${\mathbb R}^n \times {\mathbb R}^n$, is smooth, convex
and at most of quadratic growth, 
then Theorem \ref{theorem:SMP}  applies to the minimization of the above quantity. The saddle point 
of the min-max problem 
(over $q^{(N)}$ and $(\psi^1,\ldots,\psi^N)$)
can be characterized via
a 6-tuple 
$$\bigl(Y^{(N)},Z^{(N)},X^{(N)},p^{(N)},k^{(N)},q^{(N)}\bigr),$$
with
\begin{equation*}\begin{array}{cc}
     Z^{(N)} = (Z^{(N),i})_{i=1,\ldots,N}, & X^{(N)}=(X^{(N),i})_{i=1,\ldots,N}, \\
     p^{(N)} = (p^{(N),i})_{i=1,\ldots,N}, & k^{(N)}=(k^{(N),i,j})_{i,j=1,\ldots,N},
\end{array}
\end{equation*}
solution of
(using the notation  $B^j_t(\omega^1,\ldots,w^N)\coloneqq\omega^j_t$, for $j=1,\ldots,N$ and $t \in [0,T]$)
\begin{equation}
\label{eq:example3:FBSDE:N}
    \left\{ \begin{array}{rl}
- \dd Y_t^{(N)} 
&= \left( \frac12 \sum_{j=1}^N \vert Z^{(N),j} \vert^2 
+ \frac{\beta}2   \sum_{j=1}^N \vert \psi_t^j \vert^2 
\right) \dd t - \sum_{j=1}^N Z_t^{(N),j} \cdot \dd B_t^j,
\\[0.5em]
\dd X_t^{(N),i} &= \psi_t^i \dd t 
+ \dd B_t^i,
\\[0.5em]
- \dd p_t^{(N),i} &= - \sum_{i=1}^N k_t^{(N),i,j} \cdot 
\dd B_t^j, 
\\[0.5em]
\dd q_t^{(N)} &= q_t^{(N)} \sum_{i=1}^N Z_t^{(N),i} \cdot \dd B_t^i,
   \end{array} \right. 
   \end{equation}
   for $t \in [0,T]$, 
   with 
   the optimality condition 
   $\psi_t^i = -  [ \beta q_t^{(N)}  ]^{-1} p_t^{(N),i}$
   and 
   the boundary conditions
\begin{equation}
\label{eq:example3:BC:N}
    \left\{ \begin{array}{rl}
Y_T^{(N)} &= \frac{\beta}{2N}
\sum_{i,j=1}^N G\left(X_T^{(N),i},X_T^{(N),j}\right)
\\[0.5em]
p_T^{(N),i}
&= q_T^{(N)} \frac{\beta}{2N}
\sum_{j=1}^N
\left[ 
\partial_x G\left(X_T^{(N),i},X_T^{(N),j}\right)
+
\partial_y G\left(X_T^{(N),j},X_T^{(N),i}\right)
\right]. 
\end{array}
\right.
\end{equation}

Similar to the discussion initiated in Subsection 
\ref{sec:example-application}, the question here is to understand, at least informally, how the above system is connected to the mean field control problem described in Subsection \ref{sec:mean-field-control}. 
To better appreciate 
the intuitive arguments that we present, it is worth mentioning from the analysis carried out in Subsection \ref{sec:necessary-sufficient} (see in particular Lemma \ref{lemma:BSDE-Y-Z})
that the solution of the BSDE 
\eqref{eq:BSDE:compact:Y:Z} is understood via the product $qY =(q_t Y_t)_{t \in [0,T]}$. This prompts us to
consider, here, the product $q^{(N)} Y^{(N)}$. The aforementioned Lemma 
\ref{lemma:BSDE-Y-Z} says that 
$q^{(N)} Y^{(N)}$ is a semi-martingale under 
the probability measure
$q^{(N)}_T {\mathbb P}$, satisfying 
\begin{equation}
\label{eq:qN:YN}
    \left\{ \begin{array}{rl}
    - \dd \left[ q^{(N)}_t Y_t^{(N)} \right] &= 
 \left(- \frac12 \sum_{j=1}^N q^{(N)}_t \vert Z^{(N),j} \vert^2
+ \frac{\beta}2   \sum_{j=1}^N q^{(N)}_t\vert \psi_t^j \vert^2 
\right) \dd t 
\\
&\hspace{15pt} -
q^{(N)}_t
(1+ Y_t^{(N)}) 
\sum_{j=1}^N  Z_t^{(N),j}  \cdot \dd B_t^j,
\\
    q_T^{(N)} Y_T^{(N)}
    &=\frac{\beta}{2N}
q_T^N \sum_{i,j=1}^N G\left(X_T^{(N),i},X_T^{(N),j}\right).
\end{array}
\right.
\end{equation}
\vskip 5pt

\noindent \textit{Corresponding robust MFC problem.}
In parallel, consider the robust mean field control problem (MFC), which we recall below for convenience: 
\begin{equation*}
     \inf_{\psi} \sup_{q} \left\{  {\mathbb E} 
     \left[
\frac{\beta}2  
     \int_{{\mathbb R}^n}
     G\left(x,y\right) (q_T {\mathbb P}_{X_T^\psi})^{\otimes 2}(\dd x,\dd y)  + \frac{\beta}2 \int_0^T q_s \vert \psi_s \vert^2 \dd s \right] - 
     H\left( q_T {\mathbb P} \vert {\mathbb P} \right)
     \right\},
\end{equation*}
where for simplicity we do not specify the sets of admissibility in which 
$q$ and $\psi$ are taken. Generally speaking, this problem is defined on the original probability space 
$(\Omega,{\mathcal F},{\mathbb P})$
equipped with the Brownian motion 
$W$, but we can consider, for each 
$i =1,\ldots,N$, the same problem 
but 
on the $i$th factor 
$\Omega^{N}$ and thus
with respect to the Brownian
motion $B^i$ instead of 
$W$. For each $i = 1,\ldots,N$, we then call $(\tilde q^i, \tilde X^i, \tilde \psi^i)$ the saddle point of the corresponding problem, which exists and is unique under the assumptions of Corollary~\ref{corollary:mckean-vlasov}. At this stage, these assumptions are taken for granted, but we will discuss its meaning in more depth at the end of this paragraph.
 Importantly, we observe that 
$(\tilde q^i,\tilde X^i,\tilde \psi^i)$ is a function of the sole $\omega^i$. The Pontryagin system characterizing 
$(\tilde q^i,\tilde X^i,\tilde \psi^i)$ reads
\begin{equation}
\label{eq:example3:FBSDE:MF}
    \left\{ \begin{array}{rl}
- \dd \tilde Y_t^{i} 
&= \left( \frac12 \vert \tilde Z^{i}_t \vert^2 
+ \frac{\beta}2     \vert \tilde \psi_t^i \vert^2 
\right) \dd t - \tilde Z_t^{i} \cdot \dd B_t^i,
\\
[0.5em]
\dd \tilde X_t^{i} &= \tilde \psi_t^i \dd t 
+ \dd B_t^i,
\\[0.5em]
- \dd \tilde p_t^{i} &= -  \tilde k_t^{i} \cdot 
\dd B_t^i,
\\[0.5em]
\dd \tilde q_t^{i} &= \tilde q_t^i  \tilde Z_t^{i} \cdot \dd B_t^i,
   \end{array} \right. 
   \end{equation}
   for $t \in [0,T]$, 
   with 
   the optimality condition 
   $\tilde \psi_t^i = -  [ \beta \tilde q_t^{i}  ]^{-1} \tilde p_t^{i}$
   and 
   the boundary conditions
\begin{equation}
\label{eq:example3:BC:MF}
    \left\{ \begin{array}{rl}
\tilde Y_T^{i} &= \displaystyle \tfrac{\beta}{2} 
\int_{{\mathbb R}^n}
\left[
G(\tilde X_T^i, z )
+ G(z,\tilde X_T^i )
\right]
\dd (\tilde q_T^i {\mathbb P})_{\tilde X_T^i}(z) 
\\
[1em]
\tilde p_T^{i}
&=
\tfrac{\beta}2 
\displaystyle \tilde q_T^{i} 
\int_{{\mathbb R}^n} \left[ 
\partial_x G\left(\tilde X_T^{i},z\right)
+
\partial_y G\left(z,\tilde X_T^{i}\right)
\right] \dd (\tilde q_T^i {\mathbb P})_{\tilde X_T^i}(z). 
\end{array}
\right.
\end{equation}
We then let 
\begin{equation*} 
\tilde Y_t^{(N)}
= \sum_{i=1}^N 
\tilde Y_t^i, 
\quad 
\tilde q_t^{(N)}
 = \prod_{i=1}^N 
 q_t^i, \quad t \in [0,T],
 \end{equation*} 
 and, following 
\eqref{eq:qN:YN}, we consider the process 
$\tilde q^{(N)} \tilde Y^{(N)}$. It satisfies 
\begin{equation}
\label{eq:tildeqN:tildeYN}
\left\{
\begin{array}{rl}
    - \dd \left[ \tilde q^{(N)}_t \tilde Y_t^{(N)} \right] &= 
 \left(- \frac12 \tilde q^{(N)}_t \sum_{i=1}^N
 \vert 
 \tilde Z^{i} \vert^2
+ \frac{\beta}2   \tilde q^{(N)}_t
\sum_{i=1}^N 
\vert  \tilde \psi_t^i \vert^2 
\right) \dd t 
\\
&\hspace{15pt} - \tilde q^{(N)}_t
(1+ Y_t^{(N)})  
\sum_{i=1}^N \tilde Z_t^{i} \cdot \dd B_t^i,
\\
[.5em]
\tilde q_T^{(N)} \tilde Y_T^{(N)}
    &=
    \frac{\beta}{2}
      \tilde q_T^{(N)}
\sum_{i=1}^N 
\int_{{\mathbb R}^n}
    \left[
G(\tilde X_T^{i},z)
+
G(z,\tilde X_T^{i})
\right]
\dd (\tilde q_T^i {\mathbb P})_{\tilde X_T^i}(z).
\end{array}\right.
\end{equation}
\vskip 5pt

\noindent \textit{Connecting the two problems.} Of course, in the above right-hand side, 
$(\tilde q_T^i {\mathbb P})_{\tilde X_T^i}$ is independent of $i$ and can be replaced by $(\tilde q_T^1 {\mathbb P})_{\tilde X_T^1}$. For simplicity, we remove below the index $1$ and merely write 
$(\tilde q_T {\mathbb P})_{\tilde X_T}$. The connection between
the form of the boundary condition for
$\tilde q_T^{(N)} \tilde Y_T^{(N)}$
in \eqref{eq:tildeqN:tildeYN}
and the form of the boundary condition for 
$q_T^{(N)} Y_T^{(N)}$ in 
\eqref{eq:qN:YN}
can be 
better understood by applying the weak law of large numbers under the probability $\tilde q_T^{(N)} {\mathbb P}$. Indeed, 
since the random variables $\tilde X_T^1,\ldots,\tilde X_T^N$ are independent under 
$\tilde q_T^{(N)}$, with 
$(\tilde q {\mathbb P})_{\tilde X_T}$ as common distribution, 
and because 
$\int_{{\mathbb R}^n \times {\mathbb R}^n}
G(x,y) (\tilde q_T {\mathbb P}_{\tilde X_T})^{\otimes 2}(\dd x,\dd y) < + \infty$ (as a consequence of 
Lemma \ref{lemma:reg-X-psi-A}),
we have 
\begin{equation} 
\label{eq:MF:LLN}
\begin{split}
\lim_{N \rightarrow + 
\infty}
{\mathbb E}
\Biggl[&\tilde q_T^{(N)}
\Biggl\vert 
\frac1{N^2}
\sum_{i,j=1}^N 
G(\tilde X_T^i, \tilde X_T^j)
\\
&\hspace{15pt} - 
\frac1{2N} 
\sum_{i=1}^N 
\int_{{\mathbb R}^n} 
\left( G(\tilde X_T^i,z) +
G(z,\tilde X_T^i) 
\right) (\tilde q_T {\mathbb P})_{\tilde X_T})(\dd z)
\Biggr\vert
\Biggr] = 0.
\end{split}
\end{equation}
Pay attention to the fact that the boundary conditions for 
$\tilde q_T^{(N)} \tilde Y_T^{(N)}$ in 
\eqref{eq:tildeqN:tildeYN}
and 
$q_T^{(N)} Y_T^{(N)}$ in 
\eqref{eq:qN:YN}
are of order $N$, whereas the two terms in the above difference are of order $1$. That said, the above display shows that $\tilde q_T^{(N)} \tilde Y_T^{(N)}$ satisfies
a boundary condition similar to the one satisfied by $q_T^{(N)} Y_T^{(N)}$ in 
\eqref{eq:qN:YN}, up to a remainder of order $o(N) = \varepsilon_N N$  with $\varepsilon_N$ converging to $0$ in $L^1$ under ${\mathbb P}$. This makes it possible to view the process $\tilde q^{(N)} \tilde Y^{(N)}$ as a `nearly solution' of the equation satisfied by $q^{(N)} Y^{(N)}$, but with $\psi^i$ replaced by $\tilde \psi^i$ in the generator of the backward component, and $(X_T^{(N),1},\ldots,X_T^{(N),N})$ replaced by $(\tilde X_T^1,\ldots,\tilde X_T^N)$ in the terminal condition.

By the same argument, one can multiply each $\tilde p^i$ in 
\eqref{eq:example3:FBSDE:MF}, for 
$i \in \{1,\ldots,N\}$, by 
$\tilde q^{(N)} (\tilde q^{i})^{-1}$. The resulting process 
$\tilde q^{(N)} (\tilde q^{i})^{-1} \tilde p^i$ remains a local martingale, and its boundary condition satisfies, up to a new  remainder of order $o(N)$, a boundary condition similar to the one satisfied by $p^{(N),i}$ in 
\eqref{eq:example3:BC:N}. This shows that the process 
$\tilde q^{(N)} (\tilde q^{i})^{-1} \tilde p^i$ is a `nearly solution' of the equation satisfied by $p^{(N),i}$, but with $(X_T^{(N),1},\ldots,X_T^{(N),N})$ replaced by $(\tilde X_T^1,\ldots,\tilde X_T^N)$. 
Next, rewriting the identity
\[
\tilde \psi_t^i = - (\beta \tilde q_t^{i})^{-1} \tilde p_t^{i}
\]
in the form
\[
\tilde \psi_t^i
= - (\beta \tilde q_t^{(N)})^{-1}
\tilde q_t^{(N)} (\tilde q_t^{i})^{-1} \tilde p_t^{i},
\]
we observe that $\tilde \psi^i$ can be expressed in terms of 
$\tilde q^{(N)}$ and 
$\tilde q^{(N)} (\tilde q^{i})^{-1} \tilde p^{i}$ via the same function that allows one to express $\psi^i$ in terms of $q^{(N)}$ and $p^{(N),i}$.

Altogether, this shows that the tuple
\[
(\tilde Y^{(N)}, \tilde Z^1,\ldots,\tilde Z^N,
\tilde X^1,\ldots,\tilde X^N,
\tilde q^{(N)},
\tilde q^{(N)} (\tilde q^{1})^{-1} \tilde p^1,\ldots,
\tilde q^{(N)} (\tilde q^{N})^{-1} \tilde p^N)
\]
is a nearly solution of the forward--backward system solved by the tuple
\[
(Y^{(N)}, Z^{(N),1},\ldots,Z^{(N),N},
X^{(N),1},\ldots,X^{(N),N},
p^{(N),1},\ldots,p^{(N),N}, q^{(N)}),
\]
which makes the connection between 
\eqref{eq:example3:def}
and the robust MFC problem. 
\vskip 5pt

\noindent \textit{Assumptions on $G$.}
We now comment on the assumptions needed to apply
Corollary~\ref{corollary:mckean-vlasov} in the analysis of the robust MFC
problem.

Quite surprisingly, although the convexity of $G$ suffices to apply
Theorem~\ref{theorem:SMP} in order to characterize the saddle points of
\eqref{eq:example3:def} --because the map
$(q,x_1,\ldots,x_n) \mapsto q \sum_{i,j=1} G(x_i,x_j)$
is linear in $q$ and convex in $(x_1,\ldots,x_n)$--, it does not suffice to apply
Corollary~\ref{corollary:mckean-vlasov}. Indeed, the map
\[
\mu \in \mathcal P_2(\mathbb R^n) \mapsto
\int_{(\mathbb R^n)^2} G(x,y)\,\dd \mu^{\times 2}(x,y)
\]
is displacement convex --as a consequence of the convexity of $G$-- but may fail to be
flat concave.
For instance, if
 $G(x,y) = \varphi(x)\varphi(y)$, 
for some non-negative convex function
$\varphi : \mathbb R^n \to \mathbb R$, then $G$ is convex. However, for any
$\mu \in \mathcal P_2(\mathbb R^n)$,
\[
\Gamma(\mu) \coloneqq 
\int_{(\mathbb R^n)^2} G(x,y)\,\dd \mu^{\times 2}(x,y)
= \left(
\int_{\mathbb R^n} \varphi(x)\,\dd \mu(x)
\right)^2,
\]
which shows that the function $\Gamma : \mu \in \mathcal P_2(\mathbb R^n) \mapsto \Gamma(\mu)$ is flat convex.

Additional conditions are therefore required to apply
Corollary~\ref{corollary:mckean-vlasov}. Although, in the previous paragraph, we did not provide a complete proof but only some intuition to justify the passage from
\eqref{eq:example3:FBSDE:N}–\eqref{eq:example3:BC:N} to
\eqref{eq:example3:FBSDE:MF}–\eqref{eq:example3:BC:MF}, we believe that the need for extra assumptions to ensure existence and uniqueness of a solution to the mean field problem reflects the price to pay for passing to the limit (as $N \to \infty$) in the original problem~\eqref{eq:example3:def}.

Following the discussion in the previous subsection, we now provide an example of a class of convex functions $G$ for which
$\mu \in \mathcal P_2(\mathbb R^n) \mapsto \Gamma(\mu)$
is flat concave. If $G$ itself is not convex but $\Gamma$ is flat concave, one may replace
$G$ by the function
$(x,y) \mapsto G(x,y) + a (|x|^2 + |y|^2)$, for $a>0$ large enough, in order to enforce displacement convexity while preserving flat concavity. Thus, the remaining task is to provide an example of a function $G$ for which $\Gamma$ is concave. 
One such example is given by any function of the form
\[
(x,y) \mapsto - \sum_{i=1}^k \lambda_i\, h_i(x)\, h_i(y),
\]
where $k \ge 0$, $\lambda_i>0$, and $h_i$ is a smooth function with bounded derivative, for each $i=1,\ldots,k$.

\paragraph{Robust approximation of a Gibbs measure 
on the path space.}
We now present another example, building on the previous one, but which
corresponds to the robustification, with respect to the central planner’s
strategy $\psi$, of a control problem defined on Nature’s state $q$.
It is inspired by recent works on 
stochastic algorithms (a more precise list of references is given below). 

Given a potential 
${\mathcal W}$ defined on the Wiener path space $\Omega\coloneqq{\mathcal C}([0,T],{\mathbb R}^d)$, one wants to approximate the normalized Gibbs probability measure
\begin{equation}
\label{eq:P:target:ex}
{\mathbb P}^{\rm target}
\coloneqq
\frac{1}{\mathcal Z} \exp \left( - {\mathcal W} \right)  {\mathbb P},
\qquad {\rm with}
\quad 
{\mathcal Z} \coloneqq {\mathbb E}[ \exp(-{\mathcal W})],
\end{equation}
by the law of a controlled diffusion process of the form 
(say to simplify that $X_0^\phi = 0$)
\begin{equation}
\label{eq:X:psi:ex}
\dd X_t^{\phi} = \phi_t \dd t + \dd B_t, \quad t \in [0,T]. 
\end{equation}
(Here, we use the notation $B$ instead of $W$ for the canonical process, with is a Brownian motion under the Wiener measure ${\mathbb P}$; this to avoid confusion with the potential ${\mathcal W}$.) Typically, ${\mathcal W}(\omega)$, where $\omega=(\omega_t)_{t \in [0,T]}$ denotes the generic element of the space $\Omega$, is chosen 
as
\begin{equation}
\label{eq:prototype:mathcal W}
{\mathcal W}(\omega) = G(\omega_\tau) + \int_0^\tau F(\omega_t) \dd t, 
\end{equation}
where $\tau$ is the realization, at $\omega$, of a stopping time, 
usually chosen as the first exit time of $\omega$ from a given domain. 
Obviously, the structure of ${\mathcal W}$ described above is especially 
adapted to Markovian dynamics, which leads us to choose, in this situation,  
the control $(\phi_t)_{t \in [0,T]}$ 
in a Markov feedback form $(\phi_t=\Phi(t,X_t))_{t \in [0,T]}$. 
\vskip 5pt

\noindent \textit{Exact solution to the targeting problem.}
In fact, under standard assumptions covering the Markovian framework, 
one can find a control $\bar{\psi}$ such that
the law 
$\mathbb{P} \circ (X^{\bar{\psi}})^{-1}$ 
of 
$(X_t^{\bar{\psi}})_{t \in [0,T]}$ under 
${\mathbb P}$
perfectly matches
the target distribution ${\mathbb P}^{\rm target}$, i.e. 
\begin{equation}
\label{eq:importance sampling}
\mathbb{P} \circ (X^{\bar{\psi}})^{-1}
= 
{\mathbb P}^{\rm target}. 
\end{equation}
Assume indeed that one can solve the FBSDE
system (for simplicity, we do not specify the spaces in which solutions are taken 
because this would be useless for the rest of the paragraph)
\begin{equation} \label{eq:FBSDE:matching}
    \left\{ \begin{array}{rll}
    - \dd Y_t & = \frac12 \vert Z_t \vert^2 \dd t - 
    Z_t \cdot \dd B_t, & Y_T = {\mathcal W}(X), \\[0.5em]
       \dd X_t & = -Z_t \dd t + \dd B_t, & X_0 = 0. 
    \end{array} \right.
\end{equation}
Then, the backward equation can be reformulated as
\begin{equation}
\label{eq:FBSDE:Y0}
\exp \left( - {\mathcal W}\left(X\right) \right) 
{\mathcal E}_T \left(  \int_0^\cdot 
Z_r \cdot \dd B_r
\right)
= 
\exp \left( - Y_0 \right), 
\end{equation}
where we recall that $Y_0$ is deterministic
(as it is the initial value of the BSDE in \eqref{eq:importance sampling}). 
We deduce that, 
for any bounded and measurable function 
$\Phi : {\mathcal C}([0,T],{\mathbb R}^d) \rightarrow {\mathbb R}$, 
\begin{equation*}
{\mathbb E}
\left[
\exp \left( Y_0 - {\mathcal W}\left(X \right) \right) 
{\mathcal E}_T \left(  \int_0^\cdot 
Z_r \cdot \dd B_r
\right)
\Phi \left(X \right)
\right]
= 
{\mathbb E}
\left[ \Phi \left(X \right)\right].
\end{equation*}
Thanks to the forward equation in
\eqref{eq:FBSDE:matching}
and provided that the Girsanov transformation can be rigorously applied, we 
observe that 
the left-hand side is equal to 
${\mathbb E}[e^{Y_0-{\mathcal W}} \Phi]$, because
the law of 
$X$ under ${\mathcal E}_T(\int_0^\cdot Z_r \cdot 
\dd B_r) {\mathbb P}$
is the same as the law of $B$ under ${\mathbb P}$.
Since the function $\Phi$ is arbitrary, this proves that the law of $X$ under ${\mathbb P}$ is the Gibbs measure ${\mathbb P}^{\rm target}$, as
required. 

The analysis of the 
FBSDE \eqref{eq:FBSDE:matching} is standard in the Markovian setting. 
In this case, there exists a function $\Psi$, given as
the solution of an auxiliary nonlinear parabolic PDE (see 
\cite[Chapter 3]{carmona2018probabilistic-v1}),
such that 
$(Z_t=-\Psi(t,X_t))_{t \in [0,T]}$. In particular, one can express 
$\mathbb{P} \circ X^{-1}$ as 
${\mathcal E}_T(\int_0^{\cdot} \psi_t \cdot \dd B_t)  {\mathbb P}$, 
where $\psi_t(\omega)=\Psi(t,\omega_t)$ (the latter is different from 
$-Z_t(\omega)=\Psi(t,X_t(\omega))$). 
\vskip 5pt

\noindent{\textit{Reformulation as a Nature optimization problem.}}
Interestingly, this targeting problem can be recast as 
a minimization problem in 
the space of probability measures. Indeed, using 
Donsker-Varadhan's lemma, 
it holds, for any control 
$(\phi_t)_{t \in [0,T]}$
such that the measure 
${\mathbb P}^{\phi} \coloneqq 
{\mathcal E}_T(\int_0^{\cdot} \phi_t \cdot \dd B_t)  
{\mathbb P}
$
has a relative finite entropy 
$H({\mathbb P}^\phi \vert {\mathbb P})$, 
\begin{equation}
\label{eq:donsker:varadhan}
\begin{split}
- \ln \left( {\mathbb E}\left[ \exp(-{\mathcal W})\right]
\right) 
&\leq 
{\mathbb E}
\left[ 
\frac{\dd {\mathbb P}^\phi}{\dd {\mathbb P}}
{\mathcal W}
\right]
+ 
 \mathrm{H} \left(  \left. {\mathbb P}^\phi \right \vert 
{\mathbb P} \right).
\end{split}
\end{equation}
When $\phi$ is equal to $\psi$, 
the right-hand side becomes 
\begin{equation*}
\begin{split}
{\mathbb E}
\left[ 
\frac{\dd {\mathbb P}^\psi}{\dd {\mathbb P}}
{\mathcal W}
\right]
+ 
 \mathrm{H} \left(  \left. {\mathbb P}^\psi \right \vert 
{\mathbb P} \right)
&=
{\mathbb E}\left[ {\mathcal W}\left( X^\psi \right) \right]
+ \mathrm{H} \left(  \left. 
\mathbb{P} \circ (X^\psi)^{-1} \right \vert {\mathbb P} \right)
\\
&= {\mathbb E}\left[ {\mathcal W}\left( X  \right) 
+ \frac12 
\int_0^T \vert Z_t \vert^2 
\dd t \right],
\end{split}
\end{equation*}
which is equal (thanks to \eqref{eq:FBSDE:matching}) 
to 
${\mathbb E}[Y_0]= {\mathbb E}[\exp(-{\mathcal W})]$
(with the latter following from 
\eqref{eq:FBSDE:Y0}
and a new application of Girsanov's formula). 
Therefore, 
$\psi$ solves the minimization problem 
\begin{equation}
\label{eq:matching:tuning}
\inf_{\phi}
\left\{ 
{\mathbb E}
\left[ 
\frac{\dd {\mathbb P}^\phi}{\dd {\mathbb P}}
{\mathcal W}
\right]
+ 
 \mathrm{H} \left( \left. {\mathbb P}^\phi   \right \vert 
{\mathbb P} \right) \right\},
\end{equation}
hence connecting the targeting problem 
\eqref{eq:importance sampling}
and the minimization 
problem 
\eqref{eq:matching:tuning}. 
For example, 
these two problems are tackled
in control based importance sampling methods
for diffusion processes 
(see for instance
\cite{NuksenRichter,BorrellQuerRichterSchutte}
and 
\cite[Chapter 6]{Schütte_Klus_Hartmann_2023}, from which we borrowed part of the presentation) and 
in 
diffusion based models for generative adversarial networks 
(see for instance 
the fine tuning analysis provided in 
\cite{tang2024finetuningdiffusionmodelsstochastic,uehara2024finetuningcontinuoustimediffusionmodels}
and the MFC interpretation of score matching
approaches 
\cite{zhang2023meanfieldgameslaboratorygenerative}). 

In fact, 
the 
connection between 
the targeting problem 
\eqref{eq:importance sampling}
and the minimization problem 
\eqref{eq:matching:tuning}
can be better understood 
by reformulating the latter, 
and then by observing that 
$\psi$ solves 
\begin{equation}
\label{eq:H:target}
\inf_{\phi}
\mathrm{H} \left( 
{\mathbb P}^\phi  
\left \vert
\frac1{\mathcal Z} e^{-{\mathcal W}}  
{\mathbb P}
\right.\right),
\end{equation}
the optimal value being equal to $0$. Above, 
we recall that 
${\mathcal Z}
= {\mathbb E}[\exp(-W)]$.
Rephrased in our framework, 
the density 
$\dd {\mathbb P}^\phi/\dd {\mathbb P}$
appearing in both 
\eqref{eq:matching:tuning}
and 
\eqref{eq:H:target}
must be identified with Nature's state $q_T$ at terminal time. 
Therefore, the two problems can be regarded 
as optimal control problem for 
Nature
(even though the original problem 
\eqref{eq:importance sampling}) 
is formulated as a control problem for the player). 
\vskip 5pt

\noindent \textit{Robust version.}
Now, consistently with the robust
approach introduced in this work, 
one can think of a situation where there is some uncertainty on the 
precise form of the potential ${\mathcal W}$ in the targeting measure 
$e^{-{\mathcal W}} \cdot {\mathbb P}$ in 
\eqref{eq:P:target:ex}. We thus change 
${\mathcal W}$ into ${\mathcal W}(X^\psi)$, with $X^\psi$
as in \eqref{eq:X:psi:ex}. 
Intuitively, this says that 
there is some uncertainty on the `observed values' of ${\mathcal W}$, say for instance because the potential is computed along an approximation of the canonical process (as in the stochastic algorithms cited above).

Next, we introduce two related min-max problems.
The first problem is a robust version of 
\eqref{eq:matching:tuning}:
\begin{equation}
\label{eq:ex:2:1st:pb}
\inf_{q \in {\mathcal Q}}
\sup_{\psi \in {\mathcal A}} 
\left \{\mathrm{H} \left(\left. \mathbb{Q}^q
 \right \vert {\mathbb P}  \right) + 
{\mathbb E} \left[ q_T {\mathcal W}(X^\psi) - \frac12 q_T \int_0^T 
\vert \psi_t \vert^2 \dd t 
\right] 
  \right \}, \quad   \mathbb{Q}^q= q_T  {\mathbb P},
\end{equation}
which is quite similar to the second example in Subsection \ref{sec:example-application} and to the first example in this subsection. 
The second one is a robust version of 
\eqref{eq:H:target}:
\begin{equation}
\label{eq:ex:2:2nd:pb}
\inf_{q \in {\mathcal Q}}
\sup_{\psi \in {\mathcal A}} 
\left \{ 
\mathrm{H} \left( 
{\mathbb Q}^q  \left \vert  \mathbb{G}^{\psi}
\right. \right)
- \frac12 {\mathbb E} \left[q_T \int_0^T 
\vert \psi_t \vert^2 \dd t 
\right] 
  \right \}, \quad  \mathbb{G}^{\psi} = \frac1{\mathcal Z^\psi}  e^{-{\mathcal W}(X^\psi)}  
{\mathbb P},
\end{equation}
with ${\mathcal Z}^\psi \coloneqq {\mathbb E}[\exp(-{\mathcal W}(X^\psi))]$.  

The two problems are not the same because of the presence of the normalization constant in the second one. 
In both situations, 
the penalty term 
$-\frac12 q_T \int_0^T \vert \psi_t \vert^2 \dd t$
should be regarded as a regularization of the problem that just amounts in replacing 
the original potential 
${\mathcal W}(X^\psi)$ by 
the effective one 
${\mathcal W}(X^\psi) - \frac12 q_T \int_0^T \vert \psi_t \vert^2 \dd t$.
Also, the term 
${\mathcal W}(X^\psi)$ itself can be chosen as a function of 
the realization of the path but also of its statistical distribution under 
${\mathbb P}$. For instance, 
a typical choice, 
consistent with the 
previous example on Feynman-Kac models, 
is 
\begin{equation*}
{\mathcal W}( X_T^\psi ) = W\left(X_T^\psi,\mathbb{P} \circ (X_T^\psi)^{-1}\right).
\label{eq:ex:importance:terminal:condition}
\end{equation*}

\noindent \textit{Explicit computation.} When ${\mathcal W}$ is (say) a concave function of the terminal state, the first problem
\eqref{eq:ex:2:1st:pb} satisfies our concavity-convexity conditions and there exists a (unique) saddle point to the min-max problem. 
To better illustrate the 
result, we just focus on the case when 
$d=1$ and
${\mathcal W}(\omega)= -\beta \omega_T$, for some parameter
$\beta \in {\mathbb R}$. Then, very similar to the 
risk averse portfolio management problem addressed in Subsection \ref{sec:example-application},  
the unique saddle point 
$(\bar q,\bar \psi)$ 
can be found explicitly. 
Here, 
\begin{equation}
\label{ex:W:linear:1}
    \bar \psi_t = - \beta,
\end{equation}
and 
\begin{equation}
\label{ex:W:linear:2}
\bar q_T = \frac1{\bar{\mathcal Z}}
\exp \left( - {\mathcal W}(X^{\bar{\psi}}) - 
\frac12 \int_0^T \vert \bar \psi_t \vert^2 
\dd t \right)
=
\frac1{\bar{\mathcal Z}}
\exp \left( \beta (B_T - T \beta) - \frac12 T \beta^2 \right),
\end{equation}
where 
\begin{equation}
\label{ex:W:linear:3}
\bar{\mathcal Z}
= {\mathbb E} \left[ \exp \left( \beta (B_T - T \beta) - \frac12 T \beta^2 \right) \right] = \exp( - T \beta^2 ). 
\end{equation}

The second problem \eqref{eq:ex:2:2nd:pb} is more difficult to handle. 
Using the explicit form of ${\mathcal Z}^\psi$
(and expanding the various logarithms inside the definition of the entropy), 
it can be rewritten as 
\begin{equation*}
\begin{split}
\inf_{q \in {\mathcal Q}}
\sup_{\psi \in {\mathcal A}} 
\biggl\{
{\mathbb E} & \left[
  q_T {\mathcal W}(X^\psi) - \frac12 q_T \int_0^T 
\vert \psi_t \vert^2 \dd t \right]
\\
&\hspace{15pt} + \ln \left( {\mathbb E}
\left[ \exp \left( - {\mathcal W}(X^\psi) \right) \right]
\right)
 + \mathrm{H} \left( 
\mathbb{Q}^q 
 \big\vert
{\mathbb P}
\right)
  \biggr\}.
  \end{split}
\end{equation*}
Here we 
recall from 
\eqref{eq:donsker:varadhan}
that the cumulant generating function 
appearing on the second line 
of the right-hand can be reformulated 
as the supremum (over $\phi$)
of ${\mathbb E}[- (\dd {\mathbb P}^\phi/\dd {\mathbb P}) {\mathcal W}(X^\psi)]
- \mathrm{H}({\mathbb P}^\phi \vert {\mathbb P})$. In particular, if ${\mathcal W}$ is
concave, then $-{\mathcal W}$ is convex (in $X_T^\psi$) and the supremum (over $\phi$) is also convex in $X_T^\psi$. 
As a result, it is not clear whether the cost is concave in $\psi$, which
prevents any application of the results obtained in the article. 

Nevertheless, one can use the saddle point 
$(\bar q,\bar \psi)$ obtained for the problem 
\eqref{eq:ex:2:1st:pb} in order to gain some insight into 
the problem \eqref{eq:ex:2:2nd:pb}. Indeed, by the saddle point property, 
we have
\begin{equation*}
\begin{split}
&{\mathbb E} \left[ \bar q_T {\mathcal W}\left(X^{\bar \psi} \right) 
- 
\frac12 \bar q_T \int_0^T \vert \bar \psi_t \vert^2 \dd t \right]
= \sup_{\psi} {\mathbb E} \left[ \bar q_T {\mathcal W}\left(X^{\psi} \right) 
- 
\frac12 \bar q_T \int_0^T \vert \psi_t \vert^2 \dd t \right].
\end{split}
\end{equation*}
By concavity of the cost function (with respect to $\psi$), one deduces that, for any 
$\psi$
\begin{equation*}
\begin{split}
{\mathbb E} \left[ \bar q_T {\mathcal W}\left(X^{\bar \psi} \right)
- 
\frac12 \bar q_T \int_0^T \vert \bar \psi_t \vert^2 \dd t \right]
\geq \; &{\mathbb E} \left[ \bar q_T {\mathcal W}\left(X^{\psi} \right) 
- 
\frac12 \bar q_T \int_0^T \vert \psi_t \vert^2 \dd t \right]
\\
& + 
\frac12 {\mathbb E} \left[\bar q_T \int_0^T 
\bigl\vert \psi_t - \bar \psi_t \bigr\vert^2 \dd t \right]. 
\end{split}
\end{equation*}
And then, by expanding the logarithm inside the definition of the entropy, 
\begin{equation*}
\begin{split}
& \mathrm{H} \left( 
{\mathbb Q}^{\bar q}  \left \vert
\mathbb{G}^{\bar\psi}
\right. \right)
- \frac12 {\mathbb E} \left[\bar q_T \int_0^T 
\vert \bar \psi_t \vert^2 \dd t 
\right]
\\
= \; & {\mathbb E} \left[ \bar q_T {\mathcal W}\left(X^{\bar \psi} \right)
- 
\frac12 \bar q_T \int_0^T \vert \bar \psi_t \vert^2 \dd t \right]
+ \ln \left( {\mathbb E} \left[ 
\exp \left( - {\mathcal W}(X^{\bar{\psi}}) \right) \right]
\right) 
+ \mathrm{H} \left({\mathbb Q}^{\bar q} \vert {\mathbb P} \right)\\
\geq \; &
\mathrm{H} \left( 
{\mathbb Q}^{\bar q}  \left \vert
\mathbb{G}^\psi
\right. \right)
- \frac12 {\mathbb E} \left[\bar q_T \int_0^T 
\vert \psi_t \vert^2 \dd t 
\right] + \Delta(\psi,\bar \psi),
\end{split}
\end{equation*}
with 
\begin{equation*}
     \Delta(\psi,\bar \psi) \coloneqq - \ln \left( \frac{{\mathbb E} \left[ 
\exp \left( - {\mathcal W}(X^{{\psi}})\right) \right]}{
{\mathbb E} \left[ 
\exp \left( - {\mathcal W}(X^{\bar{\psi}}) \right)\right]}
\right) 
+ 
\frac12 {\mathbb E} \left[\bar q_T \int_0^T \bigl\vert \psi_t - \bar \psi_t \bigr\vert^2 \dd t \right],
\end{equation*}
which gives a way to control 
the variation $ \Delta(\psi,\bar \psi)$ 
of the cost when $\psi$ is deviating 
from $\bar \psi$. 
We can illustrate this idea in this example, by means of 
in \eqref{ex:W:linear:1}--\eqref{ex:W:linear:2}--\eqref{ex:W:linear:3}. 
We have
\begin{equation*}
\begin{split}
\Delta(\psi,\bar \psi)
= &
- \ln \left( \frac{{\mathbb E} \left[ 
\exp \left( \beta X^{{\psi}}\right) \right]}{
{\mathbb E} \left[ 
\exp \left( \beta X^{\bar{\psi}}\right) \right]}
\right) 
+ 
\frac12 {\mathbb E} \left[\bar q_T \int_0^T \bigl\vert \psi_t - \bar \psi_t \bigr\vert^2 \dd t \right]
\\
= & - \ln \biggl\{ {\mathbb E}
\left[ 
\exp \left(  \beta B_T - \frac12 \beta^2 T \right) 
\exp \left( \beta \int_0^T \left[ \psi_t - \bar \psi_t
\right] \dd t \right) 
\right]
\biggr\}
\\
& +
\frac12 {\mathbb E}
\left[ 
\exp \left(  \beta B_T - \frac12 \beta^2 T \right) 
\int_0^T \bigl\vert \psi_t - \bar \psi_t \bigr\vert^2 
\dd t
\right].
\end{split}
\end{equation*}
This gives a way to control the output performance in
terms of the disturbance, which principle is underpinning 
the theory of $H^\infty$-control (see 
\cite{BasarBernhard}).

\subsection{Variational mean field games} \label{sec:variationnal-mean-field-games}

\label{sec:mfg-var}

In this section, we formulate a mean field game problem that is closely related to the robust mean field control problem introduced in the previous section, 
and that even derives from it for some specific choice of the coefficients. The latter situation is an extension, to the robust setting, of the connection that exists between mean field control problems and potential mean field games.

Generally speaking, a mean field game is defined as a fixed point problem on the distribution of
a control  problem (with the latter being solved by a so-called representative agent in a continuum of agents). In our case, the 
fixed point problem is set on a generic non-negative measure $\mu \in \mathcal{M}_{2-r}(\mathbb{R}^n)$; given $\mu$, the representative agent minimizes a risk-averse objective functional
\begin{equation*}
    \inf_{\psi \in \mathcal{A}} \sup_{q \in \mathcal{Q}} \mathcal{J}[\mu](q,\psi), 
    \quad
    \mathcal{J}[\mu](q,\psi) \coloneq \mathbb{E}\left[q_T g( \mu, X^\psi_T) + \int_0^T q_s \ell(s,\psi_s)\dd s \right] - \mathcal{S}(q),
\end{equation*}
where the controlled state process  $(X_t^\psi)_{t \in [0,T]}$ satisfies the dynamics given in \eqref{eq:intro:X}. Assuming that, for each $\mu \in {\mathcal M}^{2-r}({\mathbb R}^n)$, the function $(q,X) \mapsto qg(\mu,X)$
satisfies the assumption of 
Theorem \ref{theorem:SMP}
(we clarify the choice of $g$ right below), we can denote by
$\psi^\mu$ and $q^\mu$ the optimal controls of the representative agent and of  Nature, respectively.
The fixed point condition requires that the measure $\mu$ coincides with the law of the terminal state $X_T^{\psi^\mu}$ under the measure $q_T^\mu {\mathbb P}$ induced by Nature, that is,
\begin{equation} \tag{MFG-eq}\label{eq:equilibrium-condition-bis}
    \mu = \left(q_T^\mu {\mathbb P}\right) \circ (X^{\psi^\mu}_T)^{-1}. 
    \end{equation}
The mean field game problem thus consists in finding a triple $(q,\psi,\mu) \in \mathcal{Q} \times \mathcal{A} \times \mathcal{M}_{2-r}(\mathbb{R}^n)$
  such that
 \begin{equation} \tag{MFG} \label{pb:mfg-bis}
     \mathcal{J}[\mu](q,\psi) = \inf_{\psi' \in \mathcal{A}} \sup_{q' \in \mathcal{Q}} \mathcal{J}[\mu](q',\psi'), \quad \mu = (q_T {\mathbb P}) \circ (X^\psi_T)^{-1}.
 \end{equation}

Here are the assumptions required on $g$.

\begin{enumerate}[label*=A\arabic*,resume] 
    \item \label{assumption:little-g} We assume that there exists a function 
    $G \colon \mathcal{M}_{2-r}(\mathbb{R}^n) \to \mathbb{R}$,  satisfying Assumption \ref{assumption:g}, such that 
    the mapping 
    $g \colon \mathcal{M}_{2-r}(\mathbb{R}^n) \times \mathbb{R}^n \to \mathbb{R}$ satisfies  $g(\mu,x)=\delta G/\delta \mu(\mu,x)$ and, thus, $\nabla_x g(\mu,x) = \partial_\mu G(\mu,x)$.
\end{enumerate}

Generally speaking, a mean field game problem is said to be variational if the associated mean field game system can be interpreted as the first-order optimality condition of a variational problem. Usually (i.e., in standard mean field games), the criterion of the variational problem involves a potential functional whose derivative --understood in a suitable sense-- coincides with the interaction cost of the game (see, for instance, \cite{briani2018stable} when the mean field game is formulated as a system of PDEs, and \cite[Chapter~6]{carmona2018probabilistic-v1} for the probabilistic counterpart).
Here, the mean field mapping $G \colon \mathcal{M}_{2-r}(\mathbb{R}^n) \to \mathbb{R}$ introduced in the above assumption plays the role of the potential, with the derivative understood in the flat sense for Nature and in the Lions sense for the representative player.

\begin{corollary} \label{corollary:mfg}
Let Assumptions 
\ref{hyp:b}--\ref{hyp:gamma} and \ref{assumption:little-g} be satisfied. 
    Then, there exists 
    a unique mean field game equilibrium
    $({\psi},{q}, {\mu}) \in \mathcal{A} \times \mathcal{Q} \times \mathcal{M}_{2-r}(\mathbb{R}^n)$, to the problem \eqref{pb:mfg-bis}, where we recall that $r$ is defined in Assumption \ref{hyp:sigma}. The equilibrium is fully characterized as the solution to the system formed by \eqref{optim:condition-dual}--\eqref{optim:condition-primal}
    with the terminal conditions in the first two systems being replaced by
    \begin{equation}
    \label{eq:cor:18:conditions:terminales}
    p_T = q_T \nabla_x g(X_T^\psi,\mu), \quad Y_T = g(X_T^\psi,\mu),
    \end{equation}
    complemented by the equilibrium condition \eqref{eq:equilibrium-condition-bis}, namely $\mu = (q_T {\mathbb P})_{X_T^\psi}$.
\end{corollary}

\begin{proof}
\noindent \textit{Step 1: Necessary and sufficient condition for equilibrium.}
Let $\mu \in \mathcal{M}_{2-r}(\mathbb{R}^n)$. Applying Corollary \ref{corollary:mckean-vlasov} to the parametrized mean field mapping  $G[\mu] \colon \mathcal{M}_{2-r}(\mathbb{R}^n) \to \mathbb{R}$ defined as follows
    \begin{equation*}
        G[\mu](\nu) = \int_{{\mathbb R}^n} g(\mu,x) \dd \nu(x), 
        \quad \nu \in {\mathcal M}_{2-r}({\mathbb R}^n),
    \end{equation*}
    we deduce that the 
    the system  
\eqref{optim:condition-dual}--\eqref{optim:condition-primal}, with the terminal conditions
\eqref{eq:cor:18:conditions:terminales}, is 
a necessary and sufficient condition of for equilibrium when the interaction term $\mu$ is frozen. 

When
complemented by the equilibrium condition $\mu = (q_T \mathbb{P})_{X_T^{\psi}}$, they provide a characterization of the solutions to the mean field game 
\eqref{pb:mfg-bis}.\vskip 4pt
    
    \noindent \textit{Step 2: Uniqueness.} By Corollary \ref{corollary:mckean-vlasov}, there exists a unique solution 
    $( {\psi}, {q}) \in \mathcal{A} \times \mathcal{Q}$ to the  system \eqref{optim:condition-dual}-\eqref{optim:condition-primal}
    with the terminal conditions
    \begin{equation*}
    p_T = q_T \nabla_x g \left(( q_T \mathbb{P})_{X_T^\psi}, X_T^\psi \right), \quad Y_T = g\left(( q_T \mathbb{P})_{X_T^\psi},X_T^\psi \right).
    \end{equation*}
    This system coincides with the necessary and sufficient condition identified in the first step, which proves that there exists a unique solution to \eqref{pb:mfg-bis} in the space mentioned in the statement.
    \end{proof}

\paragraph{Perspectives.}

We conclude this section with a brief discussion about mean field game model beyond the variational case.
A natural question arises as to how one might treat mean field games that lack an underlying variational structure. The monotonicity assumptions imposed on the flat and Lions derivatives of $G$ in \ref{assumption:g} (and thus on $g$ in \ref{assumption:little-g}) in the MFC problem already suggest the type of conditions that can be imposed on the interaction terms to ensure uniqueness of solutions, in the spirit of the classical Lasry–Lions monotonicity condition for standard mean field games. We refer to our companion work \cite{DelarueLavigne2} for complete results in this direction.

\section{Proof of Theorem \ref{theorem:SMP}}
\label{sec:proof}
In this section, we establish all the intermediate results used in the proof of Theorem \ref{theorem:SMP}. The presentation is organized into three subsections. 
In Subsection \ref{sec:properties-J}, we establish the existence of a min--max solution to the problem \eqref{pb:min-max-G-c1-c2}, corresponding to Step~1 and Step~2 in the proof of Theorem \ref{theorem:SMP}. 
Subsection \ref{sec:necessary-sufficient} provides the necessary and sufficient conditions for the control problem solved by Nature, thus covering the arguments developed in Step~3 of the proof. 
Finally, Subsection \ref{sec:central-planner} focuses on the central planner and forms the basis of Step~4 in the proof of Theorem \ref{theorem:SMP}.

\subsection{Existence of a saddle point to \eqref{pb:min-max-G-c1-c2}}
\label{sec:properties-J}

This subsection is dedicated to the proof of the existence of a saddle point to the problem \eqref{pb:min-max-G-c1-c2}, for given values of $c_1,c_2>0$.
This  corresponds to the first step in the proof of Theorem \ref{theorem:SMP}. 
Without any loss of generality, we can assume that 
\begin{equation}
\label{eq:lower:bound:c_1}
    c_1 >  
   {\mathbb E} \int_0^T  q_t^0 
 f(t,0,0) \dd t,
\end{equation}
where 
$\bar q^0$ denotes the solution of 
\begin{equation}
\label{eq:barq:0}
\dd  q_t^0 = q_t^0
\partial_y f(t,0,0)  \dd t 
+ 
  q_t^0 \partial_z f(t,0,0) 
\cdot \dd W_t, \quad t \in [0,T].
\end{equation}
We notice that the right-hand side on \eqref{eq:lower:bound:c_1} is equal to ${\mathcal S}(q^0)$. Indeed 
\begin{equation*}
{\mathcal S}( q^0) 
= {\mathbb E}
\int_0^T   q_t^0 f^\star\left(t,\partial_y f(t,0,0),\partial_z f(t,0,0) 
\right) \dd t
= - {\mathbb E}
\int_0^T q_t^0 f(t,0,0) \dd t.
\end{equation*}
The purpose is thus to establish the following statement:
\begin{lemma} \label{lemma:existence-of-saddle-point-p-prime}
    There exists a solution $(\psi,q) \in \mathcal{A}_{c_2} \times \mathcal{Q}_{c_1}$ to \eqref{pb:min-max-G-c1-c2}.
\end{lemma}  

Before we provide a sketch of the proof of this result, we introduce a variant of the Nature optimization problem. 
Existence of a saddle point is proven by means of weak compactness 
arguments (in $L^p$ spaces), which are developed in this subsection. In this regard, 
the 
nonlinear form of the 
state equation \eqref{eq:q}
causes additional difficulties, as the product form of the coefficients 
is not appropriate for 
weak convergence arguments. 
For this reason, it is easier to 
consider weak limits of the two products $(q_t Y^\star_t)_{t \in [0,T]}$ and 
$(q_t Z_t^\star)_{t \in [0,T]}$, each being viewed as a single process. However,
this makes more difficult the identification of the limit points as solutions of an equation of the form \eqref{eq:q} (because weak limits of the products must be shown to have a product form, which writing may be difficult to establish if the weak limit of $(q_t)_{t \in [0,T]}$ vanishes).
This prompts us to introduce a variant of the problem \eqref{pb:min-max-G-c1-c2}, and in particular to define the perspective function $q f^\star (\omega,t,y^\star/q,z^\star/q)$ for any 
$(\omega,t,q,y^\star,z^\star) \in \Omega \times [0,T]  \times \mathbb{R} \times \mathbb{R} \times \mathbb{R}^d \to \mathbb{R}$ of $f^\star$ (see \cite{bonnans2019convex,combettes2018perspective} for a presentation) with respect to its last two variables. We further introduce its lower semi-continuous envelope (or its bidual) $\tilde{f}^\star \colon \Omega \times [0,T]  \times \mathbb{R}  \times \mathbb{R} \times \mathbb{R}^d \to \mathbb{R}$, 
{\begin{equation}
\label{eq:def:perspective:function:with:rec}
    \tilde{f}^\star(\omega,t,q,y^\star,z^\star) = \left\{\begin{array}{ll}
        q f^\star \left(\omega,t,\frac{y^\star}{q},\frac{z^\star}{q} \right), & q >0, \\
       \mathrm{rec} f^\star(\omega,t,q,y^\star,z^\star),  & q = 0, \\
       +\infty, & \textrm{q < 0.}
    \end{array} \right.
\end{equation}
Here $\mathrm{rec} f^\star(\omega,t,q,\cdot,\cdot)$ denotes the recession function of $f^\star(\omega,t,q,\cdot,\cdot)$ (with respect to the last two variables of $f^\star$). By 
\cite[Lemma 1.156]{bonnans2019convex}, it coincides with the support function of 
$f^\star(\omega,t,0,\cdot,\cdot)$, i.e., 
\begin{equation*}
    \mathrm{rec} f^\star(\omega,t,0,y^\star,z^\star) = \sup_{(y,z) \in \mathbb{R} \times \mathbb{R}^d} \left \{\langle y,y^\star\rangle + \langle z,z^\star\rangle, \; f(\omega,t,0,y,z) < +\infty \right\}.
\end{equation*}
Because $f(\omega,t,q,\cdot,\cdot)$ has full support, the recession function at $q = 0$ is given by 
\begin{equation*}
    \mathrm{rec} f^\star(\omega,t,q,y^\star,z^\star) = \left\{ \begin{array}{ll}
         0, &  (y^\star,z^\star) = 0, \\
         +\infty, & \mathrm{otherwise.}
    \end{array} \right.
\end{equation*}
Finally, the function $\tilde{f}^\star$ is equal to
\begin{equation}
    \label{def:perspective-function}\tilde{f}^\star(\omega,t,q,y^\star,z^\star) = \left\{\begin{array}{ll}
        q f^\star \left(\omega,t,\frac{y^\star}{q},\frac{z^\star}{q} \right), & q >0, \\
       0,  & (q,y^\star,z^\star) = 0, \\
       +\infty, & \textrm{otherwise}.
    \end{array} \right.
\end{equation}
For simplicity, we call $\tilde{f}^\star$ the perspective function of $f^\star$ when there is no ambiguity. 
Because $f^\star(\omega,t,\cdot,\cdot)$ is convex and lower semi-continuous, its perspective function  $\tilde{f}^\star$ is convex with respect to its three last variables and lower semi-continuous. For any 
non-negative valued It{\^o} 
process $q = (q_t)_{t \in [0,T]}$, 
satisfying 
 ${\mathbb E}[q_T^*] < + \infty$,
and
admitting the expansion
\begin{equation}
    \label{eq:ito:q}
     \dd q_t = \tilde{Y}^\star_t \dd t + \tilde{Z}^\star_t \cdot \dd W_t,
     \quad t \in [0,T],
\end{equation}
for some (uniquely defined) 
${\mathbb F}$-progressively measurable process 
$\tilde{Y}^\star = (\tilde{Y}^\star_t)_{t \in [0,T]}$ and $\tilde{Z}^\star = (\tilde{Z}^\star_t)_{t \in [0,T]}$, with values in $\mathbb{R}$ and $\mathbb{R}^d$ respectively and satisfying 
\begin{equation}
\label{eq:local:tildeY:tildeZ}
    {\mathbb P}\left(\left\{\int_0^T (\vert \tilde Y_t^\star \vert + \vert \tilde Z_t^\star \vert^2) \dd t < + \infty\right\} \right)=1,
\end{equation}
we define the perspective generalized entropy 
of $q$ by letting
    \begin{equation} \label{eq:tilde:S}\tilde{\mathcal{S}}(q) \coloneqq \mathbb{E}\left[\int_0^T  \tilde{f}^\star(t,q_t,\tilde{Y}^\star_t,\tilde{Z}^\star_t) \dd t \right].
    \end{equation}
Recalling the lower bound
\eqref{ineq:duality-fstar} and using the fact that 
${\mathbb E}[q_T^*] < + \infty$, we notice that the expectation right above is well-defined; it belongs to 
$(-\infty,+\infty]$. And then, 
we introduce the perspective min-max problem 
\begin{equation} \label{pb:min-max-tilde-G-c1-c2} \tag{{\~P}'}
\sup_{q \in \tilde{\mathcal{Q}}_{c_1}}  \inf_{\psi \in \mathcal{A}_{c_2}}  \tilde{\mathcal{J}}(q,\psi),
\end{equation}
where the mapping $\tilde{\mathcal{J}}$ is given by
    \begin{equation}
    \label{eq:def:tilde:J:F}\tilde{\mathcal{J}}(q,\psi)  \coloneqq \mathcal{R}(q,\psi) - \tilde{\mathcal{S}}(q);
    \end{equation}
     recall \eqref{def:F} for the definition of $\mathcal{R}$. Above, the set $\tilde{\mathcal{Q}}_{c_1}$ is 
     defined as the collection of $q \in \tilde {\mathcal Q}$ such that $\tilde{\mathcal S}(q) \leq c_1$, 
     where 
     $\tilde {\mathcal Q}$
     is
     the set of non-negative valued  measurable It\^o processes
    $q = (q_t)_{t \in [0,T]}$ satisfying \eqref{eq:ito:q}, such that 
    $q_0=1$, 
    ${\mathbb E}[q_T^*] \leq \exp(\alpha T)$, 
    and 
    $\tilde{\mathcal{S}}(q) < + \infty$.

Since we 
restricted the controlled dynamics \eqref{eq:q} to processes $(q_t)_{t \in [0,T]}$ that do not vanish, it is easy to see that 
any $q \in {\mathcal Q}_{c_1}$ belongs to  $\tilde{\mathcal Q}_{c_1}$. Indeed, ${\mathcal S}(q)$ and 
$\tilde{\mathcal S}(q)$ coincide in this setting. Moreover, the bound $ {\mathbb E}[q_T^*] \leq \exp(\alpha T)$
follows from the facts that $Y^\star$ is bounded by $\alpha$ and 
$({\mathcal E}_t(\int_0^\cdot Z_s^\star \cdot \dd W_s))_{t \in [0,T]}$ is a martingale, see Lemma 
\ref{lem:about:q}.

Existence of a saddle point to 
\eqref{pb:min-max-tilde-G-c1-c2} is established in the next subsection; see Lemma 
\ref{lemma:sion-limit}. Taking the latter for granted, 
Lemma \ref{lemma:existence-of-saddle-point-p-prime} can be derived as follows:

\begin{proof}[Proof of Lemma \ref{lemma:existence-of-saddle-point-p-prime}.]
        Since existence of a saddle point to 
    \eqref{pb:min-max-tilde-G-c1-c2}
    is provided by Lemma 
    \ref{lemma:sion-limit}, it suffices to show that any solution to \eqref{pb:min-max-tilde-G-c1-c2} is a solution to \eqref{pb:min-max-G-c1-c2}. Let $(\bar{q},\bar{\psi})$ be a solution to \eqref{pb:min-max-tilde-G-c1-c2}, that is to say 
    \begin{equation} \label{eq:min-max-c-tilde}
            \tilde{\mathcal{J}}(q,\bar{\psi}) \leq \tilde{\mathcal{J}}(\bar{q},\bar{\psi}) \leq \tilde{\mathcal{J}}(\bar{q},\psi), \quad \forall (q,\psi) \in \tilde{\mathcal{Q}}_{c_1} \times \mathcal{A}_{c_2}.
    \end{equation}
    By Lemma \ref{lemma:positive-q} (which is stated and proven in Subsection 
    \ref{subse:5.1.2} below), the process 
    $\bar q$ is positive in the sense that $\mathbb{P}(\{\inf_{t \in [0,T]} \bar q_t >0 \}) = 1$. This makes it possible to let 
$(Y^\star_t \coloneqq \tilde{Y}^\star_t/\bar q_t)_{t \in [0,T]}$ and $(Z^\star_t = \tilde{Z}^\star_t/\bar q_t)_{t \in [0,T]}$, from which we deduce
    \begin{equation*}
        \dd \bar q_t = \bar q_t Y^\star_t \dd t + \bar q_t Z^\star_t \cdot \dd W_t, \quad t \in [0,T]; \quad \bar q_0 = 1.
    \end{equation*}
    By definition of the perspective generalized entropy
    \begin{equation*}
         {\mathcal{S}}(\bar q) =  \tilde{\mathcal{S}}(\bar q) \leq c_1,
    \end{equation*}
which proves that 
$\bar q$
    belongs to $\mathcal{Q}_{c_1}$.
    Then, by the definition \eqref{eq:def:tilde:J:F} of $\tilde{\mathcal{J}}$,  by the optimality condition \eqref{eq:min-max-c-tilde} and since 
    $\tilde{\mathcal Q}_{c_1}$ contains ${\mathcal Q}_{c_1}$,  we have
    \begin{equation*} 
           \mathcal{J}(q,\bar{\psi}) \leq \mathcal{J}(\bar{q},\bar{\psi}) \leq \mathcal{J}(\bar{q},\psi), \quad \forall (q,\psi) \in \mathcal{Q}_{c_1} \times \mathcal{A}_{c_2},
    \end{equation*}
    concluding  the proof.
\end{proof}
}

\subsubsection{Trajectories of the perspective problem and positivity of the  optimal ones}
\label{subse:5.1.2}
The third item in the following lemma
was used in the proof of Lemma 
\ref{lemma:existence-of-saddle-point-p-prime}. The first two items are also used in the proof of Lemma 
    \ref{lemma:sion-limit}. 

\begin{lemma} \label{lemma:positive-q} 
Let $q \in \tilde {\mathcal Q}_{c_1}$ and
$(\tilde Y^\star,\tilde Z^\star)$
be as in the representation \eqref{eq:ito:q}. 
\begin{enumerate}[label*=\roman*.]
\item Letting 
\begin{equation}
\label{eq:representation:Ystar:Zstar}
Y_t^\star = {\mathds 1}_{\{q_t>0\}}
\frac{\tilde Y_t^\star}{q_t}, \quad 
Z_t^\star = {\mathds 1}_{\{ q_t>0\}}
\frac{\tilde Z_t^\star}{q_t},
\quad t \in [0,T],
\end{equation}
it holds 
\begin{equation}
\label{eq:representation:perspective:3}
{\mathbb P} \otimes \textrm{\rm Leb}_{[0,T]}
\left( 
\left\{ (\omega,t) \in \Omega \times [0,T], \; 
 \vert   Y_t^\star \vert > \alpha   \right\}
\right) = 0,
\end{equation}
and $q$ can be expanded as 
\begin{equation}
\label{eq:representation:perspective:4}
\dd    q_t =     q_t Y_t^\star \dd t + 
    q_t {Z}_t^\star \cdot \dd W_t, \quad t \in [0,T].
\end{equation}

\item  Moreover, letting 
$\tau\coloneqq \inf\{t \in [0,T],\;  q_t = 0\})$ (with $\inf \emptyset = + \infty$), it also holds ${\mathbb P}(\{ \sup_{t \in [\tau,T]}   q_t >0\} \cap \{\tau < T\})=0$ (i.e., $0$ is an absorbing state). 
And then, 
\begin{equation}
\label{eq:tildeS=S}
\begin{split}
 \tilde{\mathcal{S}}(q) &= \mathbb{E}\left[ \int_0^T  \tilde{f}^\star(t,q_t,\tilde{Y}^\star_t,\tilde{Z}^\star_t) \dd t \right] 
= \mathbb{E}\left[ \int_0^\tau q_t f^\star(t,Y^\star_t,Z^\star_t) \dd t \right].
\end{split}
\end{equation}
We also have ${\mathbb E}[q_T^*] < +\infty$ and there exists a constant $C$, which depends on $q$ only via $c_1$, such that 
\begin{equation}
\label{eq:representation:perspective:40}
{\mathbb E}\left[ 
\left( \int_0^T 
\vert \tilde Z_s^\star \vert^2 
\dd s\right)^{1/2}
\right] \leq C. 
\end{equation}
\item  Lastly, if 
for a certain 
$  \psi \in \mathcal{A}_{c_2}$, the  pair $( {\psi}, {q}) \in \mathcal{A}_{c_2} \times \tilde{\mathcal{Q}}_{c_1}$ is a solution to \eqref{pb:min-max-tilde-G-c1-c2}. Then ${\mathbb P}(\{\inf_{t \in [0,T]}   q_t >0\})=1$, 
and (in particular) $  q \in {\mathcal Q}_{c_1}$. 
\end{enumerate}
\end{lemma}

\begin{remark}
The following two comments are in order:
\begin{enumerate}
\item 
In dimension $d=1$, the CIR model, i.e., 
\begin{equation*}
\dd q_t = \sqrt{q_t} \dd W_t, \quad t \in [0,\tau),
\end{equation*}
provides an interesting example 
in which $q$ may vanish even if the entropy, which is here equal to  
${\mathbb E} [\int_0^{\tau \wedge T} q_t^{-1} q_t \dd t] = {\mathbb E}[\tau \wedge T]$, is finite. 
\item When $q$ vanishes, it does not  make sense to represent it in the form of a (weighted) Doléans-Dade exponential martingale. This observation causes additional difficulites in the analysis.
\end{enumerate}
\end{remark}
\color{black}
\begin{proof} 
\noindent \textit{Step 1: Representation of $  q \in \tilde{\mathcal Q}_{c_1}$.}
We recall that 
$q$ can be represented as 
\begin{equation*}
  q_t = 1 + \int_0^t \tilde{Y}^\star_s  \dd s+ \int_0^t \tilde{Z}^\star_s \cdot \dd W_s,
\quad t \in [0,T],\end{equation*} 
with
\begin{equation*}
\tilde{\mathcal{S}}(  q) = \mathbb{E}\left[\int_0^T  \tilde{f}^\star(t,q_t,\tilde{Y}^\star_t,\tilde{Z}^\star_t) \dd t \right]
\in (-\infty,c_1].
\end{equation*}
Using \eqref{ineq:duality-fstar} together with the bound ${\mathbb E}[ q^*_T] \leq \exp(\alpha T)$, we deduce that 
\begin{equation}
{\mathbb E} \left[
\int_0^T {\mathds 1}_{\{  q_t >0\}}
   q_t \left( \chi_{\mathcal B}\left( \frac{\tilde Y^\star_t}{\alpha    q_t} \right) 
+ \frac1{2 \beta}
\left\vert \frac{\tilde Z_t^\star}{  q_t}
\right\vert^2 
\right) \dd t \right] 
< + \infty.
\label{eq:representation:0}
\end{equation}
This proves in particular that 
\begin{equation}
\label{eq:representation:perspective:1}
{\mathbb P} \otimes \textrm{\rm Leb}_{[0,T]}
\left( 
\left\{ (\omega,t) \in \Omega \times [0,T], \; 
  q_t >0, \quad \vert \tilde Y_t^\star \vert > \alpha   q_t \right\}
\right) = 0.
\end{equation}Moreover, recalling the definition 
\eqref{def:perspective-function}
of $\tilde f^\star$, we also have
\begin{equation}
\label{eq:representation:perspective:2}
{\mathbb P} \otimes \textrm{\rm Leb}_{[0,T]}
\left( 
\left\{ (\omega,t) \in \Omega \times [0,T], \;   
   q_t =0, \  \vert \tilde Y_t^\star \vert
+
\vert \tilde Z_t^\star \vert
> 0\right\}
\right) = 0.
\end{equation}
With the notation 
\eqref{eq:representation:Ystar:Zstar}, \eqref{eq:representation:perspective:3}
and
\eqref{eq:representation:perspective:4} easily follow. 
\vskip 4pt

\noindent \textit{Step 2:
Proving that $  q$
stays in $0$ once it has touched it.}
Recall that  
$\tau \coloneqq \inf\{t \in [0,T],\;   q_t =0\}$
(with $\inf \emptyset= + \infty)$. We want  to prove that ${\mathbb P}(\{ \sup_{t \in [\tau, T]}   q_t>0 \} \cap 
\{\tau < T\})=0$.
The proof is as follows. 
For any $\epsilon >0$, let $\varrho^\epsilon\coloneqq 
\inf\{t \in [\tau,T], \;   q_t = \epsilon\}$, with the convention that 
$\varrho^\epsilon = +\infty$ if $\tau=+\infty$
or if $\tau \leq T$ and $  q$ does not touch $\epsilon$ between 
$\tau$
and $T$. Using 
\eqref{eq:representation:perspective:3}
and 
\eqref{eq:representation:perspective:4}, 
we then notice that 
\begin{equation*}
\begin{split}
\dd \left( \exp(\alpha t)   q_t \right) \geq 
\exp(\alpha t) 
  q_t Z_t^\star \cdot \dd W_t, \quad t \in [0,T].
\end{split}
\end{equation*}
By localization (use \eqref{eq:local:tildeY:tildeZ} together with the fact that $\tilde Z_t^\star = 
  q_t Z_t^\star$), we can find a non-decreasing sequence of stopping times $(\sigma_k)_{k \geq 1}$, converging to $T$ (almost surely), such that 
\begin{equation*}
\forall k \geq 1, \quad    {\mathbb E}
\left[ \int_0^{\sigma_k}
  q_t^2 
\vert Z_t^\star \vert^2 
\dd t \right] < + \infty.
\end{equation*}
And then, 
\begin{equation*}
\exp\left( \alpha 
\tau \wedge \sigma_k 
\right) 
 {q}_{\tau \wedge \sigma_k}
\geq 
{\mathbb E}\left[
\exp\left( \alpha 
\varrho^{\epsilon}
\wedge \sigma_k 
\right) 
 {q}_{\varrho^{\epsilon} \wedge \sigma_k} \vert {\mathcal F}_{\tau} \right]. 
\end{equation*}
Letting $k$ tend to $+\infty$ and using a conditional version of Fatou's lemma, we obtain, ${\mathbb P}$-almost surely,
\begin{equation*}
\exp( \alpha \tau \wedge T)  {q}_{\tau \wedge T}
\geq 
{\mathbb E}
\left[
\exp(\alpha 
\varrho^\epsilon \wedge T) 
  q_{ \varrho^{\epsilon} \wedge T}
\vert 
{\mathcal F}
_{\tau} \right].
\end{equation*}
We deduce that there exists a constant $c>0$, only depending on $T$ and $\alpha$, such that 
\begin{equation*}
  q_{\tau \wedge T}
\geq c 
{\mathbb E}\left[   q_{\varrho^{\epsilon} \wedge T} \vert {\mathcal F}_{\tau}
\right]. 
\end{equation*}
Multiply both sides by 
${\mathds 1}_{\{\tau <T\}}$ and take expectation under ${\mathbb P}$. Since 
$ {q}_{\tau}=0$ when 
$\tau < T$, we get 
\begin{equation*}
{\mathbb E}\left[ 
{\mathds 1}_{\{\tau <T\}}  {q}_{\varrho^{\varepsilon} \wedge T} 
\right] =0.
\end{equation*}
This shows ${\mathbb P}(\{\tau < T\} \cap \{ \varrho^\epsilon \leq T\})=0$. Letting 
$\epsilon$ tend to $0$, we derive the expected claim, that is 
${\mathbb P}(\{\tau < T\} \cap \{  \sup_{t \in [\tau,T]}   q_t >0\})=0$.
This makes it possible to prove 
\eqref{eq:tildeS=S}. 
Indeed, together with 
\eqref{eq:representation:perspective:2}, 
we obtain
\begin{equation*}
{\mathbb P} \otimes \textrm{\rm Leb}_{[0,T]}
\left( 
\left\{ (\omega,t) \in \Omega \times [0,T], \;   
   t > \tau, \  \vert \tilde Y_t^\star \vert
+
\vert \tilde Z_t^\star \vert
> 0\right\}
\right) = 0,
\end{equation*}
from which we deduce that (recalling 
\eqref{eq:def:perspective:function:with:rec})
\begin{equation*}
\mathbb{E}\left[ \int_\tau^T  \tilde{f}^\star(t,q_t,\tilde{Y}^\star_t,\tilde{Z}^\star_t) \dd t \right]
=0.
\end{equation*}
Identity \eqref{eq:tildeS=S} 
easily follows. 

We now prove that ${\mathbb E}[q_T^*] < + \infty$. Letting 
$(\tilde q_t \coloneqq 
q_t \exp(- \int_0^t Y_s^\star \dd s))_{t \in [0,T]}$, 
we deduce from  \eqref{eq:representation:perspective:4} that
\begin{equation*}
\dd  \tilde q_t 
= \tilde q_t Z_t^\star \cdot \dd W_t, \quad t \in [0,T].
\end{equation*}
By It\^o's formula, 
\begin{equation*}
\dd \left[ 
\tilde q_t 
\ln \left( \tilde q_t \right) \right]
= 
\frac12 \tilde q_t \vert   Z_t^\star \vert^2 \dd t + 
\left[  \ln(\tilde q_t) + 1
\right] \tilde q_t Z_t^\star \cdot \dd W_t, \quad t \in [0,\tau). 
\end{equation*}
By a localization argument (together with \eqref{eq:representation:0}), we deduce that ${\mathbb E}[\tilde q_\tau \ln(\tilde q_\tau)]< + \infty$. And then, by $L\log(L)$-Doob's maximal inequality, we obtain ${\mathbb E}[\tilde q_T^*] = 
{\mathbb E}[\sup_{t \in [0,\tau]} \tilde q_t] \leq C
$, for a constant $C$ that depends on $q$ only via $c_1$. We deduce that 
${\mathbb E}[q_T^*]
< C \exp(\alpha T) $.
By Burkholder-Davis-Gundy inequalities, 
\eqref{eq:representation:perspective:40}
easily follows.
\vskip 4pt

\noindent \textit{Step 3: Contradicting the fact that $\tau \leq T$, when $(\psi, q)$ is a saddle-point.}
We now prove
the final result, that is ${\mathbb P}(\{\tau \leq T\})=0$ when $  q$ satisfies the optimality property of a saddle-point.
For $\theta \in (0,1)$, we let $q^\theta \coloneqq 
\theta   q+
(1-\theta) {q^0}$, where we recall \eqref{eq:barq:0}
for the definition of $q^0$. 
We have
\begin{align*}
\dd q_t^\theta &= \left[ \theta   q_t Y^\star_t
+ (1-\theta)   q_t^0 \partial_y f(t,0,0)\right]
\dd t+ \left[  \theta   q_t Z^\star_t 
+ 
(1-\theta) q_t^0 
\partial_z f(t,0,0)
\right]
\cdot \dd W_t
\\[0.5em]
& =:  \tilde Y^{\star,\theta}_t \dd t  + \tilde Z^{\star,\theta}_t 
\cdot \dd W_t,
\end{align*}
for any $t\in [0,T]$.
Since $q^\theta$ is positive valued, for 
$\theta \in [0,1)$, we can let 
$Y^{\star,\theta}_t=\tilde Y^{\star,\theta}_t/q_t^\theta$
and 
$Z^{\star,\theta}_t=\tilde Z^{\star,\theta}_t/q_t^\theta$, for 
$t \in [0,T]$.
By 
\eqref{eq:lower:bound:c_1}, we know that 
$\tilde{\mathcal S}( q^0)= {\mathcal S}( q^0) < c_1$. 
By convexity of
$\tilde{\mathcal S}$ (see Step 2 in the proof of Proposition \ref{prop:concave-J}), we deduce that 
$\tilde{\mathcal S}(q^\theta) \leq c_1$, and then 
$q^\theta \in \tilde {\mathcal Q}_{c_1}$ for all $\theta \in [0,1]$. 

By definition of $\tilde{{\mathcal J}}$ (see \eqref{eq:tilde:S}
and
\eqref{eq:def:tilde:J:F}),
\begin{align*}
    \tilde{{\mathcal J}}(  q^\theta,\psi) = & 
    {\mathcal G}(q_T^\theta,X_T^\psi) +    
    {\mathbb E} 
    \left[ \int_0^T q_s^\theta \ell(s,  \psi_s) \dd s 
    \right] - {\mathbb E}\left[ \int_0^T 
\tilde{f}^\star\left(s,q_s^\theta,\tilde  Y^{\star,\theta}_s, \tilde  Z^{\star,\theta}_s\right) \dd s
\right].
\end{align*}
We then subtract
$\tilde{{\mathcal J}}(  q,\psi)$ on both sides. Recalling the definition \eqref{def:perspective-function} of $\tilde{f}^\star$ we obtain 
\begin{align*}
 \tilde{{\mathcal J}}(   q^\theta,\psi) 
- \tilde{{\mathcal J}}(    q,\psi) = &
   \left( {\mathcal G}(q_T^\theta,X_T^{  \psi}) 
    -  {\mathcal G}(  q_T,X_T^{  \psi}) \right) +  {\mathbb E} 
    \left[ \int_0^T \left( q_s^\theta -   q_s\right) \ell(s,  \psi_s) \dd s 
    \right]
    \\
& - {\mathbb E}\left[ \int_0^T \left( 
\tilde f^\star\left(s,q_s^\theta,\tilde Y^{\star,\theta}_s,\tilde Z^{\star,\theta}_s
 \right) -  
\tilde f^\star\left(s,  q_s, \tilde Y^\star_s, \tilde Z^\star_s
\right) \right)\dd s
\right] 
\\
=: & S_1^\theta+ S_2^\theta- S_3^\theta.
\end{align*}
Using the fact that $q_t^\theta-  q_t=(1-\theta)( q^0_t -  q_t)$ together with the
regularity of ${\mathcal G}$ in the variable $q$ and the integrability 
properties of $\psi$, and then applying Lemma \ref{lemma:reg-X-psi-A}, we deduce that there exists a positive constant 
$C$ such that, for any $\theta \in (0,1)$, 
\begin{equation*}
\vert S_1^\theta \vert 
+ 
\vert S_2^\theta \vert \leq C (1-\theta). 
\end{equation*}
Similarly, by strong convexity of $f^\star$ (in the last variable), 
see Remark \ref{eq:assumption1-nabla-f-star}, there exists $c >0$ (independent of $\theta$) such that
\begin{align*}
& {\mathbb E}  \left[\int_0^T \tilde
f^\star\left(s,q_s^\theta, \tilde Y^{\star,\theta}_s,\tilde Z^{\star,\theta}_s
\right)
\dd s\right]
\\
&={\mathbb E}  \left[\int_0^T q_s^\theta
f^\star\left(s,
\frac{\theta   q_s}{q_s^\theta}
Y_s^{\star,\theta}
+ \frac{(1-\theta) q_s^0}{q_s^\theta}
\partial_y f_s^0,
\frac{\theta   q_s}{q_s^\theta}
Z_s^{\star,\theta}
+ \frac{(1-\theta) q_s^0}{q_s^\theta}
\partial_z f_s^0
\right)
\dd s\right]
\\
&\leq   {\mathbb E}  \left[  \int_0^T  
\left( \theta
\tilde f^\star\left(s,  q_s,\tilde Y_s^\star,\tilde Z^\star_s
\right) +   (1- \theta) q_s^0 
f_s^0
- c   (1-\theta) \frac{ \theta   q_s q_s^0}{q_s^\theta}
\vert Z_s^\star 
- \partial_z f_s^0 \vert^2 \right)\dd s\right],
\end{align*}
where we have used the shorthand notations 
$\partial_y f_s^0 \coloneqq 
\partial_y f(s,0,0)
$ and 
$\partial_z f_s^0\coloneqq \partial_z f(s,0,0)$, and the duality identity $f^\star(s,\partial_y f_s^0,\partial_z f_s^0) = f(s,0,0)=f_s^0$.
The last line, together with the fact that $\|f^0\|_{L^\infty(\mathbb{F})}  < + \infty$ and 
$\|\partial_z f^0\|_{L^\infty(\mathbb{F})}  < + \infty$ and the standard inequality $\vert Z_s^\star - \partial_z f_s^0\vert^2 \geq \tfrac12 \vert Z_s^\star\vert^2 - \vert \partial_z f^0_s\vert^2$, yields the following lower bound
\begin{equation*}
- S_3^\theta
 \geq   (1-\theta) \left( - C -
 {\mathbb E} \left[ \int_0^T  \tilde f^\star\left(s,  q_s,\tilde Y^\star_s,\tilde Z^\star_s
\right) \dd s \right]
+ \frac{c}2   {\mathbb E} \left[ \int_0^T  
 \frac{\theta   q_s q_s^0} {q_s^\theta}
|Z^\star_s|^2 
\dd s \right]\right). 
\end{equation*}
Using the fact that 
$\tilde{\mathcal S}(  q) \leq c_1$, 
we deduce in the end that (for a possibly new value of $C$)
\begin{align}
\tilde{{\mathcal J}}(  q^\theta,\psi) 
- 
 \tilde{{\mathcal J}}(    q,\psi) 
&\geq  (1-\theta) \left(-C + 
\frac{c}2  {\mathbb E} \left[ \int_0^T  
 \frac{\theta   q_s q_s^0}{q_s^\theta}
\vert Z^\star_s\vert^2 
\dd s \right]\right).    \label{eq:proof:positivity:edit:1}
\end{align}
It then remains to observe that (whether the right-hand side is finite or not)
\begin{equation}
\lim_{\theta \rightarrow 1}
 {\mathbb E} \left[ \int_0^T  
 \frac{ \theta   q_s q_s^0}{q_s^\theta}
\vert Z^\star_s\vert^2 
\dd s\right]
=  {\mathbb E}  \left[\int_0^\tau 
  q_s^0
\vert Z^\star_s\vert^2 
\dd s \right].
  \label{eq:proof:positivity:edit:2}
\end{equation}
Since $q$ is an optimizer of 
$\tilde{\mathcal J}( \cdot,\psi)$ over 
$\tilde{\mathcal Q}_{c_1}$, it holds 
$\tilde{\mathcal J}(q^\theta,\psi) \leq \tilde{\mathcal J}(q,\psi)$ implying that the left-hand side in \eqref{eq:proof:positivity:edit:1} is non-positive. 
Combining the last two lines, this shows that the right-hand side on the above identity is necessarily finite.
We claim that this implies that ${\mathbb P}(\{\tau \leq T\})=0$. 

Assume by a way of contradiction that  
${\mathbb P}(\{\tau \leq T\})>0$.
For $t \in [0,\tau)$, we can expand 
$\ln(q_t)$ by means of It{\^o}'s formula. We get
\begin{equation*}
\begin{split}
\dd \ln(q_t) &=  \left(Y^\star_t - \frac1{2} \vert Z^\star_t \vert^2 \right)
\dd t + Z^\star_t \cdot \dd W_t
\\
&=  \left(Y^\star_t - \frac1{2} \vert Z^\star_t \vert^2
+ 
Z_t^\star \cdot 
\partial_z f_t^0
\right) 
\dd t + Z^\star_t \cdot \dd \left( W_t
- \int_0^t \partial_z f_s^0 \dd s\right), \quad t \in [0,\tau).
\end{split}
\end{equation*}
Let 
${\mathbb Q}^0 \coloneqq {\mathcal E}_T(\int_0^\cdot \partial_z f_s^0 \cdot \dd W_s)$. 
Setting 
$\sigma^\epsilon \coloneqq \inf\{t \geq 0, \; q_t \leq \epsilon\}$
for any $\epsilon \in (0,1)$
(with the convention that $\inf \emptyset = + \infty$), we have 
\begin{equation*}
\begin{split}
- {\mathbb E}^{{\mathbb Q}^0} \left[ \ln(q_{\sigma^\epsilon \wedge T})\right]
&\leq \frac{1}{2} {\mathbb E}^{{\mathbb Q}^0} \left[
\int_0^T \vert Z^\star_t \vert^2 \dd t\right] + {\mathbb E}^{{\mathbb Q}^0}
\left[
\int_0^T \vert Z_t^\star \vert \vert 
\partial_z f_t^0 \vert \dd t
\right] + \alpha T,
\\
&\leq {\mathbb E}^{{\mathbb Q}^0} \left[
\int_0^T \vert Z^\star_t \vert^2 \dd t\right] + C,
\end{split}
\end{equation*}
for a constant $C$ depending on $\alpha$, $T$ and $\| 
\partial_z f^0\|_{L^\infty({\mathbb F})}$.
Now, if ${\mathbb P}(\{ \tau \leq T\})>0$, then
${\mathbb Q}^0(\{ \tau \leq T\})>0$,
and
$\sup_{\epsilon >0} [
-\ln(\epsilon) {\mathbb Q}^0(\{ \sigma^\epsilon \leq T\})]
=+\infty$, and thus 
the right-hand side is also infinite, 
which 
contradicts 
\eqref{eq:proof:positivity:edit:1} and     \eqref{eq:proof:positivity:edit:2}. 
\end{proof}

\subsubsection{Solvability of the perspective min-max problem}
\begin{proposition} \label{prop:concave-J}
Let $\psi \in L^\infty(\mathbb{F},\mathbb{R}^n)$ and $c_1>0$. Viewing $\tilde{{\mathcal Q}}_{c_1}$
as a subset of $L^1(\Omega \times [0,T],{\mathbb P} \otimes 
{\rm Leb}_{[0,T]})$ equipped with the weak topology 
$\sigma(L^1,L^\infty)$,
$\tilde{{\mathcal Q}}_{c_1}$ is (weakly) compact and convex, and satisfies 
\begin{equation}
\label{eq:entrop:bounded:Q}
\sup_{q \in \tilde{Q}_{c_1}} \sup_{t \in [0,T]} {\mathbb E}[h(q_t)] < + \infty.
\end{equation}
In addition, 
the mapping $\tilde{{\mathcal Q}}_{c_1} \ni q \mapsto\tilde{\mathcal {J}}(\psi,q)$ is strictly concave and upper semi-continuous (w.r.t. the weak topology).
\end{proposition}

\begin{remark}
\label{rem:convexity:tildeS:application}As a corollary of the proof, we obtain that the functional
$\tilde{\mathcal S}$ is convex, which has further applications.
Indeed, for any $\theta \in [0,1]$, let (as in the proof of
Lemma 
\ref{lemma:positive-q})
$q^\theta \coloneqq (1-\theta) q^0 + \theta q$, where $q^0$ is defined as in \eqref{eq:barq:0}. We observe that $q^\theta$
is positive valued for each $\theta \in [0,1)$. By Lemma~\ref{lemma:positive-q},
it is easy to see that, for every $\theta \in [0,1)$, $q^\theta \in \mathcal Q$.
Moreover, by convexity of $\tilde{\mathcal Q}_{c_1}$, we have
$q^\theta \in \mathcal Q_{c_1}$, provided that $c_1$ is large enough,
which is not a restriction here.
This shows that $q^\theta \in \mathcal Q_{c_1}$ and, more generally, that
$\mathcal S(q^\theta)
\leq (1-\theta)\mathcal S(q^0) + \theta \tilde{\mathcal S}(q)$.
Since $\mathcal S(q^0) \in \mathbb R$, we deduce that
\[
\limsup_{\theta \to 1} \mathcal S(q^\theta)
\leq \tilde{\mathcal S}(q).
\]
This observation allows us to extend results holding on $\mathcal Q_{c_1}$
to $\tilde{\mathcal Q}_{c_1}$, such as
the duality inequality
\eqref{ineq:S-S-star}, and 
Lemmas~\ref{lemma:reg-X-psi-A} and
\ref{lem:E:q:X*:S(q):Sstar(psi)}.
Similarly, Lemmas~\ref{lem:uniform:integrability} and \ref{lemma:weak-convergence-q}, which are invoked in
Step~5 of the proof below, can also be extended to sequences with values in
$\tilde{\mathcal Q}$, using the additional fact that the function $h$
is convex (recall \eqref{eq:entropy:definition:h} for the definition of $h$).
\end{remark}

\begin{proof}
The proof is divided into six steps. In Step 1, we show the relative weak compactness of  $\tilde{\mathcal{Q}}_{c_1}$. In Step 2, we establish the (strict) concavity of the mapping $\tilde{\mathcal{J}}$ and we prove the convexity of $\tilde{\mathcal{Q}}_{c_1}$. In Steps 3 and 4, we show the weak lower semi-continuity of $\tilde{\mathcal{S}}$ and the weak compactness of $\tilde{\mathcal{Q}}_{c_1}$. In Step 5, we establish the weak continuity of $\mathcal{R}$.
In Step 6, we conclude the proof. Throughout, the value of $\psi$ is fixed. For this reason, we omit it in many notations. For instance, we just write 
$\tilde{\mathcal{J}}(q)$
for 
$\tilde{\mathcal{J}}(q,\psi)$.

\vskip 4pt

\noindent \textit{Step 1: relative weak compactness of $\tilde{\mathcal{Q}}_{c_1}$.} 
We first establish the relative weak compactness, in 
$L^1(\Omega \times [0,T], {\mathbb P} \otimes {\rm Leb}_{[0,T]})$
 equipped with the weak topology 
 $\sigma(L^1,L^\infty)$, of any subset $\mathcal{D}$ of $L\log L(\mathbb{F})$
that is bounded in the sense that
\begin{equation*}
    \sup_{q \in {\mathcal D}} \sup_{t \in [0,T]} {\mathbb E}\bigl[h(q_t)\bigr]
    < +\infty.
\end{equation*} The argument is classic and goes as follows. By de la Vall{\'e}e Poussin Theorem \cite[Theorem VI]{poussin1915integrale}
(see also \cite[Theorem 22, page 24-II]{dellacherie:meyer}),
the set ${\mathcal D}$ is uniformly integrable on 
$ \Omega \times [0,T] $ equipped with 
${\mathbb P} \otimes \textrm{\rm Leb}_{[0,T]}$.
Then, by the Dunford-Pettis Theorem \cite[Theorem 4.30]{brezis}, the set is weakly relatively compact. The set is weakly sequentially relatively compact by the Eberlein-\v{S}mulian Theorem \cite[Section V.6.1, p.430]{dunford1988linear}. 
The last two statements can also be found combined into a single statement, see \cite[Theorem 25, page 27-II]{dellacherie:meyer}.

We now prove that 
$\tilde{Q}_{c_1}$ is a bounded subset of $L \log L({\mathbb F})$ (which 
corresponds to 
\eqref{eq:entrop:bounded:Q} in the statement).  Indeed, for any  $q \in \tilde{\mathcal{Q}}_{c_1}$,
letting 
$\tau\coloneqq\inf\{t \in [0,T], \; q_t=0\}$, we know from Lemma \ref{lemma:positive-q}
that $q$ stays equal to $0$ after $\tau$. Moreover, it satisfies  \eqref{eq:tildeS=S} (under the notation 
\eqref{eq:representation:perspective:3}), from which we deduce that  
\begin{equation}
\label{eq:bound:h:tilde:Q}
\begin{split}
c_1 \geq \tilde{\mathcal{S}}(q) 
& = \mathbb{E}\left[ \int_0^\tau q_t f^\star(t,Y^\star_t,Z^\star_t) \dd t \right] \geq -C + \frac{1}{2 \beta} \mathbb{E}\left[ \int_0^\tau q_t |Z^\star_t|^2 \dd t \right],
\end{split}
\end{equation}
where the process $(Y^\star_t,Z^\star_t)_{t \in [0,T]}$ denotes $({\mathds 1}_{\{q_t>0\}}\tilde{Y}^\star_t/q_t,
{\mathds 1}_{\{q_t>0\}} \tilde{Z}^\star_t/q_t)_{t \in [0,T]}$, and where
we used the coercivity inequality \eqref{ineq:duality-fstar} of $f^\star$.
The constant $C$ only depends on the $L^\infty$ bound for $f^0$.

We now expand 
$(h(q_t))_{t \in [0,\tau)}$ by means of It\^o's formula. We get 
\begin{equation}
\label{eq:h:qt:ito:expansion}
\dd h(q_t) = 
\left( (h(q_t)+q_t)Y_t^\star
+ \frac12 q_t \vert Z_t^\star \vert^2 
\right) \dd t
+ (h(q_t)+q_t)Z_t^\star
\cdot \dd W_t, \quad t \in [0,\tau). 
\end{equation}
Recalling that 
$Y^\star$ is bounded by $\alpha$ and using a standard localization argument, we deduce that 
\begin{equation*}
\sup_{t \in [0,T]}
{\mathbb E}
\left[ h(q_{t \wedge \tau})
\right] \leq C + C 
{\mathbb E}\left[ \int_0^\tau q_t \vert Z_t^\star \vert^2 \dd t \right], 
\end{equation*}
for a constant $C$ independent of $q$.
Observing that $h(0)=0$, the left-hand side is also equal to 
$\sup_{t \in [0,T]}{\mathbb E}[h(q_t)]$. 
Back to \eqref{eq:bound:h:tilde:Q}, we deduce that 
$\tilde{\mathcal Q}_{c_1}$ is a bounded subset of $L\log L({\mathbb F})$. 
\vskip 4pt

\noindent \textit{Step 2: (strict) concavity of $\tilde{\mathcal{Q}} \ni q \mapsto \tilde{\mathcal{J}}(q)$.}
Recalling the definition of 
$\tilde{\mathcal Q}$ on the line below 
\eqref{eq:def:tilde:J:F}, we first notice that 
$\tilde{\mathcal Q}$  is convex. Indeed, 
the decomposition 
\eqref{eq:ito:q} is linear and thus stable by convex combinations.  Moreover, the non-negativity constraint and the bound $
{\mathbb E}[q_T^*] \leq \exp(\alpha T)$, which  are both required in the definition of $\tilde {\mathcal Q}$,  are also stable by convex combinations.
It remains to see that, for  any 
$\theta \in (0,1)$, the convex combination $q^\theta \coloneqq \theta q^1 + (1-\theta) q^2$ of any two  
$q^1, q^2 \in \tilde{\mathcal{Q}}$
satisfies $\tilde{\mathcal S}(q^\theta)<+\infty$.
Denoting  by  $(\tilde{Y}^{\star,\theta}, \tilde{Z}^{\star,\theta})$,  
$(\tilde{Y}^{\star,1}, \tilde{Z}^{\star,1})$
and 
$(\tilde{Y}^{\star,2}, \tilde{Z}^{\star,2})$
the respective representation processes of 
$q^\theta$, $q^1$ and $q^2$ in 
\eqref{eq:ito:q}, we have
\begin{equation*}
    \tilde{\mathcal{S}}(q^\theta) = \mathbb{E}\left[ \int_0^T \tilde{f}^{\star}(s,q_s^\theta ,Y_s^{\star,\theta}, Z_s^{\star,\theta}) \dd s  \right],
\end{equation*}
with 
$(\tilde{Y}^{\star,\theta}, \tilde{Z}^{\star,\theta})
=
\theta 
(\tilde{Y}^{\star,1}, \tilde{Z}^{\star,1})+ (1-\theta) (\tilde{Y}^{\star,2}, \tilde{Z}^{\star,2})$.
Since $f^\star$
is strictly convex in
its last two variables, see  Remark \ref{eq:assumption1-nabla-f-star}, 
$\tilde f^\star$ is (strictly) jointly convex  in its last three arguments, see
\cite[Lemma 1.157]{bonnans2019convex}. 
It easily follows that 
$\tilde{S}(q^\theta) < + \infty$, i.e., 
$\tilde {\mathcal Q}$ is convex.
Moreover, 
$\tilde{\mathcal{S}}$ 
is strictly convex on 
$\tilde{\mathcal Q}$. 

It remains to see that, by the concavity Assumption \ref{eq:G-concave-convex}, 
 $\tilde{\mathcal{Q}}\ni q \mapsto \mathcal{R}(q)$ defined in \eqref{def:F} is concave
(with the shorthand notation 
${\mathcal R}(q)$ for ${\mathcal R}(q,\psi)$). 
We deduce that $\tilde {\mathcal Q} \ni q \mapsto \tilde{\mathcal J}(q)$ is strictly convex.
By convexity of $\tilde{\mathcal S}$, 
$\tilde{\mathcal Q}_{c_1}$ is also convex. \vskip 4pt

\noindent \textit{Step 3: weak lower semi-continuity of $\tilde{\mathcal{S}}$.} 
We begin with  further compactness properties of $\tilde {\mathcal Q}_{c_1}$. We thus consider 
a sequence $(q^k)_{k \in \mathbb{N}}$ lying in $\tilde{\mathcal{Q}}_{c_1}$. We denote by $(\tilde{Y}^{\star,k},\tilde{Z}^{\star,k})_{k \in \mathbb{N}}$ the sequence of processes associated to $(q^k)_{k \in \mathbb{N}}$, 
as given by \eqref{eq:ito:q}. We first prove that the two sequences $(\tilde Y^{\star,k})_{k \in {\mathbb N}}$ and 
$(\tilde Z^{\star,k})_{k \in {\mathbb N}}$
are relatively compact with respect to the weak topologies  on 
$L^1(\Omega \times [0,T],{\mathbb R}, {\mathbb P} \otimes \textrm{\rm Leb}_{[0,T]})$ and  $L^1(\Omega \times [0,T],{\mathbb R}^d,{\mathbb P} \otimes \textrm{\rm Leb}_{[0,T]})$ respectively. The relative compactness
of $(\tilde Y^{\star,k})_{k \in {\mathbb N}}$ is established as in Step 1, by proving that the  sequence is uniformly integrable. 
The latter is quite obvious: by 
\eqref{eq:representation:perspective:3}, 
we know that the sequence 
$(\vert \tilde{Y}^{\star,k} \vert)_{k \in {\mathbb N}}$ is dominated by 
$(\alpha q^{k})_{k \in {\mathbb N}}$; by Step 1, the sequence 
$(q^k)_{k \in {\mathbb N}}$ is 
uniformly integrable and, therefore, the sequence $(  \tilde{Y}^{\star,k}  )_{k \in {\mathbb N}}$ is also uniform integrable. 
To prove the relative compactness of the sequence $( Z^{\star,k} )_{k \in {\mathbb N}}$ in 
$L^1(\Omega \times [0,T],{\mathbb P} \otimes \textrm{\rm Leb}_{[0,T]},{\mathbb R}^d)$, we proceed as follows.
By \eqref{eq:bound:h:tilde:Q}, we observe that, for any $k \in {\mathbb N}$, 
any subset $A \subset \Omega \times [0,T]$ in the progressive 
$\sigma$-field, and any real $\varepsilon >0$, 
\begin{equation*}
\begin{split}
{\mathbb E} \left[ 
\int_0^T {\mathds 1}_A(t) \vert \tilde Z_t^{k,\star} \vert \dd t 
\right]
&= {\mathbb E} \left[
\int_0^T {\mathds 1}_A(t) {\mathds 1}_{\{ q^{k}_t >0\}} \vert \tilde Z_t^{k,\star} \vert \dd t 
\right]
\\
&\leq \frac1{\varepsilon}
{\mathbb E} \left[ 
\int_0^T {\mathds 1}_A(t) q^k_t \dd t 
\right]
+ \varepsilon {\mathbb E} \left[
\int_0^T {\mathds 1}_{\{ q^{k}_t >0\}}\frac1{q^k_t} 
\vert \tilde Z^{k,\star}_t \vert^2 \dd t 
\right]. 
\end{split}
\end{equation*}
Letting $Z^{k,\star}_t= {\mathds 1}_{\{q^k_t>0\}} \tilde Z^{k,\star}/q^k_t$, we deduce from \eqref{eq:bound:h:tilde:Q} that there exists a constant $C$, independent of $k$, $A$ and $\varepsilon$, such that 
\begin{equation*}
\begin{split}
{\mathbb E} \left[ 
\int_0^T {\mathds 1}_A(t) \vert \tilde Z^{k,\star}_t \vert  \dd t 
\right]
&\leq \frac1{\varepsilon}
{\mathbb E} \left[
\int_0^T{\mathds 1}_A(t)  q^k_t \dd t 
\right]
+ C \varepsilon. 
\end{split}
\end{equation*}
By the uniform integrability property established in the first step, we know that the first term on the right-hand side can be made as small as desired by choosing 
$\textrm{\rm Leb}_{[0,T]} \otimes {\mathbb P}(A)$ small enough. 
This proves that 
the collection $(\tilde Z^{\star,k})_{k \in {\mathbb N}}$ is 
uniformly integrable, on $\Omega \times [0,T]$, equipped with 
${\mathbb P} \otimes \textrm{\rm Leb}_{[0,T]}$.

We now come back to the proof of the lower semi-continuity of $\tilde {\mathcal S}$. Let
$(q^k)_{k \in {\mathbb N}}$
be a sequence in 
$\tilde{Q}_{c_1}$
that converges for the weak topology on 
$L^1(\Omega \times [0,T],{\mathbb P} \otimes \textrm{\rm Leb}_{[0,T]})$.
The relative compactness of 
$(\tilde Y^{\star,k},\tilde{Z}^{\star,k})_{k \in {\mathbb N}}$ allows us to extract a subsequence (still indexed by $k$) that
  converges in $L^1(\Omega \times [0,T],{\mathbb R},{\mathbb P} \otimes \textrm{\rm Leb}_{[0,T]}) \times L^1(\Omega \times [0,T],{\mathbb R}^d,{\mathbb P} \otimes \textrm{\rm Leb}_{[0,T]})$ equipped with the product of the weak topologies (on each factor). 
Since 
the objective is to prove the lower semi-continuity of the convex functional $\tilde {\mathcal S}$, we can replace $(q^k,\tilde Y^{\star,k},\tilde Z^{\star,k})_{k \in {\mathbb N}}$ by a convex combination that converges for the strong topologies.
We denote by $(q,\tilde Y^\star,\tilde Z^\star)$ its (strong) limit. 
The objective is to pass to the limit in 
\eqref{eq:ito:q} and to prove that 
$q$ is a non-negative valued process that can be expanded as
\begin{equation}
\label{eq:proof:lsc:tildeS:00}
\forall t \in [0,T], 
\quad 
q_t = 1 + \int_0^t \tilde Y_s^\star \dd s + 
\int_0^t \tilde Z_s^\star \cdot \dd W_s.
\end{equation}
In fact, the difficulty is to pass to the limit in the stochastic integrals appearing in the expansion of each $q^k$. 
By \eqref{eq:representation:perspective:40},  
\begin{equation*} 
\sup_{k \in {\mathbb N}}
{\mathbb E}
\left[ \left( 
\int_0^T \vert \tilde Z_t^{\star,k}\vert^2 \dd t \right)^{1/2}\right] < +\infty.
\end{equation*}
By Fatou's lemma, we deduce that $\tilde Z^\star$ satisfies the same bound. And then, by a new uniform integrability argument (combining with the convergence in $L^1$),
\begin{equation*}
\forall \eta \in (0,1),
\quad 
\lim_{k \rightarrow  \infty}
{\mathbb E}
\left[ \left( 
\int_0^T \vert \tilde Z_t^{\star,k}
-
\tilde Z_t^\star
\vert^2 \dd t \right)^{\eta/2}\right] =0,
\end{equation*}
which shows that 
\begin{equation*}
\forall \eta \in (0,1),
\quad 
\lim_{k \rightarrow  \infty}
{\mathbb E}
\left[ 
\sup_{t \in [0,T]}
\left\vert 
\int_0^t 
\left( 
\tilde Z_s^{\star,k}
-
\tilde Z_s^\star
\right) \cdot \dd W_s
\right\vert^{\eta}\right] =0. 
\end{equation*}
This makes it possible to derive \eqref{eq:proof:lsc:tildeS:00}.
Moreover, using 
Fatou's lemma, we deduce from the following three inequalities 
(with the last one following from \eqref{eq:bound:h:tilde:Q})
\begin{equation*}
\begin{split}
&
{\mathbb E}[q_T^{k,*}] \leq \exp(\alpha T),
\quad {\mathbb E} \left[ \int_0^T 
\frac{\vert \tilde Y^{\star,k}\vert}{q^k_t}
\dd t \right] \leq T \alpha,
\quad {\mathbb E} 
\left[ \int_0^T 
\frac{\vert \tilde Z^{\star,k}\vert^2}{q^k_t}
\dd t \right] \leq C + 2 \beta c_1,
\end{split}
\end{equation*}
that 
\begin{equation*}
\begin{split}
&
{\mathbb E}[q_T^*] \leq \exp(\alpha T),
\quad {\mathbb E} 
\left[\int_0^T 
{\mathds 1}_{\{q_t=0,\tilde Y_t^\star \not = 0\}}
\dd t \right]=0,
\quad {\mathbb E}  \left[ \int_0^T 
{\mathds 1}_{\{q_t=0,\tilde Z_t^\star \not = 0\}}
\dd t \right] =0. 
\end{split}
\end{equation*}
Letting 
$Y^{\star}_t\coloneqq{\mathds 1}_{\{q_t>0\}}\tilde Y^{\star}_t/q_t$
and $Z_t^\star\coloneqq{\mathds 1}_{\{q_t>0\}} \tilde Z_t^\star/q_t$, we deduce that \eqref{eq:representation:perspective:1} and \eqref{eq:representation:perspective:2} hold true (even though we do not have yet that $q \in \tilde{\mathcal Q}_{c_1}$ which prevents us from applying Lemma \ref{lemma:positive-q} at this stage of the proof).
We also have
\begin{equation*}
{\mathbb E} \left[ \int_0^T 
q_t \vert Z_t^\star \vert^2 \dd t \right] < + \infty.
\end{equation*}
In particular, 
$\vert Z_t^\star\vert$
is (a.e.) finite when $q_t>0$, and we can write, ${\rm Leb}_{[0,T]} \otimes {\mathbb P}$ almost everywhere, 
$\tilde Y_t^\star = q_t Y_t^\star$ and 
$\tilde Z_t^\star = q_t Z_t^\star$.
By repeating the proof of the first claim in item \textit{ii} of Lemma 
\ref{lemma:positive-q}, 
we also have that 
${\mathbb P}(\{ \sup_{t \in [0,T]} q_t >0\} \cap \{  \tau < T\})=0$, where $\tau \coloneqq \inf \{ t \in [0,T]; \ q_t =0\}$.

We now come back to the definition of $\tilde{\mathcal Q}_{c_1}$. Using the fact that $q^k \in \tilde{\mathcal Q}_{c_1}$  and letting $\tau^k\coloneqq\inf\{t \in [0,T]; \ q^k_t = 0\}$, 
for each 
$k \in {\mathbb N}$,
we deduce
from \eqref{eq:tildeS=S} (together with the first claim in item \textit{ii} of Lemma \ref{lemma:positive-q}, which allows us  to derive the third line below)
that 
\begin{equation}
\label{eq:proof:lsc:tildeS:1}
\begin{split}
\tilde{\mathcal S}(q^k) &= {\mathbb E}
\left[ \int_0^{\tau^k} q^k_s f^\star(s,Y_s^{\star,k},Z_s^{\star,k}) \dd s \right]
\\
&\geq 
{\mathbb E} \left[
\int_0^{\tau^k} q^k_s \left( Y_s Y_s^{\star,k} + Z_s \cdot Z^{\star,k}_s - f(s,Y_s,Z_s) \right) \dd s \right]
\\
&= 
{\mathbb E} \left[
\int_0^T  \left( Y_s \tilde Y_s^{\star,k} + Z_s \cdot \tilde Z^{\star,k}_s - q^k_s f(s,Y_s,Z_s) \right) \dd s \right],
\end{split}
\end{equation}
where $(Y,Z)$ is taken, for a certain $R>0$, in the ball of center $0$
and radius $R$ of the space $L^\infty({\mathbb F}) \times L^\infty({\mathbb F},{\mathbb R}^d)$, i.e. 
for almost every $(\omega,s) \in \Omega \times [0,T]$,
\begin{equation*}
    |Y_s(\omega)| + |Z_s(\omega)| \leq R. 
\end{equation*}
Letting $k$ tend to $+\infty$ in 
\eqref{eq:proof:lsc:tildeS:1}, we deduce that, for any $R>0$,
\begin{equation}
\label{eq:proof:lsc:tildeS:2}
\begin{split}
c_1 \geq \sup_{(Y,Z) \in {\mathcal B}_R}
{\mathbb E}
\left[
\int_0^T q_s
\left( Y_s Y_s^\star + Z_s \cdot Z_s^\star - f(s,Y_s,Z_s) \right) \dd s \right].
\end{split}
\end{equation}
This prompts us to define, for any $R>0$, the  mapping $f^\star_R \colon \Omega \times [0,T] \times \mathbb{R} \times \mathbb{R}^d \to \mathbb{R}$,
by letting
\begin{equation} \label{def:f-r}
    f^\star_R (t,y^\star,z^\star) \coloneqq \sup_{|y| + |z|\leq R} 
    \left\{ y^\star  y  + z^\star \cdot z - f(t,y,z) \right\}.
\end{equation}
Let us remark that $(f^\star_R)_{R \in \mathbb{N}}$ is a non-decreasing sequence, converging pointwise to $f^\star$.
By \cite[Proposition 3.78]{bonnans2019convex}, we have that 
\begin{equation}
\label{eq:proof:lsc:tildeS:3}
\begin{split}
    &\sup_{(Y,Z) \in \mathcal{B}_R} \mathbb{E}\left[ \int_0^T q_s \left( Y_s 
    Y^{\star}_s + Z_s \cdot 
    Z^{\star}_s - f(s,Y_s,Z_s) \right) \dd s  \right]\\
    & = \mathbb{E}\left[ \int_0^T q_s  f_R^{\star}(s,Y^{\star}_s,Z^{\star}_s) \dd s  \right]
    = \mathbb{E}\left[ \int_0^\tau q_s  f_R^{\star}(s,Y^{\star}_s,Z^{\star}_s) \dd s  \right],
    \end{split}
\end{equation}
where we recall that $\tau=\inf\{t \in [0,T]; \ q_t =0\}$.
Combining \eqref{eq:proof:lsc:tildeS:2}
and
\eqref{eq:proof:lsc:tildeS:3} together with Fatou's lemma, we obtain  
\begin{equation}
\label{eq:proof:lsc:tildeS:5}
\begin{split}
    c_1 &\geq 
    \liminf_{R \rightarrow + \infty} \mathbb{E}\left[ \int_0^\tau 
    q_s f_R^{\star}(s,q_s,\tilde{Y}^{\star}_s,\tilde{Z}^{\star}_s) \dd s  \right]
    \\  &\geq\mathbb{E}\left[ \int_0^\tau
    q_s  f^{\star}(s,q_s,\tilde{Y}^{\star}_s,\tilde{Z}^{\star}_s) \dd s  \right] = \tilde{\mathcal S}(q),
    \end{split}
\end{equation} 
with the last identity 
following from \eqref{eq:tildeS=S} (which holds true here,
thanks to the analysis achieved in the first part of this step). This shows that $\tilde{\mathcal{S}}$ is weakly lower semi-continuous. 

\vskip 4pt
\noindent \textit{Step 4: weak compactness of $\tilde{\mathcal{Q}}_{c_1}$.} 
Most of the work has been done in the previous steps. We know from the first step that any sequence $(q^k)_{k \in {\mathbb N}}$ is relatively compact for the weak topology on $L^1(\Omega \times [0,T],{\mathbb P} \otimes \textrm{\rm Leb}_{[0,T]})$. 
By 
\eqref{eq:proof:lsc:tildeS:00} in Step 3, we know that any weak limit $q$ can be expanded as in 
\eqref{eq:ito:q}, is non-negative valued and satisfies ${\mathbb E}[q_T^*] \leq \exp(\alpha T)$. And by 
\eqref{eq:proof:lsc:tildeS:5}, $\tilde{\mathcal S}(q) \leq c_1$. 
\vskip 4pt

\noindent \textit{Step 5: weak continuity of $\mathcal{R}$.}
We are thus given a sequence $(q^k)_{k \in \mathbb{N}} \in \tilde {\mathcal Q}_{c_1}$. By the previous step, there exists $q \in   \tilde{\mathcal Q}_{c_1}$ such that, 
up to a subsequence, 
$(q^k)_{k \in {\mathbb N}}$ weakly converges (i.e., with respect to 
$\sigma(L^1,L^\infty)$)
to $q$ as $k \to +\infty$. By concavity of $\mathcal{G}$, we have 
(with the shorthand notation ${\mathcal G}(q)$ for 
${\mathcal G}(q,X^\psi_T)$, and similarly for the derivative
$\delta_q \mathcal{G}(q_T)$)
\begin{equation}
\label{eq:lower:sc:R:qn:qinfty}
\begin{split}
    \mathcal{R}(q^k) &  = \mathcal{G}\left(q^k_T \right) +  \mathbb{E}\left[ \int_0^T q^k_s \ell_s \dd s  \right] \\
    & \leq\mathcal{G}\left(q_T  \right) +  \mathbb{E}\left[ \left(q^k_T - q_T \right) \delta_q \mathcal{G}(q_T )  + \int_0^T q^k_s \ell_s \dd s  \right] \\
    & = \mathcal{R}\left(q \right) +  \mathbb{E}\left[ \left(q^k_T - 
    q_T \right) \delta_q \mathcal{G}\left(q_T \right)  + \int_0^T 
    \left(q^k_s - q_s \right) \ell_s \dd s  \right].
\end{split}
\end{equation}
Recalling the growth Assumption \ref{eq:G-growth} on $\mathcal{G}$, we have
\begin{equation*}
    \left| \delta_q \mathcal{G}\left(q_T \right)\right|
    =
    \left| \delta_q \mathcal{G}\left(q_T ,X_T^\psi\right)\right| \leq L\left(1+|X_T^\psi|^{2-r} + \mathbb{E} \left[ q_T |X_T^\psi|^{2-r}\right]  \right),
\end{equation*}and then, $\delta_q \mathcal{G}$ admits exponential moments of all orders since $\psi$ is assumed to be  bounded. Thanks to Step 1, Lemma \ref{lem:uniform:integrability} (applicable by  
Remark 
\ref{rem:convexity:tildeS:application})
yields that the sequences 
\begin{equation*}
    (q^k_T \delta_q \mathcal{G}\left(q_T \right))_{k \in \mathbb{N}}, \quad (q^k_s \ell_s)_{k \in \mathbb{N}} 
\end{equation*}
are uniformly integrable. Since $(q^k)_{k \in \mathbb{N}}$ is assumed to converge weakly with respect to 
$\sigma(L^1,L^\infty)$, we deduce from Lemma \ref{lemma:weak-convergence-q} (which is applicable even though the sequence
$(q^k)_{k \in \mathbb N}$ takes values in
$\tilde{\mathcal Q}_{c_1}$ and possibly not in
$\mathcal Q_{c_1}$) that
\begin{equation*}
    \lim_{k \to +\infty}\mathbb{E}\left[ \left(q^k_T - q_T \right) \delta_q \mathcal{G}\left(q_T \right)  + \int_0^T 
    \left(q^k_s - q_s \right) \ell_s \dd s  \right] = 0.
\end{equation*}
Then $\limsup_{k \to \infty} \mathcal{R}(q^k) \leq  \mathcal{R}(q)$ concluding the step. 
\vskip 4pt

\noindent \textit{Step 6: Conclusion.}
The statement follows from the combination of Steps 1 to 6.
\end{proof}

\begin{lemma} \label{lemma:convexity-J} Let $c_1,c_2 >0$ and $q \in \tilde{\mathcal{Q}}_{c_1}$.
    The mapping ${\mathcal A}_{c_2} \ni \psi \mapsto \tilde{\mathcal{J}}(q,\psi)$ is convex, lower semi-continuous. In addition, ${\mathcal A}_{c_2}$ is convex and weakly compact in $L^2(\mathbb{F},{\mathbb Q}^0,\mathbb{R}^n)$, 
    where 
    ${\mathbb Q}^0\coloneqq{\mathcal E}_T(\int_0^\cdot \partial_z f(t,0,0) \cdot \dd W_t) {\mathbb P}$.
    In particular, it is bounded in $M^\eta({\mathbb F},{\mathbb R}^n,{\mathbb P})$, for any $\eta \in (0,1)$.
\end{lemma}

\begin{proof}
 \textit{Step 1: convexity and weak compactness of ${\mathcal A}_{c_2}$.} The convexity of ${\mathcal S}^\star$, regarded as 
 a $[0,+\infty]$-valued mapping defined on 
 the space of ${\mathbb F}$-progressively measurable ${\mathbb R}^n$-valued processes, 
 is a direct consequence of its definition \eqref{eq:expo:bound:psi}:  for each $q \in  {\mathcal Q}$, the mapping 
 $\psi \mapsto 
 {\mathbb E} [\int_0^T q_s \vert \psi_s \vert^2 \dd s]$ is convex since 
 $q$ takes non-negative values. 
 Therefore, 
 ${\mathcal S}^\star$
is convex as the supremum of a family of convex mappings. As a result, the set  ${\mathcal A}_{c_2}$ is convex.
    
    The relative weak compactness of ${\mathcal A}_{c_2}$ in $L^2(\mathbb{F},
    {\mathbb Q}^0,\mathbb{R}^n)$ is a direct consequence of the fact that 
    \begin{equation*}
        c_2 \geq {\mathcal S}^\star(\psi) = \sup_{q \in \mathcal{Q}} \left\{{\mathbb E} \left[\int_0^T q_t\vert \psi_t \vert^2 \dd t \right] - \frac{1}{\gamma}\mathcal{S}(q)\right\} \geq  {\mathbb E} \left[q^0_T \int_0^T \vert \psi_t \vert^2 \dd t \right] - \frac1{\gamma} {\mathcal S}(q^0),
    \end{equation*} 
    with $q^0$ as in \eqref{eq:barq:0}.
     Then, we are left to prove the weak closure property of ${\mathcal A}_{c_2}$. By convexity of the latter, it suffices to show that it is 
closed for the strong topology (on $L^2({\mathbb F},{\mathbb R}^n,{\mathbb Q}^0)$). We thus consider an ${\mathcal A}_{c_2}$-valued sequence $(\psi^k)_{k \in \mathbb{N}}$ that (strongly) converges
in 
$L^2({\mathbb F},{\mathbb R}^n,{\mathbb Q}^0)$ to some limit $\bar \psi$. Then, for all $q \in  {\mathcal Q}$, 
    \begin{equation*}
        {\mathbb E}\left[
    \int_0^T q_t \vert \psi^k_t \vert^2 \dd t \right] - \frac{1}{\gamma}{\mathcal S}(q) \leq c_2.
    \end{equation*} 
    By Fatou's lemma, the mapping $\psi \mapsto {\mathbb E}[\int_0^T q_t \vert \psi_t \vert^2 \dd t]$ is  lower semi-continuous for the strong topology on $L^2({\mathbb F},{\mathbb R}^n,{\mathbb Q}^0)$. We deduce that 
    \begin{equation*}
        {\mathbb E}\left[
    \int_0^T q_t \vert \psi_t \vert^2 \dd t \right] - \frac{1}{\gamma} {\mathcal S}(q) \leq c_2,
    \end{equation*} 
    which implies that $\bar \psi \in {\mathcal A}_{c_2}$, as required.

    Observing that $\dd {\mathbb P}/\dd {\mathbb Q}^0$ has finite exponential moments of any order under 
    ${\mathbb Q}^0$ and using again the fact that $\sup_{\psi \in {\mathcal A}_{c_2}} {\mathbb E}^{{\mathbb Q}^0}[\int_0^T \vert \psi_t \vert^2 \dd t] < + \infty$, we deduce 
that ${\mathcal A}_{c_2}$ is a bounded subset of $M^\eta({\mathbb F},{\mathbb R}^n,{\mathbb P})$ for any $\eta \in (0,2)$.    
    \vskip 4pt
    
    \noindent
     \textit{Step 2: convexity of ${\mathcal A}_{c_2}\ni \psi \mapsto \tilde{\mathcal{J}}(q,\psi)$, for a fixed $ q \in \tilde{\mathcal Q}_{c_1}$.} The convexity of the mapping
    \begin{equation*}
        {\mathcal A}_{c_2}\ni \psi \mapsto \mathbb{E}\left[\int_0^T q_s \ell(s,\psi_s) \dd s \right],
    \end{equation*}
    is a direct consequence of the convexity of $\ell$ and the non-negativity of $q$. We now turn to the convexity of the mapping $ {\mathcal A}_{c_2} \ni \psi \mapsto \mathcal{G}(q_T,X_T^\psi)$. 
    For $\psi^0,\psi^1 \in {\mathcal A}_{c_2}$,
    we denote by  $(X^i)_{i=0,1}$ the solutions to
    the two state equations associated with $\psi^0$ and $\psi^1$ respectively.
    Moreover, for any 
    $\theta \in (0,1)$, we call $X^\theta$ the solution to the state equation associated with $\theta \psi^1 + (1-\theta) \psi^0$. We notice that 
    \begin{equation*}
        X_T^\theta = \theta X_T^1 + (1-\theta)X_T^0.
    \end{equation*}
    Using the fact that $\mathcal{G}$ is convex with respect to its second variable by Assumption \ref{eq:G-concave-convex} together with the last equality, it is easy to deduce that $ [0,1] \ni \theta \mapsto \mathcal{G}(q_T,X_T^\theta)$ is convex.  This concludes the step. 
\vskip 4pt
    
    \noindent \textit{Step 3: lower semi-continuity of ${\mathcal A}_{c_2}\ni \psi \mapsto \tilde{\mathcal{J}}(q,\psi)$, for a fixed $q \in \tilde{{\mathcal Q}}_{c_1}$.} 
The proof relies on the fact that, as explained in 
Remark 
\ref{rem:convexity:tildeS:application}, the duality inequality \eqref{ineq:S-S-star} extends to elements $q \in \tilde{\mathcal Q}_{c_1}$.
 
    Let $(\psi^k)_{k \in \mathbb{N}}$
    be a sequence with values in ${\mathcal A}_{c_2}$. By the extended version of \eqref{ineq:S-S-star}, we have, for every $k \in \mathbb{N}$, 
    \begin{equation*}
        \gamma \mathbb{E}\left[\int_0^T q_t |\psi_t^k|^2 \dd t \right] \leq \mathcal{S}^\star(\psi^k) + \tilde{\mathcal{S}}(q) \leq c_1 + c_2,
    \end{equation*}
    from which we 
    deduce that the sequence $(\psi^k)_{k \in \mathbb{N}}$ lies in a weakly compact subset of 
the space of 
${\mathbb F}$-progressively measurable 
${\mathbb R}^n$-valued processes that are square integrable under the measure $\mathbb{Q}$, defined by ${\mathbb Q}(E) \coloneqq {\mathbb E}
\int_0^T {\mathds 1}_E q_t \dd t$, for any event of $\Omega \times [0,T]$; with a slight abuse of notation, we will denote this space by 
    $L^2({\mathbb F},{\mathbb R}^n,{\mathbb Q})$. Up to a subsequence, there exists a weak limit 
    in $L^2({\mathbb F},{\mathbb R}^n,{\mathbb Q})$, which we denote $\bar{\psi}$. As the purpose is to prove that ${\mathcal J}(q,\bar \psi) \leq \liminf_{k \rightarrow +\infty}
    {\mathcal J}( q,\psi^k)
    $, and the functional ${\mathcal J}$ is convex with respect to the first argument, we can  replace 
    the sequence $(\psi^k)_{k \in {\mathbb N}}$ by a sequence of convex combinations (of the 
    $(\psi_k)_{k \in {\mathbb N}}$'s) that converges in $L^2({\mathbb F},{\mathbb R}^n,{\mathbb Q})$ (equipped with the strong topology) to $\bar \psi$.
    Following the second step in the proof of Lemma \ref{lemma:positive-q}, we know that there exists a constant $c>0$ such that $q_t \geq c {\mathbb E}[q_T\vert {\mathcal F}_t]$, for all $t \in [0,T]$. In particular, if $E$ is in the progressive 
    $\sigma$-field, then 
    \begin{equation}
\label{eq:comparaison:q_T:mathbbQ}
    {\mathbb Q}(E) \geq c {\mathbb E}\left[ q_T{\mathds 1}_E\right].
    \end{equation}
     \color{black} By convexity of the function ${\mathcal G}$ in the second argument, see \ref{eq:G-concave-convex}, we also have
    \begin{align}
    \label{eq:semi-cont:J:proof}
    \forall k \in {\mathbb N}, \quad 
         \mathcal{J}(q,\psi^k) \geq \mathcal{J}(q,\bar \psi) +  \mathbb{E} \left[ A^k  + B^k \right],
    \end{align}
    where
    \begin{equation*}
        A^k \coloneqq \delta_X \mathcal{G}\left(q_T,X_T^{\bar \psi}\right) \cdot \left( X^{\psi^k}_T - X^{\bar \psi}_T\right), \quad B^k \coloneqq \int_0^T q_s \left(\ell(s,\psi^k_s) - \ell(s,\bar \psi_s) \right) \dd s.
    \end{equation*}
    We study the two sequences of random variables $(A^k)_{k \in \mathbb{N}}$ and $(B^k)_{k \in \mathbb{N}}$ separately. We start with $(A^k)_{k \in \mathbb{N}}$.  By the growth Assumption \ref{eq:G-growth} on $\delta_X \mathcal{G}$, we have
    \begin{align} \label{estim:A^k}
        |A^k| \leq  L q_T \left (1 + |X^{\bar \psi}_T|^{1-r} + \mathbb{E} \left[ q_T |X_T^{\bar{\psi}}|^{2-r}\right] \right) \left |X^{\psi^k}_T - X^{\bar \psi}_T \right |.
    \end{align}
    By Lemma \ref{lemma:reg-X-psi-A} in Appendix \ref{appendix:apriori-SDE} (together with 
    Remark 
\ref{rem:convexity:tildeS:application}), we know that 
    \begin{equation*}
        \mathbb{E} \left[ q_T |X_T^{\bar{\psi}}|^{2-r}\right] \leq C \left( 1 + \tilde{\mathcal S}(q) + \mathcal{S}^\star(\bar{\psi})\right) \leq C,
    \end{equation*}
    since $\bar{\psi} \in L^2({\mathbb F},{\mathbb R}^n,{\mathbb Q})$ and $q \in \tilde{\mathcal{Q}}_{c_1}$.
    Taking the expectation both sides of \eqref{estim:A^k} yields 
    \begin{align}
        \mathbb{E}[|A^k|] \leq C \mathbb{E}\left[ q_T \left (1 + \left|X^{\bar \psi}_T\right|^{1-r} \right) \left |X^{\psi^k}_T - X^{\bar \psi}_T \right | \right],
    \end{align}
    We distinguish between the cases $r=0$ and $r = 1$. 
\vskip 2pt
    
 \noindent \textit{Sub-step 3a: analysis of $(A^k)_{k \in \mathbb{N}}$ when $r=0$.}   When $r = 0$, we have 
    \begin{align*}
        \mathbb{E}[|A^k|] & \leq L \mathbb{E}\left[ q_T \left (1 + \left|X^{\bar \psi}_T\right| \right) \left |X^{\psi^k}_T - X^{\bar \psi}_T \right | \right] \\
        & \leq C \mathbb{E}\left[ q_T \left (1 + \left|X^{\bar \psi}_T\right| \right) \int_0^T \left|\psi^k_t - {\bar \psi}_t\right| \dd t  \right]
        \\
        & \leq C \mathbb{E}\left[ q_T \left (1 + \left|X^{\bar \psi}_T\right|^2 \right)\right]^{1/2} \mathbb{E}\left[q_T \int_0^T \left|\psi^k_t - {\bar \psi}_t\right|^2 \dd t  \right]^{1/2},
    \end{align*}
    where the last line follows from the Cauchy-Schwarz inequality, for some constant $C>0$.
    Then,  $\lim_{k \to +\infty} \mathbb{E}[|A^k|] = 0$,
    by \eqref{eq:comparaison:q_T:mathbbQ} and  strong convergence of $(\psi^k)_{k \in \mathbb{N}}$ in $L^2({\mathbb F},{\mathbb R}^n,{\mathbb Q})$.
\vskip2 pt 
    
     \noindent \textit{Sub-step 3b: analysis of $(A^k)_{k \in \mathbb{N}}$ when $r=1$.}
    When $r = 1$, there exists a constant $C>0$ such that 
    \begin{align}
        \mathbb{E}[|A^k|] & \leq C \mathbb{E}\left[ q_T \left |X^{\psi^k}_T - X^{\bar \psi}_T \right | \right] \nonumber
        \\
        & \leq C \left(\mathbb{E}\left[ q_T \int_0^T \left|\psi^k_t - {\bar \psi}_t \right| \dd t\right] + \mathbb{E}\left[ q_T \left| \int_0^T \left(\sigma(t,\psi^k_t) - \sigma(t,{\bar \psi}_t) \right) \dd W_t \right| \right] \right) \nonumber
        \\
        &=: a^k_1 + a^k_2.
    \label{eq:argument:qtheta:cv:An:00}
    \end{align}
    Clearly, $(a^k_1)_{k \in {\mathbb N}}$ converges to $0$ as $k \to +\infty$.  
    To handle $(a^{k}_2)_{k \in {\mathbb N}}$, we use the same family $(q^\theta)_{\theta \in [0,1)}$ as in Remark 
\ref{rem:convexity:tildeS:application}. 
      We recall that each $q^\theta$ is positive valued and belongs to ${\mathcal Q}_{c_1}$ (provided that $c_1$ is large enough, which is not a restriction here).
     Below, we write the expansion \eqref{eq:q} of $q^\theta$, for $\theta \in [0,1)$, in the form 
\begin{equation*}
        \dd q_t^\theta
        = q_t^\theta Y_t^{\star,\theta} \dd t+ q_t^\theta Z_t^{\star,\theta} \cdot \dd W_t, \quad t \in [0,T],
\end{equation*}
with initial condition $q_0^\theta = 1$.
 We notice that 
\begin{equation*}
\begin{split}
&a^k_2 \leq \frac1{1-\theta}
{\mathbb E}
 \left[
q_T^\theta \left| \int_0^T \left(\sigma(t,\psi^k_t) - \sigma(t,{\bar \psi}_t) \right) \dd W_t \right| 
\right] = : \frac1{1-\theta} a^{k,\theta}_2.
\end{split}
\end{equation*}
     By \eqref{eq:entrop:bounded:Q}, we have 
$\sup_{\theta \in [0,1]}{\mathbb E}[h(q_T^\theta)] < + \infty$. By Lemma 
\ref{lemma:representation-q}, we also know that 
${\mathbb Q}^\theta\coloneqq{\mathcal E}_T(\int_0^\cdot Z_t^{\star,\theta} \cdot \dd W_t) {\mathbb P}$ is a probability measure. It satisfies $\exp(-\alpha T) q_T^\theta \leq \dd {\mathbb Q}^{\theta}/ \dd {\mathbb P} \leq \exp(\alpha T) q_T^\theta$. 
Therefore, 
letting $(\tilde{W}_t^\theta \coloneqq W_t - \int_0^t Z^{\star,\theta}_s \dd s)_{t \in [0,T]}$, we deduce from Girsanov's theorem that
    \begin{equation}
        \label{eq:argument:qtheta:cv:An}
\begin{split}        
a^{k,\theta}_2 &\leq   C \mathbb{E}^{\mathbb{Q}^\theta}\left[ \left| \int_0^T \left(\sigma(t,\psi^k_t) - \sigma\left(t,\bar \psi_t\right) \right) \dd \tilde{W}_t^\theta \right| \right] \\
        &\hspace{15pt} + C \mathbb{E}^{\mathbb{Q}^\theta}\left[ \left| \int_0^T \left(\sigma(t,\psi^k_t) - \sigma\left(t,\bar\psi_t\right) \right) Z^{\star,\theta}_t \dd t \right| \right]
        \\
        &\leq  C \mathbb{E}^{\mathbb{Q}^\theta}\left[ \left| \int_0^T \left|\psi^k_t - \bar \psi_t\right|^2 \dd t \right| \right]^{1/2}
        \left( 1 + \mathbb{E}^{\mathbb{Q}^\theta}\left[   \int_0^T  |Z^{\star,\theta}_t|^2 \dd t   \right]^{1/2}\right).
    \end{split}
    \end{equation}
    By Lemma \ref{lemma:representation-q} and thanks to the bound $\sup_{\theta \in [0,1)} {\mathbb E}[h(q_T^\theta)] < +\infty$, we have
\begin{equation}
\label{eq:argument:qtheta:cv:An:1}
\sup_{\theta \in [0,1)}
{\mathbb E}^{{\mathbb Q}^\theta}
\left[
\int_0^T \vert Z_t^{\star,\theta}
\vert^2 \dd t
\right] < + \infty. 
\end{equation}
As for the first term on the last line of 
\eqref{eq:argument:qtheta:cv:An}, we notice that 
\begin{equation}
\label{eq:argument:qtheta:cv:An:2}
\begin{split}
{\mathbb E}^{{\mathbb Q}^\theta}
\left[   \int_0^T \left\vert \psi_t^k - \bar \psi_t\right\vert^2 \dd t \right]
\leq  \; &\exp(\alpha T) 
{\mathbb E} 
\left[ q_T^\theta  \int_0^T \left\vert \psi_t^k - \bar \psi_t\right\vert^2 \dd t \right]
\\
= \; & (1-\theta) \exp(\alpha T) 
{\mathbb E} 
\left[ q_T^0 \int_0^T \left\vert \psi_t^k - \bar \psi_t\right\vert^2 \dd t \right]
\\
& + 
\theta {\mathbb E} 
\left[ q_T \int_0^T \left\vert \psi_t^k - \bar \psi_t\right\vert^2 \dd t \right]. 
\end{split}
\end{equation}
By Step 1, we know that the first term on the right-hand side is less than $C(1-\theta)$. Using the fact that 
    $(\psi^k)_{k \in {\mathbb N}}$ strongly converges to $\bar \psi$ in $L^2({\mathbb F},{\mathbb R}^n,{\mathbb Q})$,
    the second one tends to $0$ as $k$ tends to $+ \infty$ (uniformly in $\theta \in [0,1)$). By \eqref{eq:argument:qtheta:cv:An}, \eqref{eq:argument:qtheta:cv:An:1} and \eqref{eq:argument:qtheta:cv:An:2}, we deduce that 
    $\lim_{\theta \rightarrow 1} \lim_{k \rightarrow + \infty} a^{k,\theta}_2=0$. 
    By 
    \eqref{eq:argument:qtheta:cv:An:00}, we obtain   $\lim_{k \to +\infty} \mathbb{E}[|A^k|] = 0$ when $r=1$. 
\vskip 2pt
    
     \noindent \textit{Sub-step 3c: analysis of $(B^k)_{k \in \mathbb{N}}$.}
     Back to \eqref{eq:semi-cont:J:proof}, 
    we now study the sequence $(B^k)_{k \in {\mathbb N}}$. Using 
the fact that 
the gradient of $\ell$ in the second variable is at most of linear growth, see
\ref{hyp:L}, and once again
the fact that
    $({\psi}^k)_{k \in \mathbb{N}}$ strongly converges to $\psi$ in $L^2({\mathbb F},{\mathbb R}^n,{\mathbb Q})$, we directly obtain
    \begin{equation*}
        \lim_{k \to +\infty} B^k =
         \lim_{k \to +\infty} \left( 
        \mathbb{E}\left[ \int_0^T q_s \ell(s,\psi^k_s) \dd s \right] -  \mathbb{E}\left[ \int_0^T q_s \ell\left(s,\bar \psi_s\right) \dd s \right]\right) =0.
    \end{equation*}
    
     \noindent \textit{Sub-step 3d: conclusion.}
     From the last three sub-steps, we deduce that 
     ${\mathbb E}[A^k+B^k]$ in the right-hand side of 
 \eqref{eq:semi-cont:J:proof} tends to $0$. We deduce that $\liminf_{k \to +\infty} \mathcal{J}(\psi^k) \geq \mathcal{J}(\bar \psi)$, which concludes the step and the proof. 
\end{proof}

\begin{lemma} \label{lemma:sion-limit} There exists a solution to problem \eqref{pb:min-max-tilde-G-c1-c2}, i.e. there exists $(\bar{\psi},\bar{q}) \in  {\mathcal A}_{c_2} \times \tilde{\mathcal{Q}}_{c_1}$ such that 
    \begin{equation} \label{eq:min-max-c}
            \min_{\psi \in {\mathcal A}_{c_2}} \max_{q \in \tilde{\mathcal{Q}}_{c_1}} \tilde{\mathcal{J}}(q,\psi) = \max_{q \in \tilde{\mathcal{Q}}_{c_1}} \min_{\psi \in {\mathcal A}_{c_2}}  \tilde{\mathcal{J}}(q,\psi) = \tilde{\mathcal{J}}(\bar{q},\bar{\psi}).
    \end{equation}
\end{lemma}

\begin{proof} 
For each $k \in {\mathbb N}$, we call ${\mathcal B}_k^\infty$ the collection of 
${\mathbb F}$-progressively measurable 
${\mathbb R}^n$-valued processes that are bounded by $k$ on a subset of full measure under ${\mathbb P} \otimes \textrm{\rm Leb}_{[0,T]}$. 
We then view 
${\mathcal A}_{c_2} \cap {\mathcal B}_k^\infty$ as a subset of $L^2({\mathbb F},{\mathbb R}^n,{\mathbb Q}^0)$, with ${\mathbb Q}^0$ being defined as in the statement of Lemma \ref{lemma:convexity-J}. By Lemma \ref{lemma:convexity-J} again,  
${\mathcal A}_{c_2} \cap {\mathcal B}_k^\infty$ is a convex and weakly compact subset of 
$L^2({\mathbb F},{\mathbb R}^n,{\mathbb Q}^0)$. 
By Theorem \ref{thm:Sion}, which we can apply thanks to Proposition \ref{prop:concave-J} and Lemma \ref{lemma:convexity-J}, we can find, for each $k \in \mathbb{N}$, a saddle point $(q^k,\psi^k) \in \tilde{\mathcal{Q}}_{c_1} \times ({\mathcal A}_{c_2}\cap \mathcal{B}^\infty_k)$ 
to the min-max problem
    \begin{equation} \label{eq:min-max-L-infty}
            \min_{\psi \in {\mathcal A}_{c_2} \cap \mathcal{B}^\infty_k} \max_{q \in \tilde{\mathcal{Q}}_{c_1}} \tilde{\mathcal{J}}(q,\psi) = \max_{q \in \tilde{\mathcal{Q}}_{c_1}} \min_{\psi \in {\mathcal A}_{c_2} \cap \mathcal{B}^\infty_k}  \tilde{\mathcal{J}}(q,\psi) = \tilde{\mathcal{J}}(q^k,\psi^k).
    \end{equation}
By compactness of $\tilde{\mathcal{Q}}_{c_1}$ and ${\mathcal A}_{c_2}$, for the weak topologies $\sigma(L^1({\mathbb F},{\mathbb R},{\mathbb P}),L^\infty({\mathbb F},{\mathbb R},{\mathbb P}))$ and $\sigma(L^2({\mathbb F},{\mathbb R}^n,{\mathbb Q}^0),L^2({\mathbb F},{\mathbb R}^n,{\mathbb Q}^0))$,  the sequence $(q^k,\psi^k)_{k \in \mathbb{N}}$ converges, up to a subsequence, for the product topology. The limit is denoted  $(\bar{q},\bar{\psi})$. The objective of the proof is thus to show that 
    \begin{equation} \label{ineq:minmax-under-forall-form}
         \tilde{\mathcal{J}}(q,\bar{\psi}) \leq  \tilde{\mathcal{J}}(\bar{q},\bar{\psi}) \leq  \tilde{\mathcal{J}}(\bar{q},\psi),
    \end{equation}
    for all $ \psi \in {\mathcal A}_{c_2}$ and $q \in \tilde{\mathcal{Q}}_{c_1}$, 
    which is known to be equivalent to the equality \eqref{eq:min-max-c}.

    \vspace{4pt}
    \noindent \textit{Step 1: $\tilde{\mathcal{J}}(\bar{q},\bar{\psi}) \leq \tilde{\mathcal{J}}(\bar{q},\psi)$ for any $\psi \in {\mathcal A}_{c_2}$.}  For $0 \leq k_0 \leq k$ and $\psi \in {\mathcal A}_{c_2}$, let $(\psi_t^{k_0} \coloneqq \psi_t \mathds{1}_{\{|\psi_t| \leq k_0\}})_{t \in [0,T]}$.
    By \eqref{eq:min-max-L-infty}, we have that 
    \begin{equation*} 
         \tilde{\mathcal{J}}(\bar{q},\psi^k) \leq  \tilde{\mathcal{J}}(q^k,\psi^k) \leq  \tilde{\mathcal{J}}(q^k,\psi^{k_0}).
    \end{equation*}
    Taking infimum and supremum limits in the last inequality, we have
    \begin{equation}\label{ineq:liminf-limsup-step1-sion}
        \liminf_{k \to +\infty} \tilde{\mathcal{J}}(\bar{q},\psi^{k}) \leq \limsup_{k \to +\infty} \tilde{\mathcal{J}}(\bar{q},\psi^{k}) \leq  \limsup_{k \to +\infty}  \tilde{\mathcal{J}}(q^k,\psi^{k_0}).
    \end{equation}
    We first handle the term on the right-hand side. By concavity of ${\mathcal G}$ in the first variable, we have
    \begin{equation}
    \label{ineq:liminf-limsup-step1-sion:en+}
         \tilde{\mathcal{J}}(q^k,\psi^{k_0}) \leq  \tilde{\mathcal{J}}(\bar{q},\psi^{k_0}) + a^k + b^k,
    \end{equation}
    where 
    \begin{align*}
        a^k \coloneqq \mathbb{E}\left[\delta_q \mathcal{G}\left(\bar{q}_T,X^{\psi^{k_0}}_T\right) (q_T^k-\bar{q}_T)  + \int_0^T (q_t^k-\bar{q}_t) \ell(t,\psi^{k_0}_t) \dd t \right], \quad b^k \coloneqq  \tilde{\mathcal{S}}(\bar{q}) - \tilde{\mathcal{S}}(q^k).
    \end{align*}
    By boundedness of $\psi^{k_0}$, the integrand $\ell(t,\psi^{k_0}_t)$ is bounded, uniformly in $t$ and $\omega$. Moreover,  $|X^{\psi^{k_0}}_T|^{2-r}$ has exponential moments of any order, which implies, from
    Lemma 
    \ref{lemma:reg-X-psi-A}
    and Remark 
    \ref{rem:convexity:tildeS:application}, that $\mathbb{E}[q_T |X^{\psi^{k_0}}_T|^{2-r}]<+ \infty$. And then, by
the    growth Assumption \ref{eq:G-growth} 
    on $\delta_q \mathcal{G}$, we deduce that 
    that $\delta_q \mathcal{G}(q_T,X^{\psi^{k_0}}_T)$ has exponential moments
    of any order. Recalling 
from 
\eqref{eq:entrop:bounded:Q}
    that 
    \begin{equation*}  
    \sup_{t \in [0,T]} \mathbb{E}\left[ h(\bar{q}_t)\right] < + \infty, \quad  \sup_{k \in {\mathbb N}} \sup_{t \in [0,T]} \mathbb{E}\left[ h(q^k_t)\right] < + \infty,
    \end{equation*} 
   Lemma  \ref{lem:uniform:integrability} yields
    \begin{equation*}
    \lim_{k \to +\infty} a^k = 0.
    \end{equation*}
    By weak lower semi-continuity of $\tilde{\mathcal{S}}$ on $\tilde{\mathcal{Q}}_{c_1}$ (see Step 3 in the proof of Proposition \ref{prop:concave-J}), we also have
    \begin{equation*}
        \limsup_{k \to +\infty} b^k \leq 0.
    \end{equation*}
    Therefore, inserting 
    the last two displays in 
    \eqref{ineq:liminf-limsup-step1-sion:en+} and then returning 
    back to \eqref{ineq:liminf-limsup-step1-sion}, 
    we obtain 
    \begin{equation*}
        \liminf_{k \to +\infty} \tilde{\mathcal{J}}(\bar{q},\psi^{k}) \leq \tilde{\mathcal{J}}(\bar{q},\psi^{k_0}).
    \end{equation*}
    By weak lower semi-continuity of ${\mathcal A}_{c_2} \ni \psi \mapsto \tilde{\mathcal{J}}(\bar{q},\psi)$ (see Lemma \ref{lemma:convexity-J}, using the fact that 
    $\bar q \in \mathcal{Q}_{c_1}$), the last inequality yields 
    \begin{equation} \label{ineq:J-psi-bar-psi-n0}
        \tilde{\mathcal{J}}(\bar{q},\bar{\psi}) \leq \tilde{\mathcal{J}}(\bar{q},\psi^{k_0}).
    \end{equation}
    It remains to pass to the limit in the right-hand side. 
    By regularity of ${\mathcal G}$, we can write
    \begin{align} \label{eq:J-fundam-clalculus-thm}
        \tilde{\mathcal{J}}(\bar{q},\psi^{k_0}) = \tilde{\mathcal{J}}(\bar{q},\psi) + \mathbb{E} \left[ A^{k_0}  + B^{k_0} \right],
    \end{align}
    where 
    \begin{equation}
    \label{eq:An0:Bn0}
    \begin{split}
        &A^{k_0} \coloneqq \int_0^1 \delta_X \mathcal{G}(q_T, X^{\lambda,\psi^{k_0}}_T) \cdot (X^{\psi^{k_0}}_T- X^{\psi}_T) \dd \lambda , 
        \\
        &B^{k_0} \coloneqq \int_0^T \bar{q}_s (\ell(s,\psi^{k_0}_s) - \ell(s,\psi_s)) \dd s,
        \end{split}
    \end{equation}
    and $X^{\lambda,\psi^{k_0}}_T \coloneqq \lambda X^{\psi^{k_0}}_T + (1-\lambda) X^{\psi}_T$. At this point, we are in a situation very similar to \eqref{eq:semi-cont:J:proof}, except for the fact that $\delta_X \mathcal{G}$ in the definition of $A^{k_0}$ is computed at $X^{\lambda,\psi^{k_0}}$. Apart from this, the context is the same. In particular, with the same abuse of notation as in the third step of the proof of Lemma \ref{lemma:convexity-J},
$\psi^{k_0}$ converges
to $\psi$
(as $k_0$ tends to $+ \infty$)
in 
$L^2({\mathbb F},{\mathbb R}^n,\bar{\mathbb Q})$, 
where $\bar{\mathbb Q}$
is defined by 
$\bar{\mathbb Q}(E) = {\mathbb E} \int_0^T  {\mathds 1}_E(t) \bar q_t \dd t$. Indeed, 
the extended version of 
\eqref{ineq:S-S-star}
(for elements $q$ in $\tilde {\mathcal Q}_{c_1}$, 
see again Remark 
\ref{rem:convexity:tildeS:application})
yields
\begin{equation*}
{\mathbb E}\left[ \bar q_T \int_0^T \vert \psi_t \vert^2 \dd t \right] < + \infty.
\end{equation*}
Recalling that 
$(\psi_t^{k_0} \coloneqq \psi \mathds{1}_{\{|\psi_t| \leq k_0\}})_{t \in [0,T]}$,
we deduce from dominated convergence theorem that
$\psi^{k_0}$ indeed converges to 
$\psi$ (as $k_0 \rightarrow + \infty$) in $L^2({\mathbb F},{\mathbb R}^n,\bar{\mathbb Q})$. Following the exact same reasoning as in Step 3 of Lemma \ref{lemma:convexity-J}, we then obtain that
\begin{equation} \label{eq:limit-A-B-n_0}
        \lim_{k_0 \to +\infty}\mathbb{E} \left[A^{k_0} + B^{k_0} \right] = 0.
\end{equation}
Combining the last result with \eqref{ineq:J-psi-bar-psi-n0}  and \eqref{eq:J-fundam-clalculus-thm}, 
    we deduce that 
    $\tilde{\mathcal{J}}(q,\bar \psi) \leq
\tilde{\mathcal{J}}(\bar q, \bar \psi),$
which completes the first step.
\vskip 4pt
    
\noindent \textit{Step 2: $\tilde{\mathcal{J}}(q,\bar{\psi}) \leq \tilde{\mathcal{J}}(\bar{q},\bar{\psi})$ for any $q \in \tilde{\mathcal{Q}}_{c_1}$.} Because the proof follows arguments 
that are similar to those in the first step, we just give a sketch of it.  For integers $0 \leq k_0 \leq k$, we deduce, again from
\eqref{eq:min-max-L-infty}, that 
    \begin{equation*}
         \tilde{\mathcal{J}}(q,\psi^k) \leq  \tilde{\mathcal{J}}(q^k,\psi^k) \leq  \tilde{\mathcal{J}}(q^k,\bar{\psi}^{k_0}),
    \end{equation*}
    for any $q \in \tilde{\mathcal{Q}}_{c_1}$ and where $\bar{\psi}^{k_0} \coloneqq \bar{\psi} \mathds{1}_{\{|\bar{\psi}| \leq k_0\}}$. By definition of the infimum and supremum limits, we have 
    \begin{equation*}
         \liminf_{k \to +\infty} \tilde{\mathcal{J}}(q,\psi^k) \leq \limsup_{k \to +\infty} \tilde{\mathcal{J}}(q^k,\bar{\psi}^{k_0}),
    \end{equation*}
    Following the same arguments as in Step 1, we have that 
    \begin{equation*}
        \limsup_{k \to +\infty} \tilde{\mathcal{J}}(q^k,\bar{\psi}^{k_0}) = \tilde{\mathcal{J}}(\bar{q},\bar{\psi}^{k_0}), \quad  \liminf_{k \to +\infty} \tilde{\mathcal{J}}(q,\psi^k) \geq \tilde{\mathcal{J}}(q,\bar{\psi}).
    \end{equation*}
    Combining the last inequalities yields
    \begin{equation*}
        \tilde{\mathcal{J}}(q,\bar{\psi}) \leq \tilde{\mathcal{J}}(\bar{q},\bar{\psi}^{k_0}).
    \end{equation*}
    Finally, using \eqref{eq:J-fundam-clalculus-thm} in the right-hand side and then taking the limit $k_0 \to +\infty$, the conclusion of the step and the proof follows by \eqref{eq:limit-A-B-n_0}.
\end{proof}

\color{black}
\subsection{Nature's control problem} \label{sec:necessary-sufficient}

Given two constants $c_1,c_2 >0$, 
we address the \textit{restricted} Nature control problem
\begin{equation} \label{pb:dual} \tag{{P}\textsubscript{N,$c_1$}}
\sup_{q' \in \mathcal{Q}_{c_1}}  \mathcal{J}(\bar{\psi},q'),
\end{equation}
under the assumption that, for some $\bar q$,
the pair 
$(\bar{q},\bar{\psi}) \in \mathcal{Q}_{c_1} \times {\mathcal A}_{c_2}$
is a saddle point of 
the problem \eqref{pb:min-max-G-c1-c2}.
In particular, 
the supremum in 
\eqref{pb:dual} is equal to $\mathcal{J}(\bar{q},\bar{\psi})$ and our 
goal becomes to characterize  $\bar{q}$ when $\bar{\psi}$ is given.

As a corollary of our analysis, we show that, when $c_1$ is sufficiently large, $\bar q$ is in fact the unique minimizer 
of the \textit{unrestricted} Nature control problem
\begin{equation}
\label{pb:dual:Q} \tag{P\textsubscript{N}}
\sup_{q' \in {\mathcal Q}} {\mathcal J}(q',\bar \psi),
\end{equation}
which, in contrast with \eqref{pb:dual}, is set over the entire set ${\mathcal Q}$.

A key step in relaxing the constraint imposed on $q'$ in \eqref{pb:dual}, and thereby passing from ${\mathcal Q}_{c_1}$ to ${\mathcal Q}$, is to show that the component $\bar q$ of any saddle point $(\bar q,\bar \psi)$ of \eqref{pb:min-max-G-c1-c2} actually lies in the interior of ${\mathcal Q}_{c_1}$, provided that $c_1$ is sufficiently large. This is the content of the following result, proved in Subsection~\ref{sec:Q-BSDE}.

\begin{proposition} \label{prop:dual:a priori}
There exists a constant $c_1'>0$, only depending on the data, such that, 
for any $c_1 >c_1'$
and any saddle point 
$(q,\psi)$ to \eqref{pb:min-max-G-c1-c2} over 
${\mathcal Q}_{c_1} \times {\mathcal A}_{c_2}$, the component $q$ of the saddle point necessarily belongs to 
${\mathcal Q}_{c_1'}$.
\end{proposition}

This a priori bound then allows us to apply perturbative arguments to characterize the solutions of \eqref{pb:dual}, and subsequently of \eqref{pb:dual:Q}, by means of a stochastic maximum principle.
The results are summarized in the main 
statement below:
\begin{theorem} \label{thm:sto-max-princ-dual}
Let $\bar \psi \in {\mathcal A}_{c_2}$ and assume $c_1>c_1'$ with $c_1'$ as in the statement of Proposition \ref{prop:dual:a priori}. 
\begin{enumerate}
\item If, for some 
$c_0 \in (0,c_1)$, 
there exists 
$q \in {\mathcal Q}_{c_0}$ that solves \eqref{pb:dual} 
(i.e., that maximizes $q' \mapsto {\mathcal J}(q',\bar \psi)$  over $\mathcal{Q}_{c_1}$), then there exists a pair $(Y,Z)$
such that $(q,Y,Z)$ belongs to the set $\mathscr{Q}$ defined in 
\eqref{def:mathscr-Q}, and solves the FBSDE \eqref{optim:condition-dual} (with $\psi=\bar \psi$ therein).
\item Conversely, assume that 
there exists a solution  
$(\bar q, \bar Y, \bar Z) \in \mathscr{Q}$ to the FBSDE
\eqref{optim:condition-dual}, then $(\bar q,\bar \psi)$ is the unique optimizer of Nature's control problem \eqref{pb:dual:Q}.
\item 
In particular, if, for some $\bar q \in {\mathcal Q}_{c_1}$, 
$(\bar q,\bar \psi)$ is a saddle point of the  
problem \eqref{pb:min-max-G-c1-c2}, then 
$\bar q$
is the unique solution to \eqref{pb:dual:Q} (i.e., is the unique maximizer of $q' \mapsto {\mathcal J}(q',\bar \psi)$ over the entire ${\mathcal Q}$).
\end{enumerate}
\end{theorem}
The proof is based on a series of lemmas, which are proved in the next paragraph.
\begin{proof}[Sketch of the proof of Theorem \ref{thm:sto-max-princ-dual}] Taking for granted the statement of Proposition 
\ref{prop:dual:a priori} and the results proven in the forthcoming  Subsubsections
\ref{subsubse:sufficient:Nature}
and
\ref{subsubse:necessary:Nature}, 
Theorem \ref{thm:sto-max-princ-dual}
can be established as follows. 

The necessary condition is addressed in Subsubsection 
\ref{subsubse:necessary:Nature}. We prove in Lemma \ref{lemma:necessary-q} 
that, for any optimizer $q$ of 
\eqref{pb:dual} that belongs to ${\mathcal Q}_{c_0}$ for some 
$c_0 \in (0,c_1)$, the 
BSDE
\eqref{eq:BSDE:compact:Y:Z}
has a solution 
$(Y,Z) \in D({\mathbb F},{\mathbb Q}) \times (\cap_{\beta (0,1)} M^{\beta}({\mathbb F},{\mathbb R}^d,{\mathbb Q}))$ that satisfies the optimality condition in the last line in \eqref{optim:condition-dual}. This shows that 
$(q,Y,Z)$ belongs to 
${\mathscr Q}$ and solves \eqref{optim:condition-dual}, and this proves the first assertion in the statement of Theorem \ref{thm:sto-max-princ-dual}. 

The second assertion (i.e., the converse) is a direct consequence of Lemma 
\ref{lemma:sufficient}. 

It remains to establish the third assertion. Given
$\bar q \in {\mathcal Q}_{c_1}$ such that $(\bar q,\bar \psi)$
is a saddle point of 
\eqref{pb:min-max-G-c1-c2}, we know from Proposition \ref{prop:dual:a priori} that $\bar q \in {\mathcal Q}_{c_1'}$ for some $c_1' \in (0,c_1)$. 
By the first assertion in the statement of Theorem \ref{thm:sto-max-princ-dual}, we deduce that, there exists a pair $(Y,Z)$ such that $(\bar q,Y,Z)$ solves \eqref{optim:condition-dual}. By the second assertion, we deduce that $\bar q$ is the    unique maximizer of $q' \mapsto {\mathcal J}(q',\bar \psi)$ over the entire set ${\mathcal Q}$.
\end{proof}

Throughout the subsection, the parameters
$c_1$ and $c_2$ appearing in 
\eqref{pb:dual}
are fixed.

\subsubsection{A priori estimate} \label{sec:Q-BSDE}

The purpose of this subsubsection is to establish the following a priori estimate, from which Proposition~\ref{prop:dual:a priori} follows as a direct consequence:

\begin{lemma} \label{lemma:adjoint-existence-uniqueness}
There exist two constants $c_1'>0$ and $C \geq 0$, only depending on the data and independent of $c_1, c_2$, such that, for any $(q,\psi) \in \mathcal{Q}_{c_1} \times {\mathcal A}_{c_2}$, with $c_1 > c_1'$,
satisfying
\begin{equation} \label{ineq:psi-0-q-q0}
    \mathcal{J}(q,0) \geq \mathcal{J}(q,\psi) \geq  \mathcal{J}(q^0,\psi),
\end{equation}
where $q^0$ solves \eqref{eq:barq:0}, it holds $q \in \mathcal{Q}_{c_1'}$ and $\sup_{t \in [0,T]} \mathbb{E}\left[ h(q_t) \right] \leq C$.
\end{lemma}
Assuming that $c_1$ satisfies 
\eqref{eq:lower:bound:c_1}, $q^0$ in the statement belongs to ${\mathcal Q}_{c_1}$.
Notice also that the condition \eqref{ineq:psi-0-q-q0} is stated for an arbitrary pair $(q,\psi) \in
{\mathcal Q}_{c_1} \times {\mathcal A}_{c_2}$, but is automatically satisfied by 
the saddle point 
$(\bar{q},\bar{\psi})$
introduced in the beginning of Subsection \ref{sec:central-planner}, see
\eqref{pb:dual}.

\begin{proof}
    \textit{Step 1.} We first establish a bound for ${\mathcal S}(q)$  in terms of ${\mathbb E}[h(q_T)]$.
    Recalling that $\mathcal{J}(q,\psi) = \mathcal{R}(q,\psi)  - \mathcal{S}(q)$, by inequality \eqref{ineq:psi-0-q-q0} we have 
\begin{equation}
\label{eq:minoration:S:q}
    \begin{split}
        \mathcal{S}(q) 
        = - {\mathcal J}(q,0) + {\mathcal R}(q,0) 
        &\leq - {\mathcal J}(q^0,\psi) + {\mathcal R}(q,0)
        \\
&=    - \mathcal{J}(q^0,\psi) + \mathcal{G}(X_T^0,q_T) +  \mathbb{E}\left[\int_0^T q_s \ell(s,0) \dd s\right].
    \end{split}
\end{equation}
We first provide a lower bound for 
${\mathcal J}(q^0,\psi)$. Recalling \eqref{eq:barq:0}, we have
\begin{equation*}
\begin{split}
{\mathcal J}\left(q^0,\psi\right)
= {\mathcal G}(q^0_T,X^\psi_T) + {\mathbb E}
\left[\int_0^T q^0_t \ell(s,\psi_s)\dd s \right]
+ {\mathbb E}\left[
\int_0^T q^0_s f(s,0,0) \dd s \right].
\end{split}
\end{equation*}
Using the convexity of ${\mathcal G}$ in the variable $X$
(see \eqref{eq:G-growth})
and the $L^{-1}$-strong convexity of $\ell$ in the variable 
$\psi$ (see \ref{hyp:L}), 
we deduce that there exists a constant $C$, independent of $\psi$, such that 
\begin{equation*}
\begin{split}
&{\mathcal J}\left(q^0,\psi\right)
\\
&\geq {\mathcal G}(q^0,0)
- C \left( 1 +  {\mathbb E}\left[ q_T^0\left\vert X_T^\psi \right\vert \right] \right)
+ \frac1{2L} {\mathbb E}\left[
\int_0^T q_s^0 \vert \psi_s \vert^2 \dd s\right]
+ {\mathbb E} \left[
\int_0^T q_s^0 f(s,0,0) \dd s \right].
\end{split}
\end{equation*}
Using the assumptions 
\ref{hyp:b} and \ref{hyp:sigma}, 
and rewriting the dynamics of $X^\psi$ under the equivalent probability measure 
${\mathcal E}_T(\int_0^{\cdot} \partial_z f(t,0,0) \cdot \dd W_t) {\mathbb P}$, 
we have, for a new value of $C$, 
\begin{equation*}
{\mathbb E}\left[q_T^0 \left\vert X_T^\psi \right\vert \right]
\leq C \left( 1+ {\mathbb E}\left[
\int_0^T q_s^0 \vert \psi_s \vert^2 \dd s \right]^{1/2} \right).
\end{equation*}
Then, by combining the last two displays, there exists (a new) constant $C>0$, only depending on the data 
and independent of $c_1,c_2$ and $q$, such that 
\begin{equation}
\label{eq:minoration:S:q:2}
{\mathcal J}\left(q^0,\psi\right) \geq - C.
\end{equation}
Back to 
\eqref{eq:minoration:S:q}, 
we now make use of the duality inequality 
\eqref{eq:ineq:duality:with:r}. By the latter, together with the growth Assumption \ref{eq:G-growth} 
and the bound 
\eqref{eq:minoration:S:q:2}, we get
\begin{align}
\nonumber
    \mathcal{S}(q) & \leq 
    - \mathcal{J}(q^0,\psi) + L\left(1 +  \mathbb{E}\left[(1+q_T) |X_T^{0}|^{2-r} \right]\right) +  \mathbb{E}\left[\int_0^T q_s \ell(s,0) \dd s\right] \\
     & \leq 
    C +  \frac{1}{\vartheta}\mathbb{E}\left[ h(q_T) \right] + \mathbb{E}\left[ \exp \left( \vartheta \left(\xi + \int_0^T \ell^0_s \dd s \right) \right) \right], \label{eq:minoration:S:q:3}
\end{align}
where $\xi \coloneqq L |X_T^{0}|^{2-r}$, $\ell_s^0 \coloneqq \ell(s,0)$ and $\vartheta > \upsilon \coloneqq \beta e^{\alpha T}$.  
We recall that 
$\alpha$ and $\beta$ are given by 
\eqref{eq:assumption1-f}.
Moreover, using the lower bound \eqref{ineq:duality-fstar} for the dual driver, we have (similar to \eqref{eq:bound:h:tilde:Q})
\begin{equation}
\label{eq:minoration:S:q:4}
    \mathcal{S}(q) \geq 
    - \mathbb{E} \left[\int_0^T q_s |f^0_s| \dd s\right] + \frac{1}{2 \beta} \mathbb{E} \left[\int_0^T q_s |Z^\star_s |^2 \dd s\right].
\end{equation}
Combining the last two inequalities
\eqref{eq:minoration:S:q:3}
and
\eqref{eq:minoration:S:q:4}, we obtain
\begin{align} \label{ineq:entropic-gradf-h-exp}
    \frac{1}{2 \beta} \mathbb{E} \left[\int_0^T q_s |Z^\star_s |^2 \dd s\right] \leq C_1 +  \frac{1}{\vartheta}\mathbb{E}\left[ h(q_T) \right] + \mathbb{E}\left[ \exp \left(\vartheta \left(\xi + \int_0^T \ell_s^0 \dd s \right) \right) \right],
\end{align}
where $C_1 >0 $ is a finite constant, defined by $C_1 \coloneqq C + T \|f^0\|_{L^\infty(\mathbb{F})}$.
\vskip 4pt

\noindent \textit{Step 2.}
We now provide another bound for the entropy which will lead us 
to the expected result when combined with the conclusion of the first step. To do so, we let $ (\tilde{q}_t \coloneqq q_t \exp \left( - \alpha t \right))_{t \in [0,T]}$.
Similar to 
\eqref{eq:h:qt:ito:expansion}, we have, 
by It{\^o}'s formula and for any stopping time $\tau$ such that $\int_0^{\tau} 
q_s \vert Z_s^\star \vert^2 \dd s$ 
and 
$\sup_{t \in [0,\tau]} \vert q_t \vert$
belong 
to $L^\infty({\mathcal F}_T)$,
\begin{equation}
\label{eq:ito:expansion:h:tildeq:+localisation}
    h(\tilde{q}_\tau) = -1 +  \int_0^\tau  \tilde{q}_s \ln (\tilde{q}_s)( Y^\star_s - \alpha) \dd s + \frac{1}{2} \int_0^\tau \tilde{q}_s |Z^\star_s|^2 \dd s +  H_T^\tau,
\end{equation}
where $(H_t^\tau \coloneqq  
\int_0^{t \wedge \tau} \tilde q_s \ln(\tilde q_s) Z_s^\star \cdot \dd W_s)_{t \in [0,T]}$.
Since $(\sqrt{\tilde{q}_t} \ln (\tilde q_t))_{0 \le t \le \tau}$
belongs to 
$L^\infty({\mathbb F})$,
$H^\tau$ is a square integrable martingale.
By boundedness of the partial derivative of the driver with respect to its first variable, i.e. $ |Y^\star_s| \leq  \alpha$, we have that
\begin{equation*}
\begin{split}
    \int_0^T  \tilde{q}_s \ln (\tilde{q}_s) ( Y^\star_s - \alpha) \dd s & \leq \int_0^T  \tilde{q}_s \ln (\tilde{q}_s) (Y^\star_s - \alpha)
    {\mathds 1}_{\{
    \tilde q_s <1
    \}}    
    \dd s
    \leq 
    2 \alpha e^{-1} T. 
\end{split}
\end{equation*}
Returning to 
\eqref{eq:ito:expansion:h:tildeq:+localisation}, 
inserting the above inequality and taking expectation, 
we get
\begin{equation} \label{ineq:q-tilde-h-grad-f:tau}
    \frac1{\beta} \mathbb{E}\left[h(\tilde{q}_{\tau})\right] \leq C_2 + \frac{1}{2 \beta} 
    {\mathbb E} \left[\int_0^\tau  \tilde{q}_s |Z^\star_s|^2 \dd s \right],
\end{equation} 
where $C_2 =  2 \alpha \beta^{-1} e^{-1} T - 1$. By a standard localization argument, we choose $\tau$ along a 
non-decreasing
sequence of stopping times $(\tau_k)_{k \in {\mathbb N}}$
converging to $T$, such that 
$\int_0^{\tau_k}
q_s \vert Z_s^\star \vert^2 \dd s$ 
and 
$\sup_{t \in [0,\tau_k]} \vert q_t \vert$
belong 
to $L^\infty({\mathcal F}_T)$ for each $k \in {\mathbb N}$. This is possible to construct such a sequence because, by finiteness of ${\mathbb E}[h(q_T)]$, 
we have 
\begin{equation*}
{\mathbb E} \left[ \int_0^T q_s \vert Z_s^\star \vert^2 \dd s  \right]
< + \infty, 
\quad
{\rm and}
\quad 
{\mathbb E}\left[ q_T^* \right]
< +\infty,
\end{equation*}
with the second inequality following from $L \log L$-Doob's inequality. 
Observing that 
$h(\tilde{q}_{\tau})$
is lower bounded by $-1/e-e$, we deduce from 
\eqref{ineq:q-tilde-h-grad-f:tau} and
Fatou's lemma that 
\begin{equation} \label{ineq:q-tilde-h-grad-f}
    \frac1{\beta} \mathbb{E}\left[h(\tilde{q}_{T})\right] \leq C_2 + \frac{1}{2 \beta} 
    {\mathbb E} \left[ \int_0^T  \tilde{q}_s |Z^\star_s|^2 \dd s\right].
\end{equation} 
Moreover, because $\tilde{q}_T =\exp(- \alpha T) q_T$,
\begin{equation} \label{eq:entropyq-tilde-q}
   h(\tilde{q}_T) = \exp(- \alpha T) h(q_T) -  \alpha T  \exp(- \alpha T) q_T.
\end{equation}
By the duality inequality \eqref{eq:ineq:duality:with:r}, we have for any $\theta >1$,  
\begin{equation} \label{ineq:h-e}
    \alpha T  \exp(- \alpha T) q_T \leq \frac{1}{\theta} \exp(- \alpha T) h(q_T) + \exp(\theta \alpha T).
\end{equation}
Then, combining \eqref{eq:entropyq-tilde-q} and \eqref{ineq:h-e} yields
 \begin{equation*}
    \left(1-\frac{1}{\theta} \right)\exp(- \alpha T) h(q_T) \leq h(\tilde{q}_T) + \exp(\theta \alpha T).
 \end{equation*}
 Combining the last inequality with \eqref{ineq:q-tilde-h-grad-f}, we obtain
 \begin{align} \label{ineq:h-grad-f}
    \left(1-\frac{1}{\theta} \right) \frac{1}{\upsilon} \mathbb{E}\left[h(q_T) \right] & \leq  C_2 + \frac{1}{2 \beta}   \mathbb{E} \left[ \int_0^T  \tilde{q}_s |Z^\star_s|^2 \dd s \right] + \exp(\theta \alpha T)\\
    & \leq  C_2 + \frac{1}{2 \beta}   \mathbb{E} \left[ \int_0^T  q_s |Z^\star_s|^2 \dd s \right] + \exp(\theta \alpha T),
 \end{align}
where we recall that $\upsilon = \beta \exp(\alpha T)$, see \eqref{eq:minoration:S:q:3}.
\vskip 4pt

\noindent \textit{Step 3: Conclusion.} Combining \eqref{ineq:entropic-gradf-h-exp} and \eqref{ineq:h-grad-f}, 
we get in the end
\begin{equation*}
    \left(\left(1-\frac{1}{\theta} \right) \frac{1}{\upsilon} - \frac{1}{\vartheta} \right)  h(q_T) \leq  C + \exp(\theta \alpha T) + \mathbb{E}\left[ \exp \left( \vartheta\left(\xi + \int_0^T \ell_s^0 \dd s \right) \right) \right],
 \end{equation*}
 where $C = C_1 + C_2$. Choosing $\theta > \left(1 - \upsilon/ \vartheta\right)^{-1}$, recalling that $\vartheta >  \upsilon$ by definition (see again \eqref{eq:minoration:S:q:3}), we deduce that 
 there exists a finite constant $c_1' > 0$ independent of $c_1$ such that $\sup_{t \in [0,T]} \mathbb{E}[h(q_t)] \leq c_1'$ provided that  $\xi + \int_0^T \ell^0_s \dd s \in   L^{1,\vartheta}_{\exp}(\mathcal{F}_T)$. 
 The latter can be verified by combining Lemma \ref{lemma:reg-X} with assumptions \ref{hyp:L} and \ref{hyp:gamma}. The argument was already outlined in Remark \ref{rem:p-L-beta-alpha}.
 On the one hand, \ref{hyp:L} says that 
 $\ell^0$ is bounded by $L$. In particular, it suffices to show 
 that $\xi$, which is here equal to $L \vert X_T^0 \vert^{2-r}$, belongs to 
$ L^{1,\vartheta}_{\exp}(\mathcal{F}_T)$, for 
some $\vartheta > v$. 
 By Lemma \ref{lemma:reg-X}, 
this is always true when 
$r=1$. When $r=0$, we need $4 \vartheta L \| \Gamma \|^2_{L^\infty({\mathbb F},{\mathbb R}^{n \times n})} \|\Gamma^{-1} \|^2_{L^\infty({\mathbb F},{\mathbb R}^{n \times n})} \| \nu \|^2_{L^\infty({\mathbb F},{\mathbb R}^{n \times d})} T < 1$. By \ref{hyp:gamma}, this is indeed possible to choose $\vartheta$ satisfying  the latter while ensuring the condition $\vartheta>v$. This completes the proof.
 
 Finally, using \eqref{eq:minoration:S:q:3} one last time, we conclude that $\mathcal{S}(q)\leq c_1'$, for a possibly new (but still independent of $c_1$) value of $c_1'$.
\end{proof}

\subsubsection{Necessary condition}
\label{subsubse:necessary:Nature}
We show the necessary condition, i.e. the first assertion, in the statement of Theorem \ref{thm:sto-max-princ-dual}. Because $\bar{\psi} \in {\mathcal A}_{c_2}$ is fixed 
throughout the subsubsection, 
we omit it in most of the notations
and merely write
\begin{equation}
\label{eq:saddle:simplified:notation}
    \mathcal{J}(q) \coloneqq \mathcal{J}(q,\bar{\psi}), \quad 
    \mathcal{R}(q) \coloneqq \mathcal{R}(q,\bar{\psi}), \quad \ell_s \coloneqq \ell(s,\bar{\psi}_s), 
\end{equation}
for $q \in {\mathcal Q}_{c_1}$.
We denote by
$(Y^\star,Z^\star)$ the representatives of $q$, 
as defined in
\eqref{eq:q}, i.e. 
$(Y^\star,Z^\star)$ is an
${\mathbb R} \times {\mathbb R}^d$-valued ${\mathbb F}$-progressively measurable pair satisfying 
$\mathcal{S}(q) < c_1$. With $q$, 
we also associate the equivalent probability measure $\mathbb{Q}$ given by $\dd \mathbb{Q} =  \mathcal{E}_T(\int_0^\cdot Z^\star_s \cdot \dd W_s) \dd \mathbb{P}$.  We also consider the adjoint BSDE, with unknown $(Y,Z)$,
\begin{equation} \label{eq:adjoint-U-V}
\left\{
\begin{array}{rl}
    - \dd Y_t & = \left( Y^\star_t Y_t + Z^\star_t\cdot Z_t  - f^\star(t,Y^\star_t,Z^\star_t) + \ell_t \right)\dd t - Z_t \cdot \dd W_t, 
    \quad t \in [0,T], 
    \\
    Y_T & = \delta_q \mathcal{G}(q_T,X_T^\psi).
    \end{array} \right.
\end{equation}
Notice that this BSDE is not the one appearing in the first-order system \eqref{optim:condition-dual}, since at this stage of the proof, the relationship \eqref{eq:Y:Z<->Ystar:Zstar} between $(Y,Z)$ and $(Y^\star,Z^\star)$ is not yet known. The additional property \eqref{eq:Y:Z<->Ystar:Zstar} forms part of the necessary condition and is shown to hold under the assumption that $q \in {\mathcal Q}_{c_0}$ for some $c_0 \in (0,c_1)$, with $q$ being a maximizer of $q' \mapsto {\mathcal J}(q')$ over ${\mathcal Q}_{c_1}$; see Lemma~\ref{lemma:necessary-q}. For the time being, we establish the following well-posedness result:

\begin{lemma}
\label{lemma:BSDE-Y-Z}
There exists a unique solution $(Y,Z) \in D({\mathbb F},{\mathbb Q}) \times (\cap_{\beta \in (0,1)} M^{\beta}({\mathbb Q},{\mathbb F},{\mathbb R}^d))$ to \eqref{eq:adjoint-U-V}. It is given by the  formula
(with the shorthand notation $\delta_q \mathcal{G}(q_T)$ in place of  
$\delta_q \mathcal{G}(q_T,X_T^\psi)$):
\begin{equation}
\label{eq:representation:process:Y:OPTN}
Y_t = q_t^{-1}
{\mathbb E} \left[ \left. q_T 
\delta_q \mathcal{G}(q_T) + \int_t^T q_s
\left( \ell_s - f^\star(s,Y_s^\star,Z_s^\star) \right) 
\dd s  \right\vert  {\mathcal F}_t
\right], 
\quad t \in [0,T].
\end{equation}
\end{lemma}
\begin{proof}

The proof is mostly taken  from 
\cite{BRIAND2003109}. We give it for the sake of completeness. We
recall that $q_t = \Lambda_t {\mathcal E}_t$, for any $t \in [0,T]$, where $(\Lambda_t \coloneqq \exp( \int_0^t Y_s^\star \dd s))_{t \in [0,T]}$ and 
$({\mathcal E}_t \coloneqq {\mathcal E}_t(\int_0^\cdot Z_s^{\star} \cdot \dd W_s))_{t \in [0,T]}$.
\vskip 4pt

\noindent \textit{Step 1:}
We claim that there exists a pair $(\tilde Y,\tilde Z) \in D({\mathbb F},{\mathbb Q}) \times (\cap_{\beta \in (0,1)} M^{\beta}({\mathbb Q},{\mathbb F},{\mathbb R}^d))$
 such that, for every $t \in [0,T]$, 
\begin{equation}
\label{eq:firststep:lem:19:aa}
    \tilde Y_t = \Lambda_T \delta_q \mathcal{G}(q_T) + \int_t^T \Lambda_s \left( \ell_s -  f^\star(s,Y^\star_s,Z^\star_s)\right) \dd s - 
    \int_t^T \tilde Z_s \cdot \dd \tilde W_s,
\end{equation}
where $(\tilde W_t \coloneqq W_t - \int_0^t Z_s^\star \dd s)_{t \in [0,T]}$ is a Brownian motion under the equivalent probability measure    ${\mathbb Q} = {\mathcal E}_T  {\mathbb P}$. 
Existence of the pair $(\tilde Y,\tilde Z)$ is proven in two steps.
\vskip 4pt

Throughout, the letter $C$ denotes a generic constant  that only depends on the assumptions listed in Subsection \ref{subsec:main-result} and that is, in particular, independent of
 $q$ and $\bar \psi$.
The first observation is that 
\begin{equation}
\label{eq:firststep:lem:19:a}
\mathbb{E}^{\mathbb{Q}}
\left[ \Lambda_T
\vert \delta_q \mathcal{G}(q_T)\vert + \int_0^T 
\Lambda_s \vert \ell_s -  f^\star(s,Y^\star_s,Z^\star_s) \vert \dd s \right] < +\infty,
\end{equation}
which is a consequence of the following three bounds. 
First, by the growth Assumption \ref{eq:G-growth}, we have
 \begin{align*}
         \mathbb{E}^{\mathbb Q} 
         \left[ \Lambda_T \left\vert  \delta_q \mathcal{G}(q_T) \right\vert \right]
          =
 \mathbb{E}  
         \left[ q_T \left\vert  \delta_q \mathcal{G}(q_T) \right\vert \right]
         \leq L \exp(\alpha T) \left( 1 + 2 \mathbb{E}\left[q_T |X_T^{\bar{\psi}}|^{2-r}\right] \right).
    \end{align*}
   By Lemma \ref{lemma:reg-X-psi-A}, the last term  satisfies the inequality
    \begin{equation*}
        \mathbb{E}\left[q_T \left|X_T^{\bar{\psi}}\right|^{2-r}\right] \leq C\left(1 + \mathcal{S}(q) + \mathcal{S}^\star(\bar{\psi}) \right).
    \end{equation*}
    Recalling that $q \in \mathcal{Q}_{c_1}$ and $\bar{\psi} \in \mathcal{A}_{c_2}$, we obtain
\begin{equation*}
    \mathbb{E}^{\mathbb Q} 
         \left[ \Lambda_T \left\vert  \delta_q \mathcal{G}(q_T) \right\vert \right] < +\infty.
\end{equation*}
By
\ref{hyp:L}
and 
 a direct application of the duality inequality \eqref{ineq:S-S-star},
we also have
\begin{equation*}
         \mathbb{E}^{\mathbb Q} 
         \left[ 
         \int_0^T \Lambda_s \vert \ell_s \vert \dd s\right]
         \leq 
         C + C
 \mathbb{E}  
         \left[  
         \int_0^T
         q_s \vert \bar{\psi}_s 
         \vert^2 \dd s \right] \leq C + C
         \left( \mathcal{S}(q) + \mathcal{S}^\star(\bar{\psi}) \right) < +\infty. 
    \end{equation*}
    It remains to observe from 
\ref{eq:f:star} that $f^\star(t,y^\star,z^\star) \geq - f_t^0\coloneqq-f(t,0,0)$, which implies 
$\vert f^\star(t,y^\star,z^\star)\vert \leq f^\star(t,y^\star,z^\star)+ \vert f_t^0\vert + f_t^0$, and then 
\begin{align} \nonumber
\mathbb{E}^{\mathbb{Q}}
\left[  \int_0^T 
\Lambda_s \vert   f^\star(s,Y^\star_s,Z^\star_s) \vert \dd s \right] &\leq  \mathbb{E}
\left[ \int_0^T 
q_s(    f^\star(s,Y^\star_s,Z^\star_s) + f^0_s + |f^0_s|) \dd s \right] \\  
& \leq C\left( 1 + \mathcal{S}(q)\right) < +\infty.\label{ineq:f-star-abs-finite}
\end{align}
The last three displays imply \eqref{eq:firststep:lem:19:a}. 

As announced, 
we now follow 
\cite[Section 6]{BRIAND2003109}.
To do so, we consider $(\xi^k,\ell^k)_{k \in {\mathbb N}}$ such that, for each 
$k \in {\mathbb N}$, 
$\xi^k$
 is a bounded ${\mathcal F}_T$-measurable random variable and 
 $\ell^k=(\ell^k_t)_{t \in [0,T]}$
 is a bounded ${\mathbb F}$-progressively measurable 
with the property that
\begin{equation*}
\lim_{k \rightarrow + \infty}
{\mathbb E}^{\mathbb Q}\left[ 
\Lambda_T \vert \xi^k - 
 \delta_q \mathcal{G}(q_T) \vert 
+ \int_0^T \Lambda_s
\vert \ell^k_s - (\ell_s - f^\star(s,Y^\star_s,Z^\star_s)) \vert \dd s 
\right] 
=0. 
\end{equation*}
Then, for each $k \in {\mathbb N}$, we can define 
$(\tilde Y^k,\tilde Z^k)$ such that 
(we recall from \cite[Theorem 2.4]{AksamitFontana} that the martingale representation 
theorem holds under 
${\mathbb Q}$, with respect to $\tilde W$)
\begin{equation*}
\tilde Y_t^k 
= \Lambda_T \xi^k + \int_t^T 
\Lambda_s \ell^k_s
\dd s
- \int_t^T \tilde Z_s^k \cdot \dd \tilde W_s, \quad 
t \in [0,T]. 
\end{equation*}
There is no difficulty to see that
\begin{equation*} 
\tilde Y^k_t = 
{\mathbb E}^{\mathbb Q} \left[ \left. \Lambda_T \xi^k + \int_t^T \Lambda_s \ell^k_s 
\dd s \right \vert {\mathcal F}_t\right], \quad t \in [0,T]. 
\end{equation*}
Since the term inside the conditional expectation appearing in the right-hand side is bounded, for each $k \in {\mathbb N}$, uniformly in $t \in [0,T]$, we easily deduce that the process 
$(\tilde Y^k,\tilde Z^k)$ satisfies the conclusion of the statement. In fact, item 2 is even satisfied in a stronger sense, as $\beta$ can be taken in $(0,+\infty)$. 

The key step is to prove that the sequences 
$(\tilde Y^k)_{k \in {\mathbb N}}$ and 
$(\tilde Z^k)_{k \in {\mathbb N}}$ are Cauchy sequences in well-chosen spaces. 
As for $(\tilde Y^k)_{k \in {\mathbb N}}$, we notice that, for any ${\mathbb F}$-stopping time $\tau$ with values in $[0,T]$,
for any $m,k \in {\mathbb N}$,
\begin{equation*}
{\mathbb E}^{\mathbb Q} \left[ \vert \tilde Y^k_{\tau} - \tilde Y^m_{\tau}\vert\right]
\leq {\mathbb E}^{\mathbb Q}
\left[
\Lambda_T \vert \xi^k - \xi^m \vert
+ 
\int_0^T \Lambda_s 
\vert \ell^k_s - \ell^m_s \vert \dd s 
\right], 
\end{equation*}
and the right-hand side tends to $0$, as $m$ and $k$ tend to $+\infty$. This shows that
\begin{equation*}
\lim_{N \rightarrow \infty} \sup_{m,k \geq N} \sup_{\tau}
{\mathbb E}^{\mathbb Q}\left[ \vert \tilde Y^k_{\tau} - \tilde Y^m_{\tau} \vert\right]=0,\end{equation*}
with $\tau$ in the left-hand side being implicitly understood as a generic stopping time with values in $[0,T]$. 
Then, the analysis carried out in \cite{BRIAND2003109}
(together with the references cited therein) permits us to show that there exists a process 
$\tilde{Y}$ satisfying item 1 in the statement such that 
\begin{equation*}
 \lim_{k \rightarrow \infty}  \sup_{\tau}
{\mathbb E}^{\mathbb Q}\left[ \vert \tilde Y^k_{\tau} - \tilde Y_{\tau} \vert\right]=0. 
\end{equation*}
It then remains to handle the martingale integrand 
$(\tilde Z^k)_{k \in {\mathbb N}}$. Writing
\begin{equation*}
\int_0^t 
\left( \tilde Z^k_s 
- \tilde Z^m_s \right) \cdot 
\dd \tilde W_s
= \left( \tilde{Y}_t^k - 
\tilde{Y}^m_t
\right) 
-\left( \tilde{Y}_0^k - 
\tilde{Y}^m_0
\right) 
- \int_0^t 
\Lambda_s\left( \ell^k_s - \ell^m_s
\right) \dd s, 
\end{equation*}
for $t \in [0,T]$, we deduce from \cite[Lemma 6.1]{BRIAND2003109} that, for any $\beta \in (0,1)$, 
\begin{equation*}
\begin{split}
&{\mathbb E}^{\mathbb Q}
\left[\sup_{t \in [0,T]}
\left\vert
\int_0^t 
\left( \tilde Z^k_s
- \tilde Z^m_s
\right) \cdot 
\dd \tilde W_s
\right\vert^{\beta}
\right]
\\
&\leq C_{\beta}
{\mathbb E}^{\mathbb Q} \left[
\left\vert
\int_0^T
\left( \tilde Z^k_s
- \tilde Z^m_s
\right) \cdot 
\dd \tilde W_s
\right\vert
\right]^{\beta}
\\
&\leq 
C_{\beta}
\left(
{\mathbb E}^{\mathbb Q}
\left[\Lambda_T\vert \xi^k - \xi^m\vert\right]
+
{\mathbb E}^{\mathbb Q}
\left[ \vert Y_0^k - Y_0^m \vert\right]
+ {\mathbb E}\left[
\int_0^T 
\Lambda_s\vert \ell^k_s 
- \ell^m_s \vert 
\dd s 
\right] \right)^{\beta},\end{split}
\end{equation*}
for a constant $C_{\beta} >0$ only depending on $\beta$. 
As a consequence, we obtain 
\begin{equation*} 
\lim_{N \rightarrow \infty}
\sup_{m,k \geq N}
{\mathbb E}^{\mathbb Q}
\left[\sup_{t \in [0,T]}
\left\vert
\int_0^t 
\left( \tilde Z^k_s
- \tilde Z^m_s
\right) \cdot 
\dd \tilde W_s
\right\vert^{\beta}
\right]
=0. 
\end{equation*}
And then, by B\"urkholder-Davis-Gundy inequality, 
it holds 
\begin{equation*} 
\lim_{N \rightarrow \infty}
\sup_{m,k \geq N}
{\mathbb E}^{\mathbb Q}
\left[
\left(
\int_0^T 
\vert \tilde Z^k_s
- \tilde Z^m_s
\vert^2   
\dd  s
\right)^{\beta/2}
\right]
=0, 
\end{equation*}
and the existence of 
$\tilde{Z}$ as in the statement follows from a new application of Cauchy's convergence criterion in complete spaces. 

\vskip 4pt

\noindent \textit{Step 2:}
We now establish uniqueness of the pair $(Y,Z)$. Multiplying 
any solution by $(\Lambda_t)_{t \in [0,T]}$, uniqueness of the pair $(Y,Z)$   is in fact equivalent to uniqueness of the pair $(\tilde Y,\tilde Z)$ in the expansion 
\eqref{eq:firststep:lem:19:aa}
(within the same space as in the statement). 

 By uniform integrability of the collection $(\tilde Y_{\tau})_{\tau}$, when $\tau$ runs over the set of ${\mathbb F}$-stopping times with values in $[0,T]$ and by a standard localization argument, we deduce that, necessarily, 
 \begin{equation} 
\label{eq:representation:tildeY:lemma19}
\tilde Y_t = 
{\mathbb E}^{\mathbb Q} \left[ \left. \Lambda_T \delta_q \mathcal{G}(q_T) + \int_t^T \Lambda_s \left( \ell_s
- f^\star(s,Y_s^\star,Z_s^\star)\right)
\dd s \right\vert {\mathcal F}_t\right], \quad t \in [0,T]. 
\end{equation}
This establishes the uniqueness of $\tilde{Y}$. We then rewrite the equation for $(\tilde Y,\tilde Z)$ in the form
\begin{equation*}
\tilde{Y}_t - \int_0^t \Lambda_s (\ell_s - f^\star(s,Y_s^\star,Z_s^\star)) \dd s = \int_0^t \tilde Z_s \cdot \dd \tilde W_s, \quad t \in [0,T]. 
\end{equation*}
The right-hand side is a local martingale (by assumption). Since the left-hand side is given, we deduce that $\tilde Z$
is unique. Recalling the two formulas
${\mathbb Q} = {\mathcal E}_T(\int_0^\cdot Z_s^\star \dd W_s) {\mathbb P}$ and $(q_t = \Lambda_t {\mathcal E}_t(\int_0^{\cdot} Z_s^\star \cdot \dd W_s))_{t \in [0,T]}$, we
easily 
derive \eqref{eq:representation:process:Y:OPTN} from \eqref{eq:representation:tildeY:lemma19}.
\end{proof}

At this stage, the notion of a solution to equation \eqref{eq:adjoint-U-V} is well defined and understood in the sense of Lemma \ref{lemma:BSDE-Y-Z}. We are now in a position to establish the first-order condition for Nature.

\begin{lemma} \label{lemma:necessary-q} For a given $c_0 \in (0,c_1)$, assume that there exists a maximizer $q \in \mathcal{Q}_{c_0}$ to the problem \eqref{pb:dual} (the latter being set over 
${\mathcal Q}_{c_1})$. Then, denoting by $(Y,Z)$ the solution of \eqref{eq:adjoint-U-V}, the triple $(q,Y,Z)$ satisfies the first-order condition \eqref{optim:condition-dual}.
\end{lemma}

\begin{proof}
    \noindent \textit{Step 1: localization procedure.}
Generally speaking, our main objective is to prove that, for prescribed directions 
$(y^\star, z^\star) \in L^\infty(\mathbb{F}) \times L^\infty(\mathbb{F}, \mathbb{R}^d)$, 
${\mathbb P}$-almost surely, for 
almost every 
$t \in [0,T]$, 
\begin{equation}
\label{eq:fstar(y+z)}
f^\star(t,Y_t^\star +  y_t^\star,Z_t^\star +  z_t^\star) 
- 
f^\star(t,Y_t^\star,Z_t^\star) 
- \left(  Y_t y_t^\star + 
Z_t  \cdot z_t^\star \right) \geq 0.
\end{equation} 
From this, we will eventually derive 
\eqref{eq:Y:Z<->Ystar:Zstar}
and then 
\eqref{optim:condition-dual}. 

Although the proof 
of \eqref{eq:fstar(y+z)} follows seemingly standard arguments, it requires some non-trivial adjustments.
We proceed by contradiction assuming that the 
 left-hand side on \eqref{eq:fstar(y+z)} is negative on an event 
$E$ of positive measure under ${\rm Leb}_{[0,T]} \otimes 
{\mathbb P}$, namely 
\begin{equation} 
\label{eq:def:E:1}
\begin{split}
&\exists \varrho >0, \quad 
\left(
\textrm{\rm Leb}_{[0,T]} \otimes {\mathbb P}
\right) 
 (E) >0,  
 \\
\textrm{\rm with}
\  
&E \coloneqq 
\left\{ 
f^\star(t,Y_t^\star +   y_t^\star,Z_t^\star +  z_t^\star) 
- 
f^\star(t,Y_t^\star,Z_t^\star) 
- \left( Y_t y_t^\star + 
Z_t  \cdot z_t^\star\right)  \leq - \varrho
\right\}.
\end{split}
\end{equation} 
For a given $A>0$, we also  consider the stopping time
\begin{align}
    \label{eq:tau:A-z-Z-star}
     &\tau_{A}  
   \\
   &\coloneqq \inf \left\{ t \in [0,T], \;       \frac1{q_t} + \left |\int_0^t z^\star_s \cdot Z_s^\star \dd s \right\vert + \left |\int_0^t z^\star_s \cdot Z_s \dd s \right\vert + 
        \left\vert \int_0^t z^\star_s \cdot \dd W_s \right | \geq  A   \right\},
        \nonumber
 \end{align}
    with the usual convention that the stopping time is equal to 
    $+\infty$ if the set inside the infimum is empty. 
    It is easy to prove 
    that 
    \begin{equation*}
    \lim_{A \rightarrow + \infty} 
    {\mathbb P}\left( \left\{ \tau_A \leq T \right\} \right) 
    = 0,
    \end{equation*}
    from which we deduce that we can choose $A$ large enough such that 
    \begin{equation}
    \label{eq:def:E:2}  
    \left(
\textrm{\rm Leb}_{[0,T]} \otimes {\mathbb P}
\right) 
\left( 
\left\{ (t,\omega) 
\in E, \; t \leq \tau_A(\omega)
\right\} 
\right) >0. 
    \end{equation} 
We then define the new 
`localized' directions
\begin{equation}
\label{eq:z:star:A}
y^{\star,E}_s(\omega) : =
y^\star_s(\omega) \mathds{1}_E(s,\omega), 
\quad 
    z^{\star,A,E}_s(\omega) \coloneqq z^{\star}_s(\omega) \mathds{1}_E(s,\omega)  \mathds{1}_{[0,\tau_{A}(\omega)]}(s), 
\end{equation}
for $s \in [0,T]$ and $\omega \in \Omega$. 
For an intensity $\varepsilon \in (0,1]$, we 
consider the 
solution 
$(q^{\varepsilon}_t\coloneqq 
\Lambda_t^{\varepsilon}
{\mathcal E}_t^{\varepsilon})_{t \in [0,T]}$ 
of the 
equation 
\eqref{eq:q} driven by 
the pair 
\begin{equation*}
(Y^{\star,\varepsilon},Z^{\star,\varepsilon})\coloneqq 
(Y^\star,Z^\star)+ 
\varepsilon (y^{\star,E},z^{\star,A,E})
\end{equation*}
and where
\begin{equation*}
\Lambda^{\varepsilon}_t = \exp\left(\int_0^t 
 Y_s^{\star,\varepsilon} \dd s \right), \quad \mathcal{E}^{\varepsilon}_t = \mathcal{E}_t \left( \int_0^\cdot  
 {Z^{\star,\varepsilon}_s}   \cdot \dd W_s \right),
 \quad t \in [0,T].
\end{equation*}
 (For simplicity, we omit to precise the dependence on $A$ and $E$.)

We also introduce the process $q^{\prime}$ which will be proved to be the variational process of $q$ in the direction $(y^{\star,E},z^{\star,A,E})$. It is defined as
\begin{equation}
\label{definition:qTprimeA}
q_t^{\prime}
\coloneqq q_t \left( \int_0^t \left( y_s^{\star,E}
- Z_s^\star \cdot z_s^{\star,A,E}
\right)\dd s
+ \int_0^t z_s^{\star,A,E} \cdot \dd W_s \right), 
\quad t \in [0,T].
\end{equation}
Using the definition 
of $z^{\star,A,E}$, we can check that, for all 
$t \in [0,T]$,
$\vert q_t^{\prime} \vert \leq C q_t$, for a constant $C$ independent of $\varepsilon$ and $t$. Moreover, 
$q^{\prime}$ solves the equation
    \begin{align} \label{definition:qTprimeA:2}
         q_t^{\prime} = \int_0^t\left( q_s^{\prime} Y_s^{\star} +  q_s y_s^{\star,E} \right)  \dd s +  \int_0^t \left( q_s^{\prime} Z_s^{\star} +  q_s z_s^{\star,A,E} \right) \cdot \dd W_s,   
         \quad t \in [0,T]. 
            \end{align}
We then let
\begin{equation*}
    \begin{array}{ll}
        \Delta q_s^{\varepsilon} \coloneqq \varepsilon^{-1}\delta q_s^{\varepsilon} - q_s^{\prime}, \quad &{\rm with} \quad \delta q_s^{\varepsilon} \coloneqq q_s^{\varepsilon} - q_s,
        \\
\\
        \delta f_s^{\star,\varepsilon} \coloneqq f^{\star,\varepsilon}_s - f^{\star}_s, \quad 
&{\rm with} \quad \left\{
\begin{array}{l}
f^{\star}_s   \coloneqq f^\star\left(s,Y_s^{\star},Z_s^{\star}\right),\\
f^{\star,\varepsilon}_s  \coloneqq f^\star\left(s,Y_s^{\star,\varepsilon},Z_s^{\star,\varepsilon} \right), 
\end{array}\right.
\end{array}
\end{equation*}
    for any $s \in [0,T]$.
We observe that, at this stage, $f^{\star,\varepsilon}$ may take the value 
    $+\infty$. 
\vskip 4pt

\noindent \textit{Step 2: $q^\varepsilon \in \mathcal{Q}_{c_1}$ for $\varepsilon$ small enough.} We show that $q^{\varepsilon}$ is an admissible controlled process for
\eqref{pb:dual}. The proof relies on the explicit formula for $q^{\varepsilon}$. For each time $t \in [0,T]$, we have
    \begin{align*}
        q^{\varepsilon}_t & = q_t \exp\left( \varepsilon \left(\int_0^t y_s^{\star,E} \dd s  +  \int_0^t z_s^{\star,A,E}  \cdot \dd W_s -  \int_0^t \left(Z^{\star}_s \cdot  z_s^{\star,A,E} + \tfrac12 \varepsilon |z_s^{\star,A,E}|^2 \right)\dd s \right) \right) \\
        & \eqqcolon q_t \exp\left( \varepsilon \varphi^{\varepsilon}_t \right).
    \end{align*}
   We can find a constant $C$, independent of $\varepsilon$, such that,
    with probability 1, 
$\sup_{t \in [0,T]}
|\varphi^{\varepsilon}_t| \leq C$. This follows from the fact that 
    $y^{\star}$ and $z^{\star}$ are bounded 
and from the definitions of $\tau_A$, $y^{\star,E}$ and $z^{\star,A,E}$
(see \eqref{eq:tau:A-z-Z-star} and \eqref{eq:z:star:A}). 
And then, for any $\varepsilon \in (0,1]$
and any $t \in [0,T]$,
\begin{equation} \label{ineq:delta-q-c-sup-q}
\vert q_t^{\varepsilon}
- q_t \vert \leq 
C \varepsilon \exp (C) 
 q_t  \leq 
 C \varepsilon \exp (C) 
 \sup_{s \in [0,T]}
 q_s.\end{equation}
To establish the desired result, we expand the term $\mathcal{S}(q^\varepsilon)$  as  
\begin{equation} \label{eq:def-S-split}
    \mathcal{S}(q^\varepsilon) = \mathbb{E}\left[\int_0^T q_s  f_s^{\star,\varepsilon} \dd s  \right]  + \mathbb{E}\left[ \int_0^T \delta q^\varepsilon_s f_s^{\star,\varepsilon} \dd s \right].
\end{equation}
We study the two terms on the right-hand side separately. We start with the first one. 
  By convexity of $f^\star$
      in the variables $y^\star$ and $z^\star$
      and by definition of the set $E$, see 
      \eqref{eq:def:E:1} and
        \eqref{eq:def:E:2}, we have
      (because $\varepsilon \in (0,1]$)
        \begin{equation}
        \label{eq:deltafstar,eps}\begin{split}
     &\frac1{\varepsilon} \delta f^{\star,\varepsilon}_s(\omega) 
     \\
     &\leq  \frac1{\varepsilon}
     {\mathds 1}_E(s,\omega)
     {\mathds 1}_{[0,\tau_A(\omega)]}(s) 
     \left( 
     f^\star\left(s,Y_s^\star + \varepsilon y_s^{\star}, Z_s^\star + 
      \varepsilon z_s^{\star} \right) 
      -      f^\star\left(s,Y_s^\star , Z_s^\star  \right) 
      \right)
      \\
      &\leq {\mathds 1}_E(s,\omega)
     {\mathds 1}_{[0,\tau_A(\omega)]}(s) 
     \left( y_s^\star Y_s^\star 
     + z_s^\star \cdot Z_s^\star - \varrho
      \right)
      \\
      &= y_s^{\star,E} Y_s^\star 
     + z_s^{\star,A,E} \cdot Z_s^\star
     - \varrho {\mathds 1}_E(s,\omega)
     {\mathds 1}_{[0,\tau_A(\omega)]}(s), 
      \end{split}
      \end{equation}
      for the same real $\varrho \geq 0$ as in 
      \eqref{eq:def:E:1}.
Multiplying both sides by $q$ and $\varepsilon$, adding 
$q_s f_s^\star$ on both sides,  integrating from $0$ to $T$ and using the fact that 
$\sup_{t \in [0,\tau_A]} \vert \int_0^t 
z_s^\star \cdot Z_s \dd s \vert \leq A$, we deduce that
(allowing the value of $C$ to vary from line to line as long as it remains independent of $
\varepsilon$)
\begin{equation*}
{\mathbb E}\left[ \int_0^T q_s f_s^{\star,\varepsilon} \dd s \right] \leq C \varepsilon + {\mathcal S}(q).
\end{equation*}
We now turn to the second term on the right-hand side of \eqref{eq:def-S-split}. 
From the inequality \eqref{ineq:delta-q-c-sup-q}, we directly deduce that
\begin{equation*}
    \mathbb{E}\left[ \int_0^T \delta q^\varepsilon_s f_s^{\star,\varepsilon} \dd s \right]\leq C \varepsilon \mathbb{E}\left[ \int_0^T q_s |f_s^{\star,\varepsilon}| \dd s \right] \leq C \varepsilon +  C \varepsilon \mathbb{E}\left[ \int_0^T q_s f_s^{\star,\varepsilon} \dd s \right] \leq C \varepsilon,
\end{equation*}
where we used the 
bound 
$\vert f^{\star,\varepsilon}_t\vert \leq f^{\star,\varepsilon}_t + C$, see \eqref{ineq:f-star-abs-finite}, the value of the constant $C$ being allowed to change from one term to another. 

Finally, plugging the last two estimates into \eqref{eq:def-S-split} yields the desired result provided that $\varepsilon$ is small enough.
\vskip 4pt

    \noindent \textit{Step 3: strong convergences of $\delta q^{\varepsilon}$ and 
    $\Delta q^{\varepsilon}$.} 
As
a direct consequence of the inequality \eqref{ineq:delta-q-c-sup-q}, we deduce that, for $A$ fixed, ${\mathbb P}$ almost surely,
        $\sup_{t \in [0,T]} \vert q_t^{\varepsilon} - q_t \vert \rightarrow 0$
         as 
        $\varepsilon \rightarrow 0$.
        Below, we also establish the
        almost sure uniform convergence of $\Delta q^{\varepsilon}$ to $0$. To do so, one can refine the argument presented in Step 2 and provide a second-order (in $\varepsilon$) expansion of $q_t^{\varepsilon}$, writing, for all $t \in [0,T]$,
\begin{equation}
\label{eq:qvareps-q(1+varphi)}
\begin{split} 
\left\vert q_t^{\varepsilon} - q_t(1 + \varepsilon \varphi_t^{\varepsilon}
)\right\vert 
&\leq 
q_t 
\left\vert 
\exp(\varepsilon \varphi_t^{\varepsilon})
- (1 + \varepsilon \varphi_t^{\varepsilon}
)
\right\vert 
 \leq   C^2 \varepsilon^2 \exp(C) 
q_t.
\end{split}
\end{equation}
Now, we  use the fact that 
$(q_t \varphi_t^{\varepsilon})_{t \in [0,T]}$ and $(q_t')_{t \in [0,T]}$ are close one from each other. Indeed, 
\begin{equation*}
\begin{split}
\dd \left[ q_t \varphi_t^{\varepsilon}\right]
&= \left( q_t \varphi_t^{\varepsilon} Y_t^\star 
+
  q_t y_t^{\star,E}
- 
\tfrac12 
\varepsilon \vert z_t^{\star,A,E}
\vert^2
\right) \dd t
+ \left( 
q_t \varphi_t^{\varepsilon} Z_t^\star + q_t z_t^{\star,A,E} \right) \cdot \dd W_t, \quad t \in [0,T].
\end{split}
\end{equation*}
And then, thanks to 
\eqref{definition:qTprimeA:2},
\begin{equation*}
\begin{split}
&\dd \left[q_t'- q_t \varphi_t^{\varepsilon}
\right]
\\
&\hspace{15pt}= \left( \left[ q_t' - q_t \varphi_t^{\varepsilon} \right]
Y_t^\star 
+ 
\tfrac12 
\varepsilon \vert z_t^{\star,A,E}
\vert^2
\right) \dd t
+ \left( 
\left[ q_t' -
q_t \varphi_t^{\varepsilon} 
\right] Z_t^\star \right) \cdot \dd W_t, \quad t \in [0,T],
\end{split}
\end{equation*}
with $0$ as initial condition. It is standard to deduce that 
\begin{equation*} 
q_t'- q_t \varphi_t^{\varepsilon} = q_t \varepsilon 
\int_0^t 
q_s^{-1} \vert z_s^{\star,A,E}\vert^2 \dd s, \quad t \in [0,T].
\end{equation*}
Returning back to 
\eqref{eq:qvareps-q(1+varphi)}, we 
deduce from the above identity  (together with the fact that $q_s^{-1} \leq A$ for $s \leq \tau_A$) that
\begin{equation*}
\begin{split} 
\left\vert \delta q_t^{\varepsilon} - \varepsilon q_t^{\prime}
\right\vert 
&\leq C \varepsilon^2 
q_t,
\quad t \in [0,T],
\end{split}
\end{equation*} 
for a new value of $C$ (still independent of 
$\varepsilon$). Dividing by $\varepsilon$, we get 
\begin{equation}
\label{eq:bound:Deltaq-q'}
\begin{split} 
\left\vert \Delta q_t^{\varepsilon}  
\right\vert 
&\leq C \varepsilon q_t, 
\quad t \in [0,T].
\end{split}
\end{equation}
\vskip 4pt

    \noindent \textit{Step 4:  derivative of the terminal and running costs.}
    In this step, we address the 
    limit (as $\varepsilon$ tends to $0$) 
    of (as explained above, we omit the dependence on $\bar \psi$ in the notations)
        \begin{align}
        \label{eq:crit-for-deriv-RG}
        \frac{1}{\varepsilon}  \left( \mathcal{R}(q^{\varepsilon}) - \mathcal{R}(q)\right) = \frac1{\varepsilon} \left( \mathcal{G}(q^{\varepsilon}_T) - \mathcal{G}(q_T) + \mathbb{E}\left[\int_0^T \delta q_s^{\varepsilon} \ell_s \dd s \right]  \right).
   \end{align} 
   We first compute the derivative of $\mathcal{G}$ along $q^{\varepsilon}_T$. 
We write
\begin{align}
&\frac1{\varepsilon}
\left[ {\mathcal G}(q_T^{\varepsilon}) - 
{\mathcal G}(q_T)
\right] \nonumber
\\
&= \frac1{\varepsilon}\int_0^1 {\mathbb E}
\left[ \delta_q \mathcal{G}\left( 
q_T + \lambda \delta q_T^{\varepsilon}
\right) \delta q_T^{\varepsilon}
\right] \dd \lambda \nonumber
\\
&= 
\int_0^1 {\mathbb E}
\left[ \delta_q \mathcal{G}\left( 
q_T + \lambda \delta q_T^{\varepsilon}
\right) q_T'
\right] \dd \lambda
+ \int_0^1 {\mathbb E}
\left[ \delta_q \mathcal{G}\left( 
q_T + \lambda \delta q_T^{\varepsilon}
\right) \Delta q_T^{\varepsilon}
\right] \dd \lambda. \label{eq:derive:G(qeps)}
\end{align}
Here, we recall from 
\ref{eq:G-growth} that 
\begin{equation*}
\left\vert \delta_q \mathcal{G}\left( 
q_T + \lambda \delta q_T^{\varepsilon}
\right) \right\vert \leq L \left( 1 + \mathbb{E}\left[(q_T + \lambda \delta q_T^{\varepsilon}) \vert X_T \vert^{2-r} \right] + 
\vert X_T \vert^{2-r}\right).
\end{equation*}
Recalling that $|\delta q_t^\varepsilon| \leq C \varepsilon q_t$ by \eqref{ineq:delta-q-c-sup-q}, 
we deduce from Lemma 
\ref{lemma:reg-X-psi-A} that 
the expectation on the right-hand side is bounded uniformly in $\varepsilon$ by
\begin{equation*}
    \mathbb{E}\left[(q_T + \lambda \delta q_T^{\varepsilon}) \vert X_T \vert^{2-r} \right] \leq C \left(1 + \mathbb{E}\left[q_T \vert X_T \vert^{2-r} \right]\right) < +\infty.
\end{equation*}
Also, from the first step, we know that $q_T' \leq Cq_T$. Then by Lemma 
\ref{lemma:reg-X-psi-A} again, we 
deduce that 
${\mathbb E}[(1+ \vert X_T \vert^{2-r}) q_T']<+\infty$. By dominated convergence and by \eqref{ineq:delta-q-c-sup-q}, we deduce that the first term on the right-hand side of \eqref{eq:derive:G(qeps)} converges to
\begin{equation*}
\lim_{\varepsilon \rightarrow 0}
 \int_0^1 {\mathbb E}
\left[ \delta_q \mathcal{G}\left( 
q_T + \lambda \delta q_T^{\varepsilon}
\right) q_T'
\right] \dd \lambda
= {\mathbb E}
\left[ \delta_q \mathcal{G}\left( 
q_T 
\right) q_T'
\right].
\end{equation*}
As for the second term on the right-hand side of \eqref{eq:derive:G(qeps)}, 
we can proceed in the same way, but using 
in addition 
\eqref{eq:bound:Deltaq-q'}. We obtain
\begin{equation*}
\lim_{\varepsilon \rightarrow 0}
 \int_0^1 {\mathbb E}
\left[ \delta_q \mathcal{G}\left( 
q_T + \lambda \delta q_T^{\varepsilon}
\right)   \Delta q_T^{\varepsilon}
\right] = 0. 
\end{equation*}
 By combining the last two displays, we obtain
    \begin{align*}
        \lim_{\varepsilon \to 0}\frac1{\varepsilon} \left(\mathcal{G}(q^{\varepsilon}_T) - \mathcal{G}(q_T)\right) & = \mathbb{E}\left[ \delta_q \mathcal{G}\left(q_T \right) q_T^{\prime} \right] = \mathbb{E}\left[ Y_T q_T^{\prime} \right],
    \end{align*}
where we recall that 
$Y_T$ is the terminal condition given by the second line on 
\eqref{eq:adjoint-U-V}. 

    We now turn to the second term in \eqref{eq:crit-for-deriv-RG}. Recalling that 
$\mathcal{S}^\star(\bar{\psi}) < + \infty$,  
we deduce from
the bound
\eqref{eq:bound:Deltaq-q'}
and
the duality inequality \eqref{ineq:S-S-star} that
    \begin{align*}  \lim_{\varepsilon \to 0} \frac1{\varepsilon }\mathbb{E}\left[ \int_0^T \delta q_s^{\varepsilon} \ell_s \dd s \right] & =  \mathbb{E}\left[\int_0^T q_s^{\prime} \ell_s \dd s \right].
    \end{align*}
    Combining the last two limits, 
    we obtain
    \begin{align} 
    \lim_{\varepsilon \to 0} \frac1{\varepsilon} \left( \mathcal{R}(q^{\varepsilon}) - \mathcal{R}(q) \right) =
\mathbb{E}\left[
Y_T q_T^{\prime} +
\int_0^T  q_s^{\prime} \ell_s  \dd s  \right]. 
\label{eq:deriv-F:0}
    \end{align}
   We then expand the first  term on the right-hand side by means of 
    It{\^{o}}'s formula, using 
    the BSDE 
    \eqref{eq:adjoint-U-V} satisfied by $(Y_t)_{t \in [0,T]}$ and 
    the equation 
\eqref{definition:qTprimeA:2} satisfied by 
$q^{\prime}$.
We get for all $t \in [0,T]$,
\begin{equation*}
\begin{split}
{\mathrm d} \left[ q_t^{\prime} Y_t \right] &= 
q_t^{\prime}
\left( 
- Y_t^\star Y_t - 
Z_t^\star \cdot Z_t
+ f^\star(t,Y_t^\star,Z_t^\star) - \ell_t \right) 
{\mathrm d} t
\\
&\hspace{15pt}
+ Y_t \left( q_t^{\prime} Y_t^\star + q_t y_t^{\star,E}  \right) 
{\mathrm d} t
+ Z_t 
\cdot \left( q_t^{\prime} Z_t^\star + q_t z_t^{\star,A,E} \right){\mathrm d} t + H_t \cdot \dd W_t,
\end{split}
\end{equation*}
with $(H_t \coloneqq q_t^{\prime} Z_t  + Y_t ( q_t^{\prime} Z_t^\star + q_t z_t^{\star,A,E} ))_{t \in [0,T]}$. Here, we need a new localization sequence to handle the local martingale. We define, for any 
$c>0$, 
$\varsigma_c \coloneqq \inf\{t \in [0,T], \; \vert Y_t \vert
+ \int_0^t \vert Z_s \vert^2 {\mathrm d} s\geq c\}$ (with $\inf \emptyset = + \infty$). Because $Y$ has continuous trajectories and $\mathbb{P}(\{ \int_0^T \vert Z_t \vert^2 {\mathrm d}t < + \infty\})=1$, it holds 
$\varsigma_c \rightarrow + \infty$ 
almost surely, as $c \rightarrow + \infty$. 
Then, recalling the identity $q_0'=0$ and the shorthand notation $f^\star_t$
for $f^\star(t,Y_t^\star,Z_t^\star)$, we can rearrange the above expansion and obtain 
\begin{align*}
q_{\varsigma_c \wedge T}^{\prime}
Y_{\varsigma_c \wedge T}+
\int_0^{\varsigma_c \wedge T}
q_s^{ \prime}
\ell_s 
{\mathrm d} s
= & 
 \int_0^{\varsigma_c \wedge T}
 \left( q_s^{\prime} f_s^\star  
+ q_sY_s    y_s^{\star,E}   
+ 
  q_s Z_s \cdot z_s^{\star,A,E}  
  \right) {\mathrm d} s \\
  & + \int_0^{\varsigma_c \wedge T}
H_s \cdot {\mathrm d} W_s.
\end{align*}
Using the three bounds 
$\vert q_t^{\prime} \vert \leq C q_t$,
for all 
$t \in [0,T]$,
$\| z^{\star,A,E} \|_{L^\infty({\mathbb F},{\mathbb R}^d)} < +\infty $
and 
${\mathbb E}[q_T\int_0^T \vert Z_s^\star \vert^2 
{\mathrm d}s]< + \infty$, together with the definition of the stopping time $\varsigma_c$, 
we can prove that 
$
{\mathbb E}[
\int_0^{\varsigma_c \wedge T}
H_s \cdot {\mathrm d} W_s ]
=0$,
from which we deduce 
\begin{equation}
\label{eq:qprimeAY+qprimeQell}
\begin{split}
&{\mathbb E}
\left[q_{\varsigma_c \wedge T}^{\prime}
Y_{\varsigma_c \wedge T}+
\int_0^{\varsigma_c \wedge T}
q_s^{ \prime}
\ell_s 
{\mathrm d}s\right]
  = 
{\mathbb E}  \left[\int_0^{\varsigma_c \wedge T}
 \left( q_s^{\prime} f_s^\star  
+ q_s Y_s    y_s^{\star,E}   
+ 
  q_s Z_s \cdot z_s^{\star,A,E}  
  \right) {\mathrm d} s \right]. 
\end{split}
\end{equation}
The point is to let $c$ tend to $+\infty$ on both sides. Thanks to the 
following three
inequalities (the first line follows from the two bounds $\vert q_t^{\prime} \vert 
\leq C q_t$
and 
$\vert f_t^\star \vert \leq f_t^\star +C$, for $t \in [0,T]$, see \eqref{ineq:f-star-abs-finite}),
the second one from the bound 
$\sup_{t \in [0,T]} {\mathbb E}[q_T \vert Y_t\vert]<+\infty$, see Lemma \ref{lemma:BSDE-Y-Z} and the definition of 
the space $D({\mathbb F},{\mathbb Q})$ in Section \ref{sec:notations}, and the third one from \eqref{eq:tau:A-z-Z-star},
\begin{equation*} 
\begin{split}
&{\mathbb E} \left[
\int_0^T \vert q_s^{\prime} f_s^\star 
\vert 
{\mathrm d}s \right]
\leq C{\mathbb E}
\left[
\int_0^T q_s f_s^\star
{\mathrm d}s 
\right] + C = C\mathcal{S}(q) + C < + \infty,
\\
&{\mathbb E} \left[\int_0^T 
q_s \vert Y_s y_s^\star \vert 
{\mathrm d} s  \right]
\leq C 
{\mathbb E}\left[ q_T\int_0^T 
|Y_s|
{\mathrm d} s \right]< + \infty,
\\
& {\mathbb E} \left[
\int_0^T 
q_s \vert Z_s \cdot z_s^{\star,A,E} \vert 
{\mathrm d}s  \right]
\leq 
{\mathbb E}
\left[q_T
\int_0^T \vert Z_s \cdot z_s^{\star,A,E} \vert 
{\mathrm d} s 
\right]
\leq A {\mathbb E}
\left[q_T\right]
 < + \infty,
\end{split}
\end{equation*}
we deduce, by dominated convergence theorem, that
\begin{align*}
&\lim_{c \rightarrow + \infty}
{\mathbb E} \left[\int_0^{\varsigma_c \wedge T}
 \left( q_s^{\prime} f_s^\star  
+ q_s Y_s    y_s^{\star,E}
+ 
  q_s Z_s \cdot z_s^{\star,A,E}  
  \right) {\mathrm d} s \right]
\\
&= {\mathbb E} \left[\int_0^{T}
 \left( q_s^{\prime} f_s^\star  
+ q_s Y_s    y_s^{\star,E}   
+ 
  q_s Z_s \cdot z_s^{\star,A,E}  
  \right) {\mathrm d} s \right].
  \end{align*}  
Similarly, 
\begin{equation*}
\lim_{c \rightarrow + \infty}
{\mathbb E} \left[\int_0^{\varsigma_c \wedge T}
q_s^{\prime} \ell_s {\mathrm d}s \right]
= {\mathbb E} \left[\int_0^{T}
q_s^{\prime} \ell_s {\mathrm d}s\right].
\end{equation*}
It remains to pass to the limit in 
${\mathbb E}[Y_{\varsigma_c \wedge T} q_{\varsigma_c \wedge T}^{\prime }]$
as $c$ tends to $+\infty$. Almost surely, $Y_{\varsigma_c \wedge T} q_{\varsigma_c \wedge T}^{\prime}
\rightarrow Y_{T} q_{T}^{\prime}$. Moreover, using  the bound $\vert q_t^{\prime} \vert \leq C q_t$, we deduce that, for any event $F \in {\mathcal F}_T$, 
${\mathbb E}[{\mathds 1}_F \vert Y_{\varsigma_c \wedge T}  q_{\varsigma_c \wedge T}^{\prime} \vert ] \leq C
{\mathbb E}^{\mathbb Q}[{\mathds 1}_F \vert Y_{\varsigma_c \wedge T}\vert ]$. 
Recalling that the family 
$(Y_{\varsigma_c \wedge T})_{c >0}$ 
is uniformly integrable under 
${\mathbb Q}$, see Lemma \ref{lemma:BSDE-Y-Z}, we notice that the last term can be rendered as small as desired by choosing ${\mathbb Q}(F)$ small enough, and thus by choosing ${\mathbb P}(F)$ small enough. This shows that the collection $(\vert Y_{\varsigma_c \wedge T} q_{\varsigma_c \wedge T}^{\prime}\vert)_{c>0}$ is uniformly integrable. Therefore, 
${\mathbb E}[Y_{\varsigma_c \wedge T} q_{\varsigma_c \wedge T}^{\prime }]
\rightarrow {\mathbb E}[Y_{T} q_{T}^{\prime}]$.
Letting $c$ tend to $+\infty$ in 
\eqref{eq:qprimeAY+qprimeQell}, we get 
\begin{equation*}
{\mathbb E}
\left[q_{T}^{\prime}
Y_{T}+
\int_0^{T}
q_s^{ \prime}
\ell_s 
{\mathrm d} s\right]
  = 
{\mathbb E} \left[ \int_0^{T}
 \left( q_s^{\prime} f_s^\star  
+ q_s Y_s    y_s^{\star,E}
+ 
  q_s Z_s \cdot z_s^{\star,A,E}  
  \right) {\mathrm d} s \right].
\end{equation*}
And then returning back to 
\eqref{eq:deriv-F:0}, we get
    \begin{align} 
        \lim_{\varepsilon \to 0} \frac1{\varepsilon} \left( \mathcal{R}(q^{\varepsilon}) - \mathcal{R}(q) \right) = \mathbb{E}\left[\int_0^T \left(q_s^{\prime} f^\star_s + q_s \left( Y_s  y^{\star,E}_s + Z_s \cdot z^{\star,A,E}_s  \right)\right)\dd s  \right]. \label{eq:deriv-F}
    \end{align}
    \vskip 5pt

    \noindent \textit{Step 5:  sub-derivative of the entropic cost.}
        We now address the subgradient of $\varepsilon \mapsto \mathcal{S}(q^{\varepsilon})$.
    By definition, 
    we have (using the notation introduced in the first step)
    \begin{align}
         \mathcal{S}(q^{\varepsilon}) - \mathcal{S}(q) & = \mathbb{E}\left[ \int_0^T \left(q^{\varepsilon}_s f^{\star,\varepsilon}_s - q_s f^{\star}_s\right) \dd s \right] \nonumber \\
         &= \mathbb{E}\left[ \int_0^T q^{\varepsilon}_s \delta f^{\star,\varepsilon}_s \dd s \right]  + \mathbb{E}\left[ \int_0^T \delta q^{\varepsilon}_s f_s^{\star} \dd s\right], \label{eq:S:Lemma:20}
    \end{align}
    where, 
    by construction, 
    \begin{equation*}
    \begin{split} 
     \delta f^{\star,\varepsilon}_s(\omega) 
     &= 
     {\mathds 1}_E(s,\omega)
     {\mathds 1}_{[0,\tau_A(\omega)]}(s) 
     \left( 
     f^\star\left(s,Y_s^\star + \varepsilon 
     y_s^{\star,E}, Z_s^\star + 
     \varepsilon 
     z_s^{\star,A,E} \right) 
      -      f^\star\left(s,Y_s^\star , Z_s^\star  \right) 
      \right)
      \\
      &= 
     {\mathds 1}_E(s,\omega)
     {\mathds 1}_{[0,\tau_A(\omega)]}(s) 
     \left( 
     f^\star\left(s,Y_s^\star + \varepsilon 
     y_s^{\star}, Z_s^\star + 
     \varepsilon 
     z_s^{\star} \right) 
      -      f^\star\left(s,Y_s^\star , Z_s^\star  \right) 
      \right). 
      \end{split}
      \end{equation*}
      By \eqref{eq:deltafstar,eps},   \begin{equation*}
        \begin{split}
     \frac1{\varepsilon} \delta f^{\star,\varepsilon}_s(\omega) 
     &\leq  y_s^{\star,E} Y_s^\star 
     + z_s^{\star,A,E} \cdot Z_s^\star
     - \varrho {\mathds 1}_E(s,\omega)
     {\mathds 1}_{[0,\tau_A(\omega)]}(s), 
      \end{split}
      \end{equation*}
      for some the same real $\varrho \geq 0$ as in 
      \eqref{eq:def:E:1}.
      Since the right-hand side is integrable under ${\mathbb Q}$, we deduce 
      that (using the bound $q_s^{\varepsilon} \leq C q_s$) 
      \begin{equation*}
      \begin{split}
      &\limsup_{\varepsilon \to 0} 
      \frac1\varepsilon
      \mathbb{E}\left[ \int_0^T q^{\varepsilon}_s \delta f^{\star,\varepsilon}_s \dd s \right]  
      \\
      &\leq \limsup_{\varepsilon \to 0} {\mathbb E} \left[ \int_0^T   q^\varepsilon_s\left( 
      y_s^{\star,E} Y_s^\star 
     + z_s^{\star,A,E} \cdot Z_s^\star\right) \dd s
     \right] - 
     \varrho  {\mathbb E} \left[ \int_0^T q^\varepsilon_s {\mathds 1}_E(s,\cdot)
     {\mathds 1}_{[0,\tau_A]}(s)
     \dd s \right].
      \\
      &\leq {\mathbb E} \left[ \int_0^T q_s\left( 
      y_s^{\star,E} Y_s^\star 
     + z_s^{\star,A,E} \cdot Z_s^\star\right) \dd s
     \right] - 
     \varrho  {\mathbb E} \left[ \int_0^T q_s {\mathds 1}_E(s,\cdot)
     {\mathds 1}_{[0,\tau_A]}(s)
     \dd s \right].
     \end{split}
     \end{equation*}
 As for the second term on the right-hand side of 
\eqref{eq:S:Lemma:20}, we know from \eqref{eq:bound:Deltaq-q'} that 
$ \vert \delta q_t^{\varepsilon} / \varepsilon - q_t^{\prime} \vert 
\leq C \varepsilon q_t$, 
which gives directly (using the bound 
$\vert f_s^{\star } \vert \leq C 
+ C f_s^{\star }$, see again \eqref{ineq:f-star-abs-finite})
\begin{equation*}
\lim_{\varepsilon \to 0}
\frac1{\varepsilon} 
\mathbb{E}\left[ \int_0^T \delta q^{\varepsilon}_s f_s^{\star} \dd s\right]
= \mathbb{E}\left[ \int_0^T   q^{\prime}_s f_s^{\star} \dd s\right].
\end{equation*}
Back to 
\eqref{eq:S:Lemma:20}, we deduce that
    \begin{equation*}
    \begin{split}
        &\limsup_{\varepsilon \to 0} \frac1{\varepsilon} \left(\mathcal{S}(q^{\varepsilon}) - \mathcal{S}(q)\right) 
        \\
        &\leq  \mathbb{E}\left[ \int_0^T \left(q_s^{\prime} f_s^{\star} + q_s  
        (y_s^{\star,E} Y_s^\star 
     + z_s^{\star,A,E} \cdot Z_s^\star) 
     - \varrho q_s {\mathds 1}_E(s,\cdot)
     {\mathds 1}_{[0,\tau_A]}(s)  \right) \dd s \right].
     \end{split}
    \end{equation*}   
    Recalling from Lemma 
    \ref{lemma:positive-q}
    that $q$ is strictly positive, we deduce from
    \eqref{eq:def:E:1} 
    and 
    \eqref{eq:def:E:2} that 
    \begin{equation*}
     \mathbb{E}\left[  
    \int_0^T q_s {\mathds 1}_E(s,\cdot)
     {\mathds 1}_{[0,\tau_A]}(s)    \dd s \right]>0,
     \end{equation*}
     and thus 
      \begin{equation}
       \label{eq:deriv-alpha}
    \begin{split}
        &\limsup_{\varepsilon \to 0} \frac1{\varepsilon} \left(\mathcal{S}(q^{\varepsilon}) - \mathcal{S}(q)\right) 
   < \mathbb{E}\left[ \int_0^T \left(q_s^{\prime} f_s^{\star} + q_s  
        (y_s^{\star,E} Y_s^\star 
     + z_s^{\star,A,E} \cdot Z_s^\star)   \right) \dd s \right].
     \end{split}
    \end{equation}   
        \vskip 4pt

 \noindent \textit{Step 6: conclusion.}
We now come back to the definition of $\mathcal{J}$, recalling that, for any $\varepsilon \in (0,1]$,
    \begin{align}
        \frac{1}{\varepsilon} \left(\mathcal{J}(q^{\varepsilon}) - \mathcal{J}(q) \right) =  \frac{1}{\varepsilon}  \left( \mathcal{R}(q^{\varepsilon}) - \mathcal{R}(q)\right)
        -  \frac{1}{\varepsilon} \left(\mathcal{S}(q^{\varepsilon}) - \mathcal{S}(q)\right).
        \label{eq:crit-for-deriv-J}
    \end{align} 
    Combining \eqref{eq:deriv-F}, \eqref{eq:deriv-alpha} and
    \eqref{eq:crit-for-deriv-J},  we obtain 
    \begin{equation*}
    \liminf_{\varepsilon \to 0} \frac1{\varepsilon} \left(\mathcal{J}(q^{\varepsilon}) - \mathcal{J}(q)\right) 
    >0.
     \end{equation*}
     However, 
 by optimality of $q$ and because $q^{\varepsilon}$ is admissible for 
 \eqref{pb:dual} for $\varepsilon$ small enough, we have     \begin{equation*}
        \lim_{\varepsilon \to 0} \frac{1}{\varepsilon}\left(\mathcal{J}(q^{\varepsilon}) - \mathcal{J}(q) \right)\leq 0,
    \end{equation*}
    which gives a contradiction with 
    the penultimate line. This contradicts the assumption 
    made in \eqref{eq:def:E:1}, as a result of which we deduce that, 
    almost surely, for 
almost every 
$t \in [0,T]$, 
\begin{equation*} 
f^\star(t,Y_t^\star +  y_t^\star,Z_t^\star +  z_t^\star) 
- 
f^\star(t,Y_t^\star,Z_t^\star) 
- \left(  Y_t y_t^\star + 
Z_t  \cdot z_t^\star \right) \geq 0.
\end{equation*} 
By construction, 
the above holds true when the perturbation  
$(y^\star,z^\star)$ is bounded, but this assumption can be easily dropped by means of a truncation argument. In particular, we can choose 
$(y^\star,z^\star)= (\partial_y f(t,Y_t,Z_t)-Y_t^\star,\partial_z f(t,Y_t,Z_t)-Z_t^\star)_{t \in [0,T]}$. With this choice, we obtain 
(with full measure under ${\rm Leb}_{[0,T]} \otimes {\mathbb P}$)
\begin{equation*} 
\begin{split} 
&f^\star\left(t,\partial_y f(t,Y_t,Z_t),\partial_z f(t,Y_t,Z_t)\right) 
- 
f^\star(t,Y_t^\star,Z_t^\star) 
\\
&\geq 
  Y_t   \partial_y f(t,Y_t,Z_t)
  + Z_t \cdot \partial_z f(t,Y_t,Z_t)
  - Y_t Y_t^\star - Z_t \cdot Z_t^\star. 
  \end{split} 
\end{equation*} 
    Recalling that 
    \begin{equation*} 
    \begin{split}
    &f(t,Y_t,Z_t) + 
    f^\star \left(t,\partial_y f(t,Y_t,Z_t),\partial_z f(t,Y_t,Z_t)\right) 
    \\
    &= 
      Y_t   \partial_y f(t,Y_t,Z_t)
  + Z_t \cdot \partial_z f(t,Y_t,Z_t),
  \end{split}  
  \end{equation*} 
  we deduce (again, with full measure under ${\rm Leb}_{[0,T]} \otimes {\mathbb P}$)
  \begin{equation*} 
  Y_t Y_t^\star + Z_t \cdot Z_t^\star
  = 
  f^\star(t,Y_t^\star,Z_t^\star) + f(t,Y_t,Z_t),
  \end{equation*} 
which is known to imply 
$ (Y^\star_t,Z^\star_t) = (\partial_y f(t,Y_t,Z_t),\partial_z f(t,Y_t,Z_t))$.      
\end{proof}

We complement the necessary condition established in Lemma \ref{lemma:necessary-q} with a lower bound on the process $Y$. This bound plays a key role in the sufficient condition proved later in Lemma \ref{lemma:sufficient}. In fact, this bound is similar to the one obtained in \cite[Theorem 2.1]{delbaen2011uniqueness}. Unfortunately, the bound established in \cite{delbaen2011uniqueness} only holds for one specific solution of the quadratic equation \eqref{optim:condition-dual} (obtained by taking the limit on truncated equations); in the absence of uniqueness for the quadratic equation, it is not possible to apply \cite{delbaen2011uniqueness} to our case.

\begin{lemma} \label{lemma:necessary-q:lower-bound} 
For a given $c_0 \in (0,c_1)$, let $q \in \mathcal{Q}_{c_0}$ be a maximizer to the problem \eqref{pb:dual}
(with $\bar \psi \in {\mathcal A}_{c_2}$ being fixed)
and   $(Y,Z)$ be the solution of \eqref{eq:adjoint-U-V}. Then, 
${\mathbb P}$-a.s., 
for any $t \in [0,T]$, 
\begin{equation*}
Y_t \geq \tilde Y_t,
\end{equation*}
where 
$(\tilde Y,\tilde Z)$ solves the BSDE
\begin{equation}
\label{eq:representation:tildeY}
\left\{
\begin{array}{rl}
- \dd \tilde Y_t
 = & \left( \ell_t + f(t,0,0) + \left( \partial_y f(t,0,0) 
\tilde Y_t + \partial_z f(t,0,0) \cdot \tilde Z_t
\right)  \right) 
\dd t 
\\
& - \tilde Z_t \cdot \dd W_t,
\quad t \in [0,T],
\\
\tilde Y_T  = &
\delta_q \mathcal{G}(q_T).
\end{array}\right.
\end{equation}
In addition,
for any other $\tilde q \in {\mathcal Q}_{c_1}$, the random variables
$(\tilde q_{\tau} Y^-_{\tau}
\coloneqq
\tilde q_{\tau}
\min (-Y_{\tau},0)
)_{\tau}$, with 
    $\tau$ running over the set of $[0,T]$-valued ${\mathbb F}$-stopping times, are uniformly integrable under ${\mathbb P}$. 
\end{lemma}

\begin{proof}

\textit{Step 1.} 
Recall  that $q$ denotes a maximizer to the problem \eqref{pb:dual}. Then, for a given $t \in (0,T]$ and an arbitrary event $E \in {\mathcal F}_t$, we 
define $q^{t,E}$ by letting
\begin{equation*}
q^{t,E}_s
\coloneqq
\begin{cases}
q_s \quad &\textrm{\rm if}
\ s \in [0,t], 
\\
q_s {\mathds 1}_{E^{\complement}}
+ q_t 
Q_{t,s}
{\mathds 1}_{E}
&\textrm{\rm if}
\ s \in (t,T],
\end{cases}
\end{equation*}
where 
\begin{equation*}
Q_{t,s}
\coloneqq
\exp \left( \int_t^s 
\partial_y f(r,0,0) 
\dd r\right) 
{\mathcal E}_s\left(
\int_t^\cdot 
\partial_z f(r,0,0) \cdot 
\dd W_r
\right), \quad s \in [t,T].
\end{equation*}
Equivalently, this means that 
\begin{equation*}
\dd q_s^{t,E}
= q_s^{t,E} Y_s^{t,E} \dd s + q_s^{t,E} Z_s^{t,E} \cdot \dd W_s, \quad s \in [0,T],
\end{equation*}
where 
\begin{equation*}
(Y_s^{t,E},Z_s^{t,E}) \coloneqq 
\begin{cases}
(Y_s^\star,Z_s^\star) 
\quad &\textrm{\rm if}
\ s \in [0,t], 
\\
(Y_s^\star,Z_s^\star) 
{\mathds 1}_{E^{\complement}}
+
\left(\partial_y f(s,0,0), 
\partial_z f(s,0,0)\right)
{\mathds 1}_{E}
\quad &\textrm{\rm if}
\ s \in (t,T].
\end{cases}
\end{equation*}
Using the fact that $(\partial_y f(r,0,0))_{r \in [0,T]}$
and $(\partial_z f(r,0,0))_{r \in [0,T]}$ are bounded, we easily deduce that, for any $p \geq 1$, there exists a (deterministic) constant $C_p$, independent of 
$t$, such that, ${\mathbb P}$-a.s, \begin{equation*}
 {\mathbb E} \left[ \left. \sup_{s \in [t,T]} Q_{t,s}^p \right \vert  {\mathcal F}_t \right] \leq C_p,
\end{equation*}
from which we deduce that 
\begin{equation}
\label{eq:h:q:E}
\sup_{s \in [t,T]}
{\mathbb E}\left[h\left(q_s^{t,E}\right)\right]
< + \infty.
\end{equation}
In fact, we claim that
there exists 
$\delta_0>0$, possibly depending on $t$, such that, for 
${\mathbb P}(E) \leq \delta_0$,
the process $q^{t,E}$ belongs to 
${\mathcal Q}_{c_1}$. 
Indeed, we have
(using Fenchel-Legendre duality to get the last line)
\begin{equation*}
\begin{split}
{\mathcal S}(q^{t,E})
&= {\mathbb E}
\left[\int_0^T q_s^{t,E} f^\star(s,Y_s^{t,E},Z_s^{t,E}) \dd s \right]
\\
&= {\mathbb E}
\left[\int_0^t q_s f^\star(s,Y_s^\star,Z_s^\star) \dd s\right]
+ 
{\mathbb E}\left[
\int_t^T q_s {\mathds 1}_{E^{\complement}} f^\star(s,Y_s^\star,Z_s^\star) \dd s \right]
\\
&\hspace{15pt} + 
{\mathbb E} \left[
\int_t^T q_t Q_{t,s} {\mathds 1}_{E} f^\star(s,\partial_yf (s,0,0),\partial_z f(s,0,0)) \dd s \right]
\\
&\leq {\mathcal S}(q) -
{\mathbb E} \left[\int_t^T 
q_t Q_{t,s}
{\mathds 1}_E f_s^0 \dd s \right].
\end{split}
\end{equation*}
And then, using again the fact that 
$(f_s^0 = f(s,0,0))_{s \in [0,T]}$
is bounded and recalling that $q \in \mathcal{Q}_{c_1'}$ with $c_1' < c_1$,
we deduce that there exists a constant $C > 0$, such that 
\begin{equation*}
\begin{split}
{\mathcal S}(q^{t,E})  \leq {\mathcal S}(q) + C \mathbb{E}
\left[ q_t {\mathds 1}_E \right] \leq c_1' + C \mathbb{E}\left[ q_t {\mathds 1}_E \right].
\end{split}
\end{equation*}
Then for $\mathbb{P}(E)$ small enough,
$c_1' + 
C \mathbb{E}\left[ q_t {\mathds 1}_E \right]$ 
is less than or equal to $c_1$, which implies that 
$q^{t,E}$ belongs to $\mathcal{Q}_{c_1}$.
\vskip 4pt

\noindent \textit{Step 2.}
Throughout, we assume that 
$t$ is fixed (in $(0,T]$) and 
${\mathbb P}(E) \leq \delta_0$, 
with $\delta_0$ as in the first step. 
Since 
$q^{t,E} \in {\mathcal Q}_{c_1}$, 
and by optimality of $q$ on ${\mathcal Q}_{c_0}$, we deduce that 
\begin{equation}
\label{eq:proof:comparison:step2}
{\mathcal G}(q_T) 
+ {\mathbb E}
\left[ \int_0^T q_s \ell_s \dd s
\right]
- {\mathcal S}(q) 
\geq 
{\mathcal G}(q_T^{t,E})
+ {\mathbb E}
\left[
\int_0^T q_s^{t,E} \ell_s 
\dd s
\right]
-
{\mathcal S}(q^{t,E}).
\end{equation}
It is clear that 
\begin{equation}
\label{eq:proof:comparison:step2:2}
 {\mathbb E}
\left[ \int_0^T (q_s^{t,E} - q_s) \ell_s 
\dd s \right]
=
{\mathbb E}
\left[
{\mathds 1}_E
\int_t^T 
\left( 
q_t
Q_{t,s} 
-q_s 
\right)
\ell_s
\dd s
\right],
\end{equation}
and 
\begin{equation}
\label{eq:proof:comparison:step2:3}
\begin{split}
&{\mathcal S}(q^{t,E})
- {\mathcal S}(q)
\\
&= {\mathbb E}
\left[ 
{\mathds 1}_E
\int_t^T \left[ q_s^{t,E}
f^*\left(s,\partial_y f(s,0,0),
\partial_z f(s,0,0) \right) 
-
q_s
f^*\left(s,Y_s^\star,
Z_s^\star\right) 
\right] \dd s
\right]
\\
&= - {\mathbb E}
\left[ 
{\mathds 1}_E
\int_t^T \left[  q_s^{t,E}
f(s,0,0) 
+
q_s
f^*\left(s,Y_s^\star,
Z_s^\star\right) 
\right] \dd s
\right].
\end{split}
\end{equation}
We now turn to the difference between the two boundary conditions in 
\eqref{eq:proof:comparison:step2}. 
From the regularity
property 
\ref{hyp:DgD_XG} 
(for
${\mathcal G}$), we have 
\begin{equation*}
\begin{split}
{\mathcal G}(q_T^{t,E})
&= {\mathcal G}(q_T)
+ 
{\mathbb E}
\left[ \left( q_T^{t,E} - 
q_T \right) 
\delta_q \mathcal{G}(q_T)
\right] + o
\left( 
{\mathbb E}\left[ (1 + |X_T^\psi|^{2-r})
\vert q_T^{t,E} - q_T\vert
\right]
\right),
\end{split}
\end{equation*}
where $o(r)/r \rightarrow 0$ as $r$ tends to $0$ (the rate being independent of $E$).  Recalling that $q^{t,E}$ and $q$ respectively belong to $\mathcal{Q}_{c_1}$ and $\mathcal{Q}_{c_0}$,
we know from the growth Assumption \ref{eq:G-growth} on $\mathcal{G}$ and from Lemma \ref{lemma:reg-X-psi-A} that the first expectation on the right-hand side is well defined. 
Applying once again Lemma \ref{lemma:reg-X-psi-A}, we deduce that second expectation is also well defined. By definition of $q^{t,E}$, we have
\begin{equation}
\label{eq:proof:comparison:step2:4}
\begin{split}
{\mathcal G}(q_T^{t,E})
&=  {\mathcal G}(q_T)
+ 
{\mathbb E}
\left[ {\mathds 1}_E \left( q_t Q_{t,T}
- q_T \right) 
\delta_q \mathcal{G}(q_T)
\right] 
\\ 
&\hspace{15pt} + o
\left( 
{\mathbb E}\left[  (1 + |X_T^\psi|^{2-r}) 
\vert q_T^{t,E} - q_T\vert
\right]
\right).
\end{split}
\end{equation}
Putting together 
\eqref{eq:proof:comparison:step2}, 
\eqref{eq:proof:comparison:step2:2},
\eqref{eq:proof:comparison:step2:3}
and
\eqref{eq:proof:comparison:step2:4}, we obtain
\begin{equation}
    \label{eq:proof:comparison:step2:5}
\begin{split}
&{\mathbb E}
\left[ 
{\mathds 1}_E
\left( q_t Q_{t,T} - q_T
\right) \delta_q \mathcal{G}(q_T)
\right]
+ {\mathbb E}
\left[
{\mathds 1}_E
\int_t^T 
\left( q_t Q_{t,s} 
- q_s 
\right) \ell_s \dd s
\right]
\\
&\hspace{5pt}
+
{\mathbb E}
\left[ 
{\mathds 1}_E
\int_t^T \left(  q_t Q_{t,s}
f(s,0,0) 
+
q_s
f^*\left(s,Y_s^\star,
Z_s^\star\right) 
\right) \dd s
\right] \\
& \hspace{5pt} \leq o
\left( 
{\mathbb E}\left[ 
 (1 + |X_T^\psi|^{2-r}) \vert q_T^{t,E} - q_T \vert
\right]
\right).
\end{split}
\end{equation}
Now, we recall
from 
\eqref{eq:representation:process:Y:OPTN}
that \begin{equation*}
\begin{split}
q_t Y_t
&= {\mathbb E}
\left[\left. q_T 
\delta_q \mathcal{G}(q_T)
+ \int_t^T q_s \left( \ell_s - 
f^\star(s,Y_s^\star,Z_s^\star)
\right) 
\dd s  \right \vert  {\mathcal F}_t
\right].
\end{split}
\end{equation*}
And then, we can rewrite 
\eqref{eq:proof:comparison:step2:5}
as 
\begin{equation*}
\begin{split}
&{\mathbb E}
\left[ 
{\mathds 1}_E q_t
\left(  Q_{t,T}  \delta_q \mathcal{G}(q_T) - Y_t \right)
\right]
+ {\mathbb E}
\left[
{\mathds 1}_E
\int_t^T   q_t Q_{t,s} 
 \ell_s \dd s
\right]
\\
&\hspace{5pt}
+
{\mathbb E}
\left[ 
{\mathds 1}_E
\int_t^T   Q_{t,s}
f(s,0,0)  
 \dd s\right]
 \leq o
\left( 
{\mathbb E}\left[  
(1 + |X_T^\psi|^{2-r}) \vert q_T^{t,E} - q_T  \vert
\right]
\right).
\end{split}
\end{equation*}
Also, it is easy to see from 
\eqref{eq:representation:tildeY} that $\tilde Y_t$
can be represented as
\begin{equation}
\label{eq:tildeY:as:conditional:expectation}
\tilde Y_t = 
{\mathbb E}
\left[ \left. Q_{t,T}
\delta_q \mathcal{G}(q_T) + 
\int_t^T Q_{t,s} \left( \ell_s
+ f(s,0,0) \right)
\dd s
 \right \vert {\mathcal F}_t
\right],
\end{equation}
which gives
\begin{equation*}
{\mathbb E}
\left[ {\mathds 1}_E q_t \left( \tilde{Y}_t - Y_t
\right) 
\right]
\leq 
o
\left( 
{\mathbb E}\left[ 
  (1 + |X_T^\psi|^{2-r})  \vert q_T^{t,E} - q_T \vert
\right]
\right).
\end{equation*}
We notice that the expectation in the right-hand side can be rewritten as 
\begin{equation*}
{\mathbb E}\left[  (1 + |X_T^\psi|^{2-r})  
\vert q_T^{t,E} - q_T\vert
\right]={\mathbb E}\left[ 
{\mathds 1}_E  (1 + |X_T^\psi|^{2-r})
\vert q_T  - q_t Q_{t,T}
\vert
\right]. 
\end{equation*}
Writing 
$o(r) = r \eta(r)$, with 
$\eta \geq 0$
and
$\lim_{r \rightarrow 0} 
\eta(r) =0$,  and letting $R_{t,T}: = (1 + |X_T^\psi|^{2-r}) 
\vert q_T  - q_t Q_{t,T}
\vert$, 
we get 
\begin{equation}
\label{eq:decomposition:E}
{\mathbb E}
\left[ 
{\mathds 1}_E
\left\{ 
q_t \left( \tilde Y_t - Y_t
\right) 
- 
R_{t,T} 
\eta
\left(
{\mathbb E}\left[ 
{\mathds 1}_E
R_{t,T}
\right]
\right)
\right\}
\right]\leq 0.
\end{equation}
The above is true for a given $t \in (0,T]$, for any event $E \in {\mathcal F}_t$ satisfying 
${\mathbb P}(E) \leq \delta_0$. The function $\eta$ is independent of $E$.
Moreover, we notice from 
\eqref{eq:h:q:E}
(with $E=\Omega$ therein)
and Lemma 
\ref{lemma:reg-X-psi-A} that 
${\mathbb E}[R_{t,T}] < + \infty$. 
\vskip 4pt

\noindent \textit{Step 3.}
We now argue by contradiction to prove that $Y_t \geq \tilde Y_t$.
Assume indeed that, for some 
$\varepsilon >0$, 
\begin{equation*}
\pi \coloneqq {\mathbb P}
\left( \left\{ q_t (\tilde Y_t - Y_t) \geq \varepsilon \right\} \right) >0.
\end{equation*}
Then, we can find $A>0$ such that the event 
\begin{equation*}
E_0 \coloneqq 
\left\{ q_t (\tilde Y_t - Y_t) \geq \varepsilon \right\} 
\cap 
\left\{
  R_{t,T} \leq A
\right\}
\end{equation*}
satisfies 
${\mathbb P}(E_0) 
\geq \pi/2$. 
By a standard uniform integrability argument (using the fact that 
${\mathbb E}[R_{t,T}] < + \infty$), notice also that there exists 
$\delta >0$ such that  
\begin{equation*}
{\mathbb P}(E) \leq \delta 
\Rightarrow 
\eta
\left(
{\mathbb E}\left[ 
{\mathds 1}_E
R_{t,T} \right]
\right) \leq \frac{\varepsilon}{2A}.
\end{equation*}
Decompose now $E_0$ as
\begin{equation*}
E_0 = \cup_{k \in {\mathbb N}}
\left( 
E_0 \cap \{W_t \in I_k\} \right),
\end{equation*}
where $(I_k)_{k \in {\mathbb N}}$ is partition of ${\mathbb R}^d$ into Borel subsets such that 
${\mathbb P}(\{W_t \in I_k\})< \delta \wedge \delta_0$ (with $\delta_0$ as in the first step), for each $k \in {\mathbb N}$. 

Applying 
\eqref{eq:decomposition:E}
with $E = E_{0,k} \coloneqq E_0 \cap \{W_t \in I_k\}$ for a given $k \in {\mathbb N}$, we obtain 
\begin{equation*}
\begin{split}
0 &\geq 
{\mathbb E}
\left[ 
{\mathds 1}_{E_{0,k}}
\left( 
q_t \left( \tilde Y_t - Y_t
\right) 
- 
R_{t,T}
\eta
\left(
{\mathbb E}\left[ 
{\mathds 1}_{E_{0,k}}
\vert q_T  - q_t Q_{t,T}
\vert
\right]
\right)
\right)
\right]
\\
&\geq 
{\mathbb E}
\left[ 
{\mathds 1}_{E_{0,k}}
\left(
\varepsilon - A \frac{\varepsilon}{2A}
\right) 
\right]
= \frac{\varepsilon}2 {\mathbb P}(E_{0,k}). 
\end{split}
\end{equation*}
This proves that 
${\mathbb P}(E_{0,k})=0$, for each $k \in {\mathbb N}$, and then 
${\mathbb P}(E_0)=0$, which contradicts the fact that 
$\pi>0$. We deduce that 
\begin{equation*}
\forall \varepsilon >0, 
\quad 
{\mathbb P}
\left(\left\{
q_t (\tilde Y_t - Y_t) \geq \varepsilon \right\}
\right) =0,
\end{equation*}
i.e., 
\begin{equation*}
{\mathbb P}
\left(\left\{
q_t (\tilde Y_t - Y_t) \leq 0\right\}
\right) =1.
\end{equation*}
Since ${\mathbb P}(\{q_t >0\})=1$, we deduce that ${\mathbb P}(\{Y_t \geq \tilde Y_t\})$. This holds true for any 
$t \in (0,T]$. By continuity of the two processes $Y$ and $\tilde Y$, we deduce that ${\mathbb P}$-a.s., 
for any $t \in [0,T]$, $Y_t \geq \tilde Y_t$, which proves the first claim in the statement. 
\vskip 4pt

\noindent \textit{Step 4.} It remains to establish that, for any other $\tilde q \in {\mathcal Q}_{c_1}$ (with decomposition $(\tilde Y^\star,\tilde Z^\star)$ in \eqref{eq:q}), the  family
$({\tilde q}_{{\tau}} Y^-_{{\tau}})_{{\tau}}$, with 
    $\tau$ running over the set of $[0,T]$-valued ${\mathbb F}$-stopping, is uniformly integrable. From the comparison principle established in Step 3, we first notice that 
\begin{equation*}
    0\leq Y_t^- \leq \tilde{Y}_t^-,
\end{equation*}
for any $t \in [0,T]$. 
Then to show the desired property, we provide a lower bound for the process $\tilde Y$ appearing on the right-hand side (recall 
\eqref{eq:representation:tildeY} for its definition, and 
\eqref{eq:tildeY:as:conditional:expectation} for its representation). We notice from the strong convexity of $\ell$, see \ref{hyp:L}, that 
$\ell + f(0,0)$ is lower bounded. There exists a constant $c_0 \in {\mathbb R}$ such that $\tilde{Y} \geq \tilde{Y}^0$, where the process $\tilde{Y}^0$ is defined as 
\begin{equation*}
\tilde{Y}_t^0 = 
\frac{1}{q_t^0}{\mathbb E}
\left[\left. q_{T}^0
\delta_q \mathcal{G}(q_T) \right \vert \mathcal{F}_t \right] + c_0 (T-t),
\end{equation*}
and where $q^0$ is the solution to 
\eqref{eq:barq:0}. Then, the following inequality holds
\begin{equation*}
    0\leq Y_t^- \leq \tilde{Y}_t^- \leq (\tilde{Y}_t^0)^-.
\end{equation*}

Now, let $E$ be an element of $\mathcal{F}_T$ and  $\tau$ a stopping time with values in $[0,T]$. For $\tilde q \in \mathcal{Q}_{c_1}$ as above, we have
\begin{equation} \label{ineq:1E-q-Y-tilde-1E-q-Y-0}
    0\leq \mathbb{E} \left[\mathds{1}_{E} \tilde q_{ \tau} Y_{ \tau}^- \right] \leq \mathbb{E} \left[\mathds{1}_{E} \tilde q_{ \tau}(\tilde{Y}_{ \tau}^0)^- \right],
\end{equation}
and we are left to establish that the right-hand side is finite for any $E \in {\mathcal F}_T$, and can be made small with 
${\mathbb P}(E)$, uniformly with respect to $\tau$. We have
\begin{equation*}
\tilde q_{\tau} \tilde{Y}_{\tau}^0 = 
\frac{\tilde q_{\tau}}{q_{\tau}^0}{\mathbb E}
\left[\left. q_{T}^0
\delta_q \mathcal{G}(q_T) \right \vert \mathcal{F}_{\tau} \right] + c_0 \tilde q_{\tau} (T-{\tau}).
\end{equation*}
We notice that the process $(\tilde{q}^{0,\tau}_t)_{t \in [0,T]}$ defined by $\tilde q_t^{0,\tau}\coloneqq \tilde q_t,$ for 
$t \in [0,\tau]$, and 
$\tilde q_t^{0,\tau} \coloneqq\tilde q_\tau q_t^0 /q^0_\tau$,
for $t \in [\tau,T]$, lies in $\mathcal{Q}$. It is indeed the solution to 
\begin{equation*}
    \dd \tilde{q}_t^{0,\tau} = \tilde{q}^{0,\tau}_t \tilde{Y}_t^{\star,0,\tau} \dd t + \tilde{q}^{0,\tau}_t \tilde{Z}_t^{\star,0,\tau} \cdot \dd W_t,
    \quad t \in [0,T]; \quad  \tilde{q}_0^{0,\tau} = 1,
\end{equation*}
where 
\begin{equation*}
(\tilde{Y}_t^{\star,0,\tau},\tilde{Z}_t^{\star,0,\tau}) \coloneqq
\left\{
\begin{array}{ll}
(\tilde Y_t^\star,\tilde Z_t^\star)
\quad
&\textrm{\rm if}
 \quad t \leq \tau, 
 \\
 (\partial_y f(t,0,0), \partial_z  f(t,0,0))
 \quad 
 &\textrm{\rm if}
 \quad
 t \in (\tau,T]. 
 \end{array}
 \right.
\end{equation*}
Following the first step, we deduce that 
\begin{equation}
\label{eq:bound:sup:S}
\sup_{\tau}
{\mathcal S}\left(\tilde q^{0,\tau}\right) < + \infty.
\end{equation}
Moreover, 
\begin{equation} \label{eq:1E-q-Y-tilde-0}
    \mathbb{E}\left[ \mathds{1}_{E} \tilde q_{\tau} \left(\tilde{Y}_{\tau}^{0}\right)^- \right] \leq {\mathbb E}
\left[  {\mathbb P}(E \vert {\mathcal F}_\tau) \tilde{q}_{T}^{0,\tau}
\left( \delta_q \mathcal{G}(q_T) \right)^- \right] + \vert c_0 \vert T \mathbb{E}[\mathds{1}_{E} \tilde q_{\tau}],
\end{equation}
As for the second term on the right-hand side, we 
know from Lemma \ref{lemma:representation-q} that ${\mathbb E}[\tilde q_T^*]< +\infty$. Therefore, the second term on the right-hand side is finite and can be made small with ${\mathbb P}(E)$, uniformly with respect to $\tau$.
We now address the first term on the right-hand side in \eqref{eq:1E-q-Y-tilde-0}. By the lower estimate \ref{eq:G-growth} on $\delta_q \mathcal{G}$, we have  
\begin{equation*}
    {\mathbb E}
\left[ 
{\mathbb P}(E \vert {\mathcal F}_{\tau})
\tilde{q}^{0,\tau}_{T}
\left( 
\delta_q \mathcal{G}(q_T) \right)^-  \right] \leq L {\mathbb E}
\left[
{\mathbb P}(E \vert {\mathcal F}_\tau)
\tilde{q}^{0,\tau}_{T}
\left(1+|X_T| + \mathbb{E}\left[q_T|X_T|^{2-r} \right]\right) \right].
\end{equation*}
By Lemma \ref{lemma:reg-X-psi-A}, we know that $\mathbb{E}\left[q_T|X_T|^{2-r} \right] < +\infty$ . We deduce that there exists $C>0$ such that 
\begin{equation}
    {\mathbb E}
\left[\mathds{1}_{E} \tilde{q}^{0,\tau}_{T}
\left( 
\delta_q \mathcal{G}(q_T) \right)^-  \right] \leq C    {\mathbb E}
\left[
{\mathbb P}(E \vert {\mathcal F}_\tau)
\tilde{q}^{0,\tau}_{T}
\left( 1 +
|X_T| \right) \right]. \label{ineq:1E-q-tilde-G}
\end{equation}
Following 
\eqref{eq:1E-q-Y-tilde-0}, we already know
that 
${\mathbb E}[{\mathbb P}(E \vert {\mathcal F}_{\tau}) \tilde{q}_T^{0,\tau}]= {\mathbb E}[{\mathds 1}_E \tilde q_{\tau}]$ tends to $0$ as ${\mathbb P}(E)$ tends to $0$. The remainder of the proof is devoted to establishing the same result for
$\mathbb{E}[\mathbb{P}(E \mid \mathcal{F}_{\tau}) \tilde{q}_T^{0,\tau} \lvert X_T \rvert]$.
To this end, we distinguish between the two cases $r = 0$ and $r = 1$.

If $r=0$, then by Cauchy-Schwarz inequality and Lemma 
\ref{lemma:reg-X-psi-A},
we have 
\begin{align}
    {\mathbb E}
\left[ 
{\mathbb P}(E \vert {\mathcal F}_\tau)
\tilde{q}_{T}^{0,\tau}
|X_T| \right]  & \leq  {\mathbb E}
\left[\tilde{q}_{T}^{0,\tau}
{\mathbb P}\left(E \vert {\mathcal F}_{\tau}
\right)
 \right]^{1/2} {\mathbb E}
\left[\tilde{q}^{0,\tau}_{T}
|X_T|^2 \right]^{1/2} \nonumber 
\\
& \leq   C 
{\mathbb E}
\left[\tilde{q}_{T}^{0,\tau}
{\mathbb P}\left(E \vert {\mathcal F}_{\tau}
\right)
 \right]^{1/2}
  \left(1 + \mathcal{S}\left(\tilde q^{0,\tau}\right) + \mathcal{S}^\star(\bar{\psi})\right)^{1/2}
 \nonumber 
 \\
& \leq  C 
{\mathbb E}
\left[\tilde{q}_{\tau} {\mathds 1}_E
 \right]^{1/2} \left(1 + \mathcal{S}\left(\tilde q^{0,\tau}\right) + \mathcal{S}^\star(\bar{\psi})\right)^{1/2}, \label{ineq:1-e-q-tilde-x}
\end{align}
where we used the inequality 
${\mathbb E}[\tilde q_T^{0,\tau} \vert {\mathcal F}_\tau] \leq C \tilde q_\tau$
to get the last line. 
The first term on the last line can be handled as the last term on 
\eqref{eq:1E-q-Y-tilde-0}. Combining \eqref{ineq:1E-q-Y-tilde-1E-q-Y-0}, 
\eqref{eq:bound:sup:S},
\eqref{eq:1E-q-Y-tilde-0}
and \eqref{ineq:1-e-q-tilde-x}, we easily complete the proof. 

If $r=1$, we return back to \eqref{ineq:1E-q-tilde-G}. In comparison with \eqref{ineq:1-e-q-tilde-x}, the only difficult comes from the stochastic integral in the definition of $X$, as its integrand grows up linearly in $\psi$, see \ref{hyp:sigma}.
By Girsanov theorem, 
we write
\begin{align}
{\mathbb E}
\left[ {\mathbb P}(E \vert {\mathcal F}_\tau) \tilde q_T^{0,\tau}
\left\vert 
\int_0^T \sigma(t,\psi_t)  \dd W_t
\right\vert
\right]
&\le {\mathbb E}
\left[ {\mathbb P}(E \vert {\mathcal F}_\tau)  \tilde q_T^{0,\tau}
\left\vert 
\int_0^T \sigma(t,\psi_t)  \dd \tilde W_t^{0,\tau}
\right\vert
\right] \nonumber
\\
&\hspace{15pt} 
+ C {\mathbb E}
\left[ {\mathbb P}(E \vert {\mathcal F}_\tau)   q_\tau
\left\vert 
\int_0^\tau \sigma(t,\psi_t)  Z_t^\star \dd t
\right\vert
\right] \label{eq:lem32:r=1}
\\
&\hspace{15pt} + {\mathbb E}
\left[ {\mathbb P}(E \vert {\mathcal F}_\tau)  \tilde q_T^{0,\tau}
\left\vert 
\int_\tau^T \sigma(t,\psi_t)  \partial_z 
f(t,0,0) \dd t
\right\vert
\right],
\nonumber
\end{align}
where $\tilde W^{0,\tau}$
is a Brownian motion under 
${\mathcal E}_T(\int_0^{\cdot} \tilde Z_s^{\star,0,\tau}
\dd s )$, and where $C$ is a constant independent of $\tau$ (which arises because ${\mathbb E}[\tilde q_T^{0,\tau}]$ may not be equal to 1). 
We first provide an upper bound for the  first term on the right-hand side.
By Young's inequality, observe that, for any $\varepsilon \in (0,1]$, 
\begin{equation*}
\begin{split}
{\mathbb E}
\left[ {\mathbb P}(E \vert {\mathcal F}_\tau)  \tilde q_T^{0,\tau}
\left\vert 
\int_0^T \sigma(t,\psi_t)  \dd \tilde W_t^{0,\tau}
\right\vert
\right]
\leq &\frac1{\varepsilon}
{\mathbb E}
\left[ {\mathbb P}(E \vert {\mathcal F}_\tau)
\tilde q_T^{0,\tau}
\right]
\\
& +  \varepsilon 
{\mathbb E}\left[ 
\tilde q_T^{0,\tau}
\left\vert \int_0^T \sigma(t,\psi_t) 
\dd \tilde W_t^{0,\tau}
\right\vert^2
\right]. 
\end{split}
\end{equation*}
Here, 
\begin{equation*}
{\mathbb E}
\left[ {\mathbb P}(E \vert {\mathcal F}_\tau)
\tilde q_T^{0,\tau}
\right]
\leq 
{\mathbb E}
\left[    {\mathbb P}(E \vert {\mathcal F}_\tau)
 q_{\tau}
\right] = {\mathbb E}
\left[    {\mathds 1}_E
 q_{\tau}
\right] \leq 
{\mathbb E}
\left[{\mathds 1}_E 
q_T^*
\right], 
\end{equation*}
and
\begin{equation*}
\begin{split}
{\mathbb E}
\left[   \tilde q_T^{0,\tau}
\left\vert 
\int_0^T \sigma(t,\psi_t)  \dd \tilde W_t^{0,\tau}
\right\vert^2
\right]
 &\leq C {\mathbb E}
\left[   \tilde q_T^{0,\tau} 
\int_0^T 
\left\vert \sigma(t,\psi_t) \right\vert^2 \dd t
\right]
\\
&\leq  C {\mathbb E}
\left[   \tilde q_T^{0,\tau} 
\int_0^T 
\left( 1 + 
\left\vert \psi_t \right\vert^2 
\right) \dd t
\right]. 
\end{split}
\end{equation*}
By \eqref{ineq:S-S-star} and \eqref{eq:bound:sup:S}, the above right-hand side is finite, uniformly with respect to $\tau$. By combining the last three displays, we easily deduce that the first-term on the right-hand side of 
\eqref{eq:lem32:r=1} tends to $0$ as ${\mathbb P}(E)$ tends to $0$, uniformly with respect to $\tau$. Using similar arguments together with the fact that $(\partial_z f(t,0,0))_{t \in [0,T]}$ is bounded, we can reach the same conclusion for the third term on the right-hand side of \eqref{eq:lem32:r=1}. It remains to handle the second term on the right-hand side of 
\eqref{eq:lem32:r=1}. To do so, it suffices to notice that 
\begin{equation}
\label{eq:lem32:last:step}
\begin{split}
{\mathbb E}
\left[ {\mathbb P}(E \vert {\mathcal F}_\tau) q_\tau 
\left\vert 
\int_0^\tau 
\sigma(t,\psi_t) 
Z_t^\star \dd t 
\right\vert 
\right]
&\leq C 
{\mathbb E}
\left[ {\mathbb P}(E \vert {\mathcal F}_\tau) q_T 
\left\vert 
\int_0^\tau 
\sigma(t,\psi_t) 
Z_t^\star \dd t 
\right\vert 
\right]
\\
&\leq 
C {\mathbb E}
\left[ \sup_{t \in [0,T]} {\mathbb P}(E \vert {\mathcal F}_t) q_T 
\int_0^T 
\left( 1 + \vert \psi_t \vert \right) 
\vert 
Z_t^\star \vert \dd t 
\right].
\end{split}
\end{equation}
By Doob's inequality, 
$\sup_{t \in [0,T]} {\mathbb P}(E \vert {\mathcal F}_t)$ tends to $0$ in probability as 
${\mathbb P}(E)$ tends to $0$. Moreover, 
\begin{equation*}
\begin{split}
&{\mathbb E}\left[ 
q_T \int_0^T
\left( 1 + 
\vert \psi_t \vert 
\right) \vert Z_t^\star \vert \dd t
\right]
\\
&\leq 
C {\mathbb E}\left[ 
q_T \int_0^T
\left( 1 + 
\vert Z_t^\star \vert^2 
\right)
 \dd t
\right]
+ C {\mathbb E}\left[ 
q_T \int_0^T
\left( 1 + 
\vert \psi_t \vert^2 
\right)
 \dd t
\right].
\end{split}
\end{equation*}
Thanks to \eqref{ineq:S-S-star}, the second term on the right-hand side is finite. By \eqref{eq:entrop:bounded:Q}, the first term is also finite. This proves that the left-hand side on 
\eqref{eq:lem32:last:step} tends to $0$ as 
${\mathbb P}(E)$ tends to $0$, uniformly in $\tau$. 
\end{proof}
\color{black}

\subsubsection{Sufficient condition}
\label{subsubse:sufficient:Nature}

We now turn to the proof of the sufficient condition, i.e. the second assertion in the statement of Theorem \ref{thm:sto-max-princ-dual}. We recall 
that $\bar \psi \in {\mathcal A}_{c_2}$ is given. Also, we use the same abbreviated notations as in \eqref{eq:saddle:simplified:notation}.

\begin{lemma} \label{lemma:sufficient}
Assume that 
there exists a triple 
$(q,Y,Z) \in {\mathscr Q}$ satisfying the first order condition \eqref{optim:condition-dual}. Then, 
$q$ is the unique maximizer of the mapping 
  $q' \in {\mathcal Q} \mapsto \mathcal{J}(q',\bar \psi)$.
\end{lemma}

\begin{proof} 
Throughout the proof, we omit the dependence on $\bar \psi$ in the various notations. 
For instance, 
we just write 
${\mathcal J}(q')$ for 
${\mathcal J}(q',\bar \psi)$.

Moreover, in addition to $(q,Y^\star,Z^\star)$,     we
    let  $(\tilde{q},\tilde{Y}^\star,\tilde{Z}^\star)$  be another arbitrary tuple of state and control in $\mathscr{Q}$, and then denote $\delta q \coloneqq \tilde{q} - q$. By definition of $\mathcal{J}$ and concavity of $\mathcal{G}$ with respect to its second variable $q$ (assumption \eqref {eq:G-concave-convex}), we have
    \begin{equation}
\label{eq:Itildeq-Iq}
    \begin{split}
        \mathcal{J}(\tilde{q}) - \mathcal{J}(q) & = \mathcal{R}(\tilde{q}) - \mathcal{R}(q) - \left(\mathcal{S}(\tilde{q}) -\mathcal{S}(q) \right) 
        \\
        & \leq \mathbb{E}\left[ \delta q_T \delta_q \mathcal{G}\left(q_T \right) + \int_0^T \delta q_t \ell_t \dd t \right] - \left(\mathcal{S}(\tilde{q}) -\mathcal{S}(q) \right).
        \end{split}
    \end{equation}
Notice that the right-hand side is finite, which can be shown in the same way as in \eqref{eq:firststep:lem:19:a}, using the duality inequality \eqref{ineq:S-S-star}, Lemma \ref{lemma:reg-X-psi-A} and the bound $\mathcal{S}(\tilde{q}) + \mathcal{S}(q) + \mathcal{S}^\star(\bar{\psi}) < +\infty$.

\vskip 4pt

\noindent \textit{Step 1: localization.}
In this first step, we proceed as in the analysis of the right-hand side on \eqref{eq:deriv-F:0} and   expand $(\delta q_t Y_t)_{t \in [0,T]}$ by means of It\^o's formula. 
Recalling \eqref{eq:adjoint-U-V}, we have
\begin{equation}
\label{eq:Ito:deltaqY:lemma:21}
\begin{split}
{\mathrm d} \left[ \delta q_t Y_t \right] &= 
\delta q_t
\left( 
- Y_t^\star Y_t - 
Z_t^\star \cdot Z_t
+ f^\star(t,Y_t^\star,Z_t^\star) - \ell_t \right) 
{\mathrm d} t
\\
&\hspace{15pt}
+ Y_t \left( 
\tilde q_t \tilde Y_t^\star 
-
q_t Y_t^\star
\right) 
{\mathrm d} t
+ Z_t 
\cdot \left(  
\tilde q_t \tilde Z_t^\star 
-
q_t Z_t^\star
\right){\mathrm d} t
\\
&\hspace{15pt} + \delta q_t  Z_t \cdot 
{\mathrm d} W_t + Y_t \left( 
\tilde q_t 
\tilde Z_t^\star 
- 
q_t Z_t^\star
\right) \cdot {\mathrm d} W_t, \quad t \in [0,T].
\end{split}
\end{equation}
For a given $A>0$, we then introduce the following 
stopping time 
\begin{equation*}
\tau_A \coloneqq 
\inf
\left\{ t \in [0,T], \; \vert Y_t \vert 
+
q_t + 1/q_t 
+ \tilde q_t 
+ \int_0^{t} \vert Z_s \vert^2 \dd s
\geq A
\right\}, 
\end{equation*}
with the standard convention that 
$\tau_A=+\infty$ if the set on the right-hand side is empty. 
With this definition in hand, we notice that, for 
$t \in [0,\tau_A]$
\begin{equation*}
\begin{split}
\vert \delta q_t \vert & \leq A \leq A^2 q_t,
\\
\tilde q_t &\leq A \leq A^2 q_t.
\end{split}
\end{equation*}
This implies in particular that, for a constant $C_A$ depending on $A$,
\begin{equation*}
\begin{split}
&{\mathbb E} \left[\int_0^{\tau_A}
\left\vert \delta q_t \left( - Y_t^\star Y_t - Z_t^\star \cdot Z_t + f^\star(t,Y_t^\star,Z_t^\star) - \ell_t \right) \right\vert {\mathrm d}t \right]
\\
&\leq C_A {\mathbb E}\left[
\int_0^{\tau_A}
q_t 
\left( 
\vert    Y_t^\star Y_t \vert  + \vert Z_t^\star \cdot Z_t \vert + 
\vert f^\star(t,Y_t^\star,Z_t^\star) \vert + \vert \ell_t \vert 
\right) {\mathrm d}t \right]
\\
&\leq 
C_A \left( 1+
{\mathbb E}\left[
\int_0^{\tau_A}
q_t 
\left( \vert    Y_t^\star  \vert 
+ \vert Z_t^\star \vert^2 +
\vert f^\star(t,Y_t^\star,Z_t^\star) 
\vert 
+ \vert  \ell_t 
\vert \right) {\mathrm d}t\right]\right)
\\
&\leq 
C_A
\left( 1 + 
{\mathbb E}\left[
\int_0^{T}
q_t 
\left( \vert    Y_t^\star  \vert +
\vert f^\star(t,Y_t^\star,Z_t^\star) 
\vert 
+ \vert  \ell_t 
\vert \right) {\mathrm d}t\right]
\right) 
< + \infty,
\end{split}
\end{equation*}
with the
third line following from the condition 
$\vert Y_t \vert + \int_0^t \vert Z_s \vert^2 \dd s \leq A$, and the last line following from \eqref{ineq:duality-fstar}, and from the fact that ${\mathcal S}(q)$ and ${\mathcal S}^\star(\bar \psi)$ are finite.

Similarly, 
\begin{equation*}
\begin{split}
&{\mathbb E} \left[
\int_0^{\tau_A}
\left\vert
Y_t 
\left(
\tilde q_t 
\tilde Y_t^\star 
- q_t Y_t^\star
\right) 
\right\vert
{\mathrm d} t \right]
\leq 
C_A
{\mathbb E} \left[
\int_0^T 
q_t
\left( 
\vert 
\tilde Y_t^\star 
\vert 
+
\vert 
Y_t^\star 
\vert 
\right) 
{\mathrm d}t \right]
< + \infty,
\\
&{\mathbb E} \left[
\int_0^{\tau_A}
\left\vert
Z_t
\cdot 
\left(
\tilde q_t 
\tilde Z_t^\star 
- q_t Z_t^\star
\right) 
\right\vert
{\mathrm d} t \right]
\leq 
C_A
\left( 1+
{\mathbb E} \left[
\int_0^{T}
q_t
\left(
 \vert
\tilde Z_t^\star
\vert^2 +
\vert  Z_t^\star
\vert^2 
\right) 
{\mathrm d} t\right]
\right) < + \infty.
\end{split}
\end{equation*}
Back to 
\eqref{eq:Ito:deltaqY:lemma:21}, this proves that the terms on the first and second lines of the right-hand side,
when integrated between $0$ and $\tau_A$, have a finite expectation.

It remains to check in a similar way that the stochastic integrals on the third line of 
\eqref{eq:Ito:deltaqY:lemma:21} have a zero expectation, when they are integrated between $0$ and $\tau_A$. To do so, we notice that
\begin{equation*}
    \begin{split}
&{\mathbb E} \left[\int_0^{\tau_A}
\vert \delta q_t \vert^2 \vert Z_t \vert^2 
{\mathrm d} t \right] 
\leq C_A 
{\mathbb E}\left[
\int_0^{\tau_A} q_t \vert Z_t \vert^2 
{\mathrm d} t\right] 
< +\infty,
\\
&{\mathbb E}\left[
\int_0^{\tau_A}
\vert Y_t \vert^2 
\vert \tilde q_t \tilde Z_t^\star- q_t Z_t^\star \vert^2 
{\mathrm d} t\right] 
\leq C_A 
{\mathbb E}\left[
\int_0^T \left( \tilde q_t \vert \tilde Z_t^\star 
\vert^2 
+
q_t 
\vert Z^\star_t \vert^2 
\right) {\mathrm d} t \right] 
< +\infty. 
    \end{split}
\end{equation*}
Therefore, we obtain 
\begin{equation*}
\begin{split}
&{\mathbb E}
\left[ 
\delta q_{T \wedge \tau_A}
Y_{T \wedge \tau_A} 
+
\int_0^{T \wedge \tau_A}
\delta q_t \ell_t 
{\mathrm d}t
\right]
\\
&= 
{\mathbb E} \left[
\int_0^{T \wedge \tau_A}
\delta q_t
\left(  - Y_t^\star Y_t - Z_t^\star \cdot Z_t + f^\star(t,Y_t^\star,Z_t^\star)\right) {\mathrm d} t
\right]
\\
&\hspace{15pt} +
{\mathbb E} \left[
\int_0^{T \wedge \tau_A}
\left(Y_t \left( 
\tilde q_t \tilde Y_t^\star 
- q_t Y_t^\star \right) 
+ Z_t \cdot 
\left( \tilde q_t \tilde Z_t^\star 
- q_t Z_t^\star\right)
\right)
{\mathrm d} t
\right].
\end{split}
\end{equation*}
Introduce now the notations 
$(\tilde f^\star_t 
= \tilde f^\star(t,\tilde Y_t^\star,\tilde Z_t^\star))_{t \in [0,T]}$
and
$(f^\star_t 
= f^\star(t,Y_t^\star,Z_t^\star))_{t \in [0,T]}$ and deduce that 
\begin{equation}
\label{eq:sign:cost:stopped:at:A:bis:bis}
\begin{split}
&{\mathbb E}
\left[ 
\delta q_{T \wedge \tau_A}
Y_{T \wedge \tau_A} 
+
\int_0^{T \wedge \tau_A}
\delta q_t \ell_t 
{\mathrm d}t
-
\int_0^{T \wedge \tau_A}
\left(
\tilde q_t 
\tilde f_t^\star 
- q_t f_t^\star 
\right) 
{\mathrm d} t
\right]\\
  &=  - \mathbb{E}\left[ \int_0^{T \wedge \tau_A} \left( \tilde{q}_t ( \tilde{f}^\star_t - f^\star_t - (\tilde{Y}^\star_t - Y^\star_t )  Y_t - ( \tilde{Z}^\star_t - Z^\star_t)\cdot Z_t) \right)\dd t  \right]. 
\end{split}
\end{equation}   
Then, by the first order condition \eqref{optim:condition-dual} and the (strict) joint convexity of $f^\star$, we deduce that the right-hand side is non-positive, 
i.e., for any $A >0$, 
\begin{equation}
\label{eq:sign:cost:stopped:at:A}
\begin{split}
{\mathbb E}
\left[ 
\delta q_{T \wedge \tau_A}
Y_{T \wedge \tau_A}
 + 
 \int_0^{T \wedge \tau_A} \delta q_t \ell_t \dd t  
 -
 \int_0^{T \wedge \tau_A}
 \left( 
 \tilde q_t \tilde f_t^\star - q_t f_t^\star \right) 
 \dd t \right]
 \leq 0.
\end{split}
\end{equation}
The key step in the rest of the proof is to let $A$ tend to $+\infty$ on the left-hand side of \eqref{eq:sign:cost:stopped:at:A}. This requires some extra care due to 
the rather weak integrability properties of 
$\tilde q$ and $Y$.
In order to proceed, we write the integrand in the form 
\begin{equation}
\label{eq:T1A:T2A:T3A:T4A}
\begin{split}
&\delta q_{T \wedge \tau_A}
Y_{T \wedge \tau_A}
 + 
 \int_0^{T \wedge \tau_A} \delta q_t \ell_t \dd t  
 -
 \int_0^{T \wedge \tau_A}
 \left( 
 \tilde q_t \tilde f_t^\star - q_t f_t^\star \right) \dd t
\\
&=\tilde q_{T \wedge \tau_A}
Y_{T \wedge \tau_A}
- 
q_{T \wedge \tau_A}
Y_{T \wedge \tau_A}
 + 
 \int_0^{T \wedge \tau_A} \delta q_t \ell_t \dd t  
 -
 \int_0^{T \wedge \tau_A}
 \left( 
 \tilde q_t \tilde f_t^\star - q_t f_t^\star \right) \dd t
 \\
 &= : T^1(A) + T^2(A) + 
 T^3(A) + T^4(A). 
 \phantom{\Bigr)}
\end{split}
\end{equation}

\noindent \textit{Step 2: limit $A \to +\infty$ in $T^2(A)$, $T^3(A)$ and $T^4(A)$.}
We claim that each of the three families of random variables $(T^2(A))_{A >0}$, $(T^3(A))_{A >0}$ and $(T^4(A))_{A >0}$ is uniformly integrable. For
$(T^2(A))_{A >0}$, we observe that there exists a constant $C \geq 0$
(whose value may change from line to line) such that, for any event $E \in {\mathcal F}_T$ and any constant $A' >0$,
\begin{equation*}
\begin{split}
{\mathbb E}
\left[ \vert
T^2(A)
\vert 
{\mathds 1}_E
\right]
& \leq 
A' {\mathbb E}
\left[   q_{T \wedge \tau_A} {\mathds 1}_E
\right]
+
{\mathbb E}\left[q_{T \wedge \tau_A}
{\mathds 1}_{\{ 
\vert Y_{T \wedge \tau_A} \vert
\geq A'
\}}
\vert Y_{T \wedge \tau_A} \vert
\right] \\
&  \leq 
A' {\mathbb E}
\left[   q_{T \wedge \tau_A} {\mathds 1}_E
\right]
+
C {\mathbb E}^{\mathbb{Q}}\left[
{\mathds 1}_{\{ 
\vert Y_{T \wedge \tau_A} \vert
\geq A'
\}}
\vert Y_{T \wedge \tau_A} \vert
\right],
    \end{split}
\end{equation*}
 where
${\mathbb Q} \coloneqq {\mathcal E}_T(\int_0^\cdot 
Z_t^\star \cdot \dd W_t) {\mathbb P}$. 
Since $Y \in D({\mathbb F},{\mathbb Q})$, the family $(Y_{T \wedge \tau_A})_{A>0}$
is uniformly integrable under ${\mathbb Q}$, which shows that the last term tends to $0$ as $A'$ tends to $+ \infty$, uniformly in $A>0$.
Therefore, to 
establish the 
uniformly integrability of the family $(T^2(A))_{A>0}$, it suffices to recall
that 
${\mathbb E}[q_T^*]< + \infty$ (see 
Lemma \ref{lemma:representation-q}), which implies in particular that 
\begin{equation*}
\label{eq:EQPEgivenFTA}
\lim_{{\mathbb P}(E) \to 0}
\sup_{A>0}
{\mathbb E}
\left[   q_{T \wedge \tau_A} {\mathds 1}_E
\right] = 0.
\end{equation*}

We now establish the 
uniform integrability of $(T^3(A))_{A>0}$. By a standard domination argument, it suffices to notice that 
\begin{equation*}
{\mathbb E}
\left[ 
\int_0^T \left\vert \delta q_t \ell_t \right\vert \dd t\right]
\leq C\left( 1+ {\mathbb E}
\left[ 
\int_0^T \left( 
q_t + \tilde q_t \right) \vert \bar{\psi}_t \vert^2 \dd t\right] \right)
<+\infty,
\end{equation*}
with the last inequality following from the fact that
$q$ and $\tilde q$ belong to 
${\mathcal Q}$
and 
$\bar \psi$ to ${\mathcal A}_{c_2}$ (together with the duality inequality 
\eqref{ineq:S-S-star}).

We handle $(T^4(A))_{A>0}$ in the same way. Indeed, by the same argument as in \eqref{ineq:f-star-abs-finite}, we have
\begin{equation*}
{\mathbb E} \left[
\int_0^T \left(\tilde q_t \vert \tilde f_t^\star \vert 
+ q_t \vert f_t^\star \vert\right) \dd t \right] < + \infty.
\end{equation*}

Combining the uniform integrability properties of $(T^2(A))_{A>0}$, 
$(T^3(A))_{A>0}$ and 
$(T^4(A))_{A >0}$ together with the time continuity of the processes
appearing in \eqref{eq:T1A:T2A:T3A:T4A}, we deduce that 
\begin{equation}
\label{eq:T2A:T3A:T4A}
\begin{split}
&\lim_{A \rightarrow + \infty}
{\mathbb E}\left[ T^2(A) + T^3(A) + 
T^4(A) \right]
\\
&= 
{\mathbb E} \left[
- q_T Y_T + \int_0^T \delta q_t \ell_t \dd t - 
\int_0^T \left( \tilde q_t 
\tilde f_t^\star - 
q_t f_t^\star \right) 
\dd t 
\right].
\end{split}
\end{equation}

\noindent \textit{Step 3: limit $A \to +\infty$ in $T^1(A)$.} 
We now explain 
how to handle 
$(T^1(A))_{A >0}$ in 
\eqref{eq:T1A:T2A:T3A:T4A}.
We first decompose $T^1(A)$ into non positive and non negative parts
\begin{equation}
\label{eq:proof:UI:T1A}
T^1(A) = 
\tilde{q}_{T \wedge \tau_A} \left( Y_{T \wedge \tau_A}
+ 
Y_{T \wedge \tau_A}^- \right) -
\tilde{q}_{T \wedge \tau_A}
Y_{T \wedge \tau_A}^-. 
\end{equation}
By Lemma \ref{lemma:necessary-q:lower-bound}, the 
family $(\tilde{q}_{T \wedge \tau_A}
Y_{T \wedge \tau_A}^-)_{A>0}$ is uniformly integrable. In particular, 
\begin{equation}
\label{eq:lim:A:tildeQ:Y-:1E}
\lim_{A \rightarrow + \infty}
{\mathbb E}
\left[
\tilde{q}_{T \wedge \tau_A}
Y^-_{T \wedge \tau_A}  
\right]=
{\mathbb E}
\left[
\tilde{q}_{T }
Y^-_{T }  
\right].
\end{equation}
We also notice that,  
$\tilde{q}_{T \wedge \tau_A} (Y_{T \wedge \tau_A} +Y_{T \wedge \tau_A}^-)$ takes values in $[0,+\infty)$. Therefore, Fatou's lemma gives
\begin{equation*}
\liminf_{A \rightarrow + \infty}
{\mathbb E}\left[
\tilde{q}_{T \wedge \tau_A}
\left( Y_{T \wedge \tau_A}
+ Y_{T \wedge \tau_A}^- 
\right)
\right]
\geq 
{\mathbb E}\left[
\tilde{q}_{T}
\left( Y_{T}
+ Y_{T}^- 
\right)
\right].
\end{equation*}
Inserting the latter into 
\eqref{eq:proof:UI:T1A}, we obtain 
\begin{equation*}
\liminf_{A \rightarrow + \infty}
{\mathbb E} 
\left[ T^1(A)
\right]
= 
\liminf_{A \rightarrow + \infty}
{\mathbb E}
\left[
\tilde q_{T \wedge \tau_A}
Y_{T \wedge \tau_A}
\right]
\geq {\mathbb E}
\left[
\tilde q_{T}
\left( 
Y_{T}+Y_T^- - Y_T^- 
\right) 
\right]
= 
{\mathbb E}
\left[
\tilde q_{T} 
Y_{T} 
\right].
\end{equation*}
And then, 
thanks to
\eqref{eq:T2A:T3A:T4A}, this gives 
\begin{equation*}
\begin{split}
&\liminf_{A \rightarrow + \infty}
{\mathbb E}\left[ T^1(A)
+ T^2(A) + T^3(A) + T^4(A) 
\right]
\\
&\geq {\mathbb E}\left[ 
\tilde q_{T}
Y_{T}
- q_T Y_T + \int_0^T (\tilde q_t - q_t) \ell_t \dd t - 
\int_0^T \left( \tilde q_t  
\tilde f_t^\star - 
q_t f_t^\star \right) 
\dd t 
\right].
\end{split}
\end{equation*}
By 
\eqref{eq:sign:cost:stopped:at:A}, the left-hand side is less than $0$, 
from which we deduce that 
\begin{equation*}
0 \geq 
{\mathbb E}\left[ 
\tilde q_{T}
Y_{T} - q_T Y_T + 
 \int_0^T (\tilde q_t - q_t) \ell_t \dd t - 
\int_0^T \left( \tilde q_t 
\tilde f_t^\star - 
q_t f_t^\star \right) 
\dd t 
\right].
\end{equation*}
By 
\eqref{eq:Itildeq-Iq}, the right-hand side is greater than 
${\mathcal J}(\tilde q) - {\mathcal J}(q)$. 
This shows
${\mathcal J}(\tilde q) - {\mathcal J}(q) \leq 0$, which proves the optimality of $q$. \color{black}

Uniqueness of the minimizer follows from the strict convexity of ${\mathcal J}$ in the variable $q$, see 
Proposition \ref{prop:concave-J}.
\end{proof}

\subsection{Central planner's control problem}  \label{sec:central-planner}

In this section, we study the problem of the central planner,
\begin{equation} \label{pb:primal} \tag{P\textsubscript{C}}
    \inf_{\psi \in {\mathcal A}_{c_2}}  \mathcal{J}(\bar{q},\psi),
\end{equation}
under the assumption that 
there exists 
$\bar{\psi}$ 
such that the pair 
$(\bar{q},\bar{\psi}) \in \mathcal{Q}_{c_1} \times {\mathcal A}_{c_2}$
is a saddle point of the 
problem \eqref{pb:min-max-G-c1-c2}.
 We recall \eqref{eq:def:Qc1}
and \eqref{eq:def:Ac2} for the definitions of the two sets 
${\mathcal Q}_{c_1}$ and ${\mathcal A}_{c_2}$. We further recall \eqref{def:mathscr-A} for the definition of the set $\mathscr{A}$.
The purpose of this section is to establish the following characterization of the problem \eqref{pb:primal}.

\begin{theorem} \label{thm:sto-max-princ-central-planner} 
There exists a constant $c_2'>0$,
only depending on the data and $c_1$, such that, 
if $c_2 >c_2'$, 
 the minimizer $\psi$ of \eqref{pb:primal} 
(over ${\mathcal A}_{c_2}$) belongs in fact to 
${\mathcal A}_{c_2'}$. Conversely, if $c_2> c_2'$, 
any $\psi \in {\mathcal A}_{c_2'}$ is a minimizer to the problem \eqref{pb:primal} over ${\mathcal A}_{c_2}$ if and only if there exists a tuple $(\psi,p,k,X) \in \mathscr{A}$ solving the FBSDE \eqref{optim:condition-primal} (with $q$ replaced by $\bar q$).
\end{theorem}

\begin{proof}[Sketch of Proof.]
In Lemma \ref{lemma:psi-in-A} below, we establish the existence of a constant $c_2'>0$, only depending on the data and $c_1$, independent of $c_2$, such that any minimizer
$\psi$ to the problem 
\eqref{pb:primal} over ${\mathcal A}_{c_2}$
is in fact in ${\mathcal A}_{c_2'}$. By assuming (without any loss of generality) that the constant $c_2$ is strictly larger than $c_2'$, we ensure that the minimizer $\psi$ is an interior solution in the sense that 
\begin{equation*}
    \mathcal{S}^\star(\psi) < c_2.
\end{equation*}
The conclusion of the statement follows from the necessary condition proved in Lemma \ref{lemma:necessary-psi} and the sufficient condition established in Lemma \ref{lem:sufficient-condition}.
\end{proof}

\subsubsection{A priori estimate} 
\label{subsubse:4.3.1:a:priori:estimate}

This subsection is devoted to proving the following a priori estimate for the component $\bar{\psi}$ of the saddle point $(\bar{q}, \bar{\psi})$.

\begin{lemma}  \label{lemma:psi-in-A}
There exists a positive constant $c_2'>0$, only depending on $c_1$ and on the data (and in particular independent of the parameter  $c_2$), such that the component $\bar \psi$ of any saddle point $(\bar q,\bar \psi) \in {\mathcal Q}_{c_1} \times {\mathcal A}_{c_2}$ to \eqref{pb:min-max-G-c1-c2} (with $c_2 > c_2'$)
belongs in fact to ${\mathcal A}_{c_2'}$. 
\end{lemma}

\begin{proof}
Let $q \in \mathcal{Q}$.
Using the convexity assumption on $\mathcal{G}$, we have 
\begin{align*} 
{\mathcal G} \left( q_T,X^{\bar \psi}_T \right) & \geq {\mathcal G} \left( q_T, 0 \right) 
+ {\mathbb E} \left[\delta_X \mathcal{G}(q_T,0) \cdot X_T^{\bar \psi} \right] 
\\
& \geq 
- L 
- L {\mathbb E} \left[q_T  | X_T^{\bar \psi} |\right], 
\end{align*}
where the second inequality follows from the growth Assumption \ref{eq:G-growth}. Therefore,  
\begin{align*} 
{\mathcal J}(q,\bar \psi) + \mathcal{S}(q)
\geq & - L - L {\mathbb E} \left[q_T  | X_T^{\bar \psi} | \right]  + 
{\mathbb E}\left[ \int_0^T q_t \ell(t,\bar \psi_t) \dd t 
\right].
\end{align*} 
Since $\ell$ grows up at least quadratically fast by Assumption \ref{hyp:L}, we can find  a constant $C>0$ such that 
\begin{equation*}
\mathbb{E}\left[   \int_0^T q_t \ell(t,\bar \psi_t)\dd t \right]
\geq - C +  \frac{1}{2L} \mathbb{E}\left[\int_0^T q_t\left \vert \bar \psi_t \right\vert^2 \dd t\right].
\end{equation*}
Moreover, we know from Lemma
\ref{lem:E:q:X*:S(q):Sstar(psi)} that, for any 
$\varepsilon \in (0,1)$, there exists 
$C_{\varepsilon}>0$ such that 
\begin{equation*}
\mathbb{E}\left[ q_T | X_T^{\bar \psi} | \right]
\leq C_{\varepsilon} + c_\varepsilon {\mathcal S}(q) + \varepsilon 
{\mathbb E} \left[ 
\int_0^T q_t \left\vert \bar \psi_t \right\vert^2 \dd t \right],
\end{equation*}
with $  c_\varepsilon  = 2 \beta e^{\alpha T} \|\Gamma\|_{L^\infty({\mathbb F},{\mathbb R}^{n \times n})} \|\Gamma^{-1}\|_{L^\infty({\mathbb F},{\mathbb R}^{n \times n})} \left(\|\nu\|_{L^\infty(\mathbb{F},\mathbb{R}^{n \times d})} + \frac{3}{\varepsilon}e^{\alpha T} \|\sigma\|_{L^\infty(\mathbb{F},\mathbb{R}^{n \times d \times n})} \right)$.
Therefore, 
choosing 
$\varepsilon=1/[{4\max(1,L)}]$ and combining the last three displays, we get 
\begin{align*} 
{\mathcal J}\left(q,\bar \psi\right)  
\geq - C  -  c_\varepsilon {\mathcal S}(q)  
+
\frac{1}{4 \max (1,L)}
{\mathbb E}\left[ \int_0^T q_t \left\vert \bar  \psi_t\right\vert^2 \dd t 
\right].
\end{align*}
And, then for the same
constants $\varepsilon$, $c_{\varepsilon}$ and $C$ as above,
\begin{align}
     \sup_{q\in \mathcal{Q}}
     \left\{ {\mathbb E}\left[ \int_0^T q_t \left|\bar \psi_t \right|^2  \dd t  
\right] - \gamma {\mathcal S}(q) \right\} &\leq   4 \max(1,L) \left[ C+  
\sup_{q\in \mathcal{Q}} {\mathcal J}\left(q,\bar \psi \right)\right].
\label{eq:bound:Ac2vartheta2}
\end{align}
where we recall that $\gamma$, defined in Assumption \ref{hyp:gamma}, is given by
\begin{align*}
    \gamma = & 4 \max(1,L) e^{\alpha T} \|\Gamma\|_{L^\infty(\mathbb{F},\mathbb{R}^{n \times n})} \|\Gamma^{-1}\|_{L^\infty(\mathbb{F},\mathbb{R}^{n \times n})} \\ & \times\left(\|\nu\|_{L^\infty(\mathbb{F},\mathbb{R}^{n \times d})} + 12 \max (1,L) e^{\alpha T} \|\sigma\|_{L^\infty(\mathbb{F},\mathbb{R}^{n \times d \times n})} \right).
\end{align*}
Recall now that 
$(\bar q,\bar \psi)$ is a saddle point of 
\eqref{pb:min-max-G-c1-c2}. By the third assertion in the statement of Theorem 
\ref{thm:sto-max-princ-dual}, we deduce that 
$\bar q$ is a maximizer of $q' \in {\mathcal Q} \mapsto {\mathcal J}(q',\bar \psi)$. Therefore, 
the supremum in the above display can be bounded as follows
\begin{equation}
\label{eq:bound:Ac2vartheta2:2}
\sup_{q \in {\mathcal Q}} {\mathcal J}(q,\bar \psi) = {\mathcal J}(\bar q,\bar \psi) 
\leq {\mathcal J}(\bar q,0),
\end{equation}
with the last inequality following from the saddle point property of  $(\bar{q},\bar \psi)$.
Since $f^\star$ is lower bounded, we have
\begin{equation*}
{\mathcal J}(\bar q,0) 
 = \mathcal{R}(\bar{q},0) - {\mathcal S} (\bar{q})  \leq C' \left(1  + \mathbb{E} \left[\bar q_T |X_T^0| + \int_0^T \bar{q}_t \ell(t,0) \dd t \right]\right), 
 \end{equation*}
 for a new constant $C'$.
We then use 
\eqref{eq:ineq:duality:with:r} to upper bound the right-hand side. By Lemma \ref{lemma:reg-X}, 
we know that there exists $\tau>0$, only depending on the data such that $X_T^0 \in S_{\textrm{\rm exp}}^{1,\tau}({\mathcal F}_T,{\mathbb R}^n)$. And then,
\eqref{eq:ineq:duality:with:r} (together with the fact that $\bar q \in {\mathcal Q}_{c_1}$) yields
\begin{align*}
    {\mathcal J}(\bar q,0)  \leq C' (1+c_1).
\end{align*}
Returning to 
\eqref{eq:bound:Ac2vartheta2}
and 
\eqref{eq:bound:Ac2vartheta2:2}, we obtain
\begin{equation*}
\sup_{q\in \mathcal{Q}}
     \left\{ {\mathbb E}\left[ \int_0^T q_t \left|\bar \psi_t \right|^2  \dd t  
\right] -  
\gamma {\mathcal S}(q) \right\} \leq  
4 \max(1,L)
\left[ C +
C'(1+c_1)
\right]. 
\end{equation*}
This completes the proof.
\end{proof}

\subsubsection{Necessary conditions}

This subsection is dedicated to
a series of lemmas 
leading eventually to Lemma \ref{lemma:necessary-psi}, which
we invoked in the 
proof of 
Theorem 
\ref{thm:sto-max-princ-central-planner}
to 
establish the necessary condition. Given the constant $c_2'$ in Lemma \ref{lemma:psi-in-A}, we assume that the constant $c_2$ in the definition 
of ${\mathcal A}_{c_2}$ in \eqref{pb:primal}
satisfies $c_2 > c_2' >0$, which condition guarantees that any 
minimizer 
to \eqref{pb:primal} --with $\bar q \in {\mathcal Q}_{c_1}$ such that, for some 
$\bar \psi \in {\mathcal A}_{c_2}$,
$(\bar q,\bar \psi)$
is a saddle-point--  
lies `in the interior' of the admissible set $\mathcal{A}_{c_2}$ (in the sense that it belongs to ${\mathcal A}_{c_2'}$). 
We insist
on the fact that, similar to 
Subsubsection 
\ref{subsubse:4.3.1:a:priori:estimate}, 
$\bar q$ is 
fixed throughout the analysis. In coherence with the convention adopted 
earlier, this makes it possible to 
omit $\bar q$
in the various notations. For instance, we  write 
 $\delta_X \mathcal{G}(X_T)$ for  $\delta_X \mathcal{G}(\bar{q}_T,X_T)$.
 
Recalling the 
definition of the pre-Hamiltonian $H$ in 
\eqref{eq:def:F:H}, and denoting $X^\psi$, for 
a
given $\psi \in \mathcal{A}_{c_2'}$, 
the associated solution to the state equation, we define the adjoint BSDE with unknown $(p,k)$,
\begin{equation} \label{eq:adjoint-p-k} 
    - \dd p_t  =   \nabla_x H\left(t,X_t^{\psi},\psi_t,p_t,k_t,\bar{q}_t\right) \dd t - k_t \dd W_t,  \quad 
    p_T =  \delta_X \mathcal{G}(\bar{q}_T,X_T^\psi).
\end{equation}
By assumptions on the mappings $b$ and $\sigma$, the derivative of the pre-Hamiltonian simplifies to
\begin{equation*}
    \nabla_x H\left(t,X_t^{\psi},\psi_t,p_t,k_t,\bar{q}_t\right) = b_t^{\top} p_t.
\end{equation*}
In the analysis carried out below, we will also use the fact that 
\begin{equation}
\label{eq:nablapsiH}
 \nabla_\psi H(t,x,\psi,p,k,q) =
 q 
 \nabla_\psi 
 \ell(t,\psi) + 
 c_t^{\top} p
 +
r {\rm Tr}\left( \sigma_t^{\top} k \right),
\end{equation}
where $r$ can be either $0$ or $1$, 
and we recall the convention \eqref{def:trace-sigma-k} for the trace 
\begin{equation}
\label{eq:sigmatopk}
{\rm Tr}\left( \sigma_t^{\top} k_t \right) =
\left( 
 \sum_{i=1}^n
 \sum_{j=1}^d
( \sigma_t)_{i,j,\ell}  (k_t)_{i,j} 
\right)_{\ell=1,\ldots,d}.
\end{equation}
To study the BSDE \eqref{eq:adjoint-p-k}, we introduce the following intermediary BSDE with unknown $(\varrho,h)$,
\begin{equation} \label{eq:r-h}
    - \dd \varrho_t  =   \left( b_t^\top \varrho_t  + \bar Y^\star_t 
    \varrho_t \right) \dd t - h_t   \dd \bar{W}_t, \quad  
    \varrho_T =  \bar{q}_T^{-1}\delta_X \mathcal{G}(\bar{q}_T,X_T^\psi),
\end{equation}
where $(\bar{W}_t = W_t - \int_0^t \bar Z^\star_s \dd s)_{t \in [0,T]}$ is a Brownian motion under the equivalent probability measure $\bar{\mathbb{Q}}$ defined by $\bar{\mathbb{Q}} = \bar{\mathcal{E}}_T   \mathbb{P}$, with $\bar{\mathcal{E}}_T \coloneqq  \mathcal{E}_T(\int_0^\cdot \bar{Z}_s^\star  \cdot \dd W_s)$. We recall that $(\bar Y^\star,\bar Z^\star)$ is associated to $\bar{q}$ through \eqref{eq:q}, and that $\bar{q}_T = \exp(\int_0^T \bar Y_s^\star \dd s) \bar{\mathcal E}_T$.

\begin{lemma} \label{lemma:r-h}
    There exists a unique solution $(\varrho,h) \in S^{2}(\mathbb{F},\mathbb{R}^{n},\bar{\mathbb{Q}}) \times M^2(\mathbb{F},\mathbb{R}^{n \times d},\bar{\mathbb{Q}})$ to \eqref{eq:r-h}. When $r = 1$, $\varrho$ belongs to $L^\infty(\mathbb{F},\mathbb{R}^{n})$ and 
    $h$ to 
    $L^2({\mathbb F},{\mathbb R}^{n \times d},\bar{\mathbb Q})$.
\end{lemma}

\begin{proof}
    Let us recall that 
    \begin{equation} \label{estim:grad-g-partial-f}\|b\|_{L^\infty(\mathbb{F},\mathbb{R}^{n \times n})} + \|\bar Y^\star\|_{L^\infty(\mathbb{F})} \leq \alpha + L,
    \end{equation}
    by Assumption \ref{hyp:b} and since $|\bar Y_s^\star| \leq \alpha$. In particular, the BSDE \eqref{eq:r-h} is a linear BSDE, with a bounded linear coefficient. The existence and uniqueness of a solution are well established once the terminal condition has been shown to be sufficiently integrable.
    \vskip 5pt
    
    \noindent 
    \textit{Step 1: $r = 0$.}
    By the growth Assumption \ref{eq:G-growth} on $\delta_X \mathcal{G}$, we have
    \begin{equation*}
        |\varrho_T| = \bar{q}_T^{-1} \left |\delta_X \mathcal{G}(\bar{q}_T,X_T^\psi)\right | \leq L\left(1 + |X_T^{\psi}| + \mathbb{E}\left[ \bar{q}_T |X_T^{\psi}|^{2} \right] \right).
    \end{equation*}
    Taking the square on both sides, we get
    (for a constant $C$ only depending on $L$ and whose value is allowed to vary from line to line) \begin{align*}\mathbb{E}^{\bar{\mathbb{Q}}}\left[|\varrho_T|^2\right] & \leq C\left( 1 +  \mathbb{E}^{\bar{\mathbb{Q}}}\left[ |X_T^{\psi}|^2\right] + \mathbb{E}\left[ \bar{q}_T |X_T^{\psi}|^{2} \right]^2 \right) \\
    & \leq C\left( 1 +  \mathbb{E}\left[ \int_0^T \bar{q}_s |\psi_s|^2 \dd s\right]^2\right) 
    \\
    & \leq C\left( 1 + \mathcal{S}(\bar{q})^2 + \mathcal{S}^\star(\psi)^2\right) < +\infty,
    \end{align*}
    where we used Lemma \ref{lemma:reg-X-psi-A} in Appendix \ref{appendix:apriori-SDE} in order to pass from the first to the second line. Together with \eqref{estim:grad-g-partial-f}, existence and uniqueness follow from the standard $L^2$-theory of BSDEs 
    (we recall from \cite[Theorem 2.4]{AksamitFontana} that the martingale representation 
theorem holds under 
$\bar{\mathbb Q}$, with respect to $\bar W$).
\vskip 5pt

    \noindent \textit{Step 2: $r = 1$.}  In this case, $\vert \delta_X \mathcal{G}\vert$ is  bounded by $C \bar q_T (1+ \mathbb{E}\left[\bar q_T|X_T|\right])$, for a constant $C$ only depending on $L$. 
    By Lemma  \ref{lemma:reg-X-psi-A},
     $\mathbb{E}\left[\bar q_T|X_T|\right]<+\infty$. Therefore, the terminal condition in \eqref{eq:r-h} is square integrable under the probability measure $\bar{\mathbb{Q}}$. The conclusion follows by the same arguments as in the first step. This concludes the proof. 
\end{proof}

\begin{lemma} \label{lemma:space-of-adjoint}
    There exists a unique solution $(p,k)$ to \eqref{eq:adjoint-p-k} in the space $ D(\mathbb{F},{\mathbb P},{\mathbb R}^n) \times (\cap_{\beta\in(0,1)} M^\beta(\mathbb{F},\mathbb P,\mathbb{R}^{n\times d}))$. It is given by 
    \begin{equation*}
        (p_t,k_t)_{t \in [0,T]} = \left(\bar q_t \varrho_t, \bar q_t ( h_t + \varrho_t \otimes \bar Z_t^\star)\right)_{t \in [0,T]},
    \end{equation*}
    where $(\varrho_t,h_t)_{t \in [0,T]}$ is the solution to \eqref{eq:r-h} and $( \varrho_t \otimes  \bar Z_t^\star)_{t \in [0,T]}$ denotes the $n \times d$ matrix with elements $(\varrho_t^{i} \bar Z_t^{\star,j})_{i  \in \{1,\ldots,n\},j\in \{1,\ldots,d\}}$.
\end{lemma}

\begin{proof}
\textit{Step 1: Existence.} 
    Repeating  the computations from the proof of Lemma \ref{lemma:r-h}, there exists a constant $C$, only depending on $L$, such that
    \begin{align*}
        \mathbb{E}\left[ \left|\delta_X \mathcal{G}\left(\bar{q}_T,X_T^{{\psi}}\right)\right|\right] & \leq  L\mathbb{E}\left[ \bar q_T\left(1+\left|X_T^{\psi} \right|^{1-r} + \mathbb{E}\left[ \bar{q}_T |X_T^{\psi}|^{2-r}\right)\right]\right] \\
        & \leq  C\left( 1 + \mathbb{E}\left[ \int_0^T \bar q_s|\psi_s|^2 \dd s\right]^2\right) \\
        &\leq C\left( 1 + \mathcal{S}({\bar q})^2 + \mathcal{S}^\star(\psi)^2 \right) 
        < +\infty.
    \end{align*}
    The existence of a solution to \eqref{eq:adjoint-p-k}, within the space specified in the statement,  follows from \cite[Proposition 6.4]{briand2003lp}. 
 \vskip 5pt
 
    \noindent \textit{Step 2: Uniqueness.}
    Now let $(p_t,k_t)_{t \in [0,T]}$ be a solution to the backward equation \eqref{eq:adjoint-p-k} (within the space mentioned in the statement). Then, one can verify that $(\varrho_t,h_t)_{t \in [0,T]} \coloneqq (\bar q_t^{-1} p_t, \bar q_t^{-1}  ( k_t - p_t \otimes \bar Z_t^\star))_{t \in [0,T]}$ solves \eqref{eq:r-h}, with the stochastic integral therein being understood as a local martingale under $\bar{\mathbb Q}$.
That said, at this stage, we do not know yet that the hence 
defined pair $(\varrho,h)$ belongs to the space 
$S^2({\mathbb F},{\mathbb R}^n,\bar{\mathbb Q}) \times M^2({\mathbb F},{\mathbb R}^{n \times d},\bar{\mathbb Q})$, which prevents us from identifying directly $(\varrho,h)$
with the (unique) solution constructed in Lemma \ref{lemma:r-h}. We thus proceed as follows. We denote by 
$(R_t)_{t \in [0,T]}$
the solution to the 
(random) ordinary differential equation
\begin{equation}
\label{eq:resolvant}
\frac{\dd }{\dd t}R_t = - b_t^\top R_t - \bar{Y}_t^\star R_t,  \quad R_0=I_n,
\end{equation}
with $I_n$
denoting the 
$n \times n$ identity matrix. The process $(R_t)_{t \in [0,T]}$ takes values in the set of $n \times n$ invertible matrices, and
\begin{equation*}
\frac{\dd }{\dd t} R_t^{-1}
    = R_t^{-1} b_t^\top + \bar Y_t^\star R_t^{-1}, \quad R_0^{-1} = I_n.
\end{equation*}
Obviously, the process $(R^{-1}_t)_{t \in [0,T]}$ is bounded by a deterministic constant. 
Also, it is straightforward to check that 
$(R_t^{-1} \varrho_t)_{t \in [0,T]}$ is a local martingale under $\bar{\mathbb Q}$. 
By a standard localization argument, we can find a non-decreasing sequence 
of stopping times $(\tau_k)_{k \geq 1}$, converging to $T$, such that, for any 
$t \in [0,T]$, 
and any integer 
$k \geq 1$,
\begin{equation}
R_{t \wedge \tau_k}^{-1} \varrho_{t \wedge \tau_k}
= {\mathbb E}^{\bar{\mathbb Q}}
\left[ R_{\tau_k}^{-1} \varrho_{ \tau_k}
\vert 
{\mathcal F}_{t}
\right ]. 
\label{eq:R(-1)rho}
\end{equation}
In order to pass to the limit (as $k$ tends to $+\infty$) in the above display, we check that the collection of random variables $(R_{\tau_k}^{-1} \varrho_{\tau_k})_{k \geq 1}$ is uniformly integrable under $\bar{\mathbb Q}$ (the same argument would show that the collection of random variables 
    $(R_{t \wedge \tau_k}^{-1} \varrho_{t \wedge \tau_k})_{k \geq 1}$
    is uniformly integrable). For any event $A$, we can find a constant $C$ such that, for any $k \geq 1$,
    \begin{equation*}
{\mathbb E}^{\bar{\mathbb Q}}
\left[ {\mathds 1}_A
\vert R_{\tau_k}^{-1} \varrho_{\tau_k} \vert
\right]
\leq C
{\mathbb E} \left[
{\mathds 1}_A
\vert q_{\tau_k}
\varrho_{\tau_k}
\vert 
\right]
= C
{\mathbb E} \left[
{\mathds 1}_A
\vert p_{\tau_k}
\vert 
\right].
    \end{equation*}
Since $
p$ belongs to $D({\mathbb F},{\mathbb P})$, the right-hand side tends to $0$, uniformly in $k$, as 
${\mathbb P}(A)$ tends to $0$. Writing
${\mathbb P}(A)  =
{\mathbb E}^{\bar{\mathbb Q}}[
(\bar{\mathcal E}_T)^{-1}
{\mathds 1}_A
]$ (here 
$\bar{\mathcal E}_T>0$ a.s., 
because 
$\bar q_T >0$ a.s.), we deduce that ${\mathbb P}(A)$ tends to $0$ as 
$\bar{\mathbb Q}(A)$ tends to $0$. 
This shows that 
the right-hand side (in the above display) tends to $0$, uniformly in $k$, as $\bar{\mathbb Q}(A)$ tends to $0$, which provides the required uniform integrability property. 
Letting 
$k$ tend to $+\infty$ in \eqref{eq:R(-1)rho}, we deduce that, for any $t \in [0,T]$,  
\begin{equation*}
R_t^{-1} \varrho_t = 
{\mathbb E}^{\bar{\mathbb Q}}\left[ \left. R_T^{-1} \bar{q}_T^{-1} 
\delta_X \mathcal{G}(\bar q_T,X_T^\psi)
\right \vert
{\mathcal F}_t
\right].
\end{equation*}
This 
provides an explicit formula for 
$\varrho$ and makes it possible to identify it (together with $h$) with the solution obtained in Lemma 
\ref{lemma:r-h}. It remains to see that  the mapping 
    \begin{equation*}
        \mathbb{R}^n \times \mathbb{R}^{n\times d} \ni (p,k) \mapsto (q^{-1} p, q^{-1}  (k - p \otimes z^{\star})),
    \end{equation*}
    is one-to-one for any $q>0$ and $z^\star \in \mathbb{R}^n$.  This 
    proves that $(p_t,k_t)_{t \in [0,T]}$ is uniquely determined by the pair $(\varrho_t,h_t)_{t \in [0,T]}$ and is thus unique. 
\end{proof}
\color{black}
\paragraph{Tangent processes.}

Let $\varphi \in L^\infty(\mathbb{F},\mathbb{R}^n)$ be such that $\psi + \varphi \in \mathcal{A}_{c_2}$ (we recall that $\psi$ is an arbitrary element in ${\mathcal A}_{c_2'}$). We introduce the following system of variational (or tangent) processes with unknown $(u,v,x)$,
the latter taking values in ${\mathbb R} \times {\mathbb R}^d \times {\mathbb R}^n$,
\begin{equation*} \label{eq:FBSDE-deriv-y} \tag{V}
    \left\{ \begin{array}{rll}
       \dd u_t & = \bar{q}_t \varphi_t \cdot \nabla_\psi \ell(t,\psi_t) \dd t - v_t \cdot \dd W_t,   \quad u_T = \delta_X \mathcal{G}(\bar{q}_T, X_T^{\psi}) \cdot x_T,\\[0.5em]
        \dd x_t &  = \left( b_t x_t + c_t \varphi_t \right) \dd t +  r D_{\psi} \sigma_t (\varphi_t) \dd W_t, \quad x_0  = 0,
    \end{array} \right.
\end{equation*}
where we recall that  $D_\psi \sigma_t(\varphi_t) = (\sum_{\ell=1}^n (\sigma_t)_{i,j,\ell} (\varphi_t)_{\ell})_{i\in \{1,\ldots,n\},j\in\{1,\ldots,d\}}$, see \ref{hyp:sigma}.

 \begin{lemma} \label{lemma:space-of-deriv-system}
Let $\varphi \in L^\infty({\mathbb F},{\mathbb R}^n)$ such that $\psi + \varphi \in {\mathcal A}_{c_2}$. Then, there exists a unique solution $(x_t)_{t \in [0,T]}$ 
to the forward equation in \eqref{eq:FBSDE-deriv-y}, lying in $S^{2}(\mathbb{F},\mathbb{R}^{n}) \cap S^{2}(\mathbb{F},\mathbb{R}^{n},\bar{\mathbb{Q}})$ (even in $L^{\infty}(\mathbb{F},\mathbb{R}^{n})$ when $r = 0$).
    And, there exists a solution $(u,v) \in {D(\mathbb{F},\mathbb{P})} \times (\cap_{\beta\in(0,1)} M^\beta(\mathbb{F},\mathbb{R}^{d},\mathbb{P}))$ to the backward equation in \eqref{eq:FBSDE-deriv-y}. 
\end{lemma}

\begin{proof}
    \textit{Step 1: Existence and uniqueness to the forward equation.} 
    Denoting by $\Gamma$ the resolvent associated with the linear part of the forward equation, 
    i.e.,
\begin{equation*}
        \frac{\dd}{\dd t} \Gamma_t =  b_t \Gamma_t,
        \quad 
        \Gamma_0 = I_n,
    \end{equation*}
    with $I_n$ standing for the $n \times n$ identity matrix, the solution 
    to the forward equation in 
    \eqref{eq:FBSDE-deriv-y}  is explicitly given by 
\begin{equation*}
x_t = 
\Gamma_t
\left( \eta + \int_0^t \Gamma_s^{-1} \left(a_s + c_s \varphi_s\right) \dd s
+
r
\int_0^t 
\Gamma_s^{-1}  
D_{\psi}\sigma_s (\varphi_s)
\dd W_s \right), 
\quad t \in [0,T].
\end{equation*}
Thus, $(x_t)_{t \in [0,T]}$ clearly belongs to $S^{2}(\mathbb{F},\mathbb{R}^{n},\mathbb P)$ (because $\varphi$ is bounded). It further belongs to $S^{2}(\mathbb{F},\mathbb{R}^{n},\bar{\mathbb{Q}})$ 
by a direct application of Lemma \ref{lemma:reg-X-psi-A} using once again the fact that $\varphi$ is bounded (where we identify the process $(\Gamma_t^{-1}  
D_{\psi} \sigma_t (\varphi_t))_{t \in [0,T]}$, which is bounded, in this proof with the process $(\nu_t)_{t \in [0,T]}$ in the statement of Lemma \ref{lemma:reg-X-psi-A}).

       \vskip 5pt
    
     \noindent \textit{Step 2: Well-posedness of the backward equation.} By the growth assumptions \ref{hyp:L} and \ref{eq:G-growth}, we have 
 \begin{align*}
        \mathbb{E}\left[|\delta_X \mathcal{G}(\bar{q}_T,X_T^{\psi}) \cdot x_T|\right] & \leq C \mathbb{E}\left[\bar{q}_T(1 + |X_T^{\psi}|^{1-r})  |x_T|\right].
    \end{align*}
When $r=1$, the right-hand side reduces to $C( 1 + {\mathbb E}[\bar{q}_T \vert x_T\vert])$, which is finite since $(x_t)_{t \in [0,T]} \in S^2({\mathbb F},{\mathbb R}^n,\bar{\mathbb Q})$. When $r=0$,
$(x_t)_{t \in [0,T]} \in L^\infty({\mathbb F},{\mathbb R}^n,\mathbb{P})$, and by Lemma \ref{lemma:reg-X-psi-A}, we have 
\begin{align*}
        \mathbb{E}\left[|\delta_X \mathcal{G}(\bar{q}_T,X_T^{\psi}) \cdot x_T|\right]  \leq C \left( 1 + \mathbb{E}\left[\bar{q}_T|X_T^{\psi}|\right]  \right) <+\infty.
    \end{align*} Since $\nabla_\psi \ell$
    is at most of linear growth in $\psi$, we further have
    \begin{equation}
        \label{eq:lastline:lemma38}
        \begin{split}
        \mathbb{E}\left[\int_0^T \bar{q}_s |\varphi_s \cdot \nabla_\psi\ell(s,\psi_s)| \dd s\right]  &\leq C\mathbb{E}\left[\int_0^T \bar{q}_s \left( 1 + |\varphi_s|^2 + |\psi_s|^2 \right) \dd s\right]
        \\
        & \leq C\left(1 + \mathcal{S}(\bar{q}) + \mathcal{S}^\star(\psi)\right)  <+\infty.
    \end{split}
    \end{equation}
The conclusion follows by \cite[Proposition 6.4]{briand2003lp}. 
\end{proof}

So far, 
we have considered $\varphi \in L^\infty({\mathbb F},{\mathbb R}^n,{\mathbb P})$ such that $\psi + \varphi \in {\mathcal A}_{c_2}$.
Since $\mathcal{A}$ is convex, we then have, for any $\varepsilon \in [0,1]$,  $ \psi + \varepsilon \varphi \in \mathcal{A}_{c_2}$. 
By optimality of the control $\psi$, we get 
\begin{equation} \label{ineq:optim-psi}
    \varepsilon^{-1} \left(\mathcal{J}(\psi + \varepsilon \varphi) -  \mathcal{J}(\psi) \right) \geq 0.
\end{equation}
We use the above inequality to prove the following statement:
\begin{lemma}  \label{lemma:variational-ineq} Let $\varphi \in L^\infty({\mathbb F},{\mathbb R}^n,{\mathbb P})$ such that $\psi + \varphi \in {\mathcal A}_{c_2}$. With $(u,v)$ being as in the statement of Lemma \ref{lemma:space-of-deriv-system}, 
the following variational inequality holds true: \begin{equation*}\mathbb{E} \left[ u_0 \right] = \mathbb{E}\left[u_T + \int_0^T \bar q_s \varphi_s \cdot \nabla_{\psi} \ell(s,\psi_s )\dd s\right]
         \geq 0.
         \end{equation*}
\end{lemma}

\begin{proof}
Since $\bar q$ is fixed, we omit to indicate it explicitly in the various functionals that depend on it. For instance, we use the shorthand notation 
${\mathcal G}(X_T^\psi)$ for 
${\mathcal G}(\bar q_T,X_T^\psi)$.

On the one hand, we have, from 
    \ref{hyp:DgD_XG}, 
    \begin{equation*}
        \mathcal{G}(X_T^{\psi + \varepsilon \varphi}) -  \mathcal{G}(X_T^\psi) = \mathbb{E}\left[\delta_X \mathcal{G}(X_T^\psi) \cdot \left( X_T^{\psi + \varepsilon \varphi} - X_T^\psi\right) \right] + O \left(\mathbb{E}\left[ \bar{q}_T|X_T^{\psi + \varepsilon \varphi} - X_T^\psi|^2 \right]\right).
    \end{equation*}
    Because $(X_t)_{t \in [0,T]}$ solves a linear SDE,  we also have
    \begin{equation*}
        \varepsilon^{-1}(X_T^{\psi + \varepsilon \varphi} - X_T^\psi) =   X_T^{\varphi} = x_T,
    \end{equation*}
    and thus
    \begin{equation}  
    \label{eq:derivative-G}\varepsilon^{-1}\left(\mathcal{G}(X_T^{\psi + \varepsilon \varphi}) -  \mathcal{G}(X_T^\psi)\right) = \mathbb{E}\left[\delta_X \mathcal{G}(X_T^\psi) \cdot x_T \right] + \varepsilon O \left(\mathbb{E}\left[ \bar{q}_T|x_T|^2 \right]\right).
    \end{equation}
In order to handle the right-hand side, we use the same estimates as in the second step of the proof of Lemma 
\ref{lemma:space-of-deriv-system}.
In particular, we already know that 
$\mathbb{E}[|\delta_X \mathcal{G}(X_T^\psi) \cdot x_T|]< + \infty$. We also know that
${\mathbb E}[\bar q_T \vert x_T \vert^2]<+\infty$, from which we deduce that the term on the second line of \eqref{eq:derivative-G} tends to $0$ with $\varepsilon$. So, we obtain

\begin{equation}  \label{lim:G-grad-X}
        \lim_{\varepsilon \to 0}\varepsilon^{-1}\left(\mathcal{G}(X_T^{\psi + \varepsilon \varphi}) -  \mathcal{G}(X_T^\psi)\right) = \mathbb{E}\left[\delta_X \mathcal{G}(X_T^\psi) \cdot x_T \right].
    \end{equation}
\color{black}    
    Since $\ell$ is assumed to be twice differentiable, with bounded second-order derivatives, we also have
\begin{align*}
        \varepsilon^{-1} \mathbb{E}\left[\int_0^T \bar{q}_s \left(\ell(s,\psi_s + \varepsilon \varphi_s) - \ell(s,\psi_s) \right) \dd s\right]  = & \mathbb{E}\left[\int_0^T \bar{q}_s \varphi_s \cdot \nabla_{\psi} \ell(s,\psi_s )\dd s\right]  \\
        & + \varepsilon o\left(\mathbb{E}\left[\int_0^T \bar{q}_s|\varphi_s|^2  \dd s \right]\right).
    \end{align*}
    Because $\varphi \in L^\infty(\mathbb{F},\mathbb{R}^n)$, we obviously have $\mathbb{E}[\int_0^T \bar{q}_s|\varphi_s|^2  \dd s]< +\infty$. As for the first term on the right-hand side, we recall from 
    \eqref{eq:lastline:lemma38}
that it is bounded. 
    Then, combining
    the above display with \eqref{lim:G-grad-X}, and recalling again the definition of the criterion $\mathcal{J}$,
    we get \begin{align*}
        \lim_{\varepsilon \to 0} \varepsilon^{-1} \left(\mathcal{J}(\psi + \varepsilon \varphi) -  \mathcal{J}(\psi) \right) & =  \mathbb{E}\left[\delta_X \mathcal{G}(X_T^\psi) \cdot x_T \right] + \mathbb{E}\left[\int_0^T \bar{q}_s \varphi_s \cdot \nabla_{\psi} \ell(s,\psi_s )\dd s\right] \\
        & = \mathbb{E}\left[u_T + \int_0^T \bar q_s \varphi_s \cdot \nabla_{\psi} \ell(s,\psi_s )\dd s\right].
        \end{align*}
        It remains to use to the backward equation in \eqref{eq:FBSDE-deriv-y}
in order to identify the last term with ${\mathbb E}[u_0]$. By localization, we can find a non-decreasing sequence of stopping times 
$(\tau_m)_{m \geq 1}$, converging to $T$, such that, for any $m\geq 1$, 
\begin{equation}
\label{eq:lemma39:lmastdisplay}
{\mathbb E}
\left[ u_{\tau_m}
+ 
\int_0^{\tau_m}
\bar{q}_s \varphi_s \cdot \nabla_\psi \ell(s,\psi_s) \dd s
\right]
= {\mathbb E}[u_0].
\end{equation}
Since $u$ belongs to 
$D({\mathbb F},{\mathbb P})$, 
${\mathbb E}[u_{\tau_m}] \rightarrow {\mathbb E}[u_T]$ as $m$ tends to $+\infty$. And by 
\eqref{eq:lastline:lemma38}, 
${\mathbb E}[\int_0^{\tau_m} \bar q_s \varphi_s \cdot \nabla_\psi \ell(s,\psi) \dd s] \rightarrow {\mathbb E} [\int_0^T \bar q_s \varphi_s \cdot \nabla_\psi \ell(s,\psi) \dd s]$. This shows that the 
left-hand side in the above display converges to the right-hand side of 
\eqref{eq:lemma39:lmastdisplay}.    Recalling inequality \eqref{ineq:optim-psi}, we complete the proof. 
\end{proof}

\begin{lemma} \label{lemma:necessary-psi} 
Let $\psi \in \mathcal{A}_{c_2'}$ be a minimizer to problem \eqref{pb:primal} and $(p,k)$ the associated solution to \eqref{eq:adjoint-p-k} provided by Lemma \ref{lemma:space-of-adjoint}. Then, $(\psi,p,k,X)$  belongs to the set $\mathscr{A}$
defined in 
\eqref{def:mathscr-A} and 
satisfies the first-order condition \eqref{optim:condition-primal} (with $\bar q$ being substituted for $q$ therein).
\end{lemma}

\begin{proof}
We start with the following preliminary remark: the fact that $(\psi,p,k,X)$ belongs to ${\mathscr A}$ is a consequence of 
Lemmas \ref{lemma:space-of-adjoint}
and \ref{lemma:reg-X-psi-A}.

Next, 
following the analysis carried out in Lemmas
\ref{lemma:space-of-deriv-system} and \ref{lemma:variational-ineq}, we consider $\varphi \in L^\infty({\mathbb F},{\mathbb R}^n,{\mathbb P})$ such that $\psi + \varphi \in {\mathcal A}_{c_2}$.
With $x$ as in Lemma \ref{lemma:space-of-deriv-system}, and
    by  It{\^o}'s formula, we have
    \begin{align*} 
        p_T \cdot x_T = \int_0^T p_s \cdot \dd x_s + \int_0^T x_s \cdot \dd p_s + r \int_0^T \Tr ((D_\psi \sigma_s(\varphi_s ))^\top k_s)  \dd s.
    \end{align*}
    For a given $A>0$, we also  consider the stopping time
    \begin{equation*}
        \tau_{A} \coloneqq \inf \left\{ t \in [0,T], \; \left |\int_0^t p_s ^\top D_\psi \sigma_s(\varphi_s) \dd W_s \right\vert +  \left |\int_0^t x_s^\top k_s \dd W_s \right\vert  \geq A   \right\}. 
    \end{equation*}
    Then, by cancellation of the expectations of the
    stopped stochastic integrals in the expansion of 
    $(p_t \cdot x_t)_{t \in [0,T]}$,
    we get
    \begin{align}  \mathbb{E}\left[p_{T \wedge \tau_A} \cdot x_{T \wedge \tau_A} \right]  =  \mathbb{E}\left[ \int_0^{T \wedge \tau_A} p_s^\top  c_s \varphi_s \dd s + r \int_0^{T \wedge \tau_A} {\rm Tr}((D_\psi \sigma_s(\varphi_s))^\top k_s)  \dd s \right].\label{eq:p-x-tau-A}
\end{align}
    We now aim to pass to the limit on both sides of the equality.
\vskip 5pt
    
    \noindent \textit{Step 1: convergence of the left-hand side of \eqref{eq:p-x-tau-A}.} To pass to the limit, we establish that the random variables $(p_{T \wedge \tau_A} \cdot x_{T \wedge \tau_A})_{A>0}$ are uniformly integrable, distinguishing between the two cases $r = 0$ and $r = 1$. Throughout, $E$ is a fixed subset of $\Omega$, belonging to $\mathcal{F}_T$.
    
    When $r = 0$, the process $x=(x_t)_{t \in [0,T]}$ is bounded (see Lemma \ref{lemma:space-of-deriv-system}).
     Then, we can find a constant $C$, independent of $A$, such that    \begin{equation*}
         \mathbb{E}\left[ \mathds{1}_{E}|p_{T \wedge \tau_A} \cdot x_{T \wedge \tau_A} | \right] \leq C \mathbb{E}\left[\mathds{1}_{E}  |p_{T \wedge \tau_A} |\right].
    \end{equation*}
    Since $p$ belongs to $D(\mathbb{F},{\mathbb R}^n,\mathbb{P})$ by Lemma \ref{lemma:space-of-adjoint}, the right-hand side tends to $0$, uniformly with respect to $A$, as $\mathbb{P}(E)$ tends to $0$, yielding the required uniform integrability property. 
    
    When $r=1$,  the process $x$ belongs to $S^{2}(\mathbb{F},\mathbb{R}^n,\bar{\mathbb{Q}})$ thanks to Lemma 
    \ref{lemma:space-of-deriv-system}.   
    By Lemmas \ref{lemma:r-h} and  \ref{lemma:space-of-adjoint}, 
    the process $p$ is equal to  $\bar{q} \varrho$ with $\varrho$ belonging to  $L^\infty(\mathbb{F},\mathbb{R}^n)$. Then, we have \begin{equation*}
        \begin{split}\mathbb{E}\left[ \mathds{1}_{E} |p_{T \wedge \tau_A} \cdot x_{T \wedge \tau_A} | \right] &= C \mathbb{E}\left[ \mathds{1}_{E} \bar q_{T \wedge \tau_A} |\varrho_{T \wedge \tau_A} \cdot x_{T \wedge \tau_A} |\right] \\
        &\leq  C \mathbb{E}\left[  \mathds{1}_{E} \bar q_{T}^* | x_{T \wedge \tau_A} |\right]
        \\
        &\leq C {\mathbb E}\left[ \bar q_T^* \mathds{1}_{E} \right]^{1/2},
        \end{split}\end{equation*}
    with the last line following 
from Cauchy-Schwarz inequality, and, once again, from the fact that $x \in S^2({\mathbb F},{\mathbb R}^n,\bar{\mathbb Q})$. In particular, the constant $C$ on the last line is allowed to depend on the $S^2$-norm of 
$x$ (under $\bar{\mathbb Q}$) and is implicitly allowed to vary from line to line.
To prove that the term on the last line of the above displays tends to $0$ as 
${\mathbb P}(E)$ tends to $0$, it suffices to 
recall from  Lemma 
\ref{lemma:representation-q}
that 
$q_T^*$
is integrable, so that the right-hand side tends to $0$ as 
${\mathbb P}(E)$ tends to $0$.
This
    yields the expected  uniform integrability property. 
    \vskip 5pt
    
    \noindent \textit{Step 2: convergence on the right-hand side of \eqref{eq:p-x-tau-A}.}
     Using the fact that the terms $c, \varphi$ and $D_{\psi}\sigma(\varphi_s)$ are uniformly bounded, and that $(p,k)$ belongs to $D(\mathbb{F},\mathbb{R}^n,{\mathbb P}) \times (\cap_{\beta\in(0,1)} M^\beta(\mathbb{F},\mathbb{R}^{n\times d},{\mathbb P}))$, see Lemma \ref {lemma:space-of-adjoint},
     we can derive the following upper-bound
\begin{align*}    &\int_0^{T \wedge \tau_A} 
\left\vert p_s^\top  c_s \varphi_s 
\right\vert \dd s + r\int_0^{T \wedge \tau_A} 
\left\vert \Tr ((D_\psi \sigma_s(\varphi_s))^\top k_s) 
\right\vert \dd s  
\\
 &\leq C\left(\int_0^{T} |p_s|\dd s +r \int_0^{T }  |k_s|  \dd s  \right).
 \end{align*}
When $r=0$, we deduce that the left-hand side is integrable, uniformly with respect to $A$. When $r=1$, the proof is more involved. We recall from Lemmas 
\ref{lemma:r-h}
and
\ref{lemma:space-of-adjoint} that $\vert p_t \vert \leq C \bar q_t$ and $\vert k_t \vert \leq \bar q_t ( \vert h_t \vert + \vert \bar Z_t^\star \vert)$, 
for all 
$t \in [0,T]$, 
where 
$h \in L^2({\mathbb F},{\mathbb R}^{n \times d},\bar{\mathbb Q})$. Since 
$\bar{Z}^\star$ belongs to 
$L^2({\mathbb F},{\mathbb R}^{d},\bar{\mathbb Q})$, we deduce from Cauchy-Schwarz inequality that $k$ belongs to $L^1({\mathbb F},{\mathbb R}^{n\times d},{\mathbb P})$. Therefore, the left-hand side in the above display is also integrable, uniformly with respect to $A>0$.    
\vskip 5pt

    \noindent \textit{Step 3: conclusion.}
     Thanks to the uniform integrability properties established in the last two steps, we can now pass to the limit in \eqref{eq:p-x-tau-A}. We get
    \begin{align*}
        \mathbb{E}\left[p_{T} \cdot x_{T} \right] =  \mathbb{E}\left[ \int_0^{T} p_s^\top  c_s \varphi_s \dd s + r \int_0^{T}  \Tr ((D_\psi \sigma_s(\varphi_s))^\top k_s)  \dd s \right].
    \end{align*}
    Consider now $(u,v)$ as in the statement of Lemma \ref{lemma:variational-ineq}. 
    Using the fact 
    that $u_T = p_{T} \cdot x_{T}$ and $\mathbb{E}[u_0] \geq 0$, together with Lemma \ref{lemma:variational-ineq}, we obtain
\begin{equation}
\label{eq:end:necessary:condition}
\begin{split}
        0 
        &\leq \mathbb{E}[u_0] 
        \\
        &= \mathbb{E} \left[p_{T} \cdot x_{T} +   \int_0^T \varphi_s \cdot  \bar q_s \nabla_\psi \ell(s,\psi_s)
        \dd s \right]\\
        & = 
        \mathbb{E} \left[ \int_0^T \left\{ \varphi_s \cdot  \left(\bar{q}_s \nabla_\psi \ell(s,\psi_s) +   c_s^\top p_s
        \right)
            + r \Tr ((D_\psi \sigma_s(\varphi_s)))^\top  k_s)\right\}
        \dd s \right] \\
        &= \mathbb{E} \left[  \int_0^T \varphi_s \cdot \nabla_\psi H(s,X_s,\psi_s,p_s,k_s,\bar q_s) \dd s \right],
    \end{split}
    \end{equation}
    where, to get the last line, we used
\eqref{eq:nablapsiH} together with the
identity
\begin{equation*}
\begin{split}
\Tr \left((D_\psi \sigma_s)^\top  k_s)\right)
= 
\sum_{i=1}^n
\sum_{j=1}^d
\sum_{\ell=1}^n
(\sigma_s)_{i,j,k} (\varphi_s)_k 
(k_s)_{j,i}
= 
\varphi_s \cdot 
\left( \sigma_s^{\top} k_s \right).
\end{split}
\end{equation*}
The sequence of inequalities 
\eqref{eq:end:necessary:condition} is true for any  arbitrary perturbation 
$\varphi \in L^\infty({\mathbb F},{\mathbb R}^n)$
such that 
$\psi + \varphi \in {\mathcal A}_{c_2}$. In fact, using the 
property that 
$\psi \in {\mathcal A}_{c_2'}$, where $c_2'<c_2$, 
it is easy to see that, for any given $\varphi \in L^\infty({\mathbb F},{\mathbb R}^n)$, there exists an $\varepsilon_0$
(depending on $\varphi$) such that 
$\psi + \varepsilon_0 \varphi \in {\mathcal A}_{c_2}$. 
Substituting 
$\varepsilon_0 
\varphi
$ for $\varphi$
in \eqref{eq:end:necessary:condition}, this proves that, 
for any 
$\varphi \in L^\infty({\mathbb F},{\mathbb R}^n)$, 
\begin{equation*}
\mathbb{E} \left[  \int_0^T \varphi_s \cdot \nabla_\psi H(s,X_s,\psi_s,p_s,k_s,\bar q_s) \dd s \right]
\geq 0.
\end{equation*}
And then, changing 
$\varphi$ into $-\varphi$, the above inequality is in fact an equality, from which we get that
\begin{equation*}
\nabla_\psi H(t,X_t,\psi_t,p_t,k_t,\bar q_t)=0, \quad \dd \mathbb{P}  \otimes \dd t \text{-a.e.} \ .
\end{equation*}
Since 
$H$ is strictly convex in the variable $\psi$, this shows 
that the third line in 
\eqref{optim:condition-primal} is satisfied and concludes the proof.
\end{proof}

\subsubsection{Sufficient conditions}
This last subsection is dedicated to the proof of the following lemma,
 which
we invoked in the 
proof of 
Theorem 
\ref{thm:sto-max-princ-central-planner} to establish the sufficient condition.
\begin{lemma} \label{lem:sufficient-condition}
    Let $(\psi,p,k,X) \in \mathscr{A}$ be a solution to the first order condition \eqref{optim:condition-primal}. Then, $\psi$ is the (unique)
    minimizer of the problem \eqref{pb:primal} (with $\bar q$ being substituted for $q$ therein).
\end{lemma}

\begin{proof}
    Let $\psi' \in {\mathcal A}_{c_2}$ and $X'\coloneqq X^{\psi'}$ be the associated state.
    By definition of $\mathcal{J}$ (which is here a shorthand notation for ${\mathcal J}(\bar{q},\cdot)$), we have (with a similar shorthand notation for ${\mathcal G}(\bar q_T,\cdot)$)
    \begin{equation}  \label{eq:delta-mathcal-J}
        \mathcal{J}(\psi') - \mathcal{J}(\psi) = \mathcal{G}(X_T') - \mathcal{G}(X_T)  + \mathbb{E}\left[ \int_0^T \bar{q}_s (\ell(s,\psi_s') - \ell(s,\psi_s)) \dd s \right].
    \end{equation}
    By Assumption \ref{eq:G-concave-convex}, the mapping $ \mathcal{G}$ is convex with respect to its last variable. Therefore,   
    \begin{align}  
        \mathcal{G}(X_T') - \mathcal{G}(X_T) & \geq \mathbb{E}\left[ \delta_X \mathcal{G}(X_T) \cdot (X_T' - X_T) \right]   = \mathbb{E}\left[p_T \cdot (X_T'- X_T) \right]. \label{eq:delta-mathcal-G}
    \end{align}
    For a given $A>0$, consider the stopping time
    \begin{equation*}
        \tau_{A} \coloneqq \inf \left\{ t \in [0,T], \; \left |\int_0^t p_s ^\top (\sigma_s(\psi_s') - \sigma_s(\psi_s)) \dd W_s \right\vert +  \left |\int_0^t (X_s'-X_s)^\top k_s \dd W_s \right\vert  \geq A   \right\}.
    \end{equation*}
    By \eqref{eq:adjoint-p-k} and It{\^o}'s formula, we have the following formula (which is the analogue of \eqref{eq:p-x-tau-A})
\begin{align} \nonumber
        &\mathbb{E}\left[p_{T \wedge \tau_{A}} \cdot (X_{T \wedge \tau_{A}}'- X_{T \wedge \tau_{A}}) \right]  \\ &= \mathbb{E}\left[ \int_0^{T \wedge \tau_{A}}\left( p_s^\top c_s (\psi_s'- \psi_s)  + \Tr\left(k_s\left(\sigma(s,\psi_s')  - \sigma(s,\psi_s) \right)^T \right) \right)\dd s\right]. \label{eq:p-X-X-prime-tau-A}
    \end{align} Following the proof of Lemma \ref{lemma:necessary-psi}, 
    we now aim to take the limit $A \to +\infty$ (but the proof is more difficult because the difference $\psi'-\psi$, which is the analogue of $\varphi$ in the proof of Lemma \ref{lemma:necessary-psi}, is not bounded). The strategy is to prove that the random variables inside the expectations are uniformly integrable with respect to $A$.
\vskip 5pt

    \noindent \textit{Step 1: left-hand side of \eqref{eq:p-X-X-prime-tau-A}.}
In this step, we check that the left-hand side on \eqref{eq:p-X-X-prime-tau-A} is uniformly integrable with respect to $A$. Throughout, we consider a fixed event $E \in {\mathcal F}_T$.
 By Lemma \ref{lemma:r-h}, 
    the process $p$ is equal to the process $\bar{q} \varrho$. Therefore,  
    \begin{align*}
        \mathbb{E}\left[ {\mathds 1}_E |p_{T \wedge \tau_A} \cdot (X_{T \wedge \tau_{A}}'- X_{T \wedge \tau_{A}}) | \right] & = C \mathbb{E}\left[  {\mathds 1}_E \bar{q}_{T \wedge \tau_A} |\varrho_{T \wedge \tau_A} \cdot (X_{T \wedge \tau_{A}}'- X_{T \wedge \tau_{A}}) |\right].
    \end{align*}
    We recall that $X,X' \in S^{2-r}(\mathbb{F},\mathbb{R}^n,\bar{\mathbb{Q}})$, see  Lemma \ref{lemma:reg-X-psi-A}. In addition, by Lemma \ref{lemma:r-h}, the process $\varrho$ belongs to $S^2(\mathbb{F},\mathbb{R}^n,\bar{\mathbb{Q}})$ when $r = 0$ and $L^\infty(\mathbb{F},\mathbb{R}^n)$ when $r = 1$.
    
    Let us first study the case $r = 0$. 
 By Young's inequality, we have, for any $\varepsilon >0$,
\begin{align*}    &\mathbb{E}\left[
    {\mathds 1}_E \bar{q}_{T \wedge \tau_A} |\varrho_{T \wedge \tau_A} \cdot (X_{T \wedge \tau_{A}}'- X_{T \wedge \tau_{A}}) |\right] 
        \\
&\leq  \varepsilon^{-1}  \mathbb{E}\left[ {\mathds 1}_E 
\bar{q}_{T \wedge \tau_A} \vert \varrho_{T \wedge \tau_A}
\vert^2 
\right]
        +
\varepsilon
\mathbb{E}\left[  
\bar{q}_{T \wedge \tau_A} \vert X_{T \wedge \tau_A}
- X_{T \wedge \tau_A}'
\vert^2 
\right]
        \\
        & \leq C \varepsilon^{-1} \mathbb{E}\left[
        {\mathbb P}(E \vert 
        {\mathcal F}_{T \wedge \tau_A})
        \bar{q}_T\vert \varrho_T^* |^2\right]   + 
        C \varepsilon
        \mathbb{E}\left[  \bar{q}_{T} | X_T^*  |^2\right]  + 
        C \varepsilon \mathbb{E}\left[  \bar{q}_{T} | (X')_T^* |^2\right].\end{align*}
        Observing from Doob's inequality that $\sup_{t \in [0,T]}{\mathbb P}(E\vert {\mathcal F}_t)$
   tends to $0$ in $L^1({\mathbb P})$ as ${\mathbb P}(E)$
   tends to $0$, we deduce that the left-hand side tends to $0$ with ${\mathbb P}(E)$, uniformly in $A$.
   This proves the expected uniform integrability property when $r=0$.
   
    When $r = 1$, we  have 
\begin{align*} \mathbb{E}\left[
         {\mathds 1}_E \bar{q}_{T \wedge \tau_A} |\varrho_{T \wedge \tau_A} \cdot (X_{T \wedge \tau_{A}}'- X_{T \wedge \tau_{A}}) |\right] &\leq
 C\mathbb{E}\left[
         {\mathds 1}_E \bar{q}_{T \wedge \tau_A} |X_{T \wedge \tau_{A}}'- X_{T \wedge \tau_{A}}) |\right]
         \\
        & \leq C \mathbb{E}\left[ 
    {\mathbb P}(E \vert {\mathcal F}_{T \wedge \tau_A})
        \bar{q}_{T} | (X - X')_T^* |\right],
\end{align*}
and we conclude as in the case $r=0$.
\color{black} We deduce that 
\begin{equation} \label{eq:lim-first-part-p-delta-X-A}
        \lim_{A \to +\infty} \mathbb{E}\left[p_{T \wedge \tau_{A}} \cdot (X_{T \wedge \tau_{A}}'- X_{T \wedge \tau_{A}}) \right] = \mathbb{E}\left[p_{T} \cdot (X_{T}'- X_{T}) \right].
\end{equation}
    \noindent \textit{Step 2: right-hand side of \eqref{eq:p-X-X-prime-tau-A}.}
    Consider now the term on the right-hand side of \eqref{eq:p-X-X-prime-tau-A}. Obviously, 
\begin{align*}
        & \int_0^{T \wedge \tau_{A}}\left | p_s^\top c_s (\psi_s'- \psi_s)  + \Tr\left(k_s\left(\sigma(s,\psi_s')  - \sigma(s,\psi_s)\right)^T \right)\right|\dd s 
        \\
        &\leq  \int_0^{T}\left | q_s \varrho_s^\top c_s (\psi_s'- \psi_s)  + q_s \Tr\left((h_s+\varrho_s \otimes \bar Z^\star_s)\left(\sigma(s,\psi_s')  - \sigma(s,\psi_s) \right)^T \right) \right |\dd s.
\end{align*}
For our purpose, it suffices to prove the integrability of the right-hand side.     
  When $r = 0$, the volatility term is independent of the control and reduces to $\sigma(t,\psi_t) = \nu_t$, see 
  Assumption \ref{hyp:sigma} . Moreover,  $(\varrho,h) \in S^2(\mathbb{F},\mathbb{R}^n,\bar{\mathbb{Q}}) \times M^2(\mathbb{F},\mathbb{R}^{n\times d}, \bar{\mathbb{Q}})$. Then, by Cauchy-Schwarz and Young inequalities,
\begin{align*}
        \mathbb{E}\left[ \int_0^{T}  \bar{q}_s \left | \varrho_s c_s (\psi_s'- \psi_s) \right |\dd s  \right]  &\leq  C \mathbb{E}\left[ \int_0^{T}  \bar{q}_s (| \varrho_s|^2 + |\psi_s'|^2  + |\psi_s|^2 )  \dd s \right]^{1/2}  < +\infty,\end{align*}
    and
    \begin{align*}
         \mathbb{E}\left[ \int_0^{T} \bar{q}_s\left  |\Tr\left((h_s + \varrho_s \otimes \bar Z^\star_s)\left(\sigma(s,\psi_s')  - \sigma(s,\psi_s) \right)^T \right) \right |\dd s\right] = 0.
    \end{align*}
    since $\sigma(s,\psi_s) = \sigma(s,\psi_s') = \nu_s$ almost surely, for all $s \in [0,T]$.
    Combining the last two displays yields
    \begin{equation} \label{ineq:second-term-hamiltonian-convergence}
        \mathbb{E}\left[ \int_0^{T}\left | p_s^\top c_s (\psi_s'- \psi_s)  + \Tr\left(k_s\left(\sigma(s,\psi_s')  - \sigma(s,\psi_s) \right)^T \right) \right|\dd s\right] < +\infty.
    \end{equation}
    Now, when $r = 1$, the volatility term $\sigma(s,\psi_s') - \sigma(s,\psi_s)$  
 is no longer null and depends linearly on the difference 
 $\psi_s' -\psi_s$. Moreover, 
 the process
 $\varrho$ is bounded, see 
 Lemma \ref{lemma:r-h}. Therefore,\begin{align*}
        \mathbb{E}\left[ \int_0^{T}  \bar{q}_s \left | \varrho_s c_s (\psi_s'- \psi_s)\dd s \right | \right]  &\leq  C \mathbb{E}\left[ \int_0^{T}  \bar{q}_s (|\psi_s'|  + |\psi_s|)  \dd s \right] < +\infty,
    \end{align*}
and,  by
Cauchy-Schwarz and Young inequalities, we further have
\begin{align*}
         & \mathbb{E}\left[ \int_0^{T} \bar{q}_s\left  |\Tr\left((h_s + \varrho_s \otimes \bar Z^\star_s)\left(\sigma(s,\psi_s')  - \sigma(s,\psi_s) \right)^T \right) \right |\dd s\right] \\ &\leq  C \mathbb{E}\left[ \int_0^{T}  \bar{q}_s (| h_s|^2 +  |Z_s^\star|^2 + |\psi_s'|^2 + |\psi_s|^2) \dd s \right]  < +\infty,
    \end{align*}
    so that \eqref{ineq:second-term-hamiltonian-convergence} is also true when $r=1$. Thus, by dominated convergence theorem, we have, for $r \in \{0,1\}$,
\begin{align} \nonumber
        &\lim_{A \to +\infty} \mathbb{E}\left[ \int_0^{T \wedge \tau_{A}}\left( p_s^\top c_s (\psi_s'- \psi_s)  + \Tr\left(k_s\left(\sigma(s,\psi_s')  - \sigma(s,\psi_s) \right)^T \right) \right)\dd s\right]  \\
         &= \mathbb{E}\left[ \int_0^{T}\left( p_s^\top c_s (\psi_s'- \psi_s)  + \Tr\left(k_s\left(\sigma(s,\psi_s')  - \sigma(s,\psi_s) \right)^T \right) \right)\dd s\right]. \label{eq:lim-second-part-psi-trace}
    \end{align}

    \noindent \textit{Step 3: conclusion.}
    By 
 \eqref{eq:lim-first-part-p-delta-X-A} and \eqref{eq:lim-second-part-psi-trace},
     we can pass to the limit in  \eqref{eq:p-X-X-prime-tau-A}, letting $A$ tend to $+\infty$ therein. Informally, this means that we can substitute $T$ for $T \wedge \tau_A$ in \eqref{eq:p-X-X-prime-tau-A}, from which we get
    \begin{equation*}
        \mathbb{E}\left[p_{T} \cdot (X_{T}'- X_{T}) \right] = \mathbb{E}\left[ \int_0^{T}\left( p_s^\top c_s (\psi_s'- \psi_s)  + \Tr\left(k_s\left(\sigma(s,\psi_s')  - \sigma(s,\psi_s) \right)^T \right) \right)\dd s\right].
\end{equation*}
     Finally, using that $\nabla_\psi H(t,X_t,\psi_t,p_t,k_t,\bar{q}_t) = 0$ (see \eqref{eq:nablapsiH}),
   we have
   $\bar q_t \nabla_{\psi}
   \ell(t,\psi_t) 
   + c_t^{\top} p_t + 
   r {\rm Tr}(\sigma_t^\top
   k_t)=0$. 
    Combining
    the above display with \eqref{eq:delta-mathcal-J} and \eqref{eq:delta-mathcal-G}, we get
    \begin{align*} 
        \mathcal{J}(\psi') - \mathcal{J}(\psi) \geq \mathbb{E}\left[ \int_0^T \bar{q}_s\left(   \ell(s,\psi_s') - \ell(s,\psi_s) -  \nabla_{\psi} \ell(s,\psi_s) \cdot (\psi_s' - \psi_s) \right) \dd s\right] \geq 0,
    \end{align*}
    where, 
    by strict convexity of $\ell$ and strict positivity of $\bar q$,
    the last inequality is strict whenever $\psi \neq \psi'$.
\end{proof}

\appendix

\section{Finite entropy and positive measures} \label{appendix:representation}

The aim of this appendix is to clarify the representation of positive measures with finite entropy, a task that is nontrivial due to their limited integrability properties.

\begin{lemma} \label{lemma:representation-q}
Let $q_T$ be an ${\mathcal F}_T$-measurable random variable with values in ${\mathbb R}_+$. 
Under the conditions  
${\mathbb E}[h(q_T)] <+\infty$,
${\mathbb P}(\{q_T >0\})=1$
and
${\mathbb E}[q_T]=1$, there exists a unique progressively measurable process 
$Z^\star$ (with values in ${\mathbb R}^d$) such that 
\begin{equation}
\label{eq:lem:42:1:rep:q_T}
q_T = 1 + 
\int_0^T q_s 
 Z_s^\star \cdot  \dd W_s.
\end{equation}
It satisfies
$\tfrac12 {\mathbb E}
[ \int_0^T q_s \vert Z_s^\star \vert^2 \dd s ]
= {\mathbb E}[h(q_T)+1]$.
Moreover, $q_T$ can be represented as 
\begin{equation}
\label{eq:lem:42:2:rep:doleans}
q_T 
= 
{\mathcal E}_T
\left(
\int_0^\cdot Z_s^\star 
\cdot \dd W_s \right).
\end{equation}
Lastly, the process 
$(q_t \coloneqq ({\mathcal E}_t(\int_0^{\cdot} Z_s^\star \cdot \dd W_s))_{t \in [0,T]}$ is true a martingale. It satisfies ${\mathbb E}[q_T^\star] < + \infty$, and is the unique solution to the SDE
\begin{equation}
\label{eq:lem:42:2:rep:doleans:bbbb}
q_t = 1 + \int_0^t q_s Z_s^\star \cdot \dd W_s, \quad 
t \in [0,T],
\end{equation}
within the class of continuous positive-valued and ${\mathbb F}$-adapted processes. 
\end{lemma}

As a consequence of the above lemma, the class $\mathcal{Q}$ is parameterized by the sole given data $q$ and $Y^\star$. Indeed, it suffices to apply the lemma above to the process $(\exp(-\int_0^t Y_s^\star \dd s) q_t)_{t \in [0,T]}$ to obtain the representation in equation \eqref{eq:q}, namely 
\begin{equation*}
        q_T = 1 + \int_0^T q_s Y_s^\star \dd s + \int_0^T q_s Z_s^\star \cdot \dd W_s.
\end{equation*}

\begin{proof}
\textit{Step 1: representation of $q_T$ in the form 
\eqref{eq:lem:42:1:rep:q_T}.}
Let 
$q_t = {\mathbb E}[q_T \vert {\mathcal F}_t]$, for $t \in [0,T]$. 
It is a strictly positive 
(continuous) martingale. 
By $L\log L$-Doob's maximal inequality, we deduce that 
${\mathbb E}[q_T^*] < + \infty.$

For any integer $m \geq 1$, we let 
$q_T^m \coloneqq q_T \wedge m$. 
By martingale representation theorem, we can 
write 
\begin{equation*}
q_T^m = 
{\mathbb E}[q_T^m]
+ \int_0^T  Z_s^m \cdot \dd W_s. 
\end{equation*}
Letting 
$q_t^m= {\mathbb E}[q_T^m \vert {\mathcal F}_t]$, for $t \in [0,T]$, we invoke $L\log L$-Doob's maximal inequality again to deduce that there exists a universal constant $C>0$
such that
\begin{equation*}
{\mathbb E}
\left[ \sup_{t \in [0,T]} \vert q_t^m - q_t \vert
\right]
\leq 
C {\mathbb E}
\left[\max\left \{0,\vert q^m_T-q_T\vert \ln\left( 
\vert q^m_T-q_T\vert \right)\right\} \right].
\end{equation*}
Obviously (since $q_T^m=q_T$ if $q_T \leq m$), 
\begin{equation*}\vert q^m_T-q_T\vert \ln\left( 
\vert q^m_T-q_T\vert \right)
\leq q_T \vert \ln (2 q_T)\vert.
\end{equation*}
By dominated convergence theorem, we deduce that
\begin{equation*}
0 \leq 
\limsup_{m \rightarrow \infty}
{\mathbb E}
\left[ \sup_{t \in [0,T]} \vert q_t^m - q_t \vert
\right]
\leq C
\lim_{m \rightarrow \infty}
{\mathbb E}
\left[\max\left\{0,\vert q^m_T-q_T\vert \ln\left( 
\vert q^m_T-q_T\vert \right) \right\} \right]
=
0.
\end{equation*}
In particular, 
\begin{equation*}
\lim_{m \rightarrow \infty}
\sup_{k \in {\mathbb N}}
{\mathbb E}
\left[ 
\sup_{t \in [0,T]} \vert q_t^{m+k} - q_t^{m} \vert
\right]
=0.
\end{equation*}
By B{\"u}rkolder-Davies-Gundy inequality, we further obtain that 
\begin{equation*}
\lim_{m \rightarrow \infty}
\sup_{k \in {\mathbb N}}
{\mathbb E}\left[ \left(
\int_0^T  
\vert Z_t^m-Z_t^{m+k}
\vert^2 \dd t
\right)^{1/2}
\right]
=0.
\end{equation*}
By completeness of 
$L^1(\Omega \times [0,T],{\mathcal R},{\mathbb P} \otimes {\rm Leb}_{[0,T]})$ 
(where 
${\mathcal R}$ is the progressive $\sigma$-field), 
we deduce that there exists 
a process 
$Z$
satisfying 
\begin{equation*}
{\mathbb E}
\left[ 
\left( \int_0^T
\vert Z_t \vert^2 
\dd t \right)^{1/2}
\right] < +\infty, 
\end{equation*}
and
\begin{equation*}
q_T = 1 + \int_0^T  Z_s
\cdot \dd W_s.
\end{equation*}
Since 
$(q_t)_{t \in [0,T]}$ is continuous and 
strictly positive, we can let 
$Z_t^\star = Z_t/q_t$. 
We then  notice, from It\^o's formula, that the process
\begin{equation*}
\left( h(q_t) - \frac12 \int_0^t 
q_s \vert Z_s^\star\vert^2 \dd s
\right)_{t \in [0,T]}
\end{equation*}
is a local martingale. In particular, one can find a sequence of stopping times 
$(\sigma_m)_{m \geq 1}$, converging 
to $T$ as $m$ tends to $+\infty$, 
such that for all $m \geq 1$,
\begin{equation}
\label{eq:patch:entropy:identification}
{\mathbb E}\left[ h\left(q_{T \wedge \sigma_m}\right)
\right] 
= \frac12 
\mathbb E\left[
\int_0^{T \wedge \sigma_m}
q_s \vert Z_s^\star\vert^2 
\dd s\right] .
\end{equation}
We now prove that 
\begin{equation*}
\sup_{m \geq 1}
{\mathbb E}\left[ h\left(q_{T \wedge \sigma_m}\right)
\right] < +\infty.
\end{equation*}
Since
\begin{equation*}
q_{T \wedge \sigma_m}
= 
{\mathbb E}
\left[ q_T \vert 
{\mathcal F}_{T \wedge \sigma_m} \right].
\end{equation*}
and because $h$ is convex, we deduce that 
\begin{equation*}
{\mathbb E}
\left[ h\left( q_{T \wedge \sigma_m}\right) \right]
\leq 
{\mathbb E}[ h(q_T)],
\end{equation*}
from which we get, by applying Fatou's lemma to the right-hand side on \eqref{eq:patch:entropy:identification}, that
\begin{equation*}
\frac12 {\mathbb E}
\left[ \int_0^T q_s \vert Z_s^\star \vert^2 \dd s \right]
\leq {\mathbb E}[h(q_T)].
\end{equation*}
Conversely, by observing 
that the function $h$ is lower bounded and then by applying Fatou's lemma to the left-hand side on 
\eqref{eq:patch:entropy:identification}, we get 
\begin{equation*}
{\mathbb E}[h(q_T)]
\leq
\frac12 {\mathbb E}
\left[ \int_0^T q_s \vert Z_s^\star \vert^2 \dd s \right].
\end{equation*}
By the last two identities, we deduce that the two terms on the above inequality are equal.
\vskip 5pt

\noindent 
\textit{Step 2: representation of $q_T$ in the form
\eqref{eq:lem:42:2:rep:doleans}.}
Recalling that 
$(q_t)_{t \in [0,T]}$, as defined in the first step, is 
 continuous and 
strictly positive, we deduce that, a.s., 
it is lower bounded by a positive constant. 
This proves that, a.s., 
\begin{equation*}
    \int_0^T \vert Z_s^\star \vert^2 \dd s < +\infty, 
    \end{equation*}
which makes it possible to define 
$({\mathcal E}_t(\int_0^\cdot Z_s^\star \cdot \dd W_s))_{t \in [0,T]}$. 

Moreover, one can compute the logarithm of the process $q$. By It\^o's
formula, we get
\begin{equation*}
\ln(q_T) 
= - \frac12 \int_0^T 
\vert Z_s^\star \vert^2 
\dd s + \int_0^T Z^\star_s  \cdot \dd W_s,
\end{equation*}
which says that 
\begin{equation*}
q_T = {\mathcal E}_T
\left(\int_0^\cdot Z_s^\star \cdot \dd W_s\right).
\end{equation*}
    This gives the expected representation of $q_T$, and more generally of 
    the process $(q_t={\mathbb E}[q_T \vert {\mathcal F}_t])_{t \in [0,T]}$.
    
Uniqueness to 
\eqref{eq:lem:42:2:rep:doleans:bbbb}
can be easily checked by computing 
$(q_t'/q_t)_{t \in [0,T]}$ and then by proving that this ratio is constant (equal to $1$), for any other solution $q'=(q_t')_{ t \in [0,T]}$.     
\end{proof}

\section{A priori estimates on a linear SDE} \label{appendix:apriori-SDE}

Let $q \in \mathcal{Q}$ and $\psi \in \mathcal{A}$.
In this section we provide technical results for controlled linear SDEs, including $S^1(\mathbb{F},\mathbb{Q})$ and $S^2(\mathbb{F},\mathbb{Q})$ regularity of the solutions, for $\mathbb{Q} = q  \mathbb{P}$.
With the same notations as in \ref{hyp:b} and \ref{hyp:sigma} , we consider the  controlled equation \eqref{eq:intro:X}, namely
\begin{equation} \label{eq:state}
    \dd X_t  = \left(a_t + b_t X_t + c_t \psi_t\right) \dd t + \left(\nu_t + r \sigma_t(\psi_t)\right)\dd W_t,  \quad X_0 = \eta.
\end{equation}
Following \eqref{hyp:gamma}, we denote by $\Gamma=(\Gamma_t)_{t \in [0,T]}$ the resolvent associated with the linear part of the equation, namely the solution (with values in 
the space of $n \times n$
matrices)
of
\begin{equation*}
        \frac{\dd}{\dd t} \Gamma_t =  b_t \Gamma_t, \quad 
        t \in [0,T],
        \quad 
        \Gamma_0 = I_n,
    \end{equation*}
    with $I_n$ standing for the $n \times n$ identity matrix.
It is well-known that $\Gamma_t$ is invertible for any 
$t \in [0,T]$ and that the solution to \eqref{eq:state} can be represented as
\begin{equation*}
X_t = 
\Gamma_t
\left( \eta + \int_0^t \Gamma_s^{-1} \left(a_s + c_s \psi_s\right) \dd s
+ \int_0^t 
\Gamma_s^{-1} \left(\nu_s + 
r \sigma_s (\psi_s)\right)
\dd W_s \right), 
\quad t \in [0,T].
\end{equation*}
Recalling from assumptions 
\ref{hyp:b} and \ref{hyp:sigma} that  $\|\eta\|_{L^\infty(\mathcal{F}_0,\mathbb{R}^{n})} + \|a\|_{L^\infty(\mathbb{F},\mathbb{R}^{n})} + \|b\|_{L^\infty(\mathbb{F},\mathbb{R}^{n \times n})} +  \|c\|_{L^\infty(\mathbb{F},\mathbb{R}^{n \times n})} +\|\nu\|_{L^\infty(\mathbb{F},\mathbb{R}^{n\times d})} + \|\sigma\|_{L^\infty(\mathbb{F},\mathbb{R}^{n\times d})}< 
 + \infty$, we deduce that there exists a 
constant $C >0$, independent of $\psi$, such that 
    \begin{align} 
    \label{ineq:X-tau-psi-M}
        X_T^* \leq C \left( 1+ \int_0^T \left\vert \psi_s \right\vert\dd s \right)  + M_T^*, \quad M_t \coloneqq \Gamma_t \int_0^t  
    \Gamma_s^{-1} (\nu_s + r \sigma_s(\psi_s)) \dd W_s.
    \end{align}


\begin{lemma} \label{lemma:reg-X-psi-A} 
The unique solution $X$ to \eqref{eq:state} belongs to $S^{2-r}(\mathbb{F},\mathbb{Q},\mathbb{R}^n)$, where $r \in \{0,1\}$ is as in 
\ref{hyp:sigma}, and
there exists a constant $C$, independent of $q$ and $\psi$, such that 
\begin{equation*}
     \mathbb{E}\left[ q_T |X_T^*|^{2-r} \right] \leq C \left( 1 + \mathcal{S}(q) + \mathcal{S}^\star(\psi)\right).
\end{equation*}
\end{lemma}

\begin{proof}
We distinguish between the two cases $r = 0$ and $r = 1$.

\noindent \textit{Step 1: case $r = 0$.} We have that
\begin{align*}
    |X_T^*|^{2} \leq 2 C\left( 1 + \int_0^T |\psi_s|^{2} \dd s\right) +
    2\sup_{t \in [0,T]} \left|\Gamma_t  \int_0^t \Gamma_s^{-1} \nu_s \dd W_s\right|^{2} .
\end{align*}
Using that $\|\nu\|_{L^\infty(\mathbb{F},\mathbb{R}^{n\times d})} + \|\Gamma^{-1}\|_{L^\infty(\mathbb{F},\mathbb{R}^{n\times n})} < +\infty$ and 
\cite[Theorem IV.37.8]{rogers2000diffusions}, we know that 
$M_T^*$ (whose square appears on the right-hand side) has sub-Gaussian tails, i.e., 
$M_T^*$ belongs to $L^{2,\vartheta}_{\exp}({\mathcal F}_T,{\mathbb R})$ for a certain $\vartheta >0$. 
Therefore, by
the duality inequalities 
\eqref{eq:ineq:duality:with:r}
and
\eqref{ineq:S-S-star}, there exists a constant $C>0$ which might increase from line to line such that 
 \begin{align*}
          \mathbb{E}\left[q_T |X_T^*|^{2} \right] & \leq C\left( 1 + \mathbb{E}\left[q_T  \int_0^T |\psi_s|^2 \dd s \right]
          + {\mathbb E}[h(q_T)]
          + {\mathbb E}\left[ \exp(\vartheta (M_T^*)^2\right]\right) \\& \leq  C\left(1+ \mathcal{S}(q) + \mathcal{S}^\star(\psi)\right)
     < +\infty.
\end{align*}
\vskip 3pt
\noindent \textit{Step 2: case $r = 1$.}  In this situation we have 
\begin{align*}
    |X_T^*| \leq C\left( 1 + \int_0^T |\psi_s| \dd s \right) + \left| M_T^* \right|.
\end{align*}
Introducing $(\tilde{W}_t = W_t - \int_0^t Z^\star_s \dd s)_{t \in [0,T]}$ and applying Girsanov's theorem (see 
Lemma \ref{lemma:representation-q}), 
we have 
 \begin{align} \nonumber
          \mathbb{E}\left[q_T  |X_T^*| \right]  \leq & C\mathbb{E}\left[q_T \left( 1 + \int_0^T |\psi_s| \dd s + \left| M_T^* \right|\right) \right] \nonumber  \\ \nonumber \leq & 
          C\mathbb{E}\left[q_T \left( 1 + \int_0^T |\psi_s| \dd s + \left| \sup_{t \in [0,T]} \Gamma_t \int_0^t \Gamma_s^{-1} (\nu_s + \sigma_s(\psi_s))\dd \tilde{W}_s \right| \right) \right] \\ & + C\mathbb{E}\left[q_T \left|\sup_{t \in [0,T]}  \Gamma_t \int_0^t  \Gamma_s^{-1} (\nu_s + \sigma_s(\psi_s))Z^\star_s \dd s\right| \right]. \label{ineq:q-X-and-girsanov}
\end{align}
On the one hand, by It{\^o}'s isometry (under the measure $\tilde{\mathbb Q} \coloneqq{\mathcal E}_T(\int_0^\cdot Z_s^\star \cdot \dd W_s) {\mathbb P}$) and the assumption  \ref{hyp:sigma}  on $\sigma$, we have that 
\begin{equation} \label{ineq:itometrie}
    \mathbb{E}\left[q_T \left|\sup_{t \in [0,T]} \Gamma_t \int_0^t  \Gamma_s^{-1} \sigma_s(\psi_s) \dd \tilde{W}_s \right|\right] \leq C \mathbb{E}^{\tilde{\mathbb{Q}}}\left[\left( \int_0^T  \left[ 1 + |\psi_s|^2
    \right] \dd s \right)^{1/2}\right].
\end{equation}
On the other hand, by Fenchel-Young inequality, we also have that 
\begin{equation} \label{ineq:q-psi-z-star}
    \mathbb{E}\left[q_T \left| \sup_{t \in [0,T]} \Gamma_t \int_0^t \Gamma_s^{-1} \sigma_s(\psi_s) Z^\star_s \dd s \right|\right] \leq C \mathbb{E}\left[q_T  \int_0^T \left[ 1 + |\psi_s|^2 +  |Z^\star_s|^2\right] \dd s\right].
\end{equation}
Combining \eqref{ineq:q-X-and-girsanov} with 
\eqref{ineq:itometrie} and \eqref{ineq:q-psi-z-star}, we obtain that
 \begin{align*}
          \mathbb{E}\left[q_T |X_T^*|\right] & \leq C \mathbb{E}\left[q_T \left(1 + \int_0^T |\psi_s|^2 \dd s +  \int_0^T|Z^\star_s|^2 \dd s \right) \right] \\& \leq C\left( 1 + \mathbb{E}\left[q_T \int_0^T |\psi_s|^2 \dd s\right] + \mathbb{E}\left[h(q_T)\right] \right) \\& \leq  C\left( 1 + \mathcal{S}(q) + \mathcal{S}^\star(\psi) \right)
           < +\infty,
\end{align*}
where the last two lines follow by duality inequality \eqref{ineq:S-S-star} between $\mathcal{S}$ and $\mathcal{S}^\star$, concluding the proof.
\end{proof}

In fact, the proof of Lemma \ref{lemma:reg-X-psi-A} can be easily  re-examined to get the following variant:
\begin{lemma}
\label{lem:E:q:X*:S(q):Sstar(psi)}
Let $r \in \{0,1\}$
be as in \ref{hyp:sigma}
and $X$ be the unique solution to \eqref{eq:state}. Then, for any $\varepsilon \in (0,1)$, there exist two constants $C_\varepsilon >0$ and  $c_\varepsilon >0$, independent of $q$ and $\psi$,  such that 
\begin{equation*}
{\mathbb E}\left[ q_T 
\left\vert X_T^* \right\vert \right]
\leq C_{\varepsilon} + c_{\varepsilon} {\mathcal S}(q) + 
\varepsilon 
\mathbb{E}\left[q_T  \int_0^T |\psi_s|^2 \dd s\right], 
\end{equation*}
where the second constant is explicitly given by
\begin{equation*}
    c_\varepsilon  = 2 \beta e^{\alpha T} \|\Gamma\|_{L^\infty(\mathbb{F},\mathbb{R}^{n \times n})} \|\Gamma^{-1}\|_{L^\infty(\mathbb{F},\mathbb{R}^{n \times n})}\left(\|\nu\|_{L^\infty(\mathbb{F},\mathbb{R}^{n \times d})} + \frac{3}{\varepsilon}e^{\alpha T} \|\sigma\|_{L^\infty(\mathbb{F},\mathbb{R}^{n \times d \times n})} \right).
\end{equation*}
\end{lemma}

\begin{proof}
It suffices to adapt the computations developed in the second step of the proof of Lemma \ref{lemma:reg-X-psi-A} (whether $r$ is equal to $0$ or $1$). 
Throughout the proof, the value of $\varepsilon \in (0,1)$ is fixed. 
In the following, we shall use repeatedly that, for any $s \in [0,T]$, $\mathbb{E}[q_T] \leq \exp(\alpha T)  \mathbb{E}[q_s]$ and 
\begin{equation*}
    \mathbb{E}\left[q_T \int_0^T |\psi_s|^2 \dd s \right] = \mathbb{E}\left[\int_0^T \mathbb{E}[q_T \vert \mathcal{F}_s] |\psi_s|^2 \dd s \right] \leq e^{\alpha T}\mathbb{E}\left[\int_0^T q_s |\psi_s|^2 \dd s \right].
\end{equation*}
The first term on the right-hand side of \eqref{ineq:q-X-and-girsanov} can be easily bounded, by means of Young's inequality. We obtain
\begin{align*}
{\mathbb E}
\left[ q_T \left( 1 + 
\int_0^T \vert \psi_s \vert \dd s 
\right) 
\right] & \leq {\mathbb E}
\left[ q_T \left( 1 + C_\varepsilon  + \frac{\varepsilon}{3} 
\int_0^T \vert \psi_s \vert^2 \dd s 
\right) 
\right]
\\
& \leq C_\varepsilon + \frac{\varepsilon}{3}{\mathbb E}
\left[
\int_0^T q_s \vert \psi_s \vert^2 \dd s  
\right].
\end{align*}
Similarly, by Jensen's and Young's inequalities, \eqref{ineq:itometrie}
yields
\begin{align*}   \mathbb{E}&\left[q_T \left|\sup_{t \in [0,T]} \Gamma_t \int_0^t  \Gamma_s^{-1} \sigma_s(\psi_s) \dd \tilde{W}_s \right|\right] \\ \leq &\mathbb{E} \left[q_T \left|  \int_0^T \Gamma_t \Gamma_s^{-1} \sigma_s(\psi_s) \dd s \right|^2\right]^{1/2} \\
 \leq & T \|\Gamma\|_{L^\infty(\mathbb{F},\mathbb{R}^{n \times n})}
\|\Gamma^{-1}\|_{L^\infty(\mathbb{F},\mathbb{R}^{n \times n})}\|\nu\|_{L^\infty(\mathbb{F},\mathbb{R}^{n \times d})} \mathbb{E} \left[q_T\right]^{1/2}\\
 & +  
r  \|\Gamma\|_{L^\infty(\mathbb{F},\mathbb{R}^{n \times n})} 
\|\Gamma^{-1}\|_{L^\infty(\mathbb{F},\mathbb{R}^{n \times n})}
 \|\sigma\|_{L^\infty(\mathbb{F},\mathbb{R}^{n \times d \times n})} 
\mathbb{E}\left[q_T  \int_0^T |\psi_s|^2 \dd s \right]^{1/2} 
\\ \leq &C_\varepsilon + \frac{\varepsilon}{3}{\mathbb E}
\left[
\int_0^T q_s \vert \psi_s \vert^2 \dd s  
\right].
\end{align*}
And, \eqref{ineq:q-psi-z-star} can be rewritten as 
\begin{equation*}
\begin{split}
    \mathbb{E}&\left[q_T \left| \sup_{t \in [0,T]} \Gamma_t \int_0^t \Gamma_s^{-1} \sigma_s(\psi_s) Z^\star_s \dd s \right|\right] 
    \\
    \leq &\|\Gamma\|_{L^\infty(\mathbb{F},\mathbb{R}^{n \times n})} \|\Gamma^{-1}\|_{L^\infty(\mathbb{F},\mathbb{R}^{n \times n})}\|\nu\|_{L^\infty(\mathbb{F},\mathbb{R}^{n \times d})}\mathbb{E}\left[q_T  \int_0^T   |Z^\star_s| \dd s\right] 
    \\
    & + 
  \|\Gamma\|_{L^\infty(\mathbb{F},\mathbb{R}^{n \times n})} \|\Gamma^{-1}\|_{L^\infty(\mathbb{F},\mathbb{R}^{n \times n})} \|\sigma\|_{L^\infty(\mathbb{F},\mathbb{R}^{n \times d \times n})} \mathbb{E}\left[q_T  \int_0^T   |\psi_s||Z^\star_s| \dd s\right] \\
     \leq & e^{\alpha T} \|\Gamma\|_{L^\infty(\mathbb{F},\mathbb{R}^{n \times n})}
     \|\Gamma^{-1}\|_{L^\infty(\mathbb{F},\mathbb{R}^{n \times n})}
     \|\nu\|_{L^\infty(\mathbb{F},\mathbb{R}^{n \times d})}\mathbb{E}\left[\int_0^T   q_s (1 + |Z^\star_s|^2) \dd s\right] \\& + 
    \frac{3}{\varepsilon}e^{2\alpha T }\|\Gamma\|_{L^\infty(\mathbb{F},\mathbb{R}^{n \times n})} 
    \|\Gamma^{-1}\|_{L^\infty(\mathbb{F},\mathbb{R}^{n \times n})}\|\sigma\|_{L^\infty(\mathbb{F},\mathbb{R}^{n \times d \times n})}  \mathbb{E}\left[ \int_0^T   q_s |Z^\star_s|^2 \dd s\right] \\
    & +
    \frac{\varepsilon}3  
    \mathbb{E}\left[\int_0^T q_s |\psi_s|^2 \dd s\right] \\
     \leq & C + c_\varepsilon \mathcal{S}(q) + \frac{\varepsilon}3  
    \mathbb{E}\left[\int_0^T q_s |\psi_s|^2 \dd s\right],
    \end{split}
\end{equation*}
for some fixed constant $C>0$.
Combining \eqref{ineq:X-tau-psi-M} and the three last displays, we obtain the desired inequality.
\end{proof}

\begin{lemma} \label{lemma:reg-X}
    Let $\psi = 0$ and $X \in S^2(\mathbb{F},\mathbb{R}^n)$ be the associated solution to the state equation \eqref{eq:state}. Then $X^*_T \in S_{\exp}^{2-r,L\vartheta}(\mathcal{F}_T,\mathbb{R}^n)$ when $r = 1$ in \ref{hyp:sigma}. 
    The result holds for $r = 0$ under the condition $4 \vartheta L \| \Gamma \|^2_{L^\infty({\mathbb F},{\mathbb R}^{n \times n})} \|\Gamma^{-1} \|^2_{L^\infty({\mathbb F},{\mathbb R}^{n \times n})} \| \nu \|^2_{L^\infty({\mathbb F},{\mathbb R}^{n \times d})} T < 1$.
\end{lemma}

\begin{proof}
    Inequality \eqref{ineq:X-tau-psi-M} (with $\psi=0$) gives \begin{align*}
        X_T^* \leq C + M_T^*, \quad \textrm{where} \quad M_t = \Gamma_t\int_0^t  \Gamma_s^{-1} \nu_s  \dd W_s.
    \end{align*}
    Raising to the power $2-r$ and multiplying by  $\vartheta L$, and then taking the exponential and the expectation on both sides, we obtain
    \begin{align*}
        \mathbb{E} \left[\exp(\vartheta L |X_T^*|^{2-r}) \right] & \leq C\mathbb{E} \left[\exp\left(2 \vartheta L|M_T^*|^{2-r}\right)  \right],
    \end{align*}
    for a constant $C>0$, which is allowed to increase. 
We recall from 
\cite[Theorem IV.37.8]{rogers2000diffusions}
that, for 
$$\kappa\coloneqq \| \Gamma \|^2_{L^\infty({\mathbb F},{\mathbb R}^{n \times n})} \|\Gamma^{-1} \|^2_{L^\infty({\mathbb F},{\mathbb R}^{n \times n})} \| \nu \|^2_{L^\infty({\mathbb F},{\mathbb R}^{n \times d})} T,$$
it holds 
${\mathbb P}(\{ M_T^\star \geq y\}) \leq \exp(-y^2/2 \kappa)$ for all $y>0$, from which we deduce that ${\mathbb E}[\exp(2\vartheta L (M_T^*)^{2-r})] < +\infty$ when $r=1$. When $r=0$, 
\begin{equation*}
        \mathbb{E} \left[\exp\left( 2\vartheta L|M_T^*|^{2-r}\right)  \right] \leq 1 + 2\vartheta L \int_{\mathbb{R}} e^{ y^{2}\left(2\vartheta L - \frac{1}{2\kappa}\right)}\dd y,
    \end{equation*}
    which is finite whenever $\vartheta L \kappa < 1/4$, hence concluding the proof.
\end{proof}

\section{Uniform integrability and convergence results} \label{appendix:uniform-integrability}
This section focuses on weak convergence of random variables  in the space $L \log L({\mathbb F})$, when tested against random variables having a finite exponential moment, and vice versa.  

\begin{lemma}
\label{lem:uniform:integrability}
Let $(q^k)_{k \in {\mathbb N}}$ be 
a ${\mathcal Q}$-valued sequence, uniformly 
bounded in $L \log L({\mathbb F})$, 
$(y^k)_{k \in {\mathbb N}}$ be an $L^\infty({\mathbb F})$-valued 
sequence, uniformly bounded in 
$L^\infty({\mathbb F})$, 
and $(\xi^k,\ell^k)_{k \in {\mathbb N}}$ be a
collection of random variables in $L^1(\mathcal{F}_T) \times L^1(\mathbb{F})$  satisfying for all $\vartheta >0$, 
\begin{equation}
\label{eq:hypothese:UI:moment:exponentiel} 
\sup_{k \in {\mathbb N}}
{\mathbb E} \left[\exp\left( \vartheta 
\vert \xi^k_T \vert + \int_0^T \exp\left( \vartheta 
\vert \ell^k_s \vert\right) \dd s\right) \right] < + \infty.  
\end{equation} Then, the sequences
of random variables
\begin{equation*}
\left( q^k_T \vert y^k_T \vert 
\vert \xi^k \vert
\right)_{k \in {\mathbb N}}, \quad \left(  q^k_s \vert y^k_s \vert 
\vert \ell^k_s \vert  
\right)_{k \in {\mathbb N}},
\end{equation*}
are respectively uniformly integrable on $\Omega $ and $\Omega \times [0,T]$, 
equipped with ${\mathbb P}$ and ${\mathbb P} \otimes {\rm Leb}_{[0,T]}$, in the sense that 
\begin{equation}
\label{eq:lem:uniform:integrability:1}
\sup_{k \in {\mathbb N}}
{\mathbb E} \left[ q^k_T \vert y^k_T \vert \vert \xi^k \vert \right]< + \infty, \quad 
\sup_{k \in {\mathbb N}}
{\mathbb E} \left[ \int_0^T q^k_s \vert y^k_s \vert \vert \ell^k_s \vert \dd s \right]< + \infty,
\end{equation}
and for every $\epsilon >0$, there exists $\delta >0$ such that
\begin{equation}
\begin{split}
\forall A \in 
{\mathcal F}_T, \quad
{\mathbb P}  (A) 
\leq \delta & \Rightarrow 
{\mathbb E} \left[
q^k_T \vert y^k_T \vert \vert \xi^k \vert 
{\mathds 1}_A\right]\leq \varepsilon, 
\\
\forall A \in {\mathcal B}([0,T]) \otimes 
{\mathcal F}_T, \ 
\left( {\mathbb P} \otimes {\rm Leb}_{[0,T]}
\right) (A) 
\leq \delta &\Rightarrow 
{\mathbb E} \left[
\int_0^T   
q^k_s \vert y^k_s \vert \vert \ell^k_s \vert 
{\mathds 1}_A  
\dd s \right]\leq \varepsilon. 
\end{split}
\label{eq:lem:uniform:integrability:2}
\end{equation} 
\end{lemma}

\begin{proof}
Obviously, we can assume without any loss of generality that 
the processes $(y^k)_{k \in {\mathbb N}}$ are all equal to 1. Moreover, we just make the proof for the sequence $\left(q^k_s \vert y^k_s \vert 
\vert \ell^k_s \vert 
\right)_{k \in {\mathbb N}}$, as the proof for $\left( q^k_T \vert y^n_T \vert 
\vert \xi^k \vert
\right)_{k \in {\mathbb N}}$ is analogous. 
\vskip 4pt

\noindent \textit{Step 1: display   \eqref{eq:lem:uniform:integrability:1} holds.}
The proof of 
\eqref{eq:lem:uniform:integrability:1}
follows from the duality inequality 
\eqref{eq:ineq:duality:with:r}.
If $\vartheta \geq 1$ in 
the latter display, the term $\ln(\vartheta) x$ therein is positive, which leaves us with 
\begin{equation}
\label{eq:duality:appli:UI}  
x^\star x \leq  \frac1{\vartheta} h(x)  + \exp(\vartheta x^\star),
\end{equation}
for every $\vartheta \geq 1$ and for all $x,x^\star >0$. 
We now apply this inequality with 
$x=q^k_T(\omega)$ and $x^\star = \int_0^T \vert \ell^k_s(\omega)\vert \dd s$ for any $\omega \in \Omega$. Choosing $\vartheta=1$, we get 
\begin{align*}
{\mathbb E} \left[ q^k_T \int_0^T  
 \vert \ell^k_s \vert 
\dd s \right] \leq {\mathbb E}\left[ h\left(q^k_T\right) 
\right]+ 
{\mathbb E} \left[ \exp \left(\int_0^T
\vert \ell^k_s \vert \dd s\right)  \right].
\end{align*}
By assumption (see \eqref{eq:hypothese:UI:moment:exponentiel}), the right-hand side is 
uniformly bounded with respect to $k
\in {\mathbb N}$. In order to 
derive 
\eqref{eq:lem:uniform:integrability:1}, it suffices to recall, from the definition of ${\mathcal Q}$, that there exists a 
constant $C>0$ such that, for any 
$k \in {\mathbb N}$, $q^k_s \leq C {\mathbb E}[q^k_T \vert{\mathcal F}_s]$ (with probability 1 under ${\mathbb P}$).
\vskip 4pt 

\noindent \textit{Step 2: display \eqref{eq:lem:uniform:integrability:2} holds.} Fix  
$\epsilon >0$, and then choose $\vartheta \geq 1$ large enough so that \begin{equation*}
\frac1{\vartheta} \sup_{k \in {\mathbb N}} {\mathbb E}
\left[h\left(q^k_T\right)\right]
\leq \frac{\epsilon}{2CT},
\end{equation*}
with $C$ as in the first step.
By \eqref{eq:duality:appli:UI}
(with $x = \vert q^k_T \vert$ and 
$x^\star = \vert \ell^k_t \vert$), we deduce that, for $A \in {\mathcal B}([0,T]) \otimes {\mathcal F}_T$, 
\begin{equation*}
\begin{split}
&{\mathbb E} \left[ 
\int_0^T  q^k_s \vert \ell^k_s \vert 
{\mathds 1}_A  \dd s \right]
\\
&=
\int_0^T
{\mathbb E} \left[ 
q^k_s \vert \ell^k_s \vert 
{\mathds 1}_A(s,\cdot) \right]  \dd s 
\\
&\leq \frac{C}{\vartheta} \int_0^T 
{\mathbb E}\left[ h(q^k_T) 
{\mathbb E}\left[ 
{\mathds 1}_A(s,\cdot) \vert {\mathcal F}_s
\right]
\right]
\dd s +
\int_0^T 
{\mathbb E}\left[
\exp \left( \vartheta \vert \ell^k_s \vert 
\right)
{\mathbb E}\left[ 
{\mathds 1}_A(s,\cdot) \vert {\mathcal F}_s
\right]
\right]
\dd s
\\
&\leq \frac{\epsilon}2
+   
\int_0^T 
{\mathbb E}\left[
\exp \left( \vartheta \vert \ell^k_s \vert 
\right)
{\mathbb E}\left[ 
{\mathds 1}_A(s,\cdot) \vert {\mathcal F}_s
\right]
\right]
\dd s.
\end{split}
\end{equation*}
By Cauchy-Schwarz' and then Jensen's inequalities, 
\begin{equation*}
\begin{split}
&{\mathbb E} \left[ 
\int_0^T  q^k_s \vert \ell^k_s \vert 
{\mathds 1}_A  \dd s \right]
\\
&\leq \frac{\epsilon}2
+
\left[\left({\rm Leb}_{[0,T]}\otimes {\mathbb P}\right)(A)\right]^{1/2}
{\mathbb E}\left[ \int_0^T
\exp \left(2\vartheta  \vert \ell^k_s \vert   \right) \dd s
\right]^{1/2}
\\
&\leq \frac{\epsilon}2
+ \sqrt{T}
\left[\left({\rm Leb}_{[0,T]}\otimes {\mathbb P}\right)(A)\right]^{1/2}
{\mathbb E}\left[ 
\exp \left(2\vartheta  \int_0^T \vert \ell^k_s \vert    \dd s \right)
\right]^{1/2},  
\end{split}
\end{equation*}
and we conclude by invoking 
\eqref{eq:hypothese:UI:moment:exponentiel}.
\end{proof}

\begin{lemma} \label{lemma:weak-convergence-q} Let $(q^k)_{k\in \mathbb{N}} \in \mathcal{Q}$ be a sequence, weakly converging to $q$ for the $\sigma(L^1,L^\infty)$ topology. For $(\xi,\ell) \in L^1(\mathcal{F}_T) \times L^1(\mathbb{F})$, assume further that the sequences
$(q_T^k \xi)_{k \in \mathbb{N}}$ and $( q^k_s \ell_s)_{k \in \mathbb{N}}$ 
are uniformly integrable. Then,
\begin{equation}
    \lim_{k \to +\infty} \mathbb{E}\left[q^k_T \xi + \int_0^T q^k_s \ell_s \dd s \right] = \mathbb{E}\left[q_T \xi + \int_0^T q_s \ell_s \dd s \right].
\end{equation}
\end{lemma}

\begin{proof}
    We only show the convergence of $(q^k_s \ell_s)_{k \in \mathbb{N}}$, the proof for $(q^k_{T} \xi)_{k \in \mathbb{N}}$ being analogous. Let $(\tilde{\xi},\tilde{\ell}) \in L^\infty(\mathcal{F}_T) \times L^\infty(\mathbb{F})$. By weak convergence, we have 
    \begin{equation} \label{eq:weak-convergence-infty}
        \lim_{k \to +\infty} \mathbb{E}\left[\int_0^T q^k_s \tilde{\ell}_s \dd s \right] = \mathbb{E}\left[\int_0^T q_s \tilde{\ell}_s \dd s \right].
    \end{equation}
    Now, by uniform integrability (see Lemma \ref{lem:uniform:integrability}), we also have    \begin{equation*}
        \lim_{a \to +\infty} R^a = 0, \quad \textrm{\rm with} \quad R^a \coloneqq \sup_{k \in \mathbb{N}} \mathbb{E} \left[\int_0^T q^k_s \ell_s \mathds{1}_{\{|\ell_s| \geq a\}} \dd s\right].
    \end{equation*}
    With this notation, we have, for any $a >0$, 
    \begin{equation}
       \label{eq:weak-convergence-infty:266} \limsup_{k \to +\infty} \mathbb{E} \left[\int_0^T q^k_s \ell_s \dd s\right] \leq \limsup_{k \to +\infty} \mathbb{E} \left[\int_0^T q^k_s \ell_s \mathds{1}_{\{|\ell_s| \leq a\}} \dd s\right] + R^a.
    \end{equation}
    Here, we can use \eqref{eq:weak-convergence-infty}
    in order to identify the superior limit in the right-hand side. And then,
    letting $a \to +\infty$, we deduce from \eqref{eq:weak-convergence-infty:266}:
        \begin{equation} \label{ineq:limsup}
            \limsup_{k \to +\infty}  \mathbb{E} \left[\int_0^T q^k_s \ell_s \dd s \right] \leq   \mathbb{E} \left[\int_0^T q_s \ell_s \dd s \right].
        \end{equation}
Changing $\ell$ into $-\ell$, 
we also have 
\begin{equation} \label{ineq:liminf}
            \liminf_{k \to +\infty}  \mathbb{E} \left[\int_0^T q^k_s \ell_s \dd s\right]  \geq  \mathbb{E} \left[\int_0^T q_s \ell_s \dd s\right].
        \end{equation}
        Combining \eqref{ineq:limsup} and  \eqref{ineq:liminf} yields that 
        \begin{equation*}
            \lim_{k \to +\infty} \mathbb{E}\left[\int_0^T q^k_s \ell_s \dd s \right] = \mathbb{E}\left[\int_0^T q_s \ell_s \dd s \right],
        \end{equation*}
        which concludes the proof
    \end{proof}

\section{Distance and differentiability on spaces of non-negative measures}
\label{se:Mp}

\subsection{Generalized Wasserstein distance}
\label{se:app:gen:wasserstein}

We here establish the equivalence between 
the notion of continuity used in Section \ref{sec:mean-field-control}
and the notion of generalized 
$p$-Wasserstein distance introduced in \cite{PiccoliRossi}
(see also 
\cite{ChizatPeyreSchmitzerVialard}). 
We recall the following 
definition (using the notations introduced in 
Section \ref{sec:mean-field-control}, in particular the distances $d_p$ and $W_p$): 
\begin{definition}
Let 
$p \geq 1$ and $\mu,\nu
\in {\mathcal M}_p({\mathbb R}^n)$. 
We call generalized 
$p$-Wasserstein distance between $\mu$ and $\nu$ the quantity
\begin{equation*}
W_{p,\textrm{\rm ext}}(\mu,\nu) 
\coloneqq
\inf_{\tilde \mu,\tilde \nu \in {\mathcal M}_p({\mathbb R}^n), \; 
\tilde{\mu}({\mathbb R}^n)
= 
\tilde{\nu}({\mathbb R}^n)
}
\left\{ W_p(\tilde \mu,\tilde \nu)
+ \| \mu - \tilde \mu\|_{\rm TV}
+ \| \nu -\tilde \nu \|_{\rm TV}
\right\}. 
\end{equation*}
\end{definition}
The fact that 
$W_{p,\textrm{\rm ext}}$
is a distance 
is established in 
\cite[Proposition 1]{PiccoliRossi}.
We state below the main result of this section. Following 
\ref{appendix:uniform-integrability}, we recall 
that a subset 
$E \subset {\mathcal M}_p({\mathbb R}^n)$ is said to be $p$-uniformly integrable if 
\begin{equation*}
\sup_{\mu \in E}
\int_{{\mathbb R}^n}
(1+ \vert x\vert^p) \dd \mu(x) < + \infty, 
\quad 
\lim_{a \rightarrow + \infty}\sup_{\mu \in E}
\int_{{\mathbb R}^n}
{\mathbf 1}_{\{ \vert x \vert \geq a\}}
\vert x\vert^p \dd \mu(x)=0.
\end{equation*}

We claim
\begin{proposition}
\label{prop:gen:wasserstein}
For a given 
$p \geq 1$, 
let $\psi : {\mathcal M}_p({\mathbb R}^n) \rightarrow {\mathbb R}$. 
Then, 
$\psi$ is continuous 
with respect to $W_{p,\textrm{\rm ext}}$ on any subset of $p$-uniform integrability, if and only the following two properties hold:
\begin{enumerate}
\item 
$\psi$ is continuous with respect to $d_p$, uniformly on subsets of 
$p$-uniform integrability;
\item on any isomass subset of ${\mathcal M}_p({\mathbb R}^n)$, 
$\psi$ is continuous with respect to $W_p$. 
\end{enumerate}
\end{proposition}
Pay attention that continuity of $\psi$ is just restricted to subsets that 
are $p$-uniformly integrable: equivalently, we  require that
$\psi(\mu_m) \rightarrow \psi(\mu)$ for any sequence $(\mu_m)_{m \geq 1}$ that 
converges to $\mu$ in ${\mathcal M}_p({\mathbb R}^n)$ with respect to $W_{p,{\rm ext}}$ and that is $p$-uniformly integrable (recall that convergence in 
${\mathcal M}_p({\mathbb R}^n)$ with respect to $W_{p,{\rm ext}}$ does not guarantee $p$-uniform integrability). 

\begin{proof}
We first prove the implication (direct sense). We thus assume that 
$\psi$ is continuous with respect to 
$W_{p,\textrm{\rm ext}}$ on any subset of $p$-uniform integrability. 
Obviously, any sequence that converges with respect to $W_p$ on an isomass subset of 
${\mathcal M}_p({\mathbb R}^n)$ 
is $p$-uniformly integrable and converges with respect to 
$W_{p,\textrm{\rm ext}}$. Therefore, $\psi$ is  continuous with respect to 
$W_p$ on any isomass subset. This is item 2 in the statement. 
In order to prove item 1,
consider a sequence 
$(\mu_m)_{m \geq 1}$ that converges to some limit $\mu$, in 
${\mathcal M}_p({\mathbb R}^n)$ equipped with $d_p$. 
Clearly, it is  
$p$-uniformly integrable. Moreover, by \cite[Theorem 3]{PiccoliRossi},  $(\mu_m)_{m \geq 1}$ converges to $\mu$ with respect to $W_{p,\textrm{\rm ext}}$. By continuity of $\psi$ with respect to $W_{p,\textrm{\rm ext}}$, this shows that $\psi(\mu_m) \rightarrow \psi(\mu)$
as $m \rightarrow + \infty$. 
Continuity is uniform on any subset of $p$-uniform integrability. This follows from a standard compactness argument, as any subset of $p$-uniform integrability is 
relatively compact for $W_{p,\textrm{\rm ext}}$. 

We now establish the converse, assuming that 
$\psi$
satisfies items 1 and 2
in the statement. We thus
consider a
$p$-uniformly integrable sequence 
$(\mu_m)_{m \geq 1}$ that converges to some limit 
$\mu$, in 
${\mathcal M}_p({\mathbb R}^n)$ equipped with ${\mathcal W}_{p,\textrm{\rm ext}}$.

If $\mu({\mathbb R}^n) =0$, then 
\cite[Theorem 4]{PiccoliRossi}
says that $(\mu^m)_{m\geq 1}$ converges to the null measure in ${\mathcal M}_p({\mathbb R}^n)$. By item 1 in the statement, we deduce that 
$\psi(\mu^m) \rightarrow \psi(\mu)$ as
$m \rightarrow + \infty$, as expected. 

We thus assume that
$\mu({\mathbb R}^n) >0$. 
By definition of $
W_{p,\textrm{\rm ext}}$, we can find two sequences 
$(\tilde \mu^m)_{m \geq 1}$ and 
$(\tilde \nu^m)_{m \geq 1}$ such that, for each 
$m \geq 1$, $\tilde \mu^m$ is dominated by 
$\mu^m$, $\tilde \nu^m$ is dominated by 
$\mu$, 
and $\tilde \mu^m({\mathbb R}^n) 
= \tilde \nu^m({\mathbb R}^n)$,
and
\begin{equation*}
\lim_{m \rightarrow \infty}
\left[ W_p(\tilde \mu^m,\tilde \nu^m) + \| \mu^m - \tilde \mu^m \|_{\rm TV} + \| \mu - \tilde \nu^m \|_{\rm TV}\right]
=0.
\end{equation*}
In particular, 
$\mu^m({\mathbb R}^n) 
\rightarrow \mu({\mathbb R}^n)$. 
Therefore, without any loss of generality, we can assume that 
$\mu^m({\mathbb R}^n)>0$ for any $m \geq 1$, which makes it possible to let 
\begin{equation*}
    \bar \mu^m(\cdot) = \frac{\mu({\mathbb R^n})}{\tilde{\mu}^m({\mathbb R}^n)} \tilde{\mu}^m(\cdot), 
    \quad 
     \bar \nu^m(\cdot) = \frac{\mu({\mathbb R^n})}{\tilde{\mu}^m({\mathbb R}^n)} \tilde{\nu}^m(\cdot).
\end{equation*}
It is easy to see that 
$W_p(\bar \mu^m,\bar \nu^m)$ tends to $0$ as
$m \rightarrow + \infty$. Moreover, 
the sequence 
$(\bar \mu^m)_{m \geq 1}$ is $p$-uniformly integrable
because
$(\mu^m)_{m \geq 1}$ is $p$-uniformly integrable, and 
each $\bar \mu^m$ is dominated by 
$C \mu^m$, for a constant $C$ independent of $m$. Obviously,
$(\bar \nu^m)_{m \geq 1}$ is also $p$-uniformly integrable (because 
$\bar \nu^m$
is dominated by 
$C \mu$, for a possibly different value of $C$, but still independent of $m$). Since 
$\psi$ is $W_p$-continuous on the subset of measures with constant mass equal to 
$\mu({\mathbb R}^n)$, it is in particular equi-continuous on any subset of $p$-uniformly integrable measures with constant mass equal to $\mu({\mathbb R}^n)$. 
Therefore, by item 2, 
\begin{equation*}
\lim_{m \rightarrow + \infty}
\vert \psi(\bar \mu^m) - 
\psi ( \bar \nu^m) \vert 
=0.
\end{equation*}
It remains to see that 
$\| \mu^m - \bar \mu^m \|_{\rm TV} \rightarrow 0$
as $m \rightarrow + \infty$. Since the two sequences 
$(\mu^m)_{m \geq 1}$ and 
$(\bar \mu^m)_{m \geq 1}$ are $p$-uniformly integrable, we deduce that 
$d_p(\mu^m,\tilde \mu^m) \rightarrow 0$ as $m$ tends to $+ \infty$. 
By item 1 (using the fact that continuous is uniform on subsets of uniform $p$-integrability), we deduce that 
\begin{equation*}
\lim_{m \rightarrow + \infty}
\vert \psi( \mu^m) - 
\psi ( \bar \mu^m) \vert 
=0.
\end{equation*}
Similarly, 
\begin{equation*}
\lim_{m \rightarrow + \infty}
\vert \psi( \mu) - 
\psi ( \bar \nu^m) \vert 
=0.
\end{equation*}
By combining the last three displays, we complete the proof. 
\end{proof}

\subsection{Differentiability}
\label{subse:differentiability:Mp}

The purpose of this subsection is to prove Lemma 
\ref{lem:differentiation:flat:lions}. 

\begin{proof}[Proof of Lemma  
\ref{lem:differentiation:flat:lions}.]
To simplify, we prove the result assuming that \ref{assumption:g} holds true without any restriction on the mass of $\mu$. 

We first establish 
\eqref{eq:derivatives:flat:expansion}.
Given $\mu \in {\mathcal M}_{2-r}({\mathbb R}^n)$, we deduce from 
\eqref{eq:flat:derivative:def}
(together with the continuity of the derivative in $d_{2-r}$) that, for any 
$t >0$ and any 
$x \in {\mathbb R}^n$,
\begin{equation*}
G\left( \mu + t \delta_x \right) 
=
t \int_0^1 \frac{\delta G}{\delta \mu}\left ( \mu + \theta t \delta _x, x \right) \dd \theta. 
\end{equation*}
And then, 
for any integer $\ell \geq 1$, any $t_1,\ldots,t_\ell >0$ and any 
$x_1,\ldots,x_\ell \in {\mathbb R}^n$, 
\begin{equation*}
\begin{split}
&G\left( \mu 
+ \sum_{i=1}^\ell t_i \delta_{x_i} \right) 
- G(\mu) 
\\
&= 
\sum_{i=1}^\ell
\left[
G\left( \mu 
+ \sum_{j=1}^i t_j \delta_{x_j} \right) 
- G\left( \mu 
+ \sum_{j=1}^{i-1} t_j \delta_{x_j} \right) 
\right]
\\
&= 
\sum_{i=1}^\ell
t_i \int_0^1 \frac{\delta G}{\delta \mu}
\left( \mu 
+ \sum_{j=1}^{i-1} t_j \delta_{x_j}
+ \theta t_i \delta_{x_i},x_i
\right) \dd \theta
\\
&= \int_0^1 \left[ 
\int_{{\mathbb R}^n} \frac{\delta G}{\delta \mu}
\left( \mu 
+ \sum_{j=1}^{i-1} t_j \delta_{x_j}
+ \theta t_i \delta_{x_i},y
\right) 
\dd \left( 
\sum_{i=1}^\ell t_i \delta_{x_i}\right)(y)
\right] \dd \theta.
\end{split}
\end{equation*}
By continuity of $\delta G/\delta \mu$
in the measure argument, we deduce that 
\begin{equation*}
\frac{\dd}{\dd \varepsilon}\vert_{\varepsilon =0+}
G\left (\mu 
+ \varepsilon 
\sum_{i=1}^\ell t_i \delta_{x_i}
\right) 
=
\int_{{\mathbb R}^n} \frac{\delta G}{\delta \mu}
\left( \mu,y
\right) 
\dd \left( 
\sum_{i=1}^\ell t_i \delta_{x_i}\right)(y).
\end{equation*}
And then, 
\begin{equation*}
\begin{split}
&G\left (\mu 
+   
\sum_{i=1}^\ell t_i \delta_{x_i}
\right) 
- G(\mu) 
\\
&=
\int_0^1
\left[
\int_{{\mathbb R}^n} \frac{\delta G}{\delta \mu}
\left( \mu
+ \theta  
\sum_{i=1}^\ell t_i \delta_{x_i}
,y
\right) 
\dd \left( 
\sum_{i=1}^\ell t_i \delta_{x_i}\right)(y)
\right] \dd \theta.
\end{split}
\end{equation*}
Now, 
we can approximate
any given
$\nu \in {\mathcal M}_{2-r}({\mathbb R}^n)$
by 
measures $(\nu^\ell)_{\ell \geq 1}$ with the same mass, but with each being supported by a finite set; the approximation holds true with respect to 
$W_{2-r}$.
For each 
$\ell \geq 1$, 
we have 
\begin{equation}
\label{eq:Gmu+varepsilonnuell}
G(\mu +   \nu^\ell) - G(\mu) 
= \int_0^1
\left[ 
\int_{{\mathbb R}^n}
\frac{\delta G}{\delta \mu}\left( \mu + 
\theta   
\nu^\ell, y \right) 
\dd \nu^\ell(y) \right] 
\dd \theta.
\end{equation}
We denote by $\pi^{\ell}$ an optimal coupling between 
$\nu$ and
$
\nu^\ell$. Using the third line in \eqref{ass:growth-G-mean-field}, we have 
\begin{equation*}
\begin{split}
&\left\vert \int_{{\mathbb R}^n} \frac{\delta G}{\delta \mu}
\left( \mu
+ \theta    \nu^\ell
,y
\right) 
\dd \left( 
\nu^\ell - \nu\right)(y)
\right\vert
\\
&\leq \int_{{\mathbb R}^n \times {\mathbb R}^n} \left\vert  \frac{\delta G}{\delta \mu}
\left( \mu
+ \theta   
\nu^\ell
,y
\right) 
-
\frac{\delta G}{\delta \mu}
\left( \mu
+ \theta \nu^\ell
,z
\right) \right\vert
\dd \pi^\ell(y,z)
\\
&\leq C \left( 1 + M_{2-r}(\mu+\nu^\ell) 
\right)
\int_{{\mathbb R}^n \times {\mathbb R}^n}
(1 + \vert y\vert^{1-r}
+ \vert z \vert^{1-r})
\vert y-z\vert 
\dd \pi^\ell(y,z).
\end{split}
\end{equation*}
Observing that the moments 
$(M_{2-r}(\nu^\ell))_{\ell \geq 1}$ are uniformly bounded (because the convergence holds true with respect to $W_{2-r}$) and using Cauchy-Schwarz inequality to handle the last term in the right-hand side when $r=0$, we deduce that the left-hand side in the above display tends to $0$ as $\ell$ tends to $+\infty$. 

Using the continuity of $\delta G/\delta \mu$ 
in $\mu$ with respect to $W_{2-r}$ on isomass subsets, and the growth condition \eqref{ass:growth-G-mean-field}, we can pass to the limit in 
\eqref{eq:Gmu+varepsilonnuell}. We get 
\begin{equation*}
G(\mu +   \nu) - G(\mu) 
= \int_0^1
\left[ 
\int_{{\mathbb R}^n}
\frac{\delta G}{\delta \mu}\left( \mu + 
\theta  
\nu, y \right) 
\dd \nu(y) \right] 
\dd \theta.
\end{equation*}
And then, for any $\varepsilon \in [0,1]$, 
\begin{equation*}
\begin{split}
G\left((1-\varepsilon) \mu + \varepsilon   \nu\right) - G\left((1-\varepsilon)  \mu\right) 
&= \varepsilon \int_0^1
\left[ 
\int_{{\mathbb R}^n}
\frac{\delta G}{\delta \mu}\left( (1-\varepsilon) \mu +
\varepsilon 
\theta  
\nu, y \right) 
\dd \nu(y) \right] 
\dd \theta
\\
&= \varepsilon \int_0^1
\left[ 
\int_{{\mathbb R}^n}
\frac{\delta G}{\delta \mu}\left( \mu, y \right) 
\dd \nu(y) \right] 
\dd \theta + o(\varepsilon),
\end{split}
\end{equation*}
where $o(\varepsilon)/\varepsilon\rightarrow 0$ as $\varepsilon$ tends to $0$, 
with the last line following from \eqref{ass:unif-flat-dG-mean-field}.
Performing a similar expansion for $\nu=\mu$, we obtain 
\begin{equation*}
\frac{\dd}{\dd \varepsilon}\vert_{\varepsilon = 0+}
G\left((1-\varepsilon) \mu + \varepsilon   \nu\right)=\int_{{\mathbb R}^n}
\frac{\delta G}{\delta \mu}\left( \mu, y \right) 
\dd \left[ \nu - \mu \right](y), 
\end{equation*}
from which we deduce that 
\begin{equation*}
G(\nu) - G(\mu)
= \int_0^1
\left[ \int_{{\mathbb R}^n}
\frac{\delta G}{\delta \mu}\left( (1-\theta) \mu + \theta \nu , y \right) 
\dd \left[ \nu - \mu \right](y) \right]
\dd \theta.
\end{equation*}
Choosing $\nu=(q' {\mathbb P})_X$ and 
$\mu = (q {\mathbb P}_X)$, this completes the proof of 
\eqref{eq:derivatives:flat:expansion}. 

It remains to prove 
\eqref{eq:derivatives:lions:expansion}. Generally speaking, it is a  consequence of \cite[Proposition 5.44]{carmona2018probabilistic-v1}, applied on the space 
$(\Omega,{\mathcal F},q{\mathbb P}/{\mathbb E}(q))$. Indeed, 
following
Remark \ref{rem:G(c)}, we  can apply 
\cite[Proposition 5.44]{carmona2018probabilistic-v1}
to the function 
$X \in L^2(\Omega,{\mathcal F}, q {\mathbb P}/c)
\mapsto G^{(c)}((q {\mathbb P}/c)_X)$, where 
$c\coloneqq {\mathbb E}[q]$
and 
$G^{(c)}(\mu)=G(c\mu)$. We obtain 
\eqref{eq:derivatives:lions:expansion}
when $X$ satisfies 
${\mathbb E}[q \vert X \vert^2] < + \infty$. 
When $X$ is just in 
$L^1(\Omega,{\mathcal F},q {\mathbb P})$ (which is the case when $r=1$), we can approximate it, in 
$L^1(q {\mathbb P})$, by a sequence in 
$L^2(\Omega,{\mathcal F},q {\mathbb P})$;
we then apply
\eqref{eq:derivatives:lions:expansion}
to the approximating subsequence and then
pass to the limit  using 
\eqref{ass:growth-G-mean-field} and \eqref{ass:unif-Lions-dG-mean-field}.
\end{proof}
\color{black}

\bibliographystyle{abbrv}
\bibliography{biblio}

\end{document}